%% file: main.tex
\documentclass[11pt]{amsart}

\usepackage{amsmath}
\usepackage{amssymb}
\usepackage{amsthm}
\usepackage{mathtools}
\usepackage[dvipsnames]{xcolor}
\usepackage[pdfusetitle,colorlinks=true,linkcolor=MidnightBlue,citecolor=ForestGreen,urlcolor=BrickRed]{hyperref}
\usepackage{pdfpages}
\usepackage{eso-pic}
\usepackage{bookmark}
\input{supplement_navigation.tex}
\AtBeginDocument{\bookmark[level=0,dest={Doc-Start}]{Main article}}

\allowdisplaybreaks
\theoremstyle{plain}
\newtheorem{thm}{Theorem}[section]
\newtheorem{lem}[thm]{Lemma}
\newtheorem{prop}[thm]{Proposition}
\newtheorem{cor}[thm]{Corollary}

\newcommand{\supplementproof}[2]{\par\smallskip\noindent\emph{Proof.}
See {\hypersetup{linkcolor=Cyan}\hyperlink{#1}{Supplementary #2}}.\par}

\theoremstyle{definition}

\theoremstyle{remark}

\newcommand{\R}{\mathbb{R}}
\newcommand{\N}{\mathbb{N}}
\newcommand{\E}{\mathbb{E}}
\newcommand{\Prob}{\mathbb{P}}
\newcommand{\ind}{\mathbf{1}}
\newcommand{\dd}{\,\mathrm{d}}
\newcommand{\ii}{\mathrm{i}}
\newcommand{\ee}{\mathrm{e}}
\newcommand{\abs}[1]{\lvert #1\rvert}
\newcommand{\norm}[1]{\lVert #1\rVert}
\newcommand{\set}[1]{\left\{#1\right\}}
\newcommand{\op}{\mathrm{op}}

\DeclareMathOperator{\Tr}{Tr}
\DeclareMathOperator{\rank}{rank}
\DeclareMathOperator{\diag}{diag}

\DeclareMathOperator{\spec}{spec}

\DeclareMathOperator{\Var}{Var}

\DeclareMathOperator{\Ai}{Ai}

\DeclareMathOperator{\cl}{cl}

\input{external_inputs.tex}
\makeatletter
\newcommand{\aliasref}[2]{%
  \AtBeginDocument{%
    \@ifundefined{r@#2}{}{%
      \expandafter\let\expandafter\NCLtmpref\csname r@#2\endcsname
      \expandafter\global\expandafter\let\csname r@#1\endcsname\NCLtmpref}}}
\makeatother
\aliasref{eq:coordinate-Y}{eq:direction-integral}
\aliasref{eq:low-homogeneous-ou-flow}{eq:recursion-flow}
\aliasref{eq:comparison-ou-flow}{eq:recursion-flow}
\aliasref{eq:comparison-edge-window}{eq:recursion-endpoints}
\aliasref{eq:comparison-opened-monomial}{eq:recursion-monomial}
\aliasref{eq:comparison-column-self-coefficient}{eq:recursion-row-self}
\aliasref{eq:comparison-row-self-coefficient}{eq:recursion-column-self}
\aliasref{eq:comparison-column-pair-coefficient}{eq:recursion-row-neutral}
\aliasref{eq:comparison-row-pair-coefficient}{eq:recursion-column-neutral}
\aliasref{eq:comparison-third-cumulant-normalization}{eq:recursion-cubic-normalization}
\aliasref{eq:comparison-old-label-error}{eq:recursion-old-cubic}
\aliasref{eq:high-recursion-zero}{eq:recursion-h0}
\aliasref{eq:comparison-recursion-unflagged}{eq:recursion-h0}
\aliasref{eq:high-recursion-one}{eq:recursion-h1}
\aliasref{eq:comparison-recursion-flagged}{eq:recursion-h1}
\aliasref{eq:high-exact-third-source}{eq:recursion-third-source}
\aliasref{eq:comparison-third-cumulant-source}{eq:recursion-third-source}
\aliasref{eq:low-homogeneous-error-envelope}{eq:high-derivative-envelope}
\aliasref{eq:low-homogeneous-integrated-error}{eq:high-integrated-envelope}
\aliasref{eq:low-normalization}{eq:main-normalization-identities}
\aliasref{eq:low-good-sum}{eq:high-good-sum}
\aliasref{eq:low-conditional}{eq:high-conditional}
\begin{document}

\title[Fractional Logarithm for Nested Covariance Matrices]
{A Law of Fractional Logarithm for Nested Complex Sample Covariance Matrices}
\author{Xiufan Yang}
\address{School of Science, Nanjing University of Posts and Telecommunications, Nanjing 210023, China}
\email{B25100020@njupt.edu.cn}

\hypersetup{
  pdftitle={A Law of Fractional Logarithm for Nested Complex Sample Covariance Matrices},
  pdfauthor={Xiufan Yang},
  pdfsubject={Almost-sure soft-edge extremes for nested complex sample covariance matrices},
  pdfkeywords={sample covariance matrices, largest eigenvalue, law of fractional logarithm, Tracy-Widom fluctuations, nested random matrices}
}

\date{}

\begin{abstract}
We prove a law of fractional logarithm for the largest eigenvalue along a
northwest-nested path of complex sample covariance matrices from one infinite
array. The entries are independent and centered, with unit variance,
vanishing complex second moment, and uniformly bounded moments of every
fixed order. The row dimension is nondecreasing, has bounded increments,
and has a positive limiting aspect ratio. After finite-size edge centering
and scaling, the almost-sure \mbox{limsup} on the $(\log N)^{2/3}$ scale is
$(1/4)^{2/3}$, and the \mbox{liminf} on the $(\log N)^{1/3}$ scale is $-4^{1/3}$.
The corresponding cluster sets in $\R$ are $[0,(1/4)^{2/3}]$ and
$[-4^{1/3},\infty)$. The proof compares the Laplace transform of a single smoothed count over
a growing grid of full nested matrices with its Gaussian counterpart.
Gaussian count concentration gives block occurrences with probability
tending to one. Dyadic tail bounds yield the endpoints, and deterministic
interpolation gives the cluster sets.
\end{abstract}

\maketitle

\section{Introduction}\label{sec:intro}

We study the almost-sure logarithmic fluctuations and cluster sets of the
largest eigenvalue of northwest-nested sample covariance matrices. All
matrix sizes are restrictions of one infinite array. Their almost-sure
extremes therefore depend on the joint law of the sequence, while
fixed-size edge limits describe its individual marginals. Our comparison
acts on one count of excursions over a growing grid and preserves the
shared entries throughout the interpolation.

Paquette and Zeitouni introduced the fractional-logarithm problem for the
GUE minor process \cite{PZ17}. Baslingker, Basu, Bhattacharjee, and
Krishnapur proved its lower endpoint \cite{BBBK25LFL}. Bao, Cipolloni,
Erd\H{o}s, Henheik, and Kolupaiev established both endpoints and cluster
sets for general Wigner minors \cite[Theorem 1.2 and Corollary 1.3]{BCEHK25}.
The complex covariance path has the same constants. The result here
concerns independent-entry rectangular arrays with a nondecreasing row
path, bounded increments, and a positive limiting aspect ratio.
The covariance linearization has deterministic zero diagonal blocks;
its variance pattern differs from the strictly positive profile in the
Wigner-type extension described in \cite[Remark 1.4]{BCEHK25}.

The two logarithmic scales reflect the complex soft-edge tail exponents
$4x^{3/2}/3$ and $x^3/12$ \cite{BBBK24}. Heuristically, the $N^{2/3}$
correlation scale gives $N^{1/3}$ effective tests in a block of $N$ levels
\cite{FN11,BCEHK25Dec}. Balancing the tail exponents against this count
gives $\frac43a^{3/2}=\frac13$ and $b^3/12=\frac13$, hence
$a=(1/4)^{2/3}$ and $b=4^{1/3}$. The proof combines concentrating Gaussian
counts at these thresholds with a comparison of their smoothed analogues.

The Gaussian input is supplied by Paper I
\cite[Corollary~\ref*{cor:gaussian-grid-moments}]{YangWishart}.
On separated grids it gives a diverging first moment and a second moment
asymptotic to its square, uniformly under bounded deterministic threshold
shifts. Its product-scale estimates are established in that paper.
Quantitative fixed-size limits and covariance edge universality appear in
\cite{EK06,Ma12,PY14,SX21}; Gaussian Wishart and Laguerre growth processes
are treated in \cite{FN11,DW09,AVMW13,FF10}, with related hard-edge Bessel
fields in \cite{BHW26}.

Schnelli and Xu already develop unmatched-index expansions for sample
covariance matrices, including their use with a smooth scalar observable
\cite[Definition 4.2, Proposition 4.3, and Section 7]{SX21}. We extend that
calculus to the derivative tensors of one global cutoff count on a
polynomially growing matrix family. Lemma~\ref{lem:all-order-ladder}
sums all matrix and cutoff indices, and
Theorem~\ref{thm:four-orientation-recursion} keeps those tensors with their
resolvent factors inside a joint expectation. Together with the covariance
local laws \cite{BEKYY14,KY17}, this yields the full-array Laplace
comparison in Proposition~\ref{prop:full-grid-comparison}.
Its application to Gaussian counts is the non-Gaussian transfer used
for Theorem~\ref{thm:main}.

We also prove conditional versions with a revealed northwest corner.
For the left tail, Gaussian interpolation removes the two cross strips
and reaches a block-diagonal reference family; the corresponding
conditional comparison has a related precedent in \cite{BLX25}.
These results describe the additional effect of retaining a prescribed
history. The main theorem uses the full-matrix comparison with an empty
old corner.

\subsection*{Main theorem}

\begin{thm}[Law of fractional logarithm and cluster sets]\label{thm:main}
Let $(x_{ij})_{i,j\geq 1}$ be independent complex random variables satisfying
\[
 \E x_{ij}=0,\qquad
 \E\abs{x_{ij}}^2=1,\qquad
 \E(x_{ij})^2=0,
\]
and, for every fixed $p\geq 1$,
\[
 \sup_{i,j\geq 1}\E\abs{x_{ij}}^p<\infty.
\]
Let $\gamma>0$, and let $(M_N)_{N\geq 1}$ be a nondecreasing sequence of positive integers such that
\[
 \frac{M_N}{N}\longrightarrow\gamma,
 \qquad
 K:=\sup_{N\geq 1}(M_{N+1}-M_N)<\infty.
\]
For every $N\geq 1$, define
\[
 X^{(N)}=(x_{ij})_{\substack{1\leq i\leq M_N\\1\leq j\leq N}},
 \qquad
 Q_N=\frac{1}{N}X^{(N)}(X^{(N)})^*,
 \qquad
 \lambda_1^{(N)}=\lambda_{\max}(Q_N),
\]
and
\[
 \mu_{+,N}=\frac{(\sqrt{M_N}+\sqrt N)^2}{N},
 \qquad
 \sigma_{+,N}
 =\frac{\sqrt{M_N}+\sqrt N}{N}
   \left(M_N^{-1/2}+N^{-1/2}\right)^{1/3},
\]
\[
 \chi_N^+
 =\frac{\lambda_1^{(N)}-\mu_{+,N}}{\sigma_{+,N}}.
\]
Put
\[
 A_2=\left(\frac14\right)^{2/3},
 \qquad
 B_2=4^{1/3},
 \qquad
 r_N=(\log N)^{2/3},
 \qquad
 s_N=(\log N)^{1/3}
\]
for $N\geq 3$. Then, with probability one,
\[
 \limsup_{N\to\infty}\frac{\chi_N^+}{r_N}=A_2,
 \qquad
 \liminf_{N\to\infty}\frac{\chi_N^+}{s_N}=-B_2,
\]
and
\[
 \bigcap_{m=3}^{\infty}
 \cl_{\R}\set{\frac{\chi_N^+}{r_N}\mid N\geq m}
 =[0,A_2],
\]
\[
 \bigcap_{m=3}^{\infty}
 \cl_{\R}\set{\frac{\chi_N^+}{s_N}\mid N\geq m}
 =[-B_2,\infty).
\]
\end{thm}

Both limits concern the largest eigenvalue at the upper spectral edge.
Throughout the applications, the array, path, and edge variables are those
of Theorem~\ref{thm:main}, unless an auxiliary statement gives its own
assumptions.

\paragraph{Comparison of the whole count.}
For a grid $\mathcal T_N$ of at most $N^\theta$ sizes, $\theta<1/3$,
Proposition~\ref{prop:full-grid-comparison} constructs a nonnegative
smoothed count $S_N$ and proves
\[
 \sup_{0\le\tau\le\tau_*}
 \big|\E e^{-\tau S_N(\mathbf X)}-\E e^{-\tau S_N(\mathbf W)}\big|
 \le C N^{-1/4}.
\]
The count is a function of the entire common array. The absolute sums of
its spectral derivative tensors cost $N^{o(1)}$, even as the number of
matrices increases. On the separated grids supplied by Paper I,
$S_N(\mathbf W)$ dominates a retained hit count $R_N$ with $m_N=\E R_N\to\infty$ and
$\Var(R_N)=o(m_N^2)$. Thus, for fixed $\tau>0$,
\[
 \E e^{-\tau S_N(\mathbf W)}
 \le 4\Var(R_N)/m_N^2+e^{-\tau m_N/2}\longrightarrow0.
\]
The support of $S_N$ then transfers a hit to the original array.
Corollary~\ref{cor:full-grid-occurrence} obtains block occurrence
probabilities tending to one for both tails. These probabilities give
infinitely many occurrences on the single original array, as shown in
Section~\ref{sec:main}.

\paragraph{Organization of the proof.}
Section~\ref{sec:setup} collects normalization and probability estimates.
Sections~\ref{sec:detectors}--\ref{sec:comparison} establish the joint
count comparison; Section~\ref{sec:full-grid-comparison} applies it to
full matrices. Sections~\ref{sec:high-occurrence} and
\ref{sec:low-occurrence} give the conditional refinements.
Sections~\ref{sec:fixed-level}--\ref{sec:recurrence} prove the summable
supercritical bounds, and Section~\ref{sec:main} assembles the endpoints
and cluster sets. Appendix~\ref{app:companion-estimates} records two
auxiliary Gaussian estimates. Appendix~\ref{app:recurrence-refinements}
collects the common-filtration construction for the conditional results.

\paragraph{Notation.}
We write $A^*$ and $A^{\mathsf T}$ for the adjoint and transpose of a
matrix $A$.

\section{Normalization and probabilistic preliminaries}\label{sec:setup}

A second-moment estimate converts the comparison of a joint count into
an occurrence bound. For the main theorem, the Gaussian counts concentrate
and the resulting block probabilities tend to one. We also record the
conditional probability tools used for prescribed histories.

\subsection{Edge normalization and submatrix comparison}

We use the singular-value normalization for smoothed spectral statistics
and coordinate restrictions. The covariance variable \(\chi_N^+\) agrees
exactly with the Gram variable \(\xi_N\), whose numerator contains the
monotone eigenvalue \(L_N\). This identity allows us to apply
Proposition~\ref{prop:tail-cluster-closure} to the endpoint limits for
\(\chi_N^+\).

\begin{lem}[Edge normalization]\label{lem:normalization-dictionary}
For positive $m,n$, write
\[
 a(m,n)=(\sqrt m+\sqrt n)^2,\qquad
 b(m,n)=(\sqrt m+\sqrt n)(m^{-1/2}+n^{-1/2})^{1/3}.
\]
\begin{enumerate}
\item\label{item:project-normalization}
For an $m$-by-$n$ matrix $X$, put $Q=n^{-1}XX^*$,
$\mu=a(m,n)/n$, $\sigma=b(m,n)/n$, and
$\chi=(\norm X_{\op}^2-a(m,n))/b(m,n)$. Then
\[
 \lambda_{\max}(Q)=n^{-1}\norm X_{\op}^2,\qquad
 (\lambda_{\max}(Q)-\mu)/\sigma=\chi.
\]
\item\label{item:gram-covariance-normalization}
For $X^{(N)}$ of size $M_N$-by-$N$, define
$L_N=\norm{X^{(N)}}_{\op}^2$, $a_N=a(M_N,N)$, and $b_N=b(M_N,N)$.
With the covariance normalization of Theorem~\ref{thm:main},
\[
 L_N=N\lambda_1^{(N)},\quad a_N=N\mu_{+,N},\quad
 b_N=N\sigma_{+,N},\quad
 \xi_N:=\frac{L_N-a_N}{b_N}=\chi_N^+.
\]
\item\label{item:finite-shifts}
Suppose $M_N/N\to\gamma>0$. For fixed $\delta>0$, a fixed positive
integer $J$, $r_N=(\log N)^{2/3}$, and $\ell_N=N^{-2/3-\delta}$,
\[
 \frac{J\ell_N}{\sigma_{+,N}r_N}\longrightarrow0.
\]
Hence, for fixed real numbers $A<q_*<q$ and large $N$,
$qr_N\sigma_{+,N}-J\ell_N\ge q_*r_N\sigma_{+,N}$.
\end{enumerate}
\end{lem}
\begin{proof}
The identity $\lambda_{\max}(XX^*)=\|X\|_{\op}^2$, followed by
scaling by $n^{-1}$, proves the first two assertions. With $\gamma_N=M_N/N$,
\[
 \sigma_{+,N}=N^{-2/3}(1+\sqrt{\gamma_N})^{4/3}\gamma_N^{-1/6}
 \asymp N^{-2/3}.
\]
Thus the ratio in the third assertion is $O(N^{-\delta}(\log N)^{-2/3})$.
Eventually it is at most $q-q_*$, which proves the last inequality.
\end{proof}

The two derivatives of the rectangular edge center correspond to the two
matrix updates. At aspect ratio \(\gamma=m/n\), the resolvent factor
matches \(\partial_nA(m,n)\) for column growth. The reciprocal aspect
ratio gives \(\partial_mA(m,n)\) for row growth.

\begin{lem}[Marchenko--Pastur transform]
\label{lem:edge-resolvent-derivatives}
Let \(\gamma>0\). Define
\[
 a_\gamma=(1-\sqrt\gamma)^2,\qquad
 b_\gamma=(1+\sqrt\gamma)^2,
\]
and, for \(z\geq b_\gamma\), define
\[
 g_\gamma(z)
 =
 \frac{z+\gamma-1-\sqrt{(z-a_\gamma)(z-b_\gamma)}}
      {2\gamma z}.
\]
Here $g_\gamma(z)=\int (z-x)^{-1}\nu_\gamma(\dd x)$, where
$\nu_\gamma$ is the row-side Marchenko--Pastur probability law,
including its atom at zero. Thus $g_\gamma$ has the opposite sign to
the Stieltjes-transform convention $\int(x-z)^{-1}\nu_\gamma(\dd x)$.
Then
\[
 \gamma g_\gamma(b_\gamma)
 =\frac{\sqrt\gamma}{1+\sqrt\gamma},
 \qquad
 \bigl[1-\gamma g_\gamma(b_\gamma)\bigr]^{-1}
 =1+\sqrt\gamma.
\]
Applying these formulas with \(\gamma\) replaced by \(\gamma^{-1}\) gives
\[
 \gamma^{-1}g_{\gamma^{-1}}(b_{\gamma^{-1}})
 =\frac1{1+\sqrt\gamma},
 \qquad
 \bigl[1-\gamma^{-1}g_{\gamma^{-1}}(b_{\gamma^{-1}})\bigr]^{-1}
 =1+\gamma^{-1/2}.
\]
If
\[
 A(m,n)=(\sqrt m+\sqrt n)^2,\qquad \gamma=\frac mn,
\]
for \(m,n>0\), then
\[
 \partial_nA(m,n)=1+\sqrt\gamma,
 \qquad
 \partial_mA(m,n)=1+\gamma^{-1/2}.
\]
\end{lem}

\supplementproof{supp.B.1}{Section B.1}

Deleting rows and columns can only decrease the operator norm.
Proposition~\ref{prop:coordinate-submatrix-transfer} expresses this
inequality in edge coordinates. When the number of deleted rows and
columns is \(o(N^{1/3}(\log N)^{2/3})\), the amplitude gap absorbs both
the change of scale and the loss in the edge center. A right-tail event
for the submatrix then implies a right-tail event for the full matrix.

\begin{prop}[Comparison with submatrices]
\label{prop:coordinate-submatrix-transfer}
Fix \(0<c_0<C_0<\infty\). For each positive integer \(N\), let \(m_N,n_N\) be positive integers satisfying
\[
 c_0N\leq m_N,n_N\leq C_0N.
\]
Let \(r_N,s_N\) be nonnegative integers such that
\[
 r_N<m_N,\qquad s_N<n_N,\qquad r_N+s_N=o(N),
\]
and let \(X_N\) be a complex \(m_N\)-by-\(n_N\) matrix. Let \(Y_N\) be any \((m_N-r_N)\)-by-\((n_N-s_N)\) coordinate submatrix obtained by deleting \(r_N\) rows and \(s_N\) columns. For \(m,n>0\), set
\[
 a(m,n)=(\sqrt m+\sqrt n)^2,
 \qquad
 b(m,n)=(\sqrt m+\sqrt n)
        (m^{-1/2}+n^{-1/2})^{1/3},
\]
and define
\[
 \chi_{X,N}
 =\frac{\norm{X_N}_{\op}^2-a(m_N,n_N)}{b(m_N,n_N)},
\]
\[
 \chi_{Y,N}
 =\frac{\norm{Y_N}_{\op}^2-a(m_N-r_N,n_N-s_N)}
        {b(m_N-r_N,n_N-s_N)}.
\]
Then
\begin{equation}\label{eq:submatrix-exact-transfer}
 \chi_{X,N}
 \geq
 \frac{b(m_N-r_N,n_N-s_N)}{b(m_N,n_N)}\chi_{Y,N}
 -
 \frac{a(m_N,n_N)-a(m_N-r_N,n_N-s_N)}
      {b(m_N,n_N)}.
\end{equation}
Uniformly under these dimension hypotheses,
\begin{equation}\label{eq:submatrix-center-loss}
 0\leq a(m_N,n_N)-a(m_N-r_N,n_N-s_N)
 \leq C(r_N+s_N),
\end{equation}
the quantity \(b(m_N,n_N)\) is bounded above and below by positive constants times \(N^{1/3}\), and
\begin{equation}\label{eq:submatrix-scale-ratio}
 \frac{b(m_N-r_N,n_N-s_N)}{b(m_N,n_N)}
 =1+O\left(\frac{r_N+s_N}{N}\right).
\end{equation}
The constants depend only on \(c_0,C_0\).

Suppose additionally that
\[
 r_N+s_N=o\bigl(N^{1/3}(\log N)^{2/3}\bigr).
\]
For fixed \(0<a_-<a_+\), all sufficiently large \(N\) satisfy
\begin{equation}\label{eq:submatrix-high-event}
 \chi_{Y,N}\geq a_+\bigl(\log(n_N-s_N)\bigr)^{2/3}
 \quad\Longrightarrow\quad
 \chi_{X,N}\geq a_-(\log n_N)^{2/3}.
\end{equation}
If \(Q_N=n_N^{-1}X_NX_N^*\) is equipped with the edge center and scale, then its edge-normalized largest eigenvalue is exactly \(\chi_{X,N}\). Thus \eqref{eq:submatrix-high-event} is a high-event implication in the edge normalization.
\end{prop}

\supplementproof{supp.G.1}{Section G.1}

\subsection{Nested matrices and cluster sets}

Northwest nesting makes the unnormalized largest eigenvalue monotone.
Bounded dimension increments control the changes in its center and scale,
so the normalized process has no asymptotically macroscopic downward
jumps. A crossing argument then fills the intervals between its endpoint
limits, giving Proposition~\ref{prop:tail-cluster-closure}.

\begin{prop}[Cluster sets from endpoint limits]
\label{prop:tail-cluster-closure}
Let \(\gamma>0\), and let \((M_N)_{N\geq1}\) be nondecreasing positive integers such that
\[
 \frac{M_N}{N}\longrightarrow\gamma,
 \qquad
 K:=\sup_N(M_{N+1}-M_N)<\infty.
\]
Let \(X^{(N)}\) be a nested sequence of complex \(M_N\)-by-\(N\) matrices, meaning that \(X^{(N)}\) is the upper-left \(M_N\)-by-\(N\) submatrix of \(X^{(N+1)}\). Put
\[
 L_N=\lambda_{\max}\bigl(X^{(N)}(X^{(N)})^*\bigr),
\]
\[
 a_N=(\sqrt{M_N}+\sqrt N)^2,\qquad
 b_N=(\sqrt{M_N}+\sqrt N)
 \left(\frac1{\sqrt{M_N}}+\frac1{\sqrt N}\right)^{1/3},
\]
and
\[
 \xi_N=\frac{L_N-a_N}{b_N},\qquad
 r_N=(\log N)^{2/3},\qquad
 s_N=(\log N)^{1/3}
\]
for \(N\geq2\). If \(A,B>0\) and
\[
 \limsup_{N\to\infty}\frac{\xi_N}{r_N}=A,
 \qquad
 \liminf_{N\to\infty}\frac{\xi_N}{s_N}=-B,
\]
then
\[
 \bigcap_{m\geq2}
 \cl_{\R}\set{\frac{\xi_N}{r_N}\mid N\geq m}
 =[0,A]
\]
and
\[
 \bigcap_{m\geq2}
 \cl_{\R}\set{\frac{\xi_N}{s_N}\mid N\geq m}
 =[-B,\infty).
\]
\end{prop}

\supplementproof{supp.G.2}{Section G.2}

After revealing a northwest rectangle, delete its rows and columns.
The remaining southeast array is independent of the revealed sigma-algebra
and retains the limiting aspect ratio and bounded dimension increments.
Lemma~\ref{lem:fresh-corner} records these properties. Together with
Proposition~\ref{prop:coordinate-submatrix-transfer}, they transfer a
right-tail occurrence from the southeast submatrix to the full future
matrix.

\begin{lem}[Independent submatrices]
\label{lem:fresh-corner}
Let \((x_{ij})_{i,j\geq1}\) be an infinite array of mutually independent complex random variables. Let \((M_N)_{N\geq1}\) be nondecreasing and satisfy
\[
 \frac{M_N}{N}\longrightarrow\gamma\in(0,\infty),
 \qquad
 K=\sup_N(M_{N+1}-M_N)<\infty.
\]
Fix a positive integer \(P\), and define
\[
 \mathcal P_P=\set{(i,j)\mid 1\leq i\leq M_P,\ 1\leq j\leq P},
 \qquad
 \mathcal F_P=\sigma(x_{ij}\mid(i,j)\in\mathcal P_P).
\]
For \(j\geq1\), put
\[
 \widehat M_j=M_{P+j}-M_P
\]
and define the \(\widehat M_j\)-by-\(j\) matrix
\[
 Y^{(j)}
 =
 \bigl(x_{M_P+i,P+h}\bigr)_
 {1\leq i\leq\widehat M_j,\ 1\leq h\leq j}.
\]
Then:

\begin{enumerate}
\item[(a)] The family \((Y^{(j)})_{j\geq1}\) is a nested rectangular path. Moreover, \(Y^{(j)}\) is the coordinate submatrix of
\[
 X^{(P+j)}
 =\bigl(x_{ih}\bigr)_
 {1\leq i\leq M_{P+j},\ 1\leq h\leq P+j}
\]
obtained by deleting its first \(M_P\) rows and first \(P\) columns.

\item[(b)] The entire independent-submatrix array
\[
 (x_{M_P+i,P+h})_{i,h\geq1}
\]
is independent of \(\mathcal F_P\). In particular, its conditional joint law given \(\mathcal F_P\) equals its unconditional product law.

\item[(c)] As \(j\to\infty\),
\[
 \frac{\widehat M_j}{j}\longrightarrow\gamma.
\]
The sequence \((\widehat M_j)\) is nondecreasing and satisfies
\[
 0\leq\widehat M_{j+1}-\widehat M_j\leq K.
\]
\end{enumerate}
\end{lem}

\supplementproof{supp.G.3}{Section G.3}

\subsection{Second-moment and comparison estimates}

In the next lemma, the terminal smoothed count \(S^T\) dominates a count
\(R\) with controlled conditional moments. The second-moment estimate
gives a lower bound for the probability that \(R\) is large, hence a
deficit in the Laplace transform of \(S^T\). Comparison transfers this
deficit to \(S^X\). Since \(S^X\) vanishes on \(E^c\), it yields a
lower bound for \(\Prob(E\mid\mathcal G)\).

\begin{lem}[Conditional second-moment bound]
\label{lem:conditional-occurrence-closure}
Let \(\mathcal G\) be a sigma-algebra, \(E\) an event, and
\(S^X,S^T,R\) nonnegative random variables.  Fix \(\tau>0\) and a deterministic constant \(0<C<\infty\).
Let \(\beta\ge0\) be \(\mathcal G\)-measurable. Suppose, almost surely,
\[
 S^T\geq R,
 \qquad
 0<m:=\E[R\mid\mathcal G]<\infty,
 \qquad
 \E[R^2\mid\mathcal G]\leq C m^2,
\]
and
\[
 \left|\E[\ee^{-\tau S^X}\mid\mathcal G]
 -\E[\ee^{-\tau S^T}\mid\mathcal G]\right|\leq\beta.
\]
If \(S^X=0\) on \(E^c\), then
\[
 \Prob(E\mid\mathcal G)
 \geq
 \frac1{4C}\left(1-\ee^{-\tau m/2}\right)-\beta.
\]
For every $0<a<1$, one also has
\[
 \Prob(E\mid\mathcal G)
 \ge 1-\frac{\Var(R\mid\mathcal G)}{(1-a)^2m^2}
         -\ee^{-\tau a m}-\beta.
\]
In particular, if $m\to\infty$, $\Var(R\mid\mathcal G)=o(m^2)$,
and $\beta=o(1)$ uniformly on the prescribed good histories, then
$\Prob(E\mid\mathcal G)\ge1-o(1)$ there.
\end{lem}

\begin{proof}
Put \(A=\{R\geq m/2\}\). Since \(0<m<\infty\) almost surely,
\[
 \frac m2\leq\E[R\ind_A\mid\mathcal G]
 \leq\E[R^2\mid\mathcal G]^{1/2}\Prob(A\mid\mathcal G)^{1/2}.
\]
Thus \(\Prob(A\mid\mathcal G)\geq1/(4C)\). Since \(S^T\geq R\),
\[
 \E[\ee^{-\tau S^T}\mid\mathcal G]
 \leq1-\frac1{4C}\left(1-\ee^{-\tau m/2}\right).
\]
On the other hand, \(S^X=0\) on \(E^c\), so
\(\Prob(E^c\mid\mathcal G)\leq
\E[\ee^{-\tau S^X}\mid\mathcal G]\).  Combining the two estimates proves
the first bound. For the second, conditional Chebyshev gives
\[
 \E[\ee^{-\tau R}\mid\mathcal G]
 \le\Prob(R<am\mid\mathcal G)+\ee^{-\tau a m}
 \le\frac{\Var(R\mid\mathcal G)}{(1-a)^2m^2}+\ee^{-\tau a m}.
\]
Use $S^T\ge R$ and
$\Prob(E^c\mid\mathcal G)\le\E[\ee^{-\tau S^X}\mid\mathcal G]$.
\end{proof}

\paragraph{Retaining a concentrating count.}
Suppose $0\le\widetilde R\le R$, $m=\E[R\mid\mathcal G]>0$, and
\[
 \E[R^2\mid\mathcal G]\le(1+u)m^2+m,\qquad
 \E[R-\widetilde R\mid\mathcal G]\le q m,\qquad 0\le q<1.
\]
Then $\widetilde m=\E[\widetilde R\mid\mathcal G]\ge(1-q)m$ and
\[
 \frac{\Var(\widetilde R\mid\mathcal G)}{\widetilde m^2}
 \le\frac{1+u+m^{-1}}{(1-q)^2}-1.
\]
This follows from $\widetilde R^2\le R^2$ and the lower bound for
$\widetilde m$. Thus removing events with $q=o(1)$ preserves vanishing
relative variance when $u=o(1)$ and $m\to\infty$. For the Gaussian counts
below, $m,u,q$ can all be chosen deterministic: the complete Gaussian
array is independent of the conditioning sigma-algebra.

The following proposition iterates a differential inequality along a
fixed finite sequence of shifted cutoffs. Each step gains a factor
\(N^{-\rho}\), while the comparison errors accumulate over a logarithmic
time interval. For fixed \(J\) with \(J\rho>1/3\), the stated strict
exponent bounds give an additive error \(N^{-\theta}\) with
\(\theta>1/3\). The parameter \(q\) labels the threshold used in the
applications.

\begin{prop}[Iterated differential inequality]
\label{prop:shifted-gronwall}
Let \(q\in\R\), and let \(N\to\infty\) through the positive integers. For each sufficiently large \(N\), let
\((\Omega_N,\mathcal F_N,\Prob_N)\) be a probability space, let
\(H_N^X,H_N^G\in\mathcal F_N\), let \(T_N\geq0\), and let \(J\) be a fixed positive integer independent of \(N\). For \(j=0,\ldots,J\), let
\[
 f_{N,j}\colon[0,T_N]\to[0,1]
\]
be absolutely continuous, and let \(g_{N,j}\in[0,1]\). Let
\(C_T,C_d,\rho\) be positive finite constants, and let
\[
 \sigma>\frac13,\qquad c_G>\frac13,\qquad D>\frac13,
 \qquad J\rho>\frac13.
\]
Assume, for all sufficiently large \(N\), that:

\begin{enumerate}
\item \(T_N\leq C_T\log N\);

\item \(\Prob_N(H_N^X)\leq f_{N,0}(0)+N^{-D}\);

\item
\[
 \abs{f_{N,j}(T_N)-g_{N,j}}\leq N^{-D},
 \qquad 0\leq j\leq J;
\]

\item for \(0\leq j\leq J-1\) and almost every \(t\in[0,T_N]\),
\[
 \abs{f_{N,j}'(t)}
 \leq C_dN^{-\rho}f_{N,j+1}(t)+C_dN^{-\sigma};
\]

\item
\[
 g_{N,0}\leq\Prob_N(H_N^G)+N^{-c_G},
 \qquad
 g_{N,j}\leq N^{-c_G}\quad(1\leq j\leq J).
\]
\end{enumerate}
Then there are \(\theta>1/3\) and \(N_0\) such that
\[
 \Prob_N(H_N^X)
 \leq\Prob_N(H_N^G)+N^{-\theta}
 \qquad(N\geq N_0).
\]
\end{prop}

\supplementproof{supp.G.4}{Section G.4}

For the common-filtration refinement in
Appendix~\ref{app:recurrence-refinements}, we use the following conditional
form of recurrence.

\begin{lem}[Conditional Borel--Cantelli lemma]\label{lem:conditional-recurrence}
Let $(\mathcal A_l)_{l\ge0}$ be a filtration, let
$E_l\in\mathcal A_l$ and $H_l\in\mathcal A_{l-1}$, and let
$p_l\in[0,1]$ be deterministic. Suppose
\[
 \Prob(E_l\mid\mathcal A_{l-1})\ge p_l\quad\hbox{on }H_l,
 \qquad \sum_l\Prob(H_l^c)<\infty,\qquad\sum_l p_l=\infty.
\]
Then $E_l$ occurs infinitely often almost surely.
\end{lem}
\begin{proof}
Put $A_l=E_l\cup H_l^c$. Since $H_l\in\mathcal A_{l-1}$,
$\Prob(A_l\mid\mathcal A_{l-1})\ge p_l$ almost surely. Successive
conditioning gives, for $m\ge n$,
\[
 \Prob\left(\bigcap_{l=n}^m A_l^c\right)
 \le\prod_{l=n}^m(1-p_l)
 \le\exp\left(-\sum_{l=n}^m p_l\right).
\]
Letting $m\to\infty$ and then taking the countable union over $n$ shows
that $A_l$ occurs infinitely often. Borel--Cantelli implies that only
finitely many $H_l^c$ occur, and proves the assertion.
\end{proof}

Appendix~\ref{app:recurrence-refinements} constructs a common filtration
for several threshold sequences.

\subsection{Monotonicity of cutoff functions}

\begin{lem}[Monotonicity in the threshold]
\label{lem:downward-threshold}\label{lem:right-transition-support}
Let $a_1\le a_0\le B$ and let $g_\omega\ge0$ be integrable on $[a_1,B]$.
Put
\[
 X_0(\omega)=\int_{a_0}^B g_\omega(y)\,\dd y,\qquad
 X_1(\omega)=\int_{a_1}^B g_\omega(y)\,\dd y.
\]
For a nonempty set $S\subset\R$ with finite $\inf S$, and
$\psi\colon\R\to[0,1]$ equal to one on $[\inf S,\infty)$,
\[
 \ind_{\{X_0(\omega)\in S\}}\le\psi(X_1(\omega)).
\]
In particular, if $\{x:\phi(x)\ne0\}\subset[s_0,\infty)$ and
$\psi=1$ on $[s_0,\infty)$, then
$\ind_{\{\phi(X_0(\omega))\ne0\}}\le\psi(X_1(\omega))$.
\end{lem}
\begin{proof}
Nonnegativity gives $X_1=X_0+\int_{a_1}^{a_0}g_\omega\ge X_0$.
On the indicated event, $X_1\ge\inf S$ and $\psi(X_1)=1$.
The last assertion follows with $S=\{x:\phi(x)\ne0\}$; if this set
is empty its indicator is zero.
\end{proof}

\section{Smooth spectral statistics and resolvent identities}\label{sec:detectors}

We approximate eigenvalue events by Poisson-smoothed counts, then apply
weighted cutoff functions and differentiate the resulting spectral statistics.
The weighted construction is stated first, followed by its application to
matrix eigenvalues and the necessary resolvent identities.
We use $\mathrm{c}$ as a subscript or superscript for cutoff parameters
and $\Phi$ for conditional Laplace transforms.

On a polynomial-size grid, we control the sum of cutoff derivatives by
assigning weight \(R^{-r}\) to level \(r\). When a derivative of order
one, two, or three is nonzero at \(r<J\), the next cutoff equals one.
The identity \(R^{-r}=R R^{-(r+1)}\) therefore bounds the differentiated
term by the next contribution to the weighted count. The remaining
terminal term carries weight \(R^{-J}\). At the comparison endpoint,
a retained eigenvalue event activates level \(r=0\), contributing one
to the same count.

\begin{lem}[Weighted cutoff functions]\label{lem:ladder-absorption}
Let \(I\) be a finite set of cardinality \(L\), let \(J\) be a positive integer, let \(R>1\), and let \(\rho\in C^3(\R;[0,1])\) satisfy
\[
 \rho(y)=1\quad\text{for }y\leq -1,
 \qquad
 \rho(y)=0\quad\text{for }y\geq 0.
\]
Put \(\lambda_r=R^{-r}\) for \(0\leq r\leq J\). For each \(i\in I\), let \(y_i\) be three times differentiable near \(h=0\), and define
\begin{equation}
 A_{i,r}(h)=\lambda_r\rho\bigl(y_i(h)+J-r\bigr),
 \qquad
 S(h)=\sum_{i\in I}\sum_{r=0}^J A_{i,r}(h).
 \label{eq:ladder-S}
\end{equation}
Then \(0\leq A_{i,r}\leq 1\). If \(0\leq r<J\), \(1\leq k\leq3\), and
\[
 \rho^{(k)}\bigl(y_i(0)+J-r\bigr)\neq0,
\]
then
\begin{equation}
 \rho\bigl(y_i(0)+J-(r+1)\bigr)=1.
 \label{eq:ladder-absorption}
\end{equation}
Consequently, writing \(C_k=\norm{\rho^{(k)}}_\infty\), if
\[
 \abs{y_i'(0)}\leq B_1,
 \qquad
 \abs{y_i''(0)}\leq B_2
 \qquad (i\in I),
\]
then
\begin{equation}
 \abs{S'(0)}
 \leq RC_1B_1S(0)+L\lambda_JC_1B_1,
 \label{eq:ladder-first}
\end{equation}
\begin{equation}
 \abs{S''(0)}
 \leq R(C_2B_1^2+C_1B_2)S(0)
   +L\lambda_J(C_2B_1^2+C_1B_2).
 \label{eq:ladder-second}
\end{equation}
After removing the third-derivative contribution
\[
 C_{\mathrm{cub}}
 =
 \sum_{i\in I}\sum_{r=0}^J
 \lambda_r\rho'\bigl(y_i(0)+J-r\bigr)y_i'''(0),
\]
the remaining third derivative \(B_{\mathrm{nc}}=S'''(0)-C_{\mathrm{cub}}\) satisfies
\begin{equation}
 \begin{split}
 \abs{B_{\mathrm{nc}}}
 &\leq R(C_3B_1^3+3C_2B_1B_2)S(0)\\
 &\quad
 +L\lambda_J(C_3B_1^3+3C_2B_1B_2).
 \end{split}
 \label{eq:ladder-third}
\end{equation}
Thus every lower-derivative term is controlled by the full weighted count, with the only unabsorbed term carrying the final weight \(R^{-J}\).

There is also an expectation-level multiplier form that retains a joint Laplace factor. Let \(\mathcal G\) be a sigma-algebra, let \(\Omega_0\) be an event, and let \(Q_{i,r}\) be complex random variables such that
\[
 \abs{Q_{i,r}}\leq K_{\mathrm{loc}}\quad\text{on }\Omega_0,
 \qquad
 \abs{Q_{i,r}}\leq K_{\mathrm{abs}}\quad\text{almost surely}.
\]
For \(\tau>0\), put
\[
 V_\tau=\exp(-\tau S(0)),\qquad
 \Phi(\tau)=\E[V_\tau\mid\mathcal G],
\]
and
\[
 D(\tau)=-\tau\partial_\tau\Phi(\tau)
 =\E[\tau S(0)V_\tau\mid\mathcal G].
\]
For \(k\in\{1,2,3\}\), define
\[
 T_k
 =
 \sum_{i\in I}\sum_{r=0}^J
 \lambda_r\rho^{(k)}\bigl(y_i(0)+J-r\bigr)Q_{i,r}.
\]
For every fixed \(\tau_*>0\) and every \(0<\tau\leq\tau_*\),
\begin{equation}
 \begin{split}
 \tau\abs{\E[V_\tau T_k\mid\mathcal G]}
 &\leq RC_kK_{\mathrm{loc}}D(\tau)\\
 &\quad+C_k\tau_*L
 \left[
 \lambda_JK_{\mathrm{loc}}
 +\frac{R}{R-1}K_{\mathrm{abs}}
   \Prob(\Omega_0^c\mid\mathcal G)
 \right].
 \end{split}
 \label{eq:ladder-expectation}
\end{equation}
The same assertion holds when the variables and bounds depend measurably on an interpolation time.

Finally, suppose that \(y_i(0)\) is a normalized target spectral coordinate for which \(y_i(0)<0\) implies the desired target event at \(i\), and suppose that at terminal time an eigenvalue event implies \(y_i(T)\leq-J-1\). Then \(S(0)>0\) implies the union of the target events, while
\[
 A_{i,0}(T)=1
\]
on the eigenvalue event. Hence the cutoff sum retains unit contribution at the comparison endpoint at level \(r=0\), although its final derivative layer has weight \(R^{-J}\).
\end{lem}

\supplementproof{supp.G.8}{Section G.8}

The statistic \(Z_i\) combines the Poisson-smoothed count in
\([a_i,B_i]\), at resolution \(\eta_i\), with a smooth upper-tail term
\(g_i\) that records eigenvalues beyond \(B_i+\eta_i\).
A squared operator norm above the threshold forces a positive count;
a sufficiently small count implies that the norm is below the threshold.
After centering and rescaling, the weighted cutoff sum has the target
support and terminal contribution needed for comparison. A logarithmic
truncation order makes \(\abs I R^{-J}\) polynomially small while
\(\delta^{-1}=O(\log N)\).

\begin{prop}[Smoothed eigenvalue counts]\label{prop:poisson-detector}
Let \(I\) be a finite index set. For every \(i\in I\), let \(m_i,n_i\) be positive integers, let \(t_i,a_i,B_i\in\R\), and let \(\eta_i>0\) satisfy
\begin{equation}
 a_i<t_i\leq B_i,
 \qquad
 t_i-a_i\geq4\eta_i.
 \label{eq:poisson-window}
\end{equation}
Choose \(g_i\in C^\infty(\R;[0,1])\) such that
\[
 g_i(x)=0\quad\text{for }x\leq B_i,
 \qquad
 g_i(x)=1\quad\text{for }x\geq B_i+\eta_i.
\]
For an \(m_i\times n_i\) complex matrix \(W\), let
\(\lambda_{i,s}(W)\), \(1\leq s\leq\min(m_i,n_i)\), denote its squared singular values, including zero values with multiplicity, and define
\[
 Z_i(W)
 =
 \sum_s\bigl[P_i(\lambda_{i,s}(W))+g_i(\lambda_{i,s}(W))\bigr],
\]
where
\begin{equation}
 P_i(x)
 =
 \frac{1}{\pi}\int_{a_i}^{B_i}
 \frac{\eta_i}{(x-u)^2+\eta_i^2}\dd u.
 \label{eq:poisson-detector}
\end{equation}
Then
\begin{equation}
 \norm{W}_{\op}^2\geq t_i
 \quad\Longrightarrow\quad
 Z_i(W)\geq\frac18.
 \label{eq:poisson-detects}
\end{equation}
Consequently,
\[
 Z_i(W)<\frac1{16}
 \quad\Longrightarrow\quad
 \norm{W}_{\op}^2<t_i.
\]
Each \(Z_i\) is a \(C^\infty\) function of the real and imaginary parts of the entries of \(W\), including at repeated singular values.

Fix \(R>1\), a function \(\rho\in C^3(\R;[0,1])\) satisfying
\[
 \rho(y)=1\quad\text{for }y\leq-1,
 \qquad
 \rho(y)=0\quad\text{for }y\geq0,
\]
and a positive integer \(J\). Put
\[
 a_*=\frac1{16},
 \qquad
 \delta=\frac{a_*}{J+2},
 \qquad
 y_i(W)=\frac{Z_i(W)-a_*}{\delta},
\]
\[
 A_{i,r}(W)
 =
 R^{-r}\rho\bigl(y_i(W)+J-r\bigr),
 \qquad 0\leq r\leq J,
\]
and
\begin{equation}
 S(W)
 =
 \sum_{i\in I}\sum_{r=0}^J A_{i,r}(W).
 \label{eq:poisson-ladder}
\end{equation}
For arbitrary target matrices \(W_i^X\),
\begin{equation}
 S\bigl((W_i^X)_{i\in I}\bigr)>0
 \quad\Longrightarrow\quad
 \norm{W_i^X}_{\op}^2<t_i
 \quad\text{for at least one }i.
 \label{eq:poisson-target}
\end{equation}
For arbitrary terminal matrices \(W_i^T\), if
\[
 Z_i(W_i^T)\leq\delta,
\]
then
\begin{equation}
 A_{i,0}(W_i^T)=1.
 \label{eq:poisson-terminal}
\end{equation}

The cutoff sum also has the following derivative support estimates. Suppose that \(W_i(h)\) are scalar \(C^3\) perturbations near \(h=0\), and write
\[
 y_i(h)=y_i(W_i(h)).
\]
Put
\[
 B_1=\max_{i\in I}\abs{y_i'(0)},
 \qquad
 B_2=\max_{i\in I}\abs{y_i''(0)},
 \qquad
 C_k=\norm{\rho^{(k)}}_\infty,
\]
and
\[
 C_{\mathrm{cub}}
 =
 \sum_{i\in I}\sum_{r=0}^J
 R^{-r}\rho'\bigl(y_i(0)+J-r\bigr)y_i'''(0).
\]
Then
\begin{equation}
 \abs{S'(0)}
 \leq RC_1B_1S(0)+\abs{I}R^{-J}C_1B_1,
 \label{eq:poisson-first}
\end{equation}
\begin{equation}
 \begin{split}
 \abs{S''(0)}
 &\leq R(C_2B_1^2+C_1B_2)S(0)\\
 &\quad+\abs{I}R^{-J}(C_2B_1^2+C_1B_2),
 \end{split}
 \label{eq:poisson-second}
\end{equation}
and
\begin{equation}
 \begin{split}
 \abs{S'''(0)-C_{\mathrm{cub}}}
 &\leq R(C_3B_1^3+3C_2B_1B_2)S(0)\\
 &\quad+\abs{I}R^{-J}(C_3B_1^3+3C_2B_1B_2).
 \end{split}
 \label{eq:poisson-third}
\end{equation}

Finally, suppose that the construction depends on an integer \(N\) tending to infinity, that
\[
 \abs{I}\leq N^\theta
\]
for a fixed \(\theta\), and fix \(A>0\). If
\[
 J=\left\lceil\zeta\log N\right\rceil,
 \qquad
 \zeta>\frac{\theta+A}{\log R},
\]
then, for all sufficiently large \(N\),
\[
 \abs{I}R^{-J}\leq N^{-A},
 \qquad
 \delta^{-1}=O(\log N).
\]
If, for a fixed \(\kappa>0\), a terminal event gives
\[
 Z_i(W_i^T)\leq2N^{-\kappa},
\]
then that event implies \eqref{eq:poisson-terminal} for all sufficiently large \(N\).
\end{prop}

\supplementproof{supp.G.9}{Section G.9}

Lemma~\ref{lem:all-order-ladder} controls the mixed derivatives of the
Laplace-weighted cutoff coefficient by an \(\ell^1\) tensor bound.
The smoothed-statistic rescaling contributes \(\delta^{-(m+1)}\) after
\(m\) derivatives. For two disjoint families, the mixed tensor is bounded
by the product of the two one-family estimates. These bounds sum the
derivative allocations arising in the cumulant expansion.

\begin{lem}[Derivative estimates]\label{lem:all-order-ladder}
Let \(I\) be a finite nonempty set, let \(J\) be a positive integer, let \(R_{\mathrm{c}}>1\), and let \(\rho\in C^\infty(\R;[0,1])\) satisfy
\[
 \rho(u)=1\quad\text{for }u\leq-1,
 \qquad
 \rho(u)=0\quad\text{for }u\geq0.
\]
Put \(\lambda_r=R_{\mathrm{c}}^{-r}\) for \(0\leq r\leq J\). For
\(y=(y_i)_{i\in I}\in\R^I\), define
\[
 A_{i,r}(y)=\lambda_r\rho(y_i+J-r),
 \qquad
 S(y)=\sum_{i\in I}\sum_{r=0}^J A_{i,r}(y).
\]
Assume that
\[
 \epsilon_{\mathrm{fin}}=\abs I\lambda_J\leq1.
\]
Fix a finite \(\tau_*>0\). For \(0<\tau\leq\tau_*\), define
\[
 G_{i,r,\tau}(y)
 =
 \tau\exp(-\tau S(y))
 \lambda_r\rho'(y_i+J-r).
\]
For every nonnegative integer \(m\), there is a finite constant
\(C_{m,\tau_*}\), depending only on
\[
 m,\quad \tau_*,\quad R_{\mathrm{c}},
 \quad\text{and}\quad
 \max_{1\leq k\leq m+1}\norm{\rho^{(k)}}_\infty,
\]
such that
\begin{equation}
 \sum_{i_0\in I}\sum_{r_0=0}^J
 \sum_{i_1,\ldots,i_m\in I}
 \left|
 \partial_{y_{i_1}}\cdots\partial_{y_{i_m}}
 G_{i_0,r_0,\tau}(y)
 \right|
 \leq C_{m,\tau_*}
 \label{eq:all-order}
\end{equation}
for every \(y\in\R^I\) and \(0<\tau\leq\tau_*\). When \(m=0\), the last index sum and the derivatives are omitted.

Let \(\delta>0\) and \(a_i\in\R\) for \(i\in I\). Put
\[
 y_i=\frac{q_i-a_i}{\delta},
 \qquad
 \widetilde G_{i,r,\tau}(q)
 =
 \delta^{-1}G_{i,r,\tau}(y(q)).
\]
Then the left side of \eqref{eq:all-order}, with \(\widetilde G\) and \(q\)-derivatives in place of \(G\) and \(y\)-derivatives, is at most
\begin{equation}
 C_{m,\tau_*}\delta^{-(m+1)}.
 \label{eq:all-order-scaled}
\end{equation}
If two disjoint copies of this construction are used, possibly with different finite index sets, cutoff sum depths, smoothed statistic scales, and Laplace parameters in \((0,\tau_*]\), then the \(\ell^1\)-sum of any mixed derivative tensor of
\[
 \widetilde G_{i,r,\tau}(q)
 \widetilde G'_{j,s,\tau'}(q')
\]
is bounded by the product of the corresponding one-copy constants and spectral-scale powers.
\end{lem}

\supplementproof{supp.G.10}{Section G.10}

The next lemma gives the endpoint formula for every fixed derivative order,
its two coordinate-axis specializations, and the bounds used in the comparison.
Each complex-entry direction varies both Hermitian-transpose positions.

\begin{lem}[Resolvent derivatives]
\label{lem:covariance-direction-derivatives}
\label{lem:covariance-coordinate-derivatives}
Let $P_{\rm col}$ project onto the first $n_{\rm col}$ coordinates and
let $H_0$ be a fixed complex matrix of dimension $n_{\rm col}+n_{\rm row}$.
Fix a column label $b$ and a row label $\alpha$. For $w=s+\ii t$, put
\begin{equation}\label{eq:direction-matrix}
 S_w=wE_{\alpha b}+\overline wE_{b\alpha}.
\end{equation}
For $\eta>0$ and $x_-<x_+$, assume that
\begin{equation}\label{eq:direction-resolvent}
 G(\theta,z)=(H_0+\theta S_w-zP_{\rm col})^{-1}
\end{equation}
exists for real $\theta$ near zero and $z=x+\ii\eta$,
$x_-\le x\le x_+$. Define
\begin{equation}\label{eq:direction-integral}Y_w(\theta)=\frac1\pi\int_{x_-}^{x_+}
   \Im\Tr(P_{\rm col}G(\theta,x+\ii\eta))\,\dd x.
\end{equation}
Write $z_\pm=x_\pm+\ii\eta$ and
$\Delta f=f(z_+)-f(z_-)$. Then, for every integer $r\ge1$,
\begin{equation}\label{eq:direction-endpoint}
 \partial_\theta^rY_w(0)
 =\frac{(-1)^r(r-1)!}{\pi}
   \Delta\Im\Tr\bigl((S_wG(0,z))^r\bigr).
\end{equation}

At each endpoint abbreviate
$a=G_{b\alpha}$, $c=G_{\alpha b}$, $p=G_{bb}$, and $d=G_{\alpha\alpha}$.
For the real and imaginary axes, respectively,
\begin{align}
 \partial_uY&=-\pi^{-1}\Delta\Im(a+c),&
 \partial_u^2Y&=\pi^{-1}\Delta\Im(a^2+c^2+2pd),
 \label{eq:coordinate-real}\\
 \partial_vY&=-\pi^{-1}\Delta\Im\bigl(\ii(a-c)\bigr),&
 \partial_v^2Y&=\pi^{-1}\Delta\Im(-a^2-c^2+2pd).
 \label{eq:coordinate-imaginary}
\end{align}
If, at both endpoints,
\begin{equation}\label{eq:coordinate-entry-bounds}
 |a|+|c|\le B_{\rm cross},\qquad |p|+|d|\le C_{\rm diag},\qquad
 |\Im p|+|\Im d|\le B_{\rm im},
\end{equation}
then either axis satisfies
\begin{equation}\label{eq:coordinate-derivative-bounds}
 |Y'(0)|\le2B_{\rm cross}/\pi,\qquad
 |Y''(0)|\le C(B_{\rm cross}^2+C_{\rm diag}B_{\rm im}).
\end{equation}
More generally, fix a positive integer $K_{\rm der}$ and finite nonnegative constants $B_{\rm dir},C_{\rm diag}$.
If $|s|+|t|\le B_{\rm dir}$, $0\le\Psi\le1$, and
\begin{align}
 |p|+|d|&\le C_{\rm diag},\label{eq:direction-diagonal}\\
 |a|+|c|+|\Im p|+|\Im d|&\le\Psi,\label{eq:direction-small}
\end{align}
then
\begin{equation}\label{eq:direction-bound}
 |\partial_\theta^rY_w(0)|\le C_{\rm der}\Psi,
 \qquad 1\le r\le K_{\rm der},
\end{equation}
where $C_{\rm der}$ depends only on the three fixed bounds.
\end{lem}
\begin{proof}
Differentiation gives $\partial_\theta G=-GS_wG$ and
$\partial_zG=GP_{\rm col}G$. Hence
\[
 \partial_\theta^r\Tr(P_{\rm col}G)
 =(-1)^rr!\Tr\bigl(P_{\rm col}G(S_wG)^r\bigr)
 =(-1)^r(r-1)!\partial_z\Tr\bigl((S_wG)^r\bigr).
\]
Integration proves \eqref{eq:direction-endpoint}. Only the block on
$\{b,\alpha\}$ enters the trace. In the order $(b,\alpha)$ its product is
\[
 T=\begin{pmatrix}\overline w c&\overline w d\\wp&wa\end{pmatrix}.
\]
The traces of $T$ and $T^2$ give the two axis identities. The bounds in
\eqref{eq:coordinate-derivative-bounds} follow from
$|\Im(pd)|\le |p||\Im d|+|d||\Im p|$.
For the higher-order estimate replace $a,c$ by zero and $p,d$ by their real
parts, obtaining
$T_0=\left(\begin{smallmatrix}0&\overline w\Re d\\w\Re p&0\end{smallmatrix}\right)$.
Its odd-power trace is zero and its even-power trace is real. Moreover,
$\|T-T_0\|=O(\Psi)$ and $\|T\|+\|T_0\|=O(1)$, uniformly under the
stated bounds. The finite telescoping identity for $T^r-T_0^r$ gives
$|\Im\Tr T^r|\le C_r\Psi$. Apply it at both endpoints in
\eqref{eq:direction-endpoint}.
\end{proof}

For covariance matrices take
$H_0=\left(\begin{smallmatrix}0&X^*\\X&-I\end{smallmatrix}\right)$.
The assumption in \eqref{eq:direction-resolvent} concerns the linearized
inverse at the spectral parameter and permits singular $H_0$.
Thus the coordinate formulas also apply to rank-deficient and wide
matrices. Normalization of the raw entries by $n^{-1/2}$ contributes
one factor $n^{-1/2}$ per derivative.

\section{Local laws and Green function comparison}\label{sec:comparison}

We first relate eigenvalue events to smoothed counts and establish
simultaneous conditional local laws. Schur complement estimates then
control Gaussian interpolation in the two cross strips. A cumulant
expansion treats the non-Gaussian entry flow. Sections~\ref{sec:high-occurrence}
and~\ref{sec:low-occurrence} apply these estimates with their respective
histories and spectral endpoints.

\subsection{Smoothing of eigenvalue counts}

The pathwise count estimate below gives both one-sided smoothing
inequalities, with the threshold shifted in the direction appropriate
to each tail.

\begin{prop}[Approximation of eigenvalue counts]\label{prop:poisson-count-dichotomy}
Let $\gamma\ge1$ and $M_N\ge N$ eventually, with $M_N/N\to\gamma$.
Let $X_N$ have independent centered complex entries with
$\E|\sqrt N X_N(\alpha,b)|^2=1$, $\E X_N(\alpha,b)^2=0$, and
uniformly bounded fixed moments of $\sqrt N X_N(\alpha,b)$. Put
\[
 L_N=X_N^*X_N,\qquad \lambda_N=\lambda_{\max}(L_N),\qquad
 E_{+,N}=(1+\sqrt{M_N/N})^2.
\]
Fix $0<\kappa<1/72$ and $C_E<\infty$. Let $(E_N)$ be deterministic and satisfy
\[
 \abs{E_N-E_{+,N}}
 \le C_E N^{-2/3}(\log N)^{2/3}.
\]
Define
\[
 \eta_N=N^{-2/3-12\kappa},\qquad
 \ell_N=N^{-2/3-6\kappa},\qquad
 B_N=E_{+,N}+4N^{-2/3+\kappa},
\]
\[
 m_N(z)=\frac1N\Tr(L_N-zI_N)^{-1},
\]
and
\[
 Z_N^+=\frac{N}{\pi}\int_{E_N+\ell_N}^{B_N}
          \Im m_N(y+\ii\eta_N)\,\dd y,\qquad
 Z_N^-=\frac{N}{\pi}\int_{E_N-\ell_N}^{B_N}
          \Im m_N(y+\ii\eta_N)\,\dd y.
\]
For every $D>0$, there are $N_0$ and events $\Omega_N$, $N\ge N_0$, such that
\[
 \Prob(\Omega_N^c)\le N^{-D},
\]
and, on $\Omega_N$,
\begin{equation}
 Z_N^+-N^{-\kappa}
 \le \#\set{\lambda\in\spec(L_N)\mid \lambda\ge E_N}
 \le Z_N^-+N^{-\kappa}.
 \label{eq:poisson-count-dichotomy}
\end{equation}
Consequently, on $\Omega_N$,
\begin{equation}
 \lambda_N<E_N\quad\Longrightarrow\quad Z_N^+\le N^{-\kappa},
 \label{eq:poisson-count-below}
\end{equation}
and
\begin{equation}
 \lambda_N\ge E_N\quad\Longrightarrow\quad Z_N^-\ge1-N^{-\kappa}.
 \label{eq:poisson-count-above}
\end{equation}
All constants are uniform over deterministic $E_N$ satisfying the displayed window condition. The same assertions hold for matching complex white Wishart matrices.
\end{prop}

\supplementproof{supp.H.3}{Section H.3}

\begin{cor}[Smoothing inequalities]
\label{prop:left-smoothing}\label{prop:right-smoothing}
Use the matrix assumptions, $\eta_N$, $\ell_N$, $B_N$, and $m_N$ of
Proposition~\ref{prop:poisson-count-dichotomy}. Let
\[
 \sigma_{+,N}=\frac{\sqrt{M_N}+\sqrt N}{N}
       (M_N^{-1/2}+N^{-1/2})^{1/3},\qquad
 \ell_{1,N}=N^{3\kappa}\eta_N.
\]
The notation $l_N=\ell_N$ and $l_{1,N}=\ell_{1,N}$ is interchangeable.
Fix $q>0$, an integer $J\ge1$, and $0<a_F<b_F<1$.
For a Borel function $F:\R\to[0,1]$, the following assertions hold.

\emph{Left tail.} Put $s_N=(\log N)^{1/3}$,
$E_N=E_{+,N}-qs_N\sigma_{+,N}$, and
\[
 E_{N,j}=E_N+(2j+1)\ell_N,\qquad
 Z_{N,j}=\frac N\pi\int_{E_{N,j}}^{B_N}\Im m_N(y+\ii\eta_N)\,\dd y.
\]
If $F=1$ on $(-\infty,a_F]$ and $F=0$ on $[b_F,\infty)$, then, for
$0\le j\le J$ and all sufficiently large $N$,
\[
 \begin{aligned}
 \Prob(\lambda_N<E_N)&\le\E F(Z_{N,0})+N^{-D},\\
 \E F(Z_{N,j})&\le\Prob(\lambda_N<E_N+(2j+2)\ell_N)+N^{-D}.
 \end{aligned}
\]
\emph{Right tail.} Put $r_N=(\log N)^{2/3}$,
$E_N=E_{+,N}+qr_N\sigma_{+,N}$, and now set
$E_{N,j}=E_N-(2j+1)\ell_N$ in the same integral. If $F=0$ on
$(-\infty,a_F]$ and $F=1$ on $[b_F,\infty)$, then
\[
 \begin{aligned}
 \Prob(\lambda_N\ge E_N)&\le\E F(Z_{N,0})+N^{-D},\\
 \E F(Z_{N,j})&\le\Prob(\lambda_N\ge E_N-(2j+2)\ell_N)+N^{-D}.
 \end{aligned}
\]
Here $D>0$ is arbitrary. Both clauses also hold for the matching complex
white Wishart matrix.
\end{cor}
\begin{proof}
Apply Proposition~\ref{prop:poisson-count-dichotomy} at $E_N$ and the
finitely many thresholds $E_N\pm(2j+2)\ell_N$, using failure exponent $D+1$.
All these thresholds remain in its logarithmic window. Increase $N$ so that
$N^{-\kappa}<\min(a_F,1-b_F)$ and the union failure probability is at
most $N^{-D}$.
On the common good event, $\lambda_N<E_N$ forces the left-tail $Z_{N,0}$
to be at most $N^{-\kappa}$, while
$\lambda_N\ge E_N+(2j+2)\ell_N$ forces the left-tail $Z_{N,j}$ to be
at least $1-N^{-\kappa}$. Thus the left-tail cutoff is one on the first
event and zero on the second. For the right-tail integral,
$\lambda_N\ge E_N$ forces $Z_{N,0}\ge1-N^{-\kappa}$, and
$\lambda_N<E_N-(2j+2)\ell_N$ forces $Z_{N,j}\le N^{-\kappa}$.
The right-tail cutoff gives the corresponding inequalities. Take
expectations. In particular, the strict and non-strict inequalities do
not require an atom-free entry distribution.
\end{proof}

\subsection{Entrywise local laws along the interpolation}

We write $z=x+\ii\eta$, with $\eta>0$. The local laws use $G_N(t,z)$
for the linearized resolvent, $m_N(t,z)$ for the empirical column
Stieltjes transform, and $\widetilde m_N(z)$ for its deterministic
approximation, as defined below.

The right and left cutoffs use edge distances of order
$N^{-2/3}(\log N)^{2/3}$ and $N^{-2/3}(\log N)^{1/3}$, respectively.
The comparison differentiates all cutoff levels along one flow and over
a polynomial-size matrix family. We therefore use one event on which
the local law holds uniformly in time, spectral position, entry index,
and both endpoint families.

\begin{thm}[Entrywise local law]\label{thm:flow-local-law}
Fix $\gamma\ge1$, $q_R,q_L>0$, $C_J<\infty$, and
$0<\kappa_0<1/288$. Let $M_N\ge N$, $\rho_N=M_N/N\to\gamma$, and let
$X_N$ be an $M_N\times N$ complex random matrix with independent entries satisfying
\[
 \E X_N(\alpha,b)=0,\qquad
 \E\abs{\sqrt N X_N(\alpha,b)}^2=1,\qquad
 \E X_N(\alpha,b)^2=0,
\]
and uniformly bounded normalized fixed moments. Let $W_N$ be independent of
$X_N$ with independent centered circular complex Gaussian entries of variance $1/N$. For
$0\le t\le T_N:=8\log N$, define
\[
 X_N(t)=\ee^{-t/2}X_N+\sqrt{1-\ee^{-t}}\,W_N.
\]
This is the tall orientation $M_N\ge N$. A wide matrix must first be conjugate-transposed and, if its entry normalization changes, rescaled.

Put
\[
 L_N(t)=X_N(t)^*X_N(t),\qquad
 \mathcal L_N(t)=X_N(t)X_N(t)^*,
\]
\[
 R_N(t,z)=(L_N(t)-zI_N)^{-1},\qquad
 \mathcal R_N(t,z)=(\mathcal L_N(t)-zI_{M_N})^{-1},
\]
\[
 H_N(t,z)=
 \begin{pmatrix}
 -zI_N&X_N(t)^*\\
 X_N(t)&-I_{M_N}
 \end{pmatrix},
 \qquad
 G_N(t,z)=H_N(t,z)^{-1}.
\]
The row-first linearization is the block permutation
\[
 \mathcal H_N(t,z)=
 \begin{pmatrix}
 -I_{M_N}&X_N(t)\\
 X_N(t)^*&-zI_N
 \end{pmatrix},
\]
and its inverse is the corresponding block permutation of $G_N$. The Schur complement gives
\begin{equation}
 G_N(t,z)=
 \begin{pmatrix}
 R_N(t,z)&X_N(t)^*\mathcal R_N(t,z)\\
 X_N(t)R_N(t,z)&z\mathcal R_N(t,z)
 \end{pmatrix}.
 \label{eq:flow-schur}
\end{equation}
Let $P_N$ project onto the first $N$ coordinates and define
\[
 m_N(t,z)=\frac1N\Tr R_N(t,z),
 \qquad
 \underline G_N(t,z)=\frac1N\Tr(P_NG_N(t,z))=m_N(t,z),
\]
\[
 \underline{\mathcal G}_N(t,z)
 =\frac1{M_N}\sum_{\alpha=N+1}^{N+M_N}G_N(t,z)_{\alpha,\alpha}
 =\frac{z}{M_N}\Tr\mathcal R_N(t,z).
\]
Let
\[
 E_{\pm,N}=(1\pm\sqrt{\rho_N})^2,
\]
\[
 \sigma_{+,N}
 =\frac{\sqrt{M_N}+\sqrt N}{N}
  \left(\frac1{\sqrt{M_N}}+\frac1{\sqrt N}\right)^{1/3}.
\]
Let $\widetilde m_N(z)$ be the solution with positive imaginary part and
$\widetilde m_N(z)=-z^{-1}+o(\abs z^{-1})$ of
\begin{equation}
 z\widetilde m_N(z)^2+(z+1-\rho_N)\widetilde m_N(z)+1=0,
 \label{eq:mp-quadratic}
\end{equation}
and put
\[
 \Pi_N(z)=\diag\left(
 \widetilde m_N(z)I_N,
 -\frac1{1+\widetilde m_N(z)}I_{M_N}
 \right).
\]
Define
\[
 \begin{aligned}
 r_N&=(\log N)^{2/3}, & s_N&=(\log N)^{1/3},\\
 \eta_N&=N^{-2/3-12\kappa_0}, &
 \ell_N&=N^{-2/3-6\kappa_0}.
 \end{aligned}
\]
\[
 B_N=E_{+,N}+4N^{-2/3+\kappa_0},
\]
\[
 A_{R,N}=E_{+,N}+q_Rr_N\sigma_{+,N},\qquad
 A_{L,N}=E_{+,N}-q_Ls_N\sigma_{+,N}.
\]
Let $0\le J_N\le C_J\log N$, and set
\[
 I_{R,N}=[A_{R,N}-(2J_N+1)\ell_N,B_N],\qquad
 I_{L,N}=[A_{L,N},B_N].
\]
All endpoints
\[
 A_{R,N}-(2j+1)\ell_N,\qquad
 A_{L,N}+(2j+1)\ell_N,\qquad 0\le j\le J_N,
\]
belong to the respective windows for all sufficiently large $N$.

Fix $12\kappa_0<\nu<1/12$ and $D>0$. There are $C<\infty$, $N_0$, and events $\Omega_N$ such that
\[
 \Prob(\Omega_N^c)\le N^{-D},
\]
and, on $\Omega_N$, simultaneously for every $0\le t\le T_N$ and every
$x\in I_{R,N}\cup I_{L,N}$, with $z=x+\ii\eta_N$,
\begin{equation}
 \max_{i,j}\abs{G_N(t,z)_{i,j}-\Pi_N(z)_{i,j}}
 \le C N^{-1/3+\nu},
 \label{eq:flow-entry-law}
\end{equation}
\begin{equation}
 \abs{\underline G_N(t,z)-\widetilde m_N(z)}
 +\abs{\underline{\mathcal G}_N(t,z)
        +(1+\widetilde m_N(z))^{-1}}
 \le C N^{-1/3+\nu}.
 \label{eq:flow-trace-law}
\end{equation}
Uniformly in column labels $b$ and row labels $\alpha$,
\begin{equation}
 \abs{G_N(t,z)_{b,\alpha}}+\abs{G_N(t,z)_{\alpha,b}}
 \le C N^{-1/3+\nu},
 \label{eq:flow-cross-law}
\end{equation}
\begin{equation}
 \abs{G_N(t,z)_{b,b}}+\abs{G_N(t,z)_{\alpha,\alpha}}\le C,
 \label{eq:flow-diagonal-law}
\end{equation}
\begin{equation}
 \abs{\Im G_N(t,z)_{b,b}}+\abs{\Im G_N(t,z)_{\alpha,\alpha}}
 \le C N^{-1/3+\nu}.
 \label{eq:flow-imaginary-law}
\end{equation}
With
\[
 m_{+,N}=-(1+\sqrt{\rho_N})^{-1},\qquad
 c_{+,N}=-\frac{1+\sqrt{\rho_N}}{\sqrt{\rho_N}},
\]
one also has
\begin{equation}
 \abs{\underline G_N(t,z)-m_{+,N}}
 +\abs{\underline{\mathcal G}_N(t,z)-c_{+,N}}
 \le C N^{-1/3+\nu},
 \label{eq:flow-edge-background}
\end{equation}
and
\begin{equation}
 \E\left[
 \sup_{\substack{0\le t\le T_N\\x\in I_{R,N}\cup I_{L,N}}}
 \Im m_N(t,x+\ii\eta_N)\right]
 \le C N^{-1/3+\nu}.
 \label{eq:flow-expected-imaginary}
\end{equation}

A deterministic angular-time mesh of spacing $N^{-4}$ and a spectral mesh of spacing $N^{-4}$ suffice. The time mesh has at most $CN^4$ points, the spectral mesh at most $CN^{10/3+\kappa_0}$ points, and their product at most $CN^{22/3+\kappa_0}$ points. Including all $O(N^2)$ coordinate pairs gives at most $CN^{28/3+\kappa_0}$ fixed-time tests. For an additional family of at most $N^A$ proportional-dimensional flows, the conclusions hold simultaneously by requesting fixed-time failure exponent $\Gamma=D+A+11$.

Finally, for $e=(\alpha,b)$, vary either the real or imaginary coordinate of
$X_N(t)_{\alpha,b}$, translated so that the current value corresponds to zero. Put
\[
 a_{R,j}=A_{R,N}-(2j+1)\ell_N,\qquad
 a_{L,j}=A_{L,N}+(2j+1)\ell_N,
\]
\[
 Y_{K,j,e}(\theta)
 =\frac1\pi\int_{a_{K,j}}^{B_N}
   \Im\Tr(P_NG_{N,e}(\theta,x+\ii\eta_N))\,\dd x,
 \qquad K\in\{R,L\}.
\]
Let $\Delta_{K,j}$ denote evaluation at $B_N+\ii\eta_N$ minus evaluation at
$a_{K,j}+\ii\eta_N$, and at either endpoint put
\[
 a=G_N(t,z)_{b,\alpha},\quad
 c=G_N(t,z)_{\alpha,b},\quad
 p=G_N(t,z)_{b,b},\quad
 d=G_N(t,z)_{\alpha,\alpha}.
\]
Then
\[
 \partial_uY_{K,j,e}(0)
 =-\frac1\pi\Delta_{K,j}\Im(a+c),
\]
\[
 \partial_u^2Y_{K,j,e}(0)
 =\frac1\pi\Delta_{K,j}\Im(a^2+c^2+2pd),
\]
\[
 \partial_vY_{K,j,e}(0)
 =-\frac1\pi\Delta_{K,j}\Im\bigl(\ii(a-c)\bigr),
\]
\begin{equation}
 \partial_v^2Y_{K,j,e}(0)
 =\frac1\pi\Delta_{K,j}\Im(-a^2-c^2+2pd).
 \label{eq:flow-coordinate-identities}
\end{equation}
All four derivatives are bounded by $CN^{-1/3+\nu}$ simultaneously in
$t,j,e$, and in both windows. If
\[
 y_{K,j,e}=\frac{Y_{K,j,e}-y_0}{\delta_N},
 \qquad \delta_N^{-1}\le C_\delta\log N,
\]
then, for $\kappa_0=1/1000$ and $\nu=1/48$, the four first and second derivatives are at most $N^{-3/10}$ for all sufficiently large $N$. No bound on an isolated third coordinate derivative is asserted.
\end{thm}

\begin{proof}
Apply Theorem~\ref{thm:frozen-corner-local-law} with an empty frozen
corner and write $\cos\theta=\ee^{-t/2}$. This parametrization covers the entire time interval, including $t=0$,
inside the compact angular interval $[0,\pi/2]$. The logarithmic windows and all their
shifted endpoints lie in the edge domain of that theorem for sufficiently
large $N$, since every fixed power of $\log N$ is $o(N^{\kappa_0})$.
The entry, trace, cross, diagonal, and imaginary-diagonal assertions
follow. The discriminant in \eqref{eq:mp-quadratic} gives
\[
 |\widetilde m_N(z)-m_{+,N}|
 +\left|-(1+\widetilde m_N(z))^{-1}-c_{+,N}\right|
 \le C N^{-1/3+\kappa_0/2},
\]
which is smaller than the stated error.

For the expected supremum, request the simultaneous event with failure
exponent larger than $D+3$. On its complement use the deterministic
bound $|m_N(t,x+\ii\eta_N)|\le\eta_N^{-1}$, valid for the ordinary
column resolvent. This proves \eqref{eq:flow-expected-imaginary}.
For the mesh count without cutoffs, there are at most $CN^4$ angular
points and $CN^{10/3+\kappa_0}$ spectral points; all index pairs add
$CN^2$ tests. The interpolation estimate in the proof of
Lemma~\ref{lem:continuum-promotion} applies, and
$\Gamma=D+A+11$ absorbs this union and its fixed constants.

The four endpoint identities follow from
Lemma~\ref{lem:covariance-coordinate-derivatives} and the Schur
identity \eqref{eq:flow-schur}. Cross entries are
$O(N^{-1/3+\nu})$. Products of a row and a column diagonal have
imaginary part of that order, because their real parts are bounded and
their imaginary parts have that order. Thus each displayed first or
second derivative is $O(N^{-1/3+\nu})$. At the fixed parameters this
is $O(N^{-5/16})$; division by $\delta_N$ costs at most $C\log N$,
and $N^{-5/16}\log N=o(N^{-3/10})$.
\end{proof}

\subsection{Uniform estimates and removal of one entry}

The local-law error parameter is $\Psi_N=N^{-1/3+\nu}$, with $\nu$
in the range stated in each result. The fixed choice $\nu=1/48$ gives
$\Psi_N=N^{-5/16}$.

A fixed-parameter local law extends to continuous interpolation parameters
by a deterministic net and resolvent derivative bounds. The next lemma
makes the estimates simultaneous over a polynomial-size deterministic
family, including the cutoffs, edge points, and matrix indices. For each
fixed choice of history parameters, conditional Markov bounds give a
measurable set of histories with small conditional failure probability.

\begin{lem}[Uniformity in the parameters]\label{lem:continuum-promotion}
Fix
\[
 \begin{aligned}
 0&<\phi_{\min}<\phi_{\max}<\infty,
 &0&<\kappa<1/288,\\
 12\kappa&<\nu<1/12,
 &A_{\rm fam}&\ge0,
 \end{aligned}
\]
and put
\[
 \eta_N=N^{-2/3-12\kappa},\qquad
 \Psi_N=N^{-1/3+\nu}.
\]
For every sufficiently large $N$, let $\mathcal I_N$ be deterministic with
$\abs{\mathcal I_N}\le N^{A_{\rm fam}}$. For $r\in\mathcal I_N$, let
\[
 \phi_{\min}\le\phi_r:=\frac{M_r}{N}\le\phi_{\max}.
\]
Let $\xi_{\alpha,b}^{(r)}$ and $w_{\alpha,b}^{(r)}$ have uniformly bounded fixed moments; independence is not required. For
$0\le m\le M_r$ and $0\le n\le N$, put
\[
 \mathcal P_{r,m,n}=\set{(\alpha,b)\mid \alpha\le m,\ b\le n}.
\]
For $0\le\theta\le\pi/2$, define
\[
 X_{r,m,n}(\theta)_{\alpha,b}
 =
 \begin{cases}
 N^{-1/2}\xi_{\alpha,b}^{(r)},&(\alpha,b)\in\mathcal P_{r,m,n},\\
 N^{-1/2}\bigl(\cos\theta\,\xi_{\alpha,b}^{(r)}
 +\sin\theta\,w_{\alpha,b}^{(r)}\bigr),
 &(\alpha,b)\notin\mathcal P_{r,m,n}.
 \end{cases}
\]
Let
\[
 H_{r,m,n}(\theta,z)=
 \begin{pmatrix}
 -zI_N&X_{r,m,n}(\theta)^*\\
 X_{r,m,n}(\theta)&-I_{M_r}
 \end{pmatrix},
 \qquad G_{r,m,n}(\theta,z)=H_{r,m,n}(\theta,z)^{-1}.
\]
Put
\[
 E_{+,r}=(1+\sqrt{\phi_r})^2,\qquad
 \mathcal S_r=\set{E+\ii\eta_N\mid
 \abs{E-E_{+,r}}\le5N^{-2/3+\kappa}}.
\]
Let $m_r$ solve
\[
 z=-\frac1{m_r(z)}+\frac{\phi_r}{1+m_r(z)},\qquad \Im m_r(z)>0,
\]
and set
\[
 \Pi_r(z)=\diag\left(m_r(z)I_N,-(1+m_r(z))^{-1}I_{M_r}\right).
\]
Let $T_{\rm col}$ and $T_{\rm row}$ denote the normalized traces of the column and row blocks of $G_{r,m,n}$.

Assume the following fixed-parameter local law. For every $B>0$, there are
$K_B<\infty$ and $N_B$ such that, for every deterministic choice of all displayed parameters and indices, the probability that any of
\[
 \abs{G_{r,m,n}(\theta,z)_{u,v}-\Pi_r(z)_{u,v}}\le K_B\Psi_N,
\]
\[
 \abs{G_{r,m,n}(\theta,z)_{u,u}}\le K_B,\qquad
 \abs{\Im G_{r,m,n}(\theta,z)_{u,u}}\le K_B\Psi_N,
\]
\[
 \abs{T_{\rm col}-m_r(z)}\le K_B\Psi_N,\qquad
 \abs{T_{\rm row}+(1+m_r(z))^{-1}}\le K_B\Psi_N
\]
fails is at most $K_BN^{-B}$.

Then, for every $D>0$, there are $C_D<\infty$, $N_D$, and events
$\Omega_N(D)$ with
\[
 \Prob(\Omega_N(D)^c)\le C_DN^{-D},
\]
such that all five estimates hold simultaneously for every
$r,m,n,\theta,z\in\mathcal S_r$ and all indices.

For fixed deterministic $r,m,n$, let
\[
 \mathcal F_{r,m,n}
 =\sigma\bigl(\xi_{\alpha,b}^{(r)}:(\alpha,b)\in\mathcal P_{r,m,n}\bigr)
\]
and define
\[
 \mathcal H_N(D;r,m,n)
 =\set{
 \Prob(\Omega_N(D)^c\mid\mathcal F_{r,m,n})
 \le C_DN^{-D+5}}.
\]
Then
\[
 \Prob(\mathcal H_N(D;r,m,n)^c)\le N^{-5}.
\]
Thus the exceptional histories are summable along every dyadic subsequence. Applying the conclusion with $D+5$ gives conditional failure at most
$C_{D+5}N^{-D}$ outside an event whose complement has probability at most $N^{-5}$. The conclusion only uniformizes the assumed fixed-parameter law.
\end{lem}

\begin{proof}
We construct the simultaneous event on a deterministic net before
conditioning on the past.
Choose a fixed $0<\epsilon<1/100$. The uniform moment bounds and Markov's
inequality give, with arbitrarily high polynomial probability,
\[
 \max_{r,\alpha,b}(|\xi^{(r)}_{\alpha,b}|+
 |w^{(r)}_{\alpha,b}|)\le N^\epsilon.
\]
There are $O(N^{A_{\rm fam}+2})$ variables in this maximum, so the
moment order can be chosen after the desired failure exponent. On this
event, uniformly in the cutoffs and angles,
\[
 \|X(\theta)\|_{\rm HS}+\|\partial_\theta X(\theta)\|_{\rm HS}
 \le C N^{1/2+\epsilon}.
\]
All real spectral parts under consideration lie in a fixed compact
subset of $(0,\infty)$. Singular-value decomposition of $X$ reduces the
linearized inverse to blocks
\[
 \frac1{u^2-z}\begin{pmatrix}1&u\\u&z\end{pmatrix},\quad u\ge0,
 \qquad -1/z,\quad -1.
\]
Splitting at a fixed large value of $u$ gives
$\|G(\theta,z)\|_{\op}\le C\eta_N^{-1}$, independently of the
matrix and its dimensions. The inverse-derivative identities therefore give
\[
 \|\partial_\theta G\|_{\op}
 \le C N^{1/2+\epsilon}\eta_N^{-2},\qquad
 \|\partial_E G\|_{\op}\le C\eta_N^{-2}.
\]
Take an angular net and, in each edge window, a real spectral net with
spacing at most $N^{-4}$, including their boundary points. After the
$O(N^2)$ cutoffs, $O(N^2)$ index pairs, and the matrix family are included,
the number of fixed-parameter tests is at most
$C N^{A_{\rm fam}+34/3+\kappa}$. The assumed estimate with, for example,
$B>D+A_{\rm fam}+20$ controls their union. The preceding derivative
bounds imply an interpolation error at most
\[
 C N^{-4}\eta_N^{-2}(1+N^{1/2+\epsilon})
 \le C N^{-13/6+24\kappa+\epsilon}
 =o(\Psi_N).
\]
This also controls both normalized partial traces, since normalized traces
of a block are bounded by its operator norm. For the deterministic
backgrounds, the quadratic Marchenko--Pastur equation gives
$|m_r'(z)|\le C\eta_N^{-1/2}$ on the edge windows, with the same
bound for $\partial_z(1+m_r)^{-1}$ because $1+m_r$ stays away from
zero. Their net-interpolation errors are likewise $o(\Psi_N)$.
This proves the simultaneous five estimates after increasing their
constant.

For each fixed $r,m,n$, set
$Q=\Prob(\Omega_N(D)^c\mid\mathcal F_{r,m,n})$. The estimate just
proved gives $\E Q\le C_DN^{-D}$. Conditional expectation is measurable
with respect to the displayed past, and ordinary Markov's inequality gives
\[
 \Prob(Q>C_DN^{-D+5})\le N^{-5}.
\]
This gives the stated history event for each fixed $r,m,n$.
Replacing $D$ by $D+5$ gives the prescribed conditional polynomial
failure exponent.
\end{proof}

We apply this extension with the northwest block $\mathcal P_0$ fixed.
Conditional on $\mathcal F_0$, the moving coordinates remain centered and
independent, and their Gaussian partners are independent of the past.
A rank-two resolvent identity also permits one moving entry to be removed.
The resulting event is simultaneous over the deterministic family and
parameters; the conditional history estimate applies to each fixed member.

\begin{thm}[Conditional local law]\label{thm:frozen-corner-local-law}
Fix
\[
 \begin{aligned}
 0&<\phi_{\min}<\phi_{\max}<\infty,
 &0&<\kappa<1/288,\\
 12\kappa&<\nu<1/12,
 &A_{\rm fam}&\ge0.
 \end{aligned}
\]
For every sufficiently large $N$, let
\[
 \phi_{\min}\le\frac MN\le\phi_{\max}.
\]
Let $\xi_{\alpha,b}$ and $w_{\alpha,b}$ be mutually independent complex random variables satisfying
\[
 \E\xi_{\alpha,b}=0,\qquad
 \E\abs{\xi_{\alpha,b}}^2=1,\qquad
 \E\xi_{\alpha,b}^2=0,
\]
with uniformly bounded fixed moments, and let every $w_{\alpha,b}$ be standard circular complex Gaussian. For $0\le m_0\le M$ and $0\le n_0\le N$, put
\[
 \mathcal P_0=\set{(\alpha,b)\mid \alpha\le m_0,\ b\le n_0},
 \qquad \mathcal E_{\rm mov}=\mathcal P_0^c.
\]
For $0\le\theta\le\pi/2$, define
\[
 X_{\rm fr}(\theta)_{\alpha,b}
 =
 \begin{cases}
 N^{-1/2}\xi_{\alpha,b},&(\alpha,b)\in\mathcal P_0,\\
 N^{-1/2}(\cos\theta\,\xi_{\alpha,b}
 +\sin\theta\,w_{\alpha,b}),&(\alpha,b)\in\mathcal E_{\rm mov}.
 \end{cases}
\]
Let
\[
 H_{\rm fr}(\theta,z)=
 \begin{pmatrix}
 -zI_N&X_{\rm fr}(\theta)^*\\
 X_{\rm fr}(\theta)&-I_M
 \end{pmatrix},
 \qquad G_{\rm fr}(\theta,z)=H_{\rm fr}(\theta,z)^{-1}.
\]
Put
\[
 \phi=\frac MN,\qquad E_+=(1+\sqrt\phi)^2,\qquad
 \eta=N^{-2/3-12\kappa},
\]
\[
 \mathcal S_{\rm edge}
 =\set{E+\ii\eta\mid\abs{E-E_+}\le5N^{-2/3+\kappa}}.
\]
Let $m_\phi$ solve
\[
 z=-\frac1{m_\phi(z)}+\frac{\phi}{1+m_\phi(z)},\qquad
 \Im m_\phi(z)>0,
\]
and set
\[
 \Pi(z)=\diag\left(m_\phi(z)I_N,
 -(1+m_\phi(z))^{-1}I_M\right).
\]
For a moving coordinate $e$, let $G_{\rm fr}^{[e]}$ denote the resolvent after setting that entry to zero.

For every $D>0$, there is an event $\Omega_{{\rm cfr},N}(D)$ satisfying
\[
 \Prob(\Omega_{{\rm cfr},N}(D)^c)\le C_DN^{-D},
\]
such that, simultaneously for all $m_0,n_0,\theta,z\in\mathcal S_{\rm edge}$, moving $e$, and indices $u,v$,
\begin{equation}
 \abs{G_{\rm fr}^{[e]}(\theta,z)_{u,v}-\Pi(z)_{u,v}}
 \le C_DN^{-1/3+\nu},
 \label{eq:frozen-deleted-entry-law}
\end{equation}
\begin{equation}
 \abs{G_{\rm fr}^{[e]}(\theta,z)_{u,u}}\le C_D,\qquad
 \abs{\Im G_{\rm fr}^{[e]}(\theta,z)_{u,u}}
 \le C_DN^{-1/3+\nu},
 \label{eq:frozen-deleted-diagonal-law}
\end{equation}
and both normalized partial traces differ from their corresponding entries of
$\Pi(z)$ by at most $C_DN^{-1/3+\nu}$.

The conclusions remain simultaneous for any deterministic family of at most
$N^{A_{\rm fam}}$ proportional dimensions, cutoffs, and edge windows. For a fixed member, let
\[
 \mathcal F_0=\sigma(\xi_{\alpha,b}:(\alpha,b)\in\mathcal P_0).
\]
There are $\mathcal F_0$-measurable history events
$\mathcal H_{{\rm cfr},N}(D)$ such that
\[
 \Prob(\mathcal H_{{\rm cfr},N}(D)^c)\le N^{-5},
\]
and, on $\mathcal H_{{\rm cfr},N}(D)$,
\[
 \Prob(\Omega_{{\rm cfr},N}(D)^c\mid\mathcal F_0)
 \le C_DN^{-D+5}.
\]
Thus the exceptional histories are summable along every dyadic subsequence.
\end{thm}

\begin{proof}
The fixed-parameter input is the complex anisotropic covariance local law
\cite[Theorem~3.14(i) and Section~2.1]{KY17}, including the stated
complex-entry extension. Choose its fixed domain parameter
$0<\tau<1/3-12\kappa$, small enough also for the dimension-ratio
bounds. Then $\eta=N^{-2/3-12\kappa}\ge N^{-1+\tau}$, so the
specified heights belong to its spectral domain. The entry laws need not
be identical. For identity population covariance, its inverse map is
\[
 f_\phi(m)=-\frac1m+\frac{\phi}{1+m}.
\]
The right critical point and edge are
\[
 m_+=-\frac1{1+\sqrt\phi},\qquad
 f_\phi(m_+)=(1+\sqrt\phi)^2=E_+.
\]
The other edge is $E_-=(1-\sqrt\phi)^2$, and
\[
 E_+-E_-=4\sqrt\phi,\qquad
 \abs{m_++1}=\frac{\sqrt\phi}{1+\sqrt\phi}.
\]
These quantities are uniformly separated from zero on
$[\phi_{\min},\phi_{\max}]$, so the right edge is uniformly regular.

For every deterministic $m_0,n_0,\theta$, the entries of $X_{\rm fr}(\theta)$ are independent and centered. Their variance is $N^{-1}$ and their complex second moment is zero because
\[
 \cos^2\theta+\sin^2\theta=1
\]
and both summands have zero complex second moment. Uniform fixed moments follow from
\[
 \abs{a+b}^p\le2^{p-1}(\abs a^p+\abs b^p).
\]
Thus the fixed-parameter local law applies uniformly.

The quadratic equation is
\[
 zm_\phi^2+(z+1-\phi)m_\phi+1=0
\]
with discriminant $(z-E_-)(z-E_+)$. On $\mathcal S_{\rm edge}$,
\[
 \abs{m_\phi(z)-m_+}\le C\sqrt{\abs{z-E_+}},
\]
and hence
\begin{equation}
 \Im m_\phi(z)\le CN^{-1/3+\kappa/2},\qquad
 \abs{m_\phi(z)}+\abs{1+m_\phi(z)}^{-1}\le C.
 \label{eq:frozen-mp-bounds}
\end{equation}
At height $\eta=N^{-2/3-12\kappa}$, the local-law control satisfies
\[
 \sqrt{\frac{\Im m_\phi(z)}{N\eta}}+\frac1{N\eta}
 \le CN^{-1/3+12\kappa}.
\]
Choose $\epsilon>0$ with $12\kappa+\epsilon<\nu$. Stochastic domination then gives the required five fixed-parameter estimates with error
$K_BN^{-1/3+\nu}$ and arbitrary polynomial failure exponent.

The row-first comparison matrix from the cited theorem becomes $\Pi(z)$ after block permutation. The averaged column law transfers to the row law through the exact identity
\[
 T_{\rm row}
 =\frac z\phi T_{\rm col}-\frac{\phi-1}{\phi},
\]
whose deterministic counterpart is
\[
 -\frac1{1+m_\phi(z)}
 =\frac z\phi m_\phi(z)-\frac{\phi-1}{\phi}.
\]
Equation \eqref{eq:frozen-mp-bounds} also supplies the diagonal and imaginary-diagonal bounds required by Lemma~\ref{lem:continuum-promotion}. Applying that lemma yields one undeleted event simultaneous over the polynomial family, all cutoffs, all coordinates, the continuum parameter $\theta$, and the spectral domain.

To delete one coordinate, fix
\[
 0<\epsilon_{\rm del}<\min(1/4,1/6-\nu).
\]
A moment bound and a union bound give, outside arbitrary polynomially small probability,
\[
 \max_{\alpha,b}(\abs{\xi_{\alpha,b}}+\abs{w_{\alpha,b}})
 \le N^{\epsilon_{\rm del}}.
\]
Thus every moving entry satisfies
\[
 \abs{x}\le2N^{-1/2+\epsilon_{\rm del}}.
\]
With
\[
 U=[\overline x\,e_b,\;xe_\alpha],\qquad
 V=[e_\alpha,\;e_b],
\]
the rank-two Woodbury identity gives
\[
 G_{\rm fr}^{[e]}
 =G_{\rm fr}+G_{\rm fr}U(I_2-V^*G_{\rm fr}U)^{-1}V^*G_{\rm fr}.
\]
The undeleted entries are bounded by \eqref{eq:frozen-mp-bounds} and the local law, so
\[
 \norm{V^*G_{\rm fr}U}\le CN^{-1/2+\epsilon_{\rm del}}\le\frac12
\]
for all sufficiently large $N$. Therefore
\[
 \max_{u,v}\abs{G_{\rm fr}^{[e]}(u,v)-G_{\rm fr}(u,v)}
 \le CN^{-1/2+\epsilon_{\rm del}}
 =o(N^{-1/3+\nu}).
\]
This transfers the entrywise, diagonal, and imaginary-diagonal estimates.

Averaging the diagonal Woodbury estimates transfers both normalized
partial traces. The row block retains the factor $z$ in the linearization,
as in the trace identities above.

The same argument covers all insertion segments used later. For a fixed
moving coordinate, let its normalized value vary from $0$ to $x$, or more
generally let the inserted complex value have modulus at most
$2N^{-1/2+\epsilon_{\rm del}}$. Relative to the undeleted matrix the
perturbation is supported on the same two linearization coordinates and has
size at most $4N^{-1/2+\epsilon_{\rm del}}$. The two-by-two inverse in
the Woodbury formula remains uniformly invertible. Consequently the
entrywise difference is $O(N^{-1/2+\epsilon_{\rm del}})=o(\Psi_N)$,
where $\Psi_N=N^{-1/3+\nu}$, simultaneously over the entire complex
insertion disk. The same event therefore controls every insertion parameter. Finite products of Green entries
and fixed-order endpoint derivatives are therefore controlled on the same
event. For a fixed finite collection of conditional moments, choose the
raw moment orders and the bad-event exponent after the collection has
been specified; the coarse inverse bound and H\"older's inequality then
control its complement.

Finally, conditional on $\mathcal F_0$, put
\[
 Q_N=\Prob(\Omega_{{\rm cfr},N}(D)^c\mid\mathcal F_0).
\]
The annealed estimate gives $\E Q_N\le C_DN^{-D}$. Markov's inequality at
$C_DN^{-D+5}$ yields
\[
 \Prob(Q_N>C_DN^{-D+5})\le N^{-5},
\]
which proves the history assertion and dyadic summability.
\end{proof}

These conditional resolvent and derivative estimates provide the inputs
for the strip interpolation and cumulant expansion below.

Throughout the conditional comparisons, the derivative orders, family
size exponents, and required remainder powers are fixed first. We then
choose one sufficiently large failure exponent $D$ and use the resulting
good-history event $\mathcal H_N(D)$ for all finitely many estimates
in that application.

\subsection{Schur complements and Gaussian interpolation}

Order the row and column indices so that the revealed northwest block
comes first. The future matrix then consists of the $r_p\times c_p$
block $A_p$, the fresh $q\times p$ rectangle $D_p$, and the cross strips
$B_p$ and $C_p$, of sizes $r_p\times p$ and $q\times c_p$.

At the Gaussian comparison point, $B_p,C_p,D_p$ are Gaussian and
independent of the revealed block. Conditioning also on $D_p$ reduces the
strip comparison to the exceptional Schur complement associated with
$A_p$. Its uniform invertibility controls the resolvent entries and the
first two derivatives of the smoothed statistic as either strip varies.

\begin{prop}[Schur complement estimates]\label{prop:gaussian-two-strip}
Fix
\[
 \kappa_0=\frac1{1000},\qquad
 \nu_0=\frac1{48},\qquad
 \omega_0=\frac{\nu_0+12\kappa_0}{2}=\frac{197}{12000}.
\]
Let $p\to\infty$, let $q\ge p$ with $q/p\le C_0$, and put
\[
 b_p=\sqrt{\frac qp},\qquad
 \eta_p=p^{-2/3-12\kappa_0}.
\]
Let
\[
 c_p,r_p\le P_p\le(\log p)^d.
\]
Let $A_p$ be deterministic of size $r_p\times c_p$ with
\[
 \norm{A_p}_{\op}\le p^{-5/16},
\]
and let $D_p$ be deterministic of size $q\times p$. For
$z\in\mathcal Z_p=\{z_-,z_+\}$, where $z=x+\ii\eta_p$ and
\[
 \abs{x-(1+b_p)^2}\le p^{-1/4},
\]
define
\[
 H_K(z)=
 \begin{pmatrix}
 -zI_p&D_p^*\\D_p&-I_q
 \end{pmatrix},
 \qquad R_p(z)=H_K(z)^{-1},
\]
with blocks $R_{cc},R_{cr},R_{rc},R_{rr}$. Assume
\[
 \max_k\abs{R_p(z)_{k,k}}\le C_0,\qquad
 \max_k\abs{\Im R_p(z)_{k,k}}\le C_0p^{-5/16}.
\]
\begin{equation}
 \abs{\frac1p\Tr R_{cc}(z)+\frac1{1+b_p}}
 +\abs{\frac1p\Tr R_{rr}(z)+b_p(1+b_p)}
 \le C_0p^{-5/16},
 \label{eq:strip-core-traces}
\end{equation}
\begin{equation}
 \norm{R_p(z)}_{\op}\le C_0p^{2/3+12\kappa_0},\qquad
 \norm{R_p(z)}_F\le C_0p^{2/3+\omega_0},
 \label{eq:strip-core-norms}
\end{equation}
and
\begin{equation}
 \max_k\set{\norm{R_p(z)e_k}_2,\norm{e_k^*R_p(z)}_2}
 \le C_0p^{1/6+\omega_0}.
 \label{eq:strip-core-vectors}
\end{equation}
Let $B_p$ be $r_p\times p$ and $C_p$ be $q\times c_p$, with independent centered circular complex Gaussian entries of variance $p^{-1}$, and assume $B_p$ and $C_p$ are independent. For $\alpha,\beta\in[0,1]$, put
\[
 X_p(\alpha,\beta)=
 \begin{pmatrix}
 A_p&\alpha B_p\\
 \beta C_p&D_p
 \end{pmatrix},
\]
and let $G_p(\alpha,\beta,z)$ be its covariance linearization resolvent.

For every $D_0>0$, outside probability at most $p^{-D_0}$, uniformly in
$\alpha,\beta,z$ and every coordinate belonging to either Gaussian strip, if that coordinate joins column index $b$ to row index $a$, then
\begin{equation}
 \abs{G_p(z)_{b,a}}+\abs{G_p(z)_{a,b}}
 \le Cp^{-61/200},
 \label{eq:strip-cross}
\end{equation}
\[
 \abs{G_p(z)_{b,b}}+\abs{G_p(z)_{a,a}}\le C,
\]
\begin{equation}
 \abs{\Im G_p(z)_{b,b}}+\abs{\Im G_p(z)_{a,a}}
 \le Cp^{-61/200}.
 \label{eq:strip-diagonal}
\end{equation}
The event may be chosen simultaneously for a fixed polynomial number of blocks and endpoint pairs.

For one strip coordinate, vary either its real or imaginary normalized scalar value $\theta$, and define
\[
 Y_p(\theta)=\frac1\pi\int_{\Re z_-}^{\Re z_+}
 \Im\Tr(P_{\rm col}G_p(\theta,x+\ii\eta_p))\,\dd x.
\]
Then
\begin{equation}
 \abs{\partial_\theta Y_p(0)}
 +\abs{\partial_\theta^2Y_p(0)}
 \le Cp^{-61/200}.
 \label{eq:strip-poisson-derivatives}
\end{equation}
If
\[
 y_p=\frac{Y_p-y_0}{\delta_p},\qquad
 \delta_p^{-1}\le C_\delta\log p,
\]
then
\begin{equation}
 \abs{\partial_\theta y_p(0)}
 +\abs{\partial_\theta^2y_p(0)}
 \le Cp^{-3/10}.
 \label{eq:strip-normalized-derivatives}
\end{equation}
The same conclusion holds after adding a smooth spectral term whose first two coordinate derivatives vanish on the good event.

If additionally $D_p$ is a covariance core and $\norm{D_p}_{\op}\le C_0$, then, with
\[
 Q_p=(D_p^*D_p-zI_p)^{-1},
\]
\begin{equation}
 R_p(z)=
 \begin{pmatrix}
 Q_p&Q_pD_p^*\\
 D_pQ_p&-I_q+D_pQ_pD_p^*
 \end{pmatrix},
 \label{eq:strip-core-block}
\end{equation}
and \eqref{eq:strip-core-norms}--\eqref{eq:strip-core-vectors} follow from the stated height and the preceding diagonal-imaginary bound.
\end{prop}

\supplementproof{supp.H.6}{Section H.6}

We use the normalizations $N^{-1/2}$ for the common array,
$n_k^{-1/2}$ for a full matrix, and $p_k^{-1/2}$ for its southeast
submatrix. Accordingly, the raw endpoints $\widehat A_k,\widehat B_k$
and height $\widehat\eta_k$ are divided by $N$, $n_k$, or $p_k$.
The next theorem gives the corresponding exact identities for Green
blocks, smoothed counts, and raw-coordinate derivatives.

\begin{thm}[Gaussian interpolation]\label{thm:natural-tag-interpolation}
Fix
\[
 \kappa_0=\frac1{1000},\qquad \nu_0=\frac1{48},
\]
a fixed $d>0$, and $0<c_{\rm dim}<C_{\rm dim}<\infty$. Let
$\mathcal K_N$ be a polynomial-size family of physical matrices obtained from a common raw array by restriction, possible conjugate transpose, and deterministic row and column permutations. In the chosen orientation, matrix $k$ has
\[
 n_k=p_k+c_k,\qquad m_k=q_k+r_k,
\]
where
\begin{equation}
 c_{\rm dim}N\le p_k\le q_k\le C_{\rm dim}N,\qquad
 0\le c_k+r_k\le P_N,\qquad P_N\le(\log N)^d.
 \label{eq:natural-dimensions}
\end{equation}
Assume $m_k/n_k$ converges uniformly to a positive limit and $q_k/p_k$ is uniformly bounded.

Let $\mathcal R_N,\mathcal C_N$ be common row and column universes with
\[
 \abs{\mathcal R_N}+\abs{\mathcal C_N}\le C_{\rm amb}N.
\]
Let $(x_e)$ be mutually independent centered complex variables of variance one, zero complex second moment, and uniformly bounded fixed moments. Fix past sets
$\mathcal R_{\rm past}$ and $\mathcal C_{\rm past}$ satisfying
\[
 \abs{\mathcal R_{\rm past}}+\abs{\mathcal C_{\rm past}}\le P_N,
\]
and put
\[
 \mathcal E_{\rm past}
 =\mathcal R_{\rm past}\times\mathcal C_{\rm past},\qquad
 \mathcal G=\sigma(x_e:e\in\mathcal E_{\rm past}).
\]
Every matrix contains these past sets, places them first, and has no other frozen coordinate. Thus its frozen block is the $r_k\times c_k$ copy or transpose of the common raw past matrix.

Let $(w_e)$ be an independent standard circular complex Gaussian array. Frozen coordinates remain $x_e$, while moving coordinates satisfy
\begin{equation}
 \xi_e(t)=\ee^{-t/2}x_e+\sqrt{1-\ee^{-t}}\,w_e,
 \qquad 0\le t\le8\log N.
 \label{eq:natural-ou-flow}
\end{equation}
The common-normalized matrix $X_k(t)$ has entries $N^{-1/2}x_e$ on its frozen block and $N^{-1/2}\xi_e(t)$ elsewhere. Put
\begin{equation}
 a_k=\sqrt{\frac N{n_k}},\qquad
 \widetilde X_k(t)=a_kX_k(t).
 \label{eq:natural-first-scaling}
\end{equation}

Let $\widehat A_k<\widehat T_k\le\widehat B_k$, $\widehat\eta_k>0$, and
\[
 \widehat T_k-\widehat A_k\ge4\widehat\eta_k.
\]
Define
\[
 A_k=\frac{\widehat A_k}{N},\quad
 B_k=\frac{\widehat B_k}{N},\quad
 \eta_k=\frac{\widehat\eta_k}{N},
\]
\[
 \widetilde A_k=\frac{\widehat A_k}{n_k},\quad
 \widetilde B_k=\frac{\widehat B_k}{n_k},\quad
 \widetilde\eta_k=\frac{\widehat\eta_k}{n_k}.
\]
Assume
\[
 \widehat\eta_k=p_k^{1/3-12\kappa_0},
\]
\begin{equation}
 \begin{aligned}
 &\abs{\widehat A_k-(\sqrt{m_k}+\sqrt{n_k})^2}
 +\abs{\widehat B_k-(\sqrt{m_k}+\sqrt{n_k})^2}\\
 &\qquad\le C\left[
 N^{1/3}(\log N)^{1/3}+P_N+N^{1/3+\kappa_0}\right],
 \end{aligned}
 \label{eq:natural-project-geometry}
\end{equation}
and
\begin{equation}
 \abs{\frac{p_k}{n_k}-1}
 +\abs{\frac{q_k}{m_k}-1}
 \le C\frac{P_N}{N}.
 \label{eq:natural-core-ratios}
\end{equation}
Put
\[
 E_{{\rm full},k}=(1+\sqrt{m_k/n_k})^2,
\]
\[
 \overline A_k=\frac{\widehat A_k}{p_k},\quad
 \overline B_k=\frac{\widehat B_k}{p_k},\quad
 \overline\eta_k=\frac{\widehat\eta_k}{p_k},\quad
 E_{{\rm core},k}=(1+\sqrt{q_k/p_k})^2,
\]
and assume
\begin{equation}
 \max\{\abs{\widetilde A_k-E_{{\rm full},k}},
       \abs{\widetilde B_k-E_{{\rm full},k}}\}
 \le5n_k^{-2/3+\kappa_0},
 \label{eq:natural-full-domain}
\end{equation}
\begin{equation}
 \max\{\abs{\overline A_k-E_{{\rm core},k}},
       \abs{\overline B_k-E_{{\rm core},k}}\}
 \le5p_k^{-2/3+\kappa_0}.
 \label{eq:natural-core-domain}
\end{equation}

For a matrix $X$ of the specified dimensions, let
\[
 H_X(z)=
 \begin{pmatrix}
 -zI_{n_k}&X^*\\X&-I_{m_k}
 \end{pmatrix},
 \qquad G_X(z)=H_X(z)^{-1},
\]
and define
\begin{equation}
 Y_k(X)=\frac1\pi\int_{A_k}^{B_k}
 \Im\Tr G_{X,cc}(x+\ii\eta_k)\,\dd x.
 \label{eq:natural-poisson}
\end{equation}
Let $g_k\colon\R\to[0,1]$ be smooth, zero on $(-\infty,B_k]$, and one on
$[B_k+\eta_k,\infty)$. Set
\begin{equation}
 Z_k(X)=Y_k(X)+\Tr g_k(X^*X).
 \label{eq:natural-complete-detector}
\end{equation}

Then the following hold.

\begin{enumerate}
\item At $\widetilde z=a_k^2z$,
\[
 G_{k,cc}(z)=a_k^2\widetilde G_{k,cc}(\widetilde z),\quad
 G_{k,cr}(z)=a_k\widetilde G_{k,cr}(\widetilde z),
\]
\begin{equation}
 G_{k,rc}(z)=a_k\widetilde G_{k,rc}(\widetilde z),\quad
 G_{k,rr}(z)=\widetilde G_{k,rr}(\widetilde z).
 \label{eq:natural-green-scaling}
\end{equation}
With $\widetilde g_k(u)=g_k(u/a_k^2)$,
\begin{equation}
 Z_k(X_k)=\widetilde Z_k(\widetilde X_k).
 \label{eq:natural-detector-invariance}
\end{equation}
If $h$ is a common-normalized entry and $\widetilde h=a_kh$, then
\[
 \partial_h^jZ_k
 =a_k^j\partial_{\widetilde h}^j\widetilde Z_k.
\]
For a raw coordinate $Q_e$,
\begin{equation}
 \partial_{Q_e}^jZ_k
 =N^{-j/2}\partial_h^jZ_k
 =n_k^{-j/2}\partial_{\widetilde h}^j\widetilde Z_k.
 \label{eq:natural-derivative-scaling}
\end{equation}
The deterministic column background scales by $a_k^2$, the row background is unchanged, and the cross background remains zero.

\item Put
\[
 \eta_{0,k}=n_k^{-2/3-12\kappa_0}.
\]
Then
\begin{equation}
 \widetilde\eta_k
 =[1+O(P_N/N)]\eta_{0,k}.
 \label{eq:natural-height-comparison}
\end{equation}
On $\mathcal G$-measurable histories with summable complements, the one-moving-coordinate-deleted natural resolvents at
$\widetilde A_k+\ii\widetilde\eta_k$ and
$\widetilde B_k+\ii\widetilde\eta_k$ have bounded diagonals and
$O(N^{-5/16})$ cross, imaginary-diagonal, and centered-diagonal entries, uniformly in every matrix, time, and deletion label. These estimates persist along
\[
 \widetilde h=\frac{sQ_e}{\sqrt{n_k}},\qquad0\le s\le1,
\]
on $\abs{Q_e}\le N^{1/16}$. Every fixed mixed derivative of
\eqref{eq:natural-poisson} through order six is $O(N^{-5/16})$. After division by a scale $\delta_N$ with $\delta_N^{-1}\le C\log N$, these derivatives are $O(N^{-3/10})$.

\item At the exact Gaussian comparison point,
\begin{equation}
 X_k(\alpha,\beta)
 =N^{-1/2}
 \begin{pmatrix}
 A_k^{\rm raw}&\alpha B_k^{\rm raw}\\
 \beta C_k^{\rm raw}&D_k^{\rm raw}
 \end{pmatrix}.
 \label{eq:natural-gaussian-block}
\end{equation}
Conditionally on $\mathcal G$, the three moving blocks are independent standard circular Gaussian blocks and are independent of $A_k^{\rm raw}$. Put
\begin{equation}
 b_k=\sqrt{\frac{n_k}{p_k}},\qquad
 \overline X_k=b_ka_kX_k=\sqrt{\frac N{p_k}}X_k.
 \label{eq:natural-core-scaling}
\end{equation}
The fresh core and strips then have entry variance $p_k^{-1}$, while
\[
 A_k^{\rm past}=p_k^{-1/2}A_k^{\rm raw}.
\]
There are $\mathcal G$-measurable events $\mathcal G_{{\rm past},N}$ with summable complements such that
\begin{equation}
 \norm{A_k^{\rm raw}}_{\op}\le p_k^{3/16}
 \label{eq:natural-past-bound}
\end{equation}
for every $k$. Hence
\[
 \norm{A_k^{\rm past}}_{\op}\le p_k^{-5/16}.
\]
At the two core endpoints, the Gaussian core satisfies
\eqref{eq:strip-core-traces}--\eqref{eq:strip-core-vectors} and has bounded operator norm with arbitrarily high polynomial probability. Proposition~\ref{prop:gaussian-two-strip} therefore applies uniformly for
$(\alpha,\beta)\in[0,1]^2$, including the boundary.

In exceptional/core order,
\[
 G_{EE}=S^{-1},\quad G_{EK}=-S^{-1}UR,\quad
 G_{KE}=-RU^*S^{-1},
\]
\begin{equation}
 G_{KK}=R+RU^*S^{-1}UR,
 \label{eq:natural-exact-schur}
\end{equation}
where
\[
 R=H_K^{-1},\qquad S=H_E-URU^*,
\]
\begin{equation}
 H_E=
 \begin{pmatrix}
 -zI&(A_k^{\rm past})^*\\
 A_k^{\rm past}&-I
 \end{pmatrix},
\qquad
 U=
 \begin{pmatrix}
 0&\beta\overline C_k^*\\
 \alpha\overline B_k&0
 \end{pmatrix}.
 \label{eq:natural-schur-data}
\end{equation}
The reference values
\[
 s_{\rm col}
 =-z+\beta^2\sqrt{\frac{q_k}{p_k}}
       \left(1+\sqrt{\frac{q_k}{p_k}}\right),
\]
\begin{equation}
 s_{\rm row}
 =-1+\frac{\alpha^2}{1+\sqrt{q_k/p_k}}
 \label{eq:natural-reference-values}
\end{equation}
remain uniformly separated from zero on $[0,1]^2$. Every strip-coordinate first and second insertion derivative of the Poisson statistic is
$O(N^{-61/200})$. Gaussian truncation at $N^{1/16}$ and Woodbury transfer preserve this estimate along the insertion segment. After division by $\delta_N$, the derivatives are $O(N^{-3/10})$. The same bound applies to the complete smoothed statistic when a cutoff derivative is nonzero, since then the upper spectral term is locally constant.

\item Let $\mathcal I_N=\mathcal K_N$, assume
\[
 \abs{\mathcal I_N}\le N^\theta,\qquad \theta<1/3,
\]
fix $R_{\rm c}>1$ and $\tau_*<\infty$, and choose
$\rho\in C^\infty(\R;[0,1])$ equal to one on $(-\infty,-1]$ and zero on
$[0,\infty)$. Put
\[
 \lambda_r=R_{\rm c}^{-r},
\]
choose $J_N=O(\log N)$ such that
\[
 \abs{\mathcal I_N}\lambda_{J_N}\le N^{-12},
\]
and define
\begin{equation}
 \delta_N=\frac{1/16}{J_N+2},\qquad
 y_k(X)=\frac{Z_k(X)-1/16}{\delta_N}.
 \label{eq:natural-ladder-scale}
\end{equation}
Let
\[
 \mathcal E_1
 =\mathcal R_{\rm past}\times
   (\mathcal C_N\setminus\mathcal C_{\rm past}),
\]
\[
 \mathcal E_2
 =(\mathcal R_N\setminus\mathcal R_{\rm past})
   \times\mathcal C_{\rm past},
\]
after discarding labels occurring in no matrix. Then
\begin{equation}
 \abs{\mathcal E_1\cup\mathcal E_2}
 \le C_{\rm amb}P_NN.
 \label{eq:natural-strip-cardinality}
\end{equation}
Let $X_k^{[s,t]}$ multiply the coordinates in $\mathcal E_1$ by $s$ and those in $\mathcal E_2$ by $t$. Define
\begin{equation}
 S(s,t)=\sum_{k\in\mathcal I_N}\sum_{r=0}^{J_N}
 \lambda_r\rho\bigl(y_k(X_k^{[s,t]})+J_N-r\bigr),
 \label{eq:natural-global-count}
\end{equation}
and $F=\exp(-\tau S)$. Use the paths
\[
 (s,t)=(\sqrt v,1),\qquad (s,t)=(0,\sqrt v),
 \qquad0\le v\le1.
\]
For $j\in\{1,2\}$, let
\[
 \Phi_j(v,\tau)=\E[\exp(-\tau S_j(v))\mid\mathcal G],
\]
\begin{equation}
 \mathcal D_j(v,\tau)
 =\E[\tau S_j(v)\exp(-\tau S_j(v))\mid\mathcal G].
 \label{eq:natural-laplace-observables}
\end{equation}
For $e\in\mathcal E_j$, let $F_{j,e}(h,\tau)$ be the actual coordinate section of \eqref{eq:natural-global-count}, where
\[
 h_e(v)=\frac{\sqrt v}{\sqrt N}(Q_{e,1}+\ii Q_{e,2})
\]
and the real Gaussian coordinates have variance $1/2$. Then
\begin{equation}
 \partial_v\Phi_j(v,\tau)
 =\frac1{4N}\sum_{e\in\mathcal E_j}\sum_{a=1}^2
 \E[\partial_{h_a}^2F_{j,e}(h_e(v),\tau)\mid\mathcal G],
 \label{eq:natural-heat-identity}
\end{equation}
and
\begin{equation}
 \begin{aligned}
 \partial_{h_a}^2F_{j,e}
 =F_{j,e}\Biggl[
 &-\tau\sum_{k,r}\lambda_r
 \left\{\rho' y_{k,a,a}+\rho''y_{k,a}^2\right\}\\
 &+\tau^2\left(\sum_{k,r}\lambda_r\rho'y_{k,a}\right)^2
 \Biggr].
 \end{aligned}
 \label{eq:natural-heat-expansion}
\end{equation}
Every cutoff derivative is evaluated at $y_k+J_N-r$. The identities extend continuously to $v=0$. No third-order or cubic term occurs.

There are nonnegative functions $a_j,e_j$ such that, on $\mathcal G$-measurable histories with summable complements, for almost every $v\in[0,1]$ and every fixed $0<\tau\le\tau_*$,
\begin{equation}
 \abs{\partial_v\Phi_j(v,\tau)}
 \le a_j(v)\mathcal D_j(v,\tau/2)+e_j(v),
 \label{eq:natural-differential-inequality}
\end{equation}
\begin{equation}
 \int_0^1a_j(v)\,\dd v
 \le CP_N(N^{-3/10}+N^{-3/5})\le N^{-1/4},
 \label{eq:natural-a-integral}
\end{equation}
\begin{equation}
 \int_0^1e_j(v)\,\dd v
 \le C_{\tau_*}N^{-9}.
 \label{eq:natural-e-integral}
\end{equation}
The endpoint of the first path is the beginning of the second. The endpoint of the second is
\[
 \diag(N^{-1/2}A_k^{\rm raw},N^{-1/2}D_k^{\rm raw})
\]
in every matrix.
\end{enumerate}
\end{thm}

\supplementproof{supp.C.1}{Section C.1}

With the northwest block fixed, the paths $(\sqrt v,1)$ and
$(0,\sqrt v)$ vary the strips $\mathcal E_1$ and $\mathcal E_2$.
Read in the removal direction, they join the full matrix to the
block-diagonal matrix through their common one-strip endpoint.
The Gaussian heat identity and the Schur bounds give the displayed
differential inequality, uniformly over the matrix family.

For the homogeneous Ornstein--Uhlenbeck flow, the conditional Laplace
comparison has five error contributions: $N^{-3/10}$ from the
fourth-cumulant section derivatives, $N^{-1/2+1/100}$ from unmatched
third-cumulant terms, $N^{-4/5}$ from the fifth cumulant,
$N^{-79/100}$ from the sixth-order Taylor remainder, and $N^{-9}$
from exceptional histories and truncation. Their integrals over
$[0,8\log N]$ tend to zero. We next estimate the third-cumulant terms
while keeping the same conditioning sigma-algebra.

\subsection{Cumulant expansion for unmatched terms}

In the matrix family below, the superscript $k$ is a matrix index.
For each positive integer $r$, $\operatorname{Sym}(r)$ denotes the
permutations of $\{1,\ldots,r\}$.

We expand at an unrevealed row or column index, using either displayed
position of that index in a Green factor. Each factor retains its
conjugation sign. Covariance contractions preserve odd multiplicity
of an unrevealed index. The two recursive classes distinguish the
original expansion, $h=0$, from the terms after removal of the first
relative third-cumulant factor $N^{-1/2}$, $h=1$. A second cubic
expansion then supplies the additional factor needed for the estimate.

The estimate controls a joint contraction of cutoff derivatives and
resolvent entries. For fixed initial factor and derivative budgets, repeated
expansion bounds its unmatched expectation by $O(N^{-1+\varepsilon})$.
The third-cumulant source has an exterior factor $N^{1/2}$ and is therefore
$O(N^{-1/2+\varepsilon})$. The definitions below specify how the complete
spectral tensors are retained in this expansion; Supplementary Lemma C.1
provides the finite induction.

\begin{thm}[Recursion for unmatched terms]\label{thm:four-orientation-recursion}
Fix positive constants
$c_{\rm dim},C_{\rm dim},d_{\rm past},\tau_*,A_{\rm fam}$, positive integers
$F_0,k_0,K_{\rm der}$, and
\[
 0<\kappa_0<1/288,\qquad12\kappa_0<\nu<1/12.
\]
Let $(\xi_e)_e$ and $(w_e)_e$ be mutually independent arrays,
each with independent entries. For every raw coordinate $e$, including
those in the old block, assume
\[
 \E\xi_e=0,\qquad \E|\xi_e|^2=1,\qquad \E\xi_e^2=0,
 \qquad \sup_e\E|\xi_e|^p<\infty
\]
for every fixed $p\ge1$. Let every $w_e$ be standard circular complex
Gaussian. Let $\mathcal R_{\rm old}$ and $\mathcal C_{\rm old}$ satisfy
\[
 P_N:=\abs{\mathcal R_{\rm old}}+\abs{\mathcal C_{\rm old}}
 \le(\log N)^{d_{\rm past}},
\]
put
\[
 \mathcal P_{\rm old}
 =\mathcal R_{\rm old}\times\mathcal C_{\rm old},
 \qquad
 \mathcal F_{\rm old}=\sigma(\xi_e:e\in\mathcal P_{\rm old}),
\]
and keep $\zeta_e(t)=\xi_e$ for $e\in\mathcal P_{\rm old}$.
For moving coordinates $e\notin\mathcal P_{\rm old}$, put
\begin{equation}
 \zeta_e(t)=\ee^{-t/2}\xi_e+\sqrt{1-\ee^{-t}}\,w_e,
 \qquad0\le t\le8\log N.
 \label{eq:recursion-flow}\end{equation}
Conditionally on $\mathcal F_{\rm old}$, moving coordinates are independent, centered, have unit second absolute moment, zero complex second moment, and uniformly bounded fixed moments.

There are at most $N^{A_{\rm fam}}$ matrices, endpoints, deletions,
smoothed statistics, and derivative contributions. The matrix sets,
endpoints, and cutoff parameters are fixed independently of the moving
coordinates. Matrix $k$ retains its row and column orientation, has dimensions
\[
 c_{\rm dim}N\le m_k,n_k\le C_{\rm dim}N,
\]
and
\begin{equation}\label{eq:comparison-linearization-resolvent}
 \begin{aligned}
 X^k&=n_k^{-1/2}\zeta|_{\mathcal R_k\times\mathcal C_k},\\
 H^k(z)&=
 \begin{pmatrix}
 -zI_{n_k}&(X^k)^*\\X^k&-I_{m_k}
 \end{pmatrix},
 &G^k(z)&=H^k(z)^{-1}.
 \end{aligned}
\end{equation}
At every endpoint,
\begin{equation}
 \Im z=n_k^{-2/3-12\kappa_0},\qquad
 \abs{\Re z-(1+\sqrt{m_k/n_k})^2}
 \le5n_k^{-2/3+\kappa_0}.
 \label{eq:recursion-endpoints}\end{equation}
One moving coordinate may be deleted.

For $\varsigma\in\{+1,-1\}$, put
\begin{equation}\label{eq:comparison-flagged-background}
 G^{(k,+1)}=G^k,\quad G^{(k,-1)}=\overline{G^k},\qquad
 \widetilde m^{(k,+1)}=\widetilde m^k,\quad
 \widetilde m^{(k,-1)}=\overline{\widetilde m^k}.
\end{equation}
Each Green function factor retains its own conjugation sign.

Each smoothed statistic has a deterministic affine integral part, with a
fixed number of terms and uniformly bounded coefficients, built from
projected resolvents of the form
\[
 Y=\frac1\pi\int_{x_-}^{x_+}
 \Im\Tr(P_{\rm col}G(x+\ii\eta))\,\dd x.
\]
A statistic may also contain the smooth upper-spectral term of
Proposition~\ref{prop:poisson-detector}. Its positive raw
derivatives through order $K_{\rm der}+6$ vanish on the common rigidity
event and have polynomial bounds elsewhere. These derivatives are
retained in the expansion and estimated with the exceptional-event
terms in Supplementary Section C.2. For the integral part and bounded
rank-two directions $S_1,\ldots,S_r$,
\begin{equation}
 D_{S_1}\cdots D_{S_r}Y
 =\frac{(-1)^r}{\pi r}\Delta\Im
 \sum_{\sigma\in\operatorname{Sym}(r)}
 \Tr(S_{\sigma(1)}G\cdots S_{\sigma(r)}G).
 \label{eq:recursion-mixed-endpoint}
\end{equation}
For weighted cutoff observables, use one or two sums with
\begin{equation}
 \abs{\mathcal I_N}R_{\rm c}^{-J_N}\le1,\qquad
 \delta_N^{-1}=N^{o(1)}.
 \label{eq:recursion-ladder}
\end{equation}
Every fixed-order derivative-tensor $\ell^1$ norm is then at most $N^\epsilon$ for every fixed $\epsilon>0$.

Fix the global cutoff function before expanding any coordinate. For one
weighted sum it is
\[
 V_N(Y)=\exp\left\{-\tau\sum_{i\in\mathcal I_N}
             \sum_{r=0}^{J_N}R_{\rm c}^{-r}
 \rho\left(\frac{Y_i-a_i}{\delta_N}+J_N-r\right)\right\}.
\]
For two sums we use the product of their Laplace factors, with distinct
formal spectral variables for the two copies. These variables may be
evaluated on matrices with common entries. The spectral coefficient tensors
are the derivatives of these fixed functions, resolved by their cutoff
levels as in Lemma~\ref{lem:all-order-ladder}. A fixed finite family of
scalar cutoffs with bounded derivatives is also allowed, as in
Section~\ref{sec:fixed-level}. Their exact differentiation rules and
contractions are given in Supplementary Section C.2.

A joint resolvent term is a finite linear combination of
\begin{equation}
 Q=N^{-s-\ell}\sum_{\boldsymbol v}^{*}
       \sum_{\boldsymbol i,\boldsymbol r}
 \chi_{\boldsymbol v,\boldsymbol i,\boldsymbol r}
 C_{\boldsymbol v,\boldsymbol i,\boldsymbol r}
 \mathcal M_{\boldsymbol i,\boldsymbol r}(Y)
 \prod_{u=1}^{F}Z_u.
 \label{eq:recursion-monomial}
\end{equation}
Here $\boldsymbol v=(v_1,\ldots,v_s)$ consists of free row and column
indices, each ranging over a set of size comparable with $N$. Every free
index occurs in an explicit Green factor or an opened trace, so $s\le2F$;
indices absent from all factors are averaged into $C$. The star
specifies distinct free indices of the same type, also distinct from fixed
indices of that type. The integer $\ell\ge0$ records normalized sums lost
through index identifications. Each spectral index in
$\boldsymbol i$ is a slot of $\mathcal M$, and $\boldsymbol r$ records
its distinguished cutoff levels. Matrix and endpoint assignments in the
Green factors are fixed functions of these slots or fixed external indices.
All these contractions occur inside the same conditional expectation.
The supports of differentiated cutoffs remain in their derivative tensors.

The masks $\chi$ take values in $[0,1]$ and are measurable with respect to
$\mathcal F_{\rm old}$ and the displayed discrete indices. They encode
membership, deletions, conjugation signs, and index restrictions; their
moving-coordinate derivatives vanish. The coefficients $C$ have the same
measurability, are bounded by one, and retain the coordinate indices of all
cumulants. A finite linear combination has total absolute coefficient
weight at most one. Fixed bounds are absorbed into $C_F$.

Each $Z_u$ is a Green entry with its conjugation sign, a centered diagonal, or a centered
normalized partial trace. The degree counts off-diagonal and centered
factors. A partial trace is opened as a normalized diagonal sum before
index coincidences are resolved. An identification retains its factor
$N^{-1}$ in $N^{-\ell}$. The tensors $\mathcal M$ belong to the specified
derivative families; their spectral and cutoff indices retain the complete
$\ell^1$ contraction of Lemma~\ref{lem:all-order-ladder}.

A row or column index is unrevealed when it is outside
$\mathcal R_{\rm old}$ or $\mathcal C_{\rm old}$, respectively. It is
unmatched when its multiplicity among Green-factor slots is odd, as in
\cite[Definition~4.2]{SX21}. We use terms with an unrevealed unmatched
index. The class $h=0$ consists of the unmarked terms
\eqref{eq:recursion-monomial} and their covariance descendants. An $h=1$
term is obtained by one first third-cumulant expansion of an $h=0$ term,
followed by covariance expansions. It carries the coordinate, ordered real
triple, and matrix indices of that expansion. Its normalization is written
as
\[
 n_k^{-1/2}n_h^{-1/2}n_j^{-1/2}\sum_c
 =N^{-1/2}\left(\frac{N}{n_k}\frac{N}{n_h}\frac{N}{n_j}\right)^{1/2}
       \frac1N\sum_c.
\]
Only the displayed factor $N^{-1/2}$ is removed in defining $h=1$;
the bounded cumulant and all dimension ratios stay in the coefficient.
The degree-lowering terms \eqref{eq:revision-cubic-drop}--\eqref{eq:revision-single-offdiag}
are estimated before entering this class. A second third-cumulant branch
is estimated directly. Thus the mark records exactly one such expansion.

Let $\mathcal R_{d,F,k,h}$ be the conditional essential supremum of
$|\E[Q\mid\mathcal F_{\rm old}]|$ over these terms of degree at least $d$,
with at most $F$ factors and total spectral derivative order at most $k$.
The scalar case uses the derivatives of its fixed test function in the
same definition. Conditional expectations always include $\mathcal M$ and
all its contracted indices.

For every fixed $D>0$ there are a simultaneous event $\Omega_N(D)$ and an
$\mathcal F_{\rm old}$-measurable event $\mathcal H_N(D)$ such that
\begin{equation}
 \begin{split}
 \Prob(\Omega_N(D)^c)&\le C_DN^{-D-6},\qquad
 \Prob(\mathcal H_N(D)^c)\le C_DN^{-6},\\
 \Prob(\Omega_N(D)^c\mid\mathcal F_{\rm old})&\le N^{-D}
       \quad\hbox{on }\mathcal H_N(D).
 \end{split}
 \label{eq:recursion-good-event}
\end{equation}
On $\Omega_N(D)$, all deleted, undeleted, and insertion-segment diagonals
are bounded, while off-diagonal entries, centered diagonals, imaginary
diagonals, and centered partial traces are
\begin{equation}
 O(\Psi_N),\qquad \Psi_N=N^{-1/3+\nu}.
 \label{eq:recursion-local-scale}
\end{equation}
The expansion and derivative orders are fixed first. We then choose $D$
to dominate the polynomial moment bounds of the resulting finite family,
and use this one event $\mathcal H_N(D)$ throughout the calculation.

The following assertions hold.

\begin{enumerate}
\item An unrevealed row label permits both row-side resolvent identities using only centered moving coordinates. An unrevealed column label permits both column-side identities.

\item For the selected factor in matrix $k$ with endpoint $z$ and conjugation sign $\varsigma_0$, the coefficient of the original row term is
\begin{equation}
 1+\widetilde m^{(k,\varsigma_0)}(z),
 \label{eq:recursion-row-self}\end{equation}
and the coefficient of the original column term is
\begin{equation}
 -\frac1{\widetilde m^{(k,\varsigma_0)}(z)}.
 \label{eq:recursion-column-self}\end{equation}
For a left or right displayed position of the selected index set $s_0=1$ or $s_0=-1$,
respectively. The same-degree contraction slot is
\[
 s_1=-s_0\varsigma_0\varsigma_1.
\]
A same-degree pair removes two occurrences of the unrevealed label and creates two
of the summed opposite-side label, preserving odd multiplicity.

The row same-degree coefficient is, up to the derivative sign,
\begin{equation}
 \sqrt{\frac{n_k}{n_h}}\,
 \frac1{1+\widetilde m^{(k,\varsigma_0)}(z)}
 \frac1{1+\widetilde m^{(h,\varsigma_1)}(z')},
 \label{eq:recursion-row-neutral}\end{equation}
and the column same-degree coefficient is
\begin{equation}
 \sqrt{\frac{n_k}{n_h}}\,
 \widetilde m^{(k,\varsigma_0)}(z)
 \widetilde m^{(h,\varsigma_1)}(z').
 \label{eq:recursion-column-neutral}\end{equation}
All coefficients are uniformly bounded, and the leading coefficients are
uniformly invertible. After index-collision terms are separated, every other
covariance contraction raises degree.

\item Write
\[
 q_{e,1}=\Re\zeta_e,\qquad q_{e,2}=\Im\zeta_e,
\]
so their covariance matrix is $I_2/2$. Let
$K_{e,abc}(t)$ be their ordered third cumulants. Every first third-cumulant term retains all eight ordered triples of real-coordinate derivatives. If its selected factor and two differentiated factors lie in matrices $k,h,j$, then its normalization is
\begin{equation}
 n_k^{-1/2}n_h^{-1/2}n_j^{-1/2}.
 \label{eq:recursion-cubic-normalization}\end{equation}
The fresh opposite-side label has multiplicity three. After the degree-lowering terms estimated below have been paid,
each remaining unrevealed term enters the $h=1$ class at the required
minimum degree. If it is old, its $h=0$ contribution is
\begin{equation}
 O(P_NN^{-3/2+\epsilon})=O(N^{-1}).
 \label{eq:recursion-old-cubic}\end{equation}
A second cubic in the $h=1$ class costs $N^{-1/2+\epsilon}$, and the outer first third-cumulant factor makes it $N^{-1+\epsilon}$ in the $h=0$ class. Terms with at least three raw derivatives in one expansion, fixed cumulant remainders, deletions, or free-label identifications are $O(N^{-1+\epsilon})$ in the $h=0$ class.

\item Put
\[
 F_+=4F+10,\qquad k_+=k+4F+20.
\]
If $K_{\rm der}\ge k_+$, then
\begin{equation}
 \mathcal R_{d,F,k,0}
 \le C_F\left[
 \mathcal R_{d+1,F_+,k_+,0}
 +N^{-1/2}\mathcal R_{d,F_+,k_+,1}\right]
 +C_{F,\epsilon}N^{-1+\epsilon},
 \label{eq:recursion-h0}\end{equation}
\begin{equation}
 \mathcal R_{d,F,k,1}
 \le C_F\mathcal R_{d+1,F_+,k_+,1}
 +C_{F,\epsilon}N^{-1/2+\epsilon}.
 \label{eq:recursion-h1}\end{equation}
For a smooth cutoff observable $\mathcal F_N(t,\tau)$, define
\begin{equation}
 \mathcal L_3^{\rm mov}
 =-\frac14\sum_{e\in\mathcal E_{\rm mov}}
 \E\left[
 \sum_{a,b,c=1}^2K_{e,abc}(t)
 \partial_{q_{e,a}}\partial_{q_{e,b}}\partial_{q_{e,c}}
 \mathcal F_N(t,\tau)
 \,\middle|\,\mathcal F_{\rm old}\right].
 \label{eq:recursion-third-source}\end{equation}
Each product-rule monomial contains three rank-two source insertions, so the source row and column have multiplicity three. At least one is unrevealed, and every term starts \eqref{eq:recursion-h0}.
\end{enumerate}

The endpoint and trace conventions are those of
\eqref{eq:recursion-endpoints} and \eqref{eq:comparison-linearization-resolvent}.
Corollaries~\ref{prop:high-orientation-specialization} and
\ref{prop:four-orientation-recursion} specify the changes of scale and the
structural-zero-mode corrections in the two applications.
\end{thm}

\supplementproof{supp.C.2}{Section C.2}

\paragraph{Choice of derivative order.}
For initial budgets $F_0,k_0$, define
\[
 F_{j+1}=4F_j+10,\qquad k_{j+1}=k_j+4F_j+20,\qquad 0\le j<4.
\]
Choose $K_{\rm der}\ge k_4+6$ before the moment and good-history
exponents. In a same-degree covariance branch, the multiplicity of the
selected unrevealed index decreases by two. At multiplicity one that
branch ends. Each remaining recursive step increases the degree or
changes $h$ from zero to one. Starting at degree at least one, at most
four such steps reach degree four or a directly estimated remainder.
At degree four, Lemma~\ref{lem:all-order-ladder} and the local law give
$O(N^\epsilon\Psi_N^4)$. Since $\nu<1/12$, finite iteration yields
\[
 \mathcal R_{d,F_0,k_0,1}=O(N^{-1/2+\epsilon}),\qquad
 \mathcal R_{d,F_0,k_0,0}=O(N^{-1+\epsilon}),\qquad d\ge1.
\]
For the third derivative of the global observable, one may take
$F_0=k_0=3$. Opening its three source insertions gives
$\mathcal L_3^{\rm mov}=O(N^{-1/2+\epsilon})$.
Supplementary Lemma C.1 gives the joint identities and the induction
with every tensor contraction retained.

\paragraph{Degree-lowering cubic terms.}
When both derivatives in the first third-cumulant expansion hit a centered
diagonal, the degree can decrease. We estimate this branch of
\eqref{eq:recursion-h0} using the theorem's local-law, conditional-moment,
and endpoint-derivative hypotheses.
Suppressing the matrix label, let $a$ be an unrevealed row label, let $y$ be a column
label, and use $G_{ay}(G_{aa}-g_a)$ as the degree-two initial monomial. For a new
column label $c\ne y$, differentiation in the raw entry $\zeta_{ac}$ gives
\[
 D_\zeta G_{uv}=-n^{-1/2}G_{ua}G_{cv},\qquad
 D_{\bar\zeta}G_{uv}=-n^{-1/2}G_{uc}G_{av}.
\]
The deterministic background $g_a$ has zero coordinate derivatives. Hence
\begin{equation}\label{eq:revision-cubic-drop}
 D_\zeta D_{\bar\zeta}(G_{aa}-g_a)
 =n^{-1}\left(G_{aa}^2G_{cc}+G_{ac}G_{aa}G_{ca}\right).
\end{equation}
The first term has degree zero. In the expansion of this monomial it produces terms
of the form
\[
 \frac1{n^{3/2}}\sum_{c\ne y}\kappa_{21}(a,c)
 \E\left[\mathcal M G_{cy}G_{aa}^2G_{cc}\mid\mathcal F_{\rm old}\right],
\]
where \(\kappa_{21}(a,c)=\E[\zeta_{ac}^2\overline{\zeta_{ac}}]\)
is the mixed third cumulant of the centered raw source coordinate at the
current flow time. Its possible dependence on the coordinate is retained
inside the sum; the uniform moment bound controls its absolute value.
After removal of the outer $N^{-1/2}$, this term has degree one:
its diagonal factors are uncentered. We estimate it separately before
passing the remaining terms to the $h=1$ class.

Let $\mathcal U_{F,k}$ denote the normalized conditional supremum for
joint contractions with exactly one off-diagonal Green factor, with all
other Green factors ordinary diagonals, and with an unrevealed endpoint of
the off-diagonal factor. The tensor families, masks, normalization, and
complete conditional expectations are those of
\eqref{eq:recursion-monomial}. Expand the unique off-diagonal factor
at an unrevealed endpoint, retaining the covariance term and bounding the
third-cumulant term and remainder directly. The self contraction has the
same invertible scalar coefficient as in
\eqref{eq:recursion-row-self}--\eqref{eq:recursion-column-self}; its
centered-trace part has degree two. A covariance derivative hitting an
ordinary diagonal creates an off-diagonal factor, in addition to the
transported off-diagonal factor. A derivative hitting the derivative tensor
opens a first-order endpoint factor between the source row and column;
this is again off-diagonal. A degree-one covariance contraction would require a second off-diagonal
or centered factor, which is absent from this class. Thus the
covariance contributions are bounded by $C N^\epsilon\Psi_N^2$.
The cubic contribution may be bounded with every Green entry of order
one. Its three normalization factors and its one new-label sum give
$N^{-1/2+\epsilon}$. The next remainder gives $N^{-1+\epsilon}$.
A free-label collision loses one normalized summation, and the deleted
coordinate and bad-event terms obey the same bounds after moment
truncation. Consequently
\begin{equation}\label{eq:revision-single-offdiag}
 \mathcal U_{F,k}
 \le C_{F,k,\epsilon}
 \left(N^\epsilon\Psi_N^2+N^{-1/2+\epsilon}
                         +N^{-1+\epsilon}\right).
\end{equation}
The small losses may be chosen separately at the finitely many uses of
this estimate. Since $\nu<1/12$, we have
$\Psi_N^2=N^{-2/3+2\nu}=o(N^{-1/2})$. Restoring the first cubic factor
therefore bounds the exceptional degree-two contribution by $O(N^{-1+\epsilon})$.

Away from label collisions, the remaining derivative allocations preserve
the required degree. A derivative of an off-diagonal factor retains an off-diagonal occurrence
when the new opposite-side label is distinct from the existing labels. A single derivative of a centered
diagonal creates an off-diagonal factor. With two derivatives on the
same centered diagonal at the selected index, the pure derivatives are
\[
 D_\zeta^2G_{aa}=2n^{-1}G_{aa}G_{ca}^2,\qquad
 D_{\bar\zeta}^2G_{aa}=2n^{-1}G_{aa}G_{ac}^2;
\]
only the first mixed term in \eqref{eq:revision-cubic-drop} loses one
unit of degree. Derivatives of a centered normalized partial trace can
produce a product of ordinary diagonals only when its trace index
coincides with a source index; those terms have an additional factor
$N^{-1}$ and are bounded among the collision terms. Derivatives of the scalar tensor
leave the explicit Green factors in place.
If the original degree is $d\ge3$, the dropped term is already bounded
by
\[
 N^{-1/2+\epsilon}\Psi_N^{d-1}
 \le N^{-7/6+2\nu+\epsilon}=O(N^{-1+\epsilon'}).
\]
At degree one, an unmatched monomial contains exactly one off-diagonal
factor; a centered factor by itself has even index multiplicities. Opposite index positions, column-side expansions, and conjugation signs
follow from the same inverse-derivative identities with the indices
interchanged or the Wirtinger derivatives conjugated. Matrix normalization
ratios are uniformly bounded; they do not change these powers.
The first two derivatives and the one additional expansion use at most
$F+5$ Green factors and spectral derivative order $k+5$, within $F_+=4F+10$ and
$k_+=k+4F+20$. The remaining first-cubic terms retain their required
minimum degree and enter the stated $h=1$ class.

The two rectangular covariance products have the same nonzero spectrum.
Their traces differ by the deterministic zero-mode term $-(n-m)/z$.
We include this term, and the corresponding deterministic upper-spectral
correction, when matching the smoothed counts. Their positive-order
raw-coordinate derivatives vanish.

\paragraph{Height transfer for endpoint applications.}\label{par:height-transfer}
The specified heights in Sections~\ref{sec:high-occurrence} and
\ref{sec:low-occurrence} differ from
\eqref{eq:recursion-endpoints} by bounded factors. Fix the real part and
write $G_\eta=H(x+\ii\eta)^{-1}$ and
$f_u(\eta)=(\Im G_\eta)_{uu}$. For any fixed rectangular matrix,
\begin{equation}\label{eq:height-ward}
 \Im G_\eta=\eta G_\eta^*P_{\rm col}G_\eta
             =\eta G_\eta P_{\rm col}G_\eta^*,\qquad
 \partial_\eta G_\eta=\ii G_\eta P_{\rm col}G_\eta.
\end{equation}
Cauchy--Schwarz gives
\[
 |f'_u(\eta)|\le f_u(\eta)/\eta,\qquad
 |\partial_\eta(G_\eta)_{uv}|
 \le\eta^{-1}\sqrt{f_u(\eta)f_v(\eta)}.
\]
If the local law holds at $\eta_0$, integration of the first inequality
controls $f_u$ on $c\eta_0\le\eta\le C\eta_0$; integration of the
second then gives
\begin{equation}\label{eq:height-transfer}
 |(G_\eta-G_{\eta_0})_{uv}|=O(\Psi_N).
\end{equation}
Normalized partial traces obey the same bound. The upper-edge
Marchenko--Pastur backgrounds vary by $O(\sqrt{\eta_0})=O(\Psi_N)$.
The inverse-derivative identities are unchanged at these heights, and
$1+\widetilde m$ and $\widetilde m$ remain bounded away from zero.
This verifies the height-dependent inputs of the common recursion for
each deleted and insertion-segment matrix on its initial local-law event.

\subsection{Coupling with the Gaussian ensemble}

At time $8\log N$, the moving coordinates differ from their Gaussian
partners by an $N^{-4}$ remnant. The following single coupling controls the
whole grid, including the common fixed block and prescribed deletions.
Superscripts $T$ and $G$ identify the flow endpoint and its Gaussian
reference, respectively.

\begin{prop}[Gaussian coupling]\label{prop:terminal-grid}
Fix $0<c_{\rm dim}<C_{\rm dim}<\infty$, $\tau_*<\infty$,
$R_{\rm c}>1$, $C_{\rm c},C_{\rm edge}<\infty$, and
$0<\kappa<1/72$. Let
\begin{equation}
 \theta<\frac43-25\kappa.
 \label{eq:grid-cardinality-condition}
\end{equation}
Conditionally on a sigma-algebra $\mathcal G_N$, let $X_N$ and $W_N$ be independent rectangular complex arrays with at most $C_{\rm dim}N$ rows and columns. Assume the entries of $X_N$ are independent, centered, have variance $N^{-1}$, zero complex second moment, and uniformly bounded normalized fixed moments; let $W_N$ be circular complex Gaussian with variance $N^{-1}$. Put
\[
 X_N^T=N^{-4}X_N+\sqrt{1-N^{-8}}\,W_N.
\]
Let $\mathcal I_N$ satisfy $\abs{\mathcal I_N}\le N^\theta$. For
$i\in\mathcal I_N$, construct $A_i^T,A_i^G$ by the same coordinate restrictions, possible transpose, zero padding, bounded deterministic scaling, and addition of the same $\mathcal G_N$-measurable block $C_i$. Assume that, on a
$\mathcal G_N$-measurable event $\mathcal H_N$,
\[
 \sup_i\norm{C_i}_{\op}\le N^{1/2+\nu_N},
 \qquad \nu_N\to0.
\]
Let
\[
 L_i^T=(A_i^T)^*A_i^T,\qquad L_i^G=(A_i^G)^*A_i^G.
\]
Choose endpoints and heights satisfying
\[
 \abs{B_i-a_i}\le C_{\rm edge}N^{-2/3+\kappa},
\]
\[
 c_{\rm dim}N^{-2/3-12\kappa}
 \le\eta_i\le
 C_{\rm dim}N^{-2/3-12\kappa}.
\]
Let $g_i\colon\R\to[0,1]$ have Lipschitz constant at most
$C_{\rm edge}\eta_i^{-1}$, with $g_i=0$ allowed, and define
\[
 Z_i(L)=\frac1\pi\int_{a_i}^{B_i}
 \Im\Tr(L-y-\ii\eta_iI)^{-1}\,\dd y+\Tr g_i(L).
\]
Let $\delta_N^{-1}\le C_{\rm c}\log N$, define
\[
 y_i(L)=\alpha_i+\epsilon_i\delta_N^{-1}Z_i(L),
 \qquad\epsilon_i\in\{-1,1\},
\]
and let $\rho\colon\R\to[0,1]$ be Lipschitz. For
$J_N\le C_{\rm c}\log N$, put
\[
 S_N^T=\sum_{i\in\mathcal I_N}\sum_{r=0}^{J_N}
 R_{\rm c}^{-r}\rho(y_i(L_i^T)+J_N-r),
\]
\[
 S_N^G=\sum_{i\in\mathcal I_N}\sum_{r=0}^{J_N}
 R_{\rm c}^{-r}\rho(y_i(L_i^G)+J_N-r).
\]
There is $c_{\rm term}>0$ such that, on $\mathcal H_N$ almost surely,
\begin{equation}
 \abs{
 \E[\ee^{-\tau S_N^T}\mid\mathcal G_N]
 -\E[\ee^{-\tau S_N^G}\mid\mathcal G_N]}
 \le N^{-c_{\rm term}}
 \label{eq:grid-terminal-comparison}
\end{equation}
uniformly for $0\le\tau\le\tau_*$ and all sufficiently large $N$.

For a single statistic $Z_i$ and any deterministic scalar function
$f\colon\R\to\R$ with $\norm f_\infty+\operatorname{Lip}(f)\le C_f$,
the same coupling gives, for every $\epsilon>0$,
\[
 \left|\E[f(Z_i(L_i^T))-f(Z_i(L_i^G))\mid\mathcal G_N]\right|
 \le C_{f,\epsilon}N^{-4/3+25\kappa+\epsilon}
 \quad\hbox{on }\mathcal H_N.
\]
The estimate is uniform over a fixed finite family of such functions
and endpoints. At $\kappa=1/1000$, any fixed exponent
$D_T<157/120$ is admissible after the small losses are absorbed.

Both conclusions hold after deleting a prescribed future coordinate
from both arrays, and for averages over at most $C_{\rm dim}N^2$
deletion labels. They apply to northwest-nested families, both
rectangular orientations, square endpoints, and common revealed blocks.
\end{prop}

\begin{proof}
All moment estimates in this proof are conditional on $\mathcal G_N$.
Fix a sufficiently small $\varepsilon>0$.  The normalized moment bounds,
Markov's inequality, and the fact that each maximal array has $O(N^2)$
entries give an event $\Omega_N$ such that
\[
 \Prob(\Omega_N^c\mid\mathcal G_N)\le N^{-D},\qquad
 \norm{X_N}_F+\norm{W_N}_F\le N^{1/2+\varepsilon}
 \quad\hbox{on }\Omega_N.
\]
Here $D$ is arbitrary, and the fixed moment order is chosen after $D$ and
$\varepsilon$.  Indeed, for every integer $r\ge1$, convexity gives
$\E\norm{X_N}_F^{2r}\le C_rN^r$, and the same estimate holds for $W_N$.
On this one event,
\[
 \norm{X_N^T-W_N}_F
 \le N^{-4}\norm{X_N}_F+N^{-8}\norm{W_N}_F
 \le 2N^{-7/2+\varepsilon}.
\]
Coordinate restriction, transpose, zero padding, and deletion of a prescribed
entry are contractions in Frobenius norm.  The fixed scaling changes these
bounds by a constant.  Since the same block $C_i$ is added at both endpoints,
it cancels in the difference.  On $\mathcal H_N\cap\Omega_N$, uniformly in
all matrices and all single-entry deletions, for large $N$ we therefore have
\[
 \norm{A_i^T-A_i^G}_{\op}\le CN^{-7/2+\varepsilon},\qquad
 \norm{A_i^T}_{\op}+\norm{A_i^G}_{\op}
 \le CN^{1/2+2\varepsilon}.
\]
Consequently, with $\Delta_i=L_i^T-L_i^G$,
\[
 \norm{\Delta_i}_{\op}\le CN^{-3+3\varepsilon},\qquad
 \rank\Delta_i\le CN,\qquad
 \norm{\Delta_i}_1\le CN^{-2+3\varepsilon}.
\]
The rank bound follows by writing
$\Delta_i=(A_i^T-A_i^G)^*A_i^T+(A_i^G)^*(A_i^T-A_i^G)$;
it is independent of the size of any added zero padding.

For $z=y+\ii\eta_i$, the Hermitian resolvent identity gives
\[
 \left|\Tr\bigl[(L_i^T-z)^{-1}-(L_i^G-z)^{-1}\bigr]\right|
 \le \eta_i^{-2}\norm{\Delta_i}_1
 \le CN^{-2/3+24\kappa+3\varepsilon}.
\]
For a real Lipschitz function $g$, the eigenvalue variation inequality
$\sum_j|\lambda_j(B)-\lambda_j(A)|\le\norm{B-A}_1$ yields
\[
 |\Tr g_i(L_i^T)-\Tr g_i(L_i^G)|
 \le C\eta_i^{-1}\norm{\Delta_i}_1
 \le CN^{-4/3+12\kappa+3\varepsilon}.
\]
The variation inequality follows by tracking the ordered
eigenvalues along $A+t(B-A)$: they are absolutely continuous and, almost
everywhere, their derivatives are diagonal matrix elements of $B-A$ in an
orthonormal eigenbasis (after diagonalizing within multiple eigenspaces).
Their absolute sum is at most $\Tr|B-A|$.  Integration proves the claim.
Integrating the resolvent bound over the stated interval now gives
\[
 |Z_i(L_i^T)-Z_i(L_i^G)|
 \le CN^{-4/3+25\kappa+3\varepsilon}.
\]
The Lipschitz bound for $\rho$, the bound for $\delta_N^{-1}$, and
$\sum_{r\ge0}R_{\rm c}^{-r}<\infty$ imply
\[
 |S_N^T-S_N^G|
 \le CN^{\theta-4/3+25\kappa+3\varepsilon}\log N.
\]
Both counts are nonnegative, so $s\mapsto\ee^{-\tau s}$ is uniformly
$\tau_*$-Lipschitz on their range.  Put
$g=4/3-25\kappa-\theta>0$ and choose $3\varepsilon<g/2$.
Taking conditional expectations and using the bound $1$ on
$\Omega_N^c$ proves \eqref{eq:grid-terminal-comparison}, with any
sufficiently small $c_{\rm term}>0$.
The same maximal-array event controls every prescribed deletion, so no
additional union bound is needed for deletion averages.

For the scalar assertion, the same pointwise estimate before the grid
and cutoff-level sums gives
\[
 |Z_i(L_i^T)-Z_i(L_i^G)|
 \le C_\varepsilon N^{-4/3+25\kappa+3\varepsilon}
 \quad\hbox{on }\mathcal H_N\cap\Omega_N.
\]
Multiply by $\operatorname{Lip}(f)$ on this event and bound the
complement by $2\norm f_\infty\Prob(\Omega_N^c\mid\mathcal G_N)$.
Choosing the moment order after the requested failure exponent proves
the scalar bound. Fixed endpoint shifts preserve the window and
height estimates. At $\kappa=1/1000$,
$4/3-25\kappa=157/120>1/3$.
\end{proof}

\subsection{Comparison of counts on a full-matrix grid}
\label{sec:full-grid-comparison}

We now apply the joint expansion to full northwest rectangles. Every
entry follows the same interpolation, and both old index sets are empty.
This gives a direct comparison of the two joint counts used in the main
theorem.

\begin{prop}[Comparison of joint counts]\label{prop:full-grid-comparison}
Use the entry and dimension-path assumptions of Theorem~\ref{thm:main}.
Fix $0<\theta<1/3$, $u,v>0$, and a deterministic nonempty set
$\mathcal T_N\subset[N,2N)\cap\N$ with $|\mathcal T_N|\le N^\theta$.
For $n\in\mathcal T_N$, put
\[
 \begin{gathered}
 p_n=\min(M_n,n),\qquad a_n=(\sqrt{M_n}+\sqrt n)^2,\\
 b_n=(\sqrt{M_n}+\sqrt n)(M_n^{-1/2}+n^{-1/2})^{1/3}.
 \end{gathered}
\]
\[
 t_n^+=a_n+u b_n(\log n)^{2/3},\qquad
 t_n^-=a_n-v b_n(\log n)^{1/3}.
\]
Take $\kappa=1/1000$ and define the raw spectral parameters
\[
 \widehat\eta_n=p_n^{1/3-12\kappa},\qquad
 \widehat\ell_n=p_n^{1/3-6\kappa},\qquad
 B_n=a_n+4p_n^{1/3+\kappa},
\]
\[
 A_n^+=t_n^++\widehat\ell_n,\qquad
 A_n^-=t_n^--\widehat\ell_n,\qquad
 P_n^\pm(x)=\frac1\pi\int_{A_n^\pm}^{B_n}
       \frac{\widehat\eta_n}{(x-y)^2+\widehat\eta_n^2}\,\dd y.
\]
Fix a smooth $g_0\colon\R\to[0,1]$ equal to zero on $(-\infty,0]$
and one on $[1,\infty)$, and set
$g_n(x)=g_0((x-B_n)/\widehat\eta_n)$.
For an $M_n$-by-$n$ raw matrix $U$, let $\lambda_s(U)$,
$1\le s\le p_n$, be its squared singular values, including zeros with
multiplicity, and write
\[
 Q_n^+(U)=\sum_{s=1}^{p_n}P_n^+(\lambda_s(U)),\qquad
 Z_n^-(U)=\sum_{s=1}^{p_n}
       [P_n^-(\lambda_s(U))+g_n(\lambda_s(U))].
\]
Choose $R>1$ and $\rho\in C^\infty(\R;[0,1])$ equal to one on
$(-\infty,-1]$ and zero on $[0,\infty)$. Let
\[
 J_N=\lceil\zeta\log N\rceil,\quad
 \zeta>\frac{\theta+12}{\log R},\quad
 a_*=\frac1{16},\quad \delta_N=\frac{a_*}{J_N+2},
\]
\[
 y_n^+(U)=\frac{a_*-Q_n^+(U)}{\delta_N},\qquad
 y_n^-(U)=\frac{Z_n^-(U)-a_*}{\delta_N}.
\]
For a raw array $\mathbf U$, let $U^{(n)}$ be its northwest
$M_n$-by-$n$ restriction, and define
\[
 S_N^\pm(\mathbf U)=
   \sum_{n\in\mathcal T_N}\sum_{r=0}^{J_N}
        R^{-r}\rho(y_n^\pm(U^{(n)})+J_N-r).
\]
Let $\mathbf W$ be an independent infinite array of standard circular
complex Gaussian entries. For every fixed $\tau_*<\infty$ there is $C$
such that, for all sufficiently large $N$,
\[
 \sup_{0\le\tau\le\tau_*}
 \left|\E e^{-\tau S_N^\pm(\mathbf X)}
       -\E e^{-\tau S_N^\pm(\mathbf W)}\right|
 \le C N^{-1/4}.
\]
The constant is uniform over the specified deterministic grids. It may
depend on $\theta,u,v,R,\zeta,g_0,\rho$, the dimension bounds, and the
fixed moment bounds of the entries.
\end{prop}

\begin{proof}
\noindent\emph{One array and its coordinate sections.}
Use the single interpolation
\[
 \xi_e(t)=e^{-t/2}x_e+\sqrt{1-e^{-t}}\,w_e,
 \qquad 0\le t\le T:=8\log N,
\]
on the largest rectangle, and restrict it to every $M_n$-by-$n$ matrix.
The conditioning sigma-algebra is trivial and
$\mathcal R_{\rm old}=\mathcal C_{\rm old}=\varnothing$.
There are $O(N^2)$ moving complex coordinates. For a source coordinate
$e$, set $h_e(t)=N^{-1/2}\xi_e(t)$ and let $F_e(h;t,\tau)$ be
$e^{-\tau S_N^\pm}$ with that normalized coordinate replaced by $h$.
All summands remain in this same section. A matrix not containing $e$
has zero $h$-derivatives, as expressed by its deterministic membership
mask.

\noindent\emph{Normalization and endpoints.}
For a raw rectangle $U$, its natural normalization is $n^{-1/2}U$.
Its common normalization $N^{-1/2}U$ satisfies, with
$D_n=\diag(\sqrt{n/N}\,I_n,I_{M_n})$,
\[
 H_n^{\rm com}(z)=D_nH_n^{\rm nat}(Nz/n)D_n,\qquad
 G_n^{\rm com}(z)=D_n^{-1}G_n^{\rm nat}(Nz/n)D_n^{-1}.
\]
The change of variables in each Poisson integral is exact. The ratios
$n/N$, $M_n/N$, and $p_n/N$ stay between positive constants.
In natural coordinates the heights are
\[
 \frac{\widehat\eta_n}{n}
 =\left(\frac{p_n}{n}\right)^{1/3-12\kappa}
      n^{-2/3-12\kappa},
\]
and all real endpoints lie in the window of
Theorem~\ref{thm:four-orientation-recursion} for large $N$.
Theorem~\ref{thm:frozen-corner-local-law} with an empty old block,
followed by \eqref{eq:height-ward}--\eqref{eq:height-transfer},
gives the simultaneous endpoint bounds
$\Psi_N=N^{-5/16}$, also for deleted coordinates and their insertion
segments. The inverse identities and Marchenko--Pastur coefficients
used in Supplementary Lemma C.1 hold at these transferred heights.
All matrices keep their physical row-column orientation.
The full column trace has $n-p_n$ additional zero eigenvalues; subtracting
$(n-p_n)P_n^\pm(0)$ recovers the statistic defined above. This deterministic
correction has zero positive-order coordinate derivatives.

\noindent\emph{The complete coefficient tensors.}
Fix $K_{\rm der}=2189$ and $0<\varepsilon<1/1000$ before
selecting moment orders and failure exponents.
Lemma~\ref{lem:all-order-ladder} applies because
$|\mathcal T_N|R^{-J_N}\le N^{-12}$ and $\delta_N^{-1}=O(\log N)$.
For each fixed derivative order it bounds the full sum over matrix and
cutoff indices by $C_r\delta_N^{-r}$. The chain rule contracts these
coefficients with endpoint derivatives of the individual statistics.
Lemma~\ref{lem:covariance-direction-derivatives}, polarization in the two
real entry directions, and the bounded normalization ratios bound each
such derivative by $C_r\Psi_N$ on the simultaneous local-law event.
These are pointwise bounds for the complete contractions.

For the left statistic, a nonzero coefficient with a positive-order
derivative in $Z_n^-$ requires $Z_n^-<a_*$ at that matrix. Since
$B_n-A_n^-\ge4\widehat\eta_n$ eventually, an eigenvalue in
$[B_n,B_n+\widehat\eta_n]$ contributes at least
\[
 \frac{\arctan 5-\arctan 1}{\pi}>\frac18
\]
to $P_n^-$; an eigenvalue above $B_n+\widehat\eta_n$ contributes one to
$g_n$. Thus, on the support of that coefficient, every eigenvalue is
below $B_n$ and all positive-order derivatives of the upper term vanish.
Every matrix index introduced by the chain rule carries such a cutoff
derivative. The right statistic consists entirely of Poisson integrals.
Consequently the same endpoint estimates apply to both signs.

Let $D_e^\alpha$ denote derivatives in the two real coordinates of $h$.
The chain rule, including all its spectral-index sums, now gives, for
$1\le |\alpha|\le6$,
\[
 \sup_{0\le s\le1}
 |D_e^\alpha F_e(sh_e(t);t,\tau)|
 \le C_\alpha N^\varepsilon\Psi_N
\]
on the event that also truncates the raw entries at $N^\varepsilon$.
Each chain-rule term has at least one differentiated statistic; products
of its endpoint bounds are bounded by $C\Psi_N$ since $\Psi_N\le1$.
Every matrix label is summed by the coefficient tensor. The deterministic resolvent norm bound and the fixed derivatives
of $g_0$ give polynomial moment bounds on the complement. Choose its
failure exponent after the derivative and moment orders. This yields
\[
 N^{-2}\sum_e\E\max_{1\le|\alpha|\le5}
     |D_e^\alpha F_e(h_e(t);t,\tau)|\le C N^{-3/10},
\]
\[
 N^{-2}\sum_e\E\left[(1+|\xi_e(t)|^6)
    \max_{|\alpha|=6}\sup_{0\le s\le1}
       |D_e^\alpha F_e(sh_e(t);t,\tau)|\right]
 \le C N^{1/100}.
\]
All these estimates are uniform in time and $0<\tau\le\tau_*$.

\noindent\emph{The joint cumulant estimate.}
Differentiating $\Phi(t,\tau)=\E e^{-\tau S_N^\pm(\boldsymbol\xi(t))}$
and expanding through cumulant order five gives
\[
 \partial_t\Phi=\mathcal L_3+\mathcal L_4+\mathcal L_5+\mathcal R_6.
\]
The covariance term cancels the Gaussian diffusion term. For cumulant
order $r=3,4,5$, its coefficient before the ordered real-coordinate sum
is $-[2(r-1)!]^{-1}N^{-r/2}$.
The preceding coordinate estimates give
\[
 |\mathcal L_4|\le C N^{-3/10},\qquad
 |\mathcal L_5|\le C N^{-4/5},\qquad
 |\mathcal R_6|\le C N^{-99/100}.
\]
For $\mathcal L_3$, use the actual coefficient tensors of Supplementary
Lemma C.1 with the affine signs in $y_n^\pm$. Both source indices are
unrevealed. With initial budgets $F_0=k_0=3$, that lemma gives
$\mathcal R_{1,F_0,k_0,0}\le C N^{-1+\varepsilon}$ for the joint
contraction. Restoring its $N^{1/2}$ source normalization yields
\[
 |\mathcal L_3|\le C N^{-1/2+1/100}.
\]
The matrix and cutoff indices stay inside the same expectation in this
application. The fixed budget $2189$ covers its entire finite expansion.
Integration over $[0,8\log N]$ therefore gives
\[
 |\Phi(0,\tau)-\Phi(T,\tau)|\le C N^{-1/4}.
\]

\noindent\emph{The Gaussian endpoint.}
Apply the coupling proof of Proposition~\ref{prop:terminal-grid} with
zero fixed blocks. The two signs are allowed in its affine statistic
coordinate, and the zero-mode correction is included in its deterministic
constant. Before taking expectations that proof gives
\[
 |S_N^\pm(\boldsymbol\xi(T))-S_N^\pm(\mathbf W)|
 \le C N^{\theta-4/3+25\kappa+\varepsilon}\log N
\]
outside an event of arbitrarily small polynomial probability.
Since $\theta<1/3$ and $\kappa=1/1000$, this contributes at most
$C N^{-1/4}$ to the Laplace comparison. Combining the two bounds proves
the proposition; at $\tau=0$ the difference is zero.
\end{proof}

\begin{cor}[Block occurrences]\label{cor:full-grid-occurrence}
Under the assumptions of Theorem~\ref{thm:main}, fix
$0<a<A_2$ and $0<b<B_2$. There are deterministic grids
$\mathcal T_N^+,\mathcal T_N^-\subset[N,2N)\cap\N$ such that
\[
 \begin{aligned}
 \Prob\left(\bigcup_{n\in\mathcal T_N^+}
       \{\chi_n^+\ge a(\log n)^{2/3}\}\right)&\longrightarrow1,\\
 \Prob\left(\bigcup_{n\in\mathcal T_N^-}
       \{\chi_n^+<-b(\log n)^{1/3}\}\right)&\longrightarrow1.
 \end{aligned}
\]
The limits hold as $N\to\infty$ through the positive integers.
\end{cor}
\begin{proof}
Choose $a<u<a_2<A_2$ and use the statistics of
Proposition~\ref{prop:full-grid-comparison} with $v=b$.
Set $\kappa_+=4a_2^{3/2}/3$ and $\kappa_-=b^3/12$.
Choose $\beta,\omega>0$ with
$\max(\kappa_+,\kappa_-)+\beta+\omega<1/3$.
Paper I, Corollary~\ref*{cor:gaussian-grid-moments}
\cite{YangWishart}, applied to the full Gaussian path in the physical
normalization, gives separated grids with
\[
 \begin{gathered}
 |\mathcal T_N^\pm|=\lfloor N^{\kappa_\pm+\beta}\rfloor,\qquad
 m_N^\pm=N^{\beta+o(1)},\\
 \E[(R_N^\pm)^2]\le(1+u_N)(m_N^\pm)^2+m_N^\pm,\qquad u_N\longrightarrow0.
 \end{gathered}
\]
The grids can be chosen as
\[
 \mathcal T_N^\pm=
 \{N+j\lceil N^{2/3+\omega}\rceil:
       0\le j<\lfloor N^{\kappa_\pm+\beta}\rfloor\}.
\]
Their span is $O(N^{\kappa_\pm+\beta+2/3+\omega})=o(N)$, so they
lie in the short macroscopic window of the Gaussian corollary and in
$[N,2N)$ for large $N$.
Here $R_N^+$ counts Gaussian events at the raw threshold
$a_n+a_2b_n(\log n)^{2/3}$, and $R_N^-$ counts Gaussian events at
$t_n^--2\widehat\ell_n$. The left shift in the standardized coordinate
is $-2\widehat\ell_n/b_n=O(N^{-6\kappa})$, within the bounded
shift range of that corollary. Choose
$\max(\kappa_+,\kappa_-)+\beta<\theta<1/3$ to apply
Proposition~\ref{prop:full-grid-comparison} to either grid.

Apply Proposition~\ref{prop:poisson-count-dichotomy} at the individual
full-matrix right thresholds, using the smaller covariance product and
column normalization $p_n$. Its proof is uniform over these comparable
dimensions and deterministic edge windows. For the actual matrices,
outside a union event of probability $O(N^{\theta-D})$,
\[
 \norm{X^{(n)}}_{\op}^2<t_n^+\quad\Longrightarrow\quad
 Q_n^+(X^{(n)})\le p_n^{-\kappa}.
\]
Thus $S_N^+(\mathbf X)>0$ implies the right target in the statement on
this event. For a Gaussian right hit at amplitude $a_2$, the same count
estimate uses the lower integration limit
$a_n+a_2b_n(\log n)^{2/3}-\widehat\ell_n$.
This is larger than $A_n^+$ for large $N$. Nonnegativity of the Poisson
kernel gives $Q_n^+\ge1-p_n^{-\kappa}$ on the Gaussian good event,
which activates the level-zero cutoff with value one.

For the left target, Proposition~\ref{prop:poisson-detector} applies
with $t_i=t_n^-$ and $a_i=A_n^-$, since
$t_n^--A_n^-=\widehat\ell_n\ge4\widehat\eta_n$.
It gives the pathwise implication from $S_N^-(\mathbf X)>0$ to a
left target. Apply Proposition~\ref{prop:poisson-activation} to the
full Gaussian matrices with $E_{N,n}=t_n^--2\widehat\ell_n$.
The lower integration endpoint is $E_{N,n}+\widehat\ell_n=A_n^-$.
On each retained Gaussian gap event the upper term is zero and
$Z_n^-\le p_n^{-\kappa}\le\delta_N$ for large $N$, so its
level-zero cutoff again equals one. A transpose used to evaluate an
individual covariance product leaves its singular-value statistic
unchanged; all joint observables remain functions of the original array.

Let $\widetilde R_N^\pm$ count the retained Gaussian hits and
$\widetilde m_N^\pm=\E\widetilde R_N^\pm$.
Choose $D>\theta+10$ in the smoothing events. Then
\[
 0\le \E(R_N^\pm-\widetilde R_N^\pm)\le C N^{\theta-D},\qquad
 \widetilde m_N^\pm=(1+o(1))m_N^\pm\longrightarrow\infty.
\]
The retention estimate following
Lemma~\ref{lem:conditional-occurrence-closure} gives
$\Var(\widetilde R_N^\pm)/(\widetilde m_N^\pm)^2=o(1)$.
Since $S_N^\pm(\mathbf W)\ge\widetilde R_N^\pm$, for fixed $\tau>0$,
\[
 \E e^{-\tau S_N^\pm(\mathbf W)}
 \le \frac{4\Var(\widetilde R_N^\pm)}{(\widetilde m_N^\pm)^2}
       +e^{-\tau\widetilde m_N^\pm/2}\longrightarrow0.
\]
Compare this transform with the actual one using
Proposition~\ref{prop:full-grid-comparison}. The target-support
implications give, with $E_N^\pm$ denoting the respective target unions,
\[
 \Prob((E_N^\pm)^c)
 \le \E e^{-\tau S_N^\pm(\mathbf X)}+C N^{\theta-D}
 \longrightarrow0.
\]
The final error is needed for the right smoothing event; it can also be
included in the left bound. This proves both limits.
\end{proof}

\section{Conditional right-tail events}\label{sec:high-occurrence}

\subsection*{Gaussian signal and smooth realization}

We refine the right-tail estimate by retaining a revealed northwest
corner. On good histories, the conditional probability of a hit in the
selected block tends to one. Paper~I,
Corollary~\ref*{cor:gaussian-grid-moments}, supplies the Gaussian count
with a diverging mean and vanishing relative variance.
Proposition~\ref{prop:gaussian-terminal-count} places it on the
lacunary scales used for these conditional estimates.

The amplitudes $a_0<a_1<a_2$ allow for smoothing and passage from the
southeast submatrix to the full matrix. We count Gaussian events at
$a_2$, start the Poisson window at $a_1$, and seek a full-matrix event
at $a_0$.

Retain the Gaussian events for which smoothing succeeds and embed their
count in one weighted cutoff observable of the complete unrevealed array.
The homogeneous comparison and Gaussian coupling transfer its conditional
Laplace transform. At the reference endpoint the cutoff sum dominates
the retained Gaussian count; at the actual endpoint its positivity
implies the target event, up to the stated smoothing failure.
Lemma~\ref{lem:conditional-occurrence-closure} then gives
Theorem~\ref{thm:high-occurrence}.

\begin{prop}[Gaussian right-tail counts]\label{prop:gaussian-terminal-count}
Let $\gamma>0$, and let $(M_N)_{N\geq1}$ be a nondecreasing sequence of positive integers such that
\[
 \frac{M_N}{N}\longrightarrow\gamma,
 \qquad
 K:=\sup_N(M_{N+1}-M_N)<\infty.
\]
Let $(x_{ij})_{i,j\geq1}$ be any infinite complex array. For $N\geq1$, define
\[
 X^{(N)}=(x_{ij})_{\substack{1\leq i\leq M_N\\1\leq j\leq N}},
 \qquad
 \chi_N^+
 =
 \frac{\norm{X^{(N)}}_{\op}^2-(\sqrt{M_N}+\sqrt N)^2}
 {(\sqrt{M_N}+\sqrt N)(M_N^{-1/2}+N^{-1/2})^{1/3}},
\]
and put
\[
 A_2=\left(\frac14\right)^{2/3}.
\]
Fix amplitudes
\[
 0<a_0<a_1<A_2,
 \qquad
 \kappa=\frac43a_1^{3/2},
\]
and choose $\epsilon,\eta>0$ such that
\[
 \kappa+\epsilon+2\eta<\frac13.
\]
Fix $d>0$. Then the following assertions hold.

\begin{enumerate}
\item There are integers
\[
 2\leq k_1<k_2<\cdots,
 \qquad
 k_l-k_{l-1}\geq2,
 \qquad
 k_0=0,
\]
such that, on putting
\[
 N_l=2^{k_l},
 \qquad
 P_l=2^{k_{l-1}+1},
 \qquad
 J_l=N_l-P_l,
\]
we have
\[
 P_l\leq(\log N_l)^d.
\]

\item Define the shifted dimension path
\[
 \widehat M_{l,j}=M_{P_l+j}-M_{P_l},
 \qquad j\geq1.
\]
On an extension of the probability space, let
\[
 (\zeta_{l,ih})_{i,h\geq1},
 \qquad l\geq1,
\]
be independent arrays of independent centered circular complex Gaussian variables of second absolute moment one, independent of the original array. Let $Z_l^{(j)}$ be the $\widehat M_{l,j}$-by-$j$ northwest rectangle of the $l$th Gaussian array, and let
\[
 Y_l^{(j)}
 =
 (x_{M_{P_l}+i,P_l+h})_{\substack{1\leq i\leq\widehat M_{l,j}\\1\leq h\leq j}}
\]
be the corresponding actual independent-submatrix matrix.

For a $u$-by-$v$ matrix $W$, put
\[
 p=\min(u,v),
 \qquad
 q=\max(u,v),
\]
\[
 \mu_{p,q}
 =
 \left(\sqrt{p+\frac12}+\sqrt{q+\frac12}\right)^2,
\]
\[
 \sigma_{p,q}
 =
 \left(\sqrt{p+\frac12}+\sqrt{q+\frac12}\right)
 \left(
 \left(p+\frac12\right)^{-1/2}
 +
 \left(q+\frac12\right)^{-1/2}
 \right)^{1/3},
\]
and
\[
 \overline\chi(W)
 =
 \frac{\norm W_{\op}^2-\mu_{p,q}}{\sigma_{p,q}}.
\]
There are finite grids
\[
 T_l\subset\{n:N_l\leq n<2N_l\}
\]
such that
\[
 \abs{T_l}\leq2N_l^{\kappa+2\eta}.
\]

\item Writing $j_{l,n}=n-P_l$ and
\[
 S_l^G
 =
 \sum_{n\in T_l}
 \ind_{\{
 \overline\chi(Z_l^{(j_{l,n})})
 \geq a_1(\log j_{l,n})^{2/3}
 \}},
\]
there is a deterministic $u_l\ge0$, with $u_l\to0$, such that
\[
 m_l:=\E S_l^G\geq N_l^\eta,
 \qquad
 \E[(S_l^G)^2]\leq(1+u_l)m_l^2+m_l.
\]

\item If
\[
 \mathcal F_l^{\mathrm{past}}
 =
 \sigma(x_{ij}:1\leq i\leq M_{P_l},\,1\leq j\leq P_l),
\]
then the auxiliary Gaussian arrays are independent of
$\mathcal F_l^{\mathrm{past}}$. Consequently,
\[
 \E[S_l^G\mid\mathcal F_l^{\mathrm{past}}]
 =
 m_l\geq N_l^\eta,
 \qquad
 \E[(S_l^G)^2\mid\mathcal F_l^{\mathrm{past}}]
 \leq(1+u_l)m_l^2+m_l
\]
almost surely.

\item For all sufficiently large $l$,
\[
 \bigcup_{n\in T_l}
 \left\{
 \overline\chi(Y_l^{(j_{l,n})})
 \geq a_1(\log j_{l,n})^{2/3}
 \right\}
 \subset
 \bigcup_{n\in T_l}
 \left\{
 \chi_n^+\geq a_0(\log n)^{2/3}
 \right\}.
\]
\end{enumerate}

Thus the completely fresh Gaussian corner supplies a terminal conditional first- and second-moment package with deterministic good history, while positivity of the analogous actual independent-submatrix count implies the required full-matrix high-block event.
\end{prop}

\supplementproof{supp.J.1}{Section J.1}

Independence of the auxiliary Gaussian array from
$\mathcal F_l^{\mathrm{past}}$ makes its conditional moments equal to
the deterministic grid moments. The buffer between $a_1$ and $a_0$
then transfers a right-tail event in the actual southeast submatrix
to the full matrix. We next replace the Gaussian count by a smooth
observable with the same terminal contribution.

\begin{prop}[Smooth cutoff functions for the right tail]\label{prop:high-poisson-ladder}
Let $(x_{ij})_{i,j\geq1}$ be mutually independent complex random variables satisfying
\[
 \E x_{ij}=0,
 \qquad
 \E\abs{x_{ij}}^2=1,
 \qquad
 \E x_{ij}^2=0,
\]
and, for every fixed $s\geq1$,
\[
 \sup_{i,j}\E\abs{x_{ij}}^s<\infty.
\]
Let $\gamma\in(0,\infty)$, and let $(M_N)_{N\geq1}$ be nondecreasing with
\[
 \frac{M_N}{N}\longrightarrow\gamma,
 \qquad
 \sup_N(M_{N+1}-M_N)<\infty.
\]
Define $X^{(N)}$, $\chi_N^+$, and $A_2$ as in Proposition~\ref{prop:gaussian-terminal-count}. Fix
\[
 0<a_0<a_1<a_2<A_2,
 \qquad
 \kappa_2=\frac43a_2^{3/2},
\]
and choose $\epsilon,\eta>0$ such that
\[
 \kappa_2+\epsilon+2\eta<\frac13.
\]
Fix $d>0$. Then the following assertions hold.

\begin{enumerate}
\item There are integers
\[
 2\leq k_1<k_2<\cdots,
 \qquad
 k_l-k_{l-1}\geq2,
 \qquad
 k_0=0,
\]
such that, with
\[
 N_l=2^{k_l},
 \qquad
 P_l=2^{k_{l-1}+1},
 \qquad
 J_l=N_l-P_l,
\]
we have $P_l\leq(\log N_l)^d$. Put
\[
 \widehat M_{l,j}=M_{P_l+j}-M_{P_l}.
\]
On an extension of the probability space, let
$(\zeta_{l,ih})_{i,h\geq1}$, $l\geq1$, be independent arrays of independent centered circular complex Gaussian variables of second absolute moment one, independent of the original array. Define
\[
 Y_l^{(j)}
 =
 (x_{M_{P_l}+i,P_l+h})_{\substack{1\leq i\leq\widehat M_{l,j}\\1\leq h\leq j}},
\qquad
 Z_l^{(j)}
 =
 (\zeta_{l,ih})_{\substack{1\leq i\leq\widehat M_{l,j}\\1\leq h\leq j}},
\]
and
\[
 \mathcal F_l^{\mathrm{past}}
 =
 \sigma(x_{ij}:1\leq i\leq M_{P_l},\,1\leq j\leq P_l).
\]
There are finite grids
\[
 T_l\subset\{n:N_l\leq n<2N_l\}
\]
such that, with $j_{l,n}=n-P_l$,
\[
 \abs{T_l}\leq2N_l^{\kappa_2+2\eta}.
\]

\item For a $u$-by-$v$ matrix $W$, put
\[
 p=\min(u,v),\qquad q=\max(u,v),
\]
and define $\mu_{p,q}$, $\sigma_{p,q}$, and $\overline\chi(W)$ as in Proposition~\ref{prop:gaussian-terminal-count}. Set
\[
 H_{l,n}^G
 =
 \left\{
 \overline\chi(Z_l^{(j_{l,n})})
 \geq a_2(\log j_{l,n})^{2/3}
 \right\},
\qquad
 R_l^G=\sum_{n\in T_l}\ind_{H_{l,n}^G}.
\]
The deterministic number
\[
 m_l=\E R_l^G
\]
satisfies, for deterministic $u_l\ge0$ with $u_l\to0$, almost surely,
\[
 \E[R_l^G\mid\mathcal F_l^{\mathrm{past}}]
 =
 m_l\geq N_l^\eta,
 \qquad
 \E[(R_l^G)^2\mid\mathcal F_l^{\mathrm{past}}]
 \leq(1+u_l)m_l^2+m_l.
\]

\item For all sufficiently large $l$,
\[
 \bigcup_{n\in T_l}
 \left\{
 \overline\chi(Y_l^{(j_{l,n})})
 \geq a_1(\log j_{l,n})^{2/3}
 \right\}
 \subset
 U_l(a_0),
\]
where
\[
 U_l(a_0)
 =
 \bigcup_{n\in T_l}
 \left\{
 \chi_n^+\geq a_0(\log n)^{2/3}
 \right\}.
\]

\item For $n\in T_l$, set
\[
 u_{l,n}=\widehat M_{l,j_{l,n}},
 \qquad
 v_{l,n}=j_{l,n},
\]
\[
 p_{l,n}=\min(u_{l,n},v_{l,n}),
 \qquad
 q_{l,n}=\max(u_{l,n},v_{l,n}),
\]
\[
 \mu_{l,n}=\mu_{p_{l,n},q_{l,n}},
 \qquad
 \sigma_{l,n}=\sigma_{p_{l,n},q_{l,n}},
\]
and
\[
 t_{r,l,n}
 =
 \mu_{l,n}
 +
 a_r(\log j_{l,n})^{2/3}\sigma_{l,n},
 \qquad r=1,2.
\]

\item Fix $0<\kappa_0<1/72$ and put
\[
 \widehat\eta_{l,n}
 =
 p_{l,n}^{1/3-12\kappa_0},
 \qquad
 \widehat\ell_{l,n}
 =
 p_{l,n}^{1/3-6\kappa_0},
\]
\[
 \widehat B_{l,n}
 =
 (\sqrt{p_{l,n}}+\sqrt{q_{l,n}})^2
 +
 4p_{l,n}^{1/3+\kappa_0}.
\]

\item If $W$ is a $u_{l,n}$-by-$v_{l,n}$ matrix and
$\lambda_1(W),\ldots,\lambda_{p_{l,n}}(W)$ are all its squared singular values, including zeros with multiplicity, define
\[
 Q_{l,n}(W)
 =
 \sum_{s=1}^{p_{l,n}}
 \frac1\pi
 \int_{t_{1,l,n}+\widehat\ell_{l,n}}^{\widehat B_{l,n}}
 \frac{\widehat\eta_{l,n}}
 {(\lambda_s(W)-x)^2+\widehat\eta_{l,n}^2}
 \,\dd x.
\]

\item Fix $R_{\mathrm{c}}>1$, a function
\[
 \rho_{\mathrm{c}}\in C^\infty(\R;[0,1])
\]
such that
\[
 \rho_{\mathrm{c}}(y)=1\quad\text{for }y\leq-1,
 \qquad
 \rho_{\mathrm{c}}(y)=0\quad\text{for }y\geq0,
\]
and $A_{\mathrm{fin}}>0$. Put
\[
 \theta_h=\kappa_2+2\eta+\frac{\epsilon}{2},
\]
choose
\[
 \zeta>\frac{\theta_h+A_{\mathrm{fin}}}{\log R_{\mathrm{c}}},
\]
and define
\[
 J_{\mathrm{c},l}
 =
 \left\lceil\zeta\log N_l\right\rceil,
 \qquad
 a_*=\frac1{16},
 \qquad
 \delta_l=\frac{a_*}{J_{\mathrm{c},l}+2},
\]
\[
 y_{l,n}(W)
 =
 \frac{a_*-Q_{l,n}(W)}{\delta_l},
\]
and
\[
 A_{l,n,r}(W)
 =
 R_{\mathrm{c}}^{-r}
 \rho_{\mathrm{c}}
 \bigl(y_{l,n}(W)+J_{\mathrm{c},l}-r\bigr),
 \qquad
 0\leq r\leq J_{\mathrm{c},l}.
\]

\item For a family $\mathbf W=(W_n)_{n\in T_l}$ of the indicated dimensions, put
\[
 S_l(\mathbf W)
 =
 \sum_{n\in T_l}
 \sum_{r=0}^{J_{\mathrm{c},l}}
 A_{l,n,r}(W_n).
\]

\item After increasing the recursively selected scales if necessary, there are events $\Omega_l^X$, measurable with respect to the actual independent-submatrix array and independent of $\mathcal F_l^{\mathrm{past}}$, such that
\[
 \Prob((\Omega_l^X)^c\mid\mathcal F_l^{\mathrm{past}})
 \leq N_l^{-10}
\]
almost surely and, on $\Omega_l^X$,
\[
 S_l\bigl((Y_l^{(j_{l,n})})_{n\in T_l}\bigr)>0
 \quad\Longrightarrow\quad
 U_l(a_0).
\]

\item There are Gaussian smoothing events $\Omega_{l,n}^G$ such that, with
\[
 \widetilde H_{l,n}^G
 =
 \ind_{H_{l,n}^G}\ind_{\Omega_{l,n}^G},
 \qquad
 \widetilde R_l^G
 =
 \sum_{n\in T_l}\widetilde H_{l,n}^G,
\]
and
\[
 \widetilde m_l
 =
 \E[\widetilde R_l^G\mid\mathcal F_l^{\mathrm{past}}],
\]
we have
\[
 \widetilde m_l
 =
 (1+o(1))m_l
 \geq
 N_l^\eta(1-o(1)),
\]
\[
 \E[(\widetilde R_l^G)^2\mid\mathcal F_l^{\mathrm{past}}]
 \leq
 (1+o(1))\widetilde m_l^2,
\]
and, pathwise,
\[
 S_l\bigl((Z_l^{(j_{l,n})})_{n\in T_l}\bigr)
 \geq
 \widetilde R_l^G.
\]
The $o(1)$ terms are deterministic. In particular, for all sufficiently large $l$,
\[
 \widetilde m_l\geq N_l^{\eta/2},
 \qquad
 \frac{\E[(\widetilde R_l^G)^2\mid\mathcal F_l^{\mathrm{past}}]}
 {\widetilde m_l^2}
 \leq1+o(1).
\]

Every $Q_{l,n}$, $y_{l,n}$, and $A_{l,n,r}$ is $C^\infty$ in the real and imaginary entry coordinates. The complete first-coordinate, second-coordinate, and noncubic third-coordinate derivative sums of $S_l$ obey the support-absorption estimates of Lemma~\ref{lem:ladder-absorption}. Every nonfinal sum over the grid and cutoff levels is therefore absorbed by $S_l$, while
\[
 \abs{T_l}R_{\mathrm{c}}^{-J_{\mathrm{c},l}}
 \leq N_l^{-A_{\mathrm{fin}}},
 \qquad
 \delta_l^{-1}=O(\log N_l).
\]
The third-cumulant contribution is treated by Theorem~\ref{thm:four-orientation-recursion}.
\end{enumerate}
\end{prop}

\supplementproof{supp.J.2}{Section J.2}

A retained Gaussian event at amplitude $a_2$ activates the level-$0$
cutoff with unit weight. At the actual endpoint, positivity of any level
implies the $a_0$ target on the smoothing event. The derivative estimates
sum the intermediate levels through the weighted count, leaving the
terminal error $\abs{T_l}R_{\mathrm{c}}^{-J_{\mathrm{c},l}}$.

\subsection*{One global observable and the homogeneous flow}

Lemma~\ref{lem:global-high-ladder} realizes all matrix summands as one
observable of the $N^{-1/2}$-normalized array. Differentiation in an entry
acts on precisely the matrices containing it, while each coordinate section
retains the full count. The factor $\sqrt{N/p_{l,n}}$ converts to the
natural normalization used by the local law, with exact transformations
of the spectral endpoints, height, insertion directions, and Green blocks.

\begin{lem}[Change of normalization]\label{lem:global-high-ladder}
Adopt the objects of Proposition~\ref{prop:high-poisson-ladder} with
\[
 \kappa_0=\frac1{1000},
 \qquad
 A_{\mathrm{fin}}=12.
\]
Fix a scale $l$ and write
\[
 N=N_l,
 \qquad
 J=J_{\mathrm{c},l}.
\]
Let
\[
 j_{\max}=\max_{n\in T_l}j_{l,n},
 \qquad
 u_{\max}=M_{P_l+j_{\max}}-M_{P_l},
\]
and
\[
 \mathcal E_N
 =
 \{1,\ldots,u_{\max}\}\times\{1,\ldots,j_{\max}\}.
\]
Conditionally on $\mathcal F_l^{\mathrm{past}}$, put
\[
 \xi_e(t)
 =
 \ee^{-t/2}\xi_e(0)
 +
 \sqrt{1-\ee^{-t}}\,g_e,
 \qquad
 0\leq t\leq8\log N,
\]
where $(\xi_e(0))_{e\in\mathcal E_N}$ is the unnormalized actual southeast unrevealed array and $(g_e)_{e\in\mathcal E_N}$ is its independent circular complex Gaussian counterpart. Define
\[
 X_N(t)=N^{-1/2}(\xi_e(t))_{e\in\mathcal E_N}.
\]
For $n\in T_l$, let $X_n(t)$ be the northwest $u_{l,n}$-by-$v_{l,n}$ restriction of $X_N(t)$, transposed and conjugated when necessary so that it is a $q_{l,n}$-by-$p_{l,n}$ matrix. Put
\[
 A_n=\frac{t_{1,l,n}+\widehat\ell_{l,n}}{N},
 \qquad
 B_n=\frac{\widehat B_{l,n}}{N},
 \qquad
 \eta_n=\frac{\widehat\eta_{l,n}}{N}.
\]
For a complete array $X$, define
\[
 Z_n(X)
 =
 \frac1\pi
 \int_{A_n}^{B_n}
 \Im\Tr
 \left(
 X_n^*X_n-x-\ii\eta_n
 \right)^{-1}
 \,\dd x,
\]
\[
 y_n(X)=\frac{1/16-Z_n(X)}{\delta_l},
\]
and
\begin{equation}
 S_N(X)
 =
 \sum_{n\in T_l}
 \sum_{r=0}^{J}
 R_{\mathrm{c}}^{-r}
 \rho_{\mathrm{c}}\bigl(y_n(X)+J-r\bigr).
 \label{eq:global-high-observable}
\end{equation}
Then $S_N(X)\geq0$, and
\begin{equation}
 S_N(t):=S_N(X_N(t))
 \label{eq:global-high-process}
\end{equation}
is exactly the common-$N^{-1/2}$ normalized version of the single global cutoff sum observable in Proposition~\ref{prop:high-poisson-ladder}.

For $e\in\mathcal E_N$ and $h=(h_1,h_2)\in\R^2$, let $X_N^{e,h}(t)$ be the complete array obtained from $X_N(t)$ by replacing entry $e$ with $h_1+\ii h_2$. Define
\[
 S_{N,e}(h;t)=S_N(X_N^{e,h}(t)),
\]
\begin{equation}
 F_{N,e}(h;t,\tau)
 =
 \exp[-\tau S_{N,e}(h;t)].
 \label{eq:global-high-section}
\end{equation}
If
\[
 h_e(t)
 =
 N^{-1/2}
 \bigl(\Re\xi_e(t),\Im\xi_e(t)\bigr),
\]
then, simultaneously for every $e\in\mathcal E_N$,
\begin{equation}
 S_{N,e}(h_e(t);t)=S_N(t),
 \qquad
 F_{N,e}(h_e(t);t,\tau)=\exp[-\tau S_N(t)].
 \label{eq:global-high-restoration}
\end{equation}
Thus every $F_{N,e}$ is the coordinate section of the single global observable $\exp[-\tau S_N]$. If a matrix $n$ does not contain $e$, its summand in \eqref{eq:global-high-observable} is constant as a function of $h$. If it contains $e$, all its coordinate derivatives are the corresponding derivatives of the global section. Hence derivative-active matrices are selected by an exact membership mask, while every generator term retains the complete global count.

For each matrix, put
\begin{equation}
 a_n=\sqrt{\frac{N}{p_{l,n}}},
 \qquad
 \widetilde X_n=a_nX_n,
 \qquad
 \widetilde z=a_n^2z.
 \label{eq:high-tag-rescaling}
\end{equation}
Then $\widetilde X_n$ has entry variance $p_{l,n}^{-1}$. With column-first covariance linearizations
\[
 H_n(z)
 =
 \begin{bmatrix}
 -zI&X_n^*\\
 X_n&-I
 \end{bmatrix},
\qquad
 \widetilde H_n(a_n^2z)
 =
 \begin{bmatrix}
 -a_n^2zI&\widetilde X_n^*\\
 \widetilde X_n&-I
 \end{bmatrix},
\]
and inverses $G_n,\widetilde G_n$, the four blocks satisfy
\[
 G_{n,cc}(z)
 =
 a_n^2\widetilde G_{n,cc}(a_n^2z),
 \qquad
 G_{n,cr}(z)
 =
 a_n\widetilde G_{n,cr}(a_n^2z),
\]
\begin{equation}
 G_{n,rc}(z)
 =
 a_n\widetilde G_{n,rc}(a_n^2z),
 \qquad
 G_{n,rr}(z)
 =
 \widetilde G_{n,rr}(a_n^2z).
 \label{eq:high-four-block-rescaling}
\end{equation}
The global insertion $h_1+\ii h_2$ enters $\widetilde X_n$ at speed $a_n$, with the imaginary sign reversed when the matrix was conjugate-transposed, and
\begin{equation}
 \begin{aligned}
 &\int_{A_n}^{B_n}
 \Im\Tr G_{n,cc}(x+\ii\eta_n)\,\dd x\\
 &\qquad=
 \int_{a_n^2A_n}^{a_n^2B_n}
 \Im\Tr
 \widetilde G_{n,cc}
 (\widetilde x+\ii a_n^2\eta_n)
 \,\dd\widetilde x.
 \end{aligned}
 \label{eq:high-integrated-rescaling}
\end{equation}
Moreover,
\[
 a_n^2A_n
 =
 \frac{t_{1,l,n}+\widehat\ell_{l,n}}{p_{l,n}},
 \qquad
 a_n^2B_n
 =
 \frac{\widehat B_{l,n}}{p_{l,n}},
\]
\begin{equation}
 a_n^2\eta_n
 =
 p_{l,n}^{-2/3-12\kappa_0}.
 \label{eq:high-natural-parameters}
\end{equation}
All $a_n$ and $a_n^{-1}$ are uniformly bounded. Thus the exact factor $\sqrt{N/p_{l,n}}$ is carried through the endpoints, height, insertion directions, integrated smoothed statistic, and all four Green blocks without changing the global observable \eqref{eq:global-high-observable}.
\end{lem}

\supplementproof{supp.J.3}{Section J.3}

We apply the common recursion in the physical matrix orientation,
retaining the external conditioning sigma-algebra and matching the spectral
height and deterministic zero-mode contribution.

\begin{cor}[Right-tail specialization]
\label{prop:high-orientation-specialization}
Use Proposition~\ref{prop:high-poisson-ladder} and
Lemma~\ref{lem:global-high-ladder}, with
$\kappa_0=1/1000$, $\nu_0=1/48$, and $A_{\rm fin}=12$.
Write $N=N_l$, $\mathcal G=\mathcal F_l^{\rm past}$, and retain each
physical $u_n$-by-$v_n$ fresh matrix $W_n(t)$ without transposition. Put
$p_n=\min(u_n,v_n)$ and
\[
 \begin{aligned}
 A_n&=v_n^{-1/2}W_n,&
 a_n&=(t_{1,l,n}+\widehat\ell_{l,n})/v_n,\\
 b_n&=\widehat B_{l,n}/v_n,&
 \eta_n&=p_n^{1/3-12\kappa_0}/v_n.
 \end{aligned}
\]
Use the linearization of Theorem~\ref{thm:four-orientation-recursion}
with $(m_k,n_k)=(u_n,v_n)$, and define
\begin{equation}\label{eq:high-physical-detector}
 \begin{aligned}
 Y_n&=\frac1\pi\int_{a_n}^{b_n}\Im\Tr(P_{{\rm c},n}G_n(x+\ii\eta_n))\,\dd x,\\
 Z_{0,n}&=\frac{(v_n-u_n)_+}{\pi}\int_{a_n}^{b_n}
             \frac{\eta_n}{x^2+\eta_n^2}\,\dd x.
 \end{aligned}
\end{equation}
Then the exact smoothed statistic identity is
\begin{equation}\label{eq:high-zero-mode-correction}
 Q_{l,n}(W_n)=Y_n-Z_{0,n}.
\end{equation}
Let $m_n$ be the column background, determined by
\begin{equation}\label{eq:high-mp-equation}
 z=-m_n(z)^{-1}+\frac{u_n/v_n}{1+m_n(z)},\qquad \Im m_n(z)>0.
\end{equation}
At the specified endpoints and uniformly over the stated matrices, times,
deletions, and insertion segments, the conditional local-law estimates of
Theorem~\ref{thm:four-orientation-recursion} hold with
$\Psi_N=N^{-5/16}$. In particular,
\begin{equation}\label{eq:high-physical-local-law}
 \left|v_n^{-1}\Tr(P_{{\rm c},n}G_n)-m_n\right|
 +\left|u_n^{-1}\Tr(P_{{\rm r},n}G_n)+(1+m_n)^{-1}\right|
 =O(\Psi_N).
\end{equation}
On good $\mathcal G$-histories with summable complements, the same
normalized classes, coefficients, recurrences
\eqref{eq:recursion-h0}--\eqref{eq:recursion-h1}, and cubic source
\eqref{eq:recursion-third-source} apply with
$\mathcal R_{\rm old}=\mathcal C_{\rm old}=\varnothing$ and conditioning
on $\mathcal G$. Both source labels are unrevealed. The cutoff sum bounds are
$|T_l|R_{\rm c}^{-J_{{\rm c},l}}\le N^{-12}$ and
$\delta_l^{-1}=O(\log N)$.
\end{cor}
\begin{proof}
Singular-value decomposition and the change of variable $\widehat x=v_nx$
give \eqref{eq:high-zero-mode-correction}; $Z_{0,n}$ is deterministic,
so its positive coordinate derivatives vanish. The full traces in
\eqref{eq:high-physical-local-law} keep their structural zeros.
The ratio of $\eta_n$ to $v_n^{-2/3-12\kappa_0}$ is
$(p_n/v_n)^{1/3-12\kappa_0}$, bounded above and below by positive constants.
The height-transfer argument following
Theorem~\ref{thm:four-orientation-recursion} therefore supplies the local
bounds at these endpoints. The inverse-derivative identities and bounded
leading coefficients are unchanged, so the proof of that theorem gives the
stated recurrences at the specified heights. The complete unrevealed array is
independent of $\mathcal G$, and all expectations are conditioned on
this same $\mathcal G$ while the old index sets are empty. All moving sources have
two unrevealed labels. The normalization speeds are those of
Lemma~\ref{lem:global-high-ladder}, and the last two bounds come from
Proposition~\ref{prop:high-poisson-ladder}.
\end{proof}

Theorem E.2 of the technical supplement gives the homogeneous comparison;
Section E.2 verifies its derivative and truncation hypotheses for the
right-tail family.

\begin{prop}[Right-tail comparison]
\label{prop:high-homogeneous-bound}
Use the same cutoff, complete unrevealed array, physical orientations, and
conditional history $\mathcal G$ as in
Corollary~\ref{prop:high-orientation-specialization}. Evolve all fresh
raw coordinates by \eqref{eq:recursion-flow}, divide the array by
$\sqrt N$, and use the unchanged global count $S^{\rm high}(t)$ of
\eqref{eq:global-high-observable}. For every fixed $\tau_*<\infty$, on
$\mathcal G$-histories with summable complements,
\[
 \Phi^{\rm high}(t,\tau)=\E[\ee^{-\tau S^{\rm high}(t)}\mid\mathcal G]
 \quad\hbox{satisfies}\quad
 |\partial_t\Phi^{\rm high}(t,\tau)|\le e_{{\rm high},N}(t)
\]
for almost every $0\le t\le8\log N$ and every $0<\tau\le\tau_*$.
A common error envelope for the high and low applications is
\begin{equation}\label{eq:high-derivative-envelope}
 \begin{aligned}
 \mathfrak e_N&=N^{-3/10}+N^{-1/2+1/100}+N^{-4/5}
                   +N^{-79/100}+N^{-9},\\
 e_{{\rm high},N}(t)&\le C_{\tau_*}\mathfrak e_N.
 \end{aligned}
\end{equation}
In particular,
\begin{equation}\label{eq:high-integrated-envelope}\int_0^{8\log N}e_{{\rm high},N}(t)\,\dd t
 \le C_{\tau_*}N^{-1/4}.
\end{equation}

\end{prop}

\supplementproof{supp.E.2}{Section E.2}

\subsection*{Terminal signal and second-moment argument}

Integrating the homogeneous estimate to time $8\log N$ and applying
Proposition~\ref{prop:terminal-grid} gives an additive conditional
Laplace-transform comparison with the Gaussian count. This count dominates
the retained Gaussian events, whose first two moments are controlled.
The initial cutoff sum vanishes outside the union of the right-tail
target and smoothing failure. Lemma~\ref{lem:conditional-occurrence-closure}
therefore yields the conditional probability bound below.

\begin{thm}[Conditional right-tail lower bound]\label{thm:high-occurrence}
Let the independent complex-entry covariance array, dimension path $(M_N)$, aspect ratio $\gamma$, normalized variables $\chi_N^+$, amplitudes
\[
 0<a_0<a_1<a_2<\left(\frac14\right)^{2/3},
\]
parameters
\[
 \kappa_2=\frac43a_2^{3/2},
 \qquad
 \epsilon,\eta,d>0,
 \qquad
 \kappa_2+\epsilon+2\eta<\frac13,
\]
lacunary scales $N_l$, cutoffs $P_l$, sigma-algebras $\mathcal F_l^{\mathrm{past}}$, grids $T_l$, target events
\[
 U_l(a_0)
 =
 \bigcup_{n\in T_l}
 \left\{
 \chi_n^+\geq a_0(\log n)^{2/3}
 \right\},
\]
actual smoothing events $\Omega_l^X$, Gaussian cutoff sum $S_l^G$, and retained Gaussian count $\widetilde R_l^G$ be those of Proposition~\ref{prop:high-poisson-ladder}. Use the same $C^\infty$ cutoff and global weighted cutoff sum at the actual endpoint, throughout the homogeneous flow, and at the Gaussian terminal stage.

Then there are $\mathcal F_l^{\mathrm{past}}$-measurable events
\[
 \mathcal H_l^{+}
\]
and deterministic numbers
\[
 \epsilon_l^{+}\longrightarrow0
\]
such that
\[
 \sum_l
 \Prob\bigl((\mathcal H_l^{+})^c\bigr)
 <\infty
\]
and, on $\mathcal H_l^{+}$ almost surely, for every sufficiently large $l$,
\begin{equation}
 \Prob(U_l(a_0)\mid\mathcal F_l^{\mathrm{past}})
 \geq
 1-\epsilon_l^{+}.
 \label{eq:high-occurrence-lower-bound}
\end{equation}
\end{thm}

\begin{proof}
Fix a sufficiently large selected scale and abbreviate
\[
 N=N_l,
 \qquad
 \mathcal G=\mathcal F_l^{\mathrm{past}},
 \qquad
 T=8\log N.
\]
Choose
\[
 \kappa_0=\frac1{1000},
 \qquad
 A_{\mathrm{fin}}=12
\]
in Proposition~\ref{prop:high-poisson-ladder}. Reindex the complete southeast unrevealed array as in Lemma~\ref{lem:global-high-ladder}, while retaining each matrix in its physical row-column orientation. Let
\[
 \xi_e(t)
 =
 \ee^{-t/2}\xi_e(0)
 +
 \sqrt{1-\ee^{-t}}\,g_e,
 \qquad
 0\leq t\leq T,
\]
be the homogeneous flow of every unnormalized fresh coordinate, and divide the array by $\sqrt N$. Let $S(t)$ be the global weighted cutoff observable and put
\[
 \Phi(t,\tau)
 =
 \E[\exp(-\tau S(t))\mid\mathcal G].
\]

For a fixed window $0<\tau<1$, Proposition~\ref{prop:high-homogeneous-bound} gives $\mathcal G$-measurable events $\mathcal H_l$ whose complements are summable and, on $\mathcal H_l$,
\[
 \abs{\partial_t\Phi(t,\tau)}
 \leq e_l(t)
\]
for almost every $0\leq t\leq T$, where
\begin{equation}
 \int_0^Te_l(t)\,\dd t
 \leq
 CN^{-1/4}.
 \label{eq:high-transport-error}
\end{equation}
The finite-dimensional cutoff sum is a bounded smooth function of the entry coordinates because the resolvents have positive imaginary part and the cutoff is $C^\infty$. The Ornstein--Uhlenbeck semigroup therefore gives a continuously differentiable version of $\Phi$. Integrating \eqref{eq:high-transport-error} gives, uniformly for $0<\tau<1$,
\begin{equation}
 \abs{\Phi(0,\tau)-\Phi(T,\tau)}
 \leq
 CN^{-1/4}.
 \label{eq:high-homogeneous-laplace-comparison}
\end{equation}

We next apply the terminal coupling to the same observable. Write the normalized initial unrevealed array as $X_N$ and the normalized circular Gaussian array as $W_N$. Since $T=8\log N$,
\[
 \ee^{-T/2}=N^{-4},
 \qquad
 1-\ee^{-T}=1-N^{-8}.
\]
Consequently, the terminal array is exactly
\begin{equation}
 X_N^T
 =
 N^{-4}X_N
 +
 \sqrt{1-N^{-8}}\,W_N.
 \label{eq:high-terminal-coupling}
\end{equation}

We verify the hypotheses of Proposition~\ref{prop:terminal-grid}. From
\[
 P_l\leq(\log N)^d,
 \qquad
 n\in[N,2N),
 \qquad
 \frac{M_n}{n}\longrightarrow\gamma,
\]
the two physical dimensions of every fresh matrix lie between fixed positive multiples of $N$ for all sufficiently large $l$. Every matrix is a coordinate restriction, a permitted orientation change, and possible zero padding of the same complete unrevealed array. The common fixed block is zero.

Choose $\theta$ such that
\[
 \kappa_2+2\eta<\theta<\frac13.
\]
Then
\[
 \abs{T_l}
 \leq
 2N^{\kappa_2+2\eta}
 \leq
 N^\theta
\]
eventually, and
\[
 \theta<\frac43-\frac{25}{1000}.
\]

For a matrix with squared singular values $\lambda_s$, divide its endpoints and smoothing height by $N$:
\[
 a_i=\frac{t_{1,l,n}+\widehat\ell_{l,n}}{N},
 \qquad
 B_i=\frac{\widehat B_{l,n}}{N},
 \qquad
 \eta_i=\frac{\widehat\eta_{l,n}}{N}.
\]
Because both dimensions are comparable with $N$, $\eta_i$ lies between fixed positive multiples of
\[
 N^{-2/3-12/1000}.
\]
The endpoint formulas in Proposition~\ref{prop:high-poisson-ladder}, with a fixed power absorbing the logarithmic factors, give
\[
 \abs{B_i-a_i}
 \leq
 CN^{-2/3+1/1000}
\]
eventually. If $L$ is the squared singular-value matrix after division by $\sqrt N$, then the change of variables $x=Ny$ gives
\[
 \begin{aligned}
 &\frac1\pi
 \int_{a_i}^{B_i}
 \frac{\eta_i}
 {(\lambda_s/N-y)^2+\eta_i^2}
 \,\dd y\\
 &\qquad=
 \frac1\pi
 \int_{t_{1,l,n}+\widehat\ell_{l,n}}^{\widehat B_{l,n}}
 \frac{\widehat\eta_{l,n}}
 {(\lambda_s-x)^2+\widehat\eta_{l,n}^2}
 \,\dd x.
 \end{aligned}
\]
Thus the terminal smoothed statistic is exactly $Q_{l,n}$. Its affine variable is exactly the endpoint variable with coefficient
\[
 \frac{a_*}{\delta_l}
\]
and sign $-1$. Moreover,
\[
 \delta_l^{-1}=O(\log N),
 \qquad
 J_{\mathrm{c},l}=O(\log N),
\]
and the ratio $R_{\mathrm{c}}$ and cutoff $\rho_{\mathrm{c}}$ are unchanged. The terminal comparison therefore gives $c_{\mathrm{term}}>0$ such that, uniformly for $0\leq\tau\leq1$,
\begin{equation}
 \abs{\Phi(T,\tau)-\Phi_l^G(\tau)}
 \leq
 N^{-c_{\mathrm{term}}},
 \label{eq:high-terminal-laplace-comparison}
\end{equation}
where
\[
 \Phi_l^G(\tau)
 =
 \E[\exp(-\tau S_l^G)\mid\mathcal G].
\]

Set $\mathcal L_l=U_l(a_0)\cup(\Omega_l^X)^c$. The target support
in Proposition~\ref{prop:high-poisson-ladder} gives
\begin{equation}
 S(0)=0\quad\hbox{on }\mathcal L_l^c.
 \label{eq:high-target-support}
\end{equation}
The same proposition gives $S_l^G\ge\widetilde R_l^G$ and
\[
 \widetilde m_l=\E[\widetilde R_l^G\mid\mathcal G]\ge N^{\eta/2},
 \qquad
 v_l^+=\frac{\Var(\widetilde R_l^G\mid\mathcal G)}{\widetilde m_l^2}
 \longrightarrow0.
\]
The quantities $\widetilde m_l$ and $v_l^+$ are deterministic. Their
bounds retain the product-scale second moment of Paper~I,
Corollary~\ref*{cor:gaussian-grid-moments}, through the Gaussian smoothing
restriction. Fix $\tau\in(0,1)$ and put
$\beta_l^+=CN^{-1/4}+N^{-c_{\rm term}}$.
The comparisons \eqref{eq:high-homogeneous-laplace-comparison} and
\eqref{eq:high-terminal-laplace-comparison}, followed by the variance
bound in Lemma~\ref{lem:conditional-occurrence-closure} with $a=1/2$, give
\begin{equation}
 \Prob(\mathcal L_l\mid\mathcal G)
 \ge1-4v_l^+-\exp(-\tau\widetilde m_l/2)-\beta_l^+.
 \label{eq:high-augmented-event-bound}
\end{equation}
The smoothing failure has conditional probability at most $N_l^{-10}$.
On the homogeneous-comparison history $\mathcal H_l$, therefore,
\[
 \Prob(U_l(a_0)\mid\mathcal G)\ge1-\epsilon_l^+,
 \qquad
 \epsilon_l^+=4v_l^++e^{-\tau N_l^{\eta/2}/2}
                   +\beta_l^++N_l^{-10}\longrightarrow0.
\]
Take $\mathcal H_l^+=\mathcal H_l$. These histories have summable
complements, as required.
\end{proof}

For the left tail, we use a Gaussian gap count at a specified negative
threshold. The comparison retains the northwest block and removes the
two cross strips after the homogeneous flow, with the zero-mode
contribution included in the trace normalization.

\section{Conditional left-tail events}\label{sec:low-occurrence}

\subsection*{Gaussian gap signal at specified thresholds}

We retain the northwest past and prove that the conditional probability
of a left-tail hit tends to one on good histories. The full-matrix
target and the block-diagonal Gaussian reference are joined by the two
strip interpolations.
Corollary~\ref*{cor:gaussian-grid-moments} of Paper~I provides the
Gaussian left-tail count and its first two moments.

A left-tail event requires every squared singular value to lie below the
threshold, and hence a small smoothed count above it. The full-matrix
target uses $t_{l,n}^X$; the Gaussian southeast rectangle is tested at
$t^\sharp_{l,n}=t_{l,n}^X-2\widehat\ell_{l,n}$. The comparison also
retains the northwest block and treats the two cross strips. For a wide
matrix, the column trace includes a deterministic zero-mode contribution.

We remove the Gaussian events on which smoothing fails and embed the
retained count in a full-matrix cutoff sum. The homogeneous flow and
terminal coupling reach a Gaussian moving array with the past fixed.
Two Gaussian interpolations then remove the cross strips and give the
block-diagonal reference family. Its cutoff sum dominates the retained
Gaussian count, while positivity of the target cutoff sum implies the
left-tail event. Lemma~\ref{lem:conditional-occurrence-closure} completes
the probability estimate.

\begin{prop}[Approximation of Gaussian left-tail events]
\label{prop:poisson-activation}
Let \(N_l\to\infty\). For each \(l\), let \(I_l\) be a finite set satisfying
\[
 \abs{I_l}\leq N_l^\theta
\]
for a fixed \(\theta\). For \(i\in I_l\), let \(p_{l,i}\) and \(q_{l,i}\) be integers satisfying
\begin{equation}\label{eq:low-activation-dimensions}
 c_0N_l\leq p_{l,i}\leq q_{l,i}\leq C_0N_l
\end{equation}
for fixed positive constants \(c_0,C_0\), and let \(Z_{l,i}\) be a \(q_{l,i}\)-by-\(p_{l,i}\) matrix of independent centered circular complex Gaussian variables of second absolute moment one. The matrices associated with different \(i\) may be dependent.

Let \(E_{l,i}\) be deterministic thresholds satisfying
\begin{equation}\label{eq:low-activation-wall-window}
 \left|
 \frac{E_{l,i}}{p_{l,i}}
 -\left(1+\sqrt{\frac{q_{l,i}}{p_{l,i}}}\right)^2
 \right|
 \leq
 C_Ep_{l,i}^{-2/3}(\log p_{l,i})^{1/3}.
\end{equation}
Fix \(0<\kappa<1/72\), and define
\[
 \widehat\eta_{l,i}=p_{l,i}^{1/3-12\kappa},
 \qquad
 \widehat\ell_{l,i}=p_{l,i}^{1/3-6\kappa},
 \qquad
 a_{l,i}=E_{l,i}+\widehat\ell_{l,i},
\]
\begin{equation}\label{eq:low-activation-upper-cutoff}
 B_{l,i}
 =(\sqrt{p_{l,i}}+\sqrt{q_{l,i}})^2
  +4p_{l,i}^{1/3+\kappa},
\end{equation}
and
\begin{equation}\label{eq:low-activation-poisson-mass}
 Q_{l,i}
 =
 \sum_s\frac1\pi
 \int_{a_{l,i}}^{B_{l,i}}
 \frac{\widehat\eta_{l,i}}
 {(\lambda_{l,i,s}-u)^2+\widehat\eta_{l,i}^2}\,\dd u,
\end{equation}
where \(\lambda_{l,i,s}\) are the eigenvalues of \(Z_{l,i}^*Z_{l,i}\). Put
\[
 H_{l,i}
 =
 \ind_{\{\lambda_{\max}(Z_{l,i}^*Z_{l,i})\leq E_{l,i}\}},
 \qquad
 R_l=\sum_{i\in I_l}H_{l,i},
 \qquad
 m_l=\E R_l.
\]
Assume that, for fixed \(\alpha>0\) and \(C_R<\infty\),
\begin{equation}\label{eq:low-activation-moment-assumptions}
 m_l\geq N_l^\alpha,
 \qquad
 \E[R_l^2]\leq C_Rm_l^2.
\end{equation}

For every fixed \(D>0\), there are events \(\Omega_{l,i}\), measurable with respect to \(Z_{l,i}\), such that
\begin{equation}\label{eq:low-activation-good-event-probability}
 \Prob(\Omega_{l,i}^c)\leq p_{l,i}^{-D}
\end{equation}
for all sufficiently large \(l\), and, almost surely,
\begin{equation}\label{eq:low-activation-detector-bound}
 H_{l,i}\ind_{\Omega_{l,i}}=1
 \quad\Longrightarrow\quad
 Q_{l,i}\leq p_{l,i}^{-\kappa}.
\end{equation}
If \(D>\theta+2\) and
\[
 \widetilde H_{l,i}=H_{l,i}\ind_{\Omega_{l,i}},
 \qquad
 \widetilde R_l=\sum_{i\in I_l}\widetilde H_{l,i},
 \qquad
 \widetilde m_l=\E\widetilde R_l,
\]
then
\begin{equation}\label{eq:low-activation-retained-mean}
 \widetilde m_l
 =(1+o(1))m_l
 \geq N_l^\alpha(1-o(1)),
\end{equation}
and
\begin{equation}\label{eq:low-activation-retained-second-moment}
 \E[\widetilde R_l^2]
 \leq C_Rm_l^2
 =(C_R+o(1))\widetilde m_l^2.
\end{equation}
If additionally $\E[R_l^2]\le(1+u_l)m_l^2+m_l$ with deterministic
$u_l\to0$, then
\[
 \frac{\Var(\widetilde R_l)}{\widetilde m_l^2}=o(1).
\]
Indeed, $\E(R_l-\widetilde R_l)\le C N_l^{\theta-D}$, so the
retention estimate following Lemma~\ref{lem:conditional-occurrence-closure}
applies with $q_l\le C N_l^{\theta-D}/m_l\to0$.
If the entire Gaussian family is independent of a sigma-algebra
\(\mathcal G_l\), then \eqref{eq:low-activation-retained-mean}--\eqref{eq:low-activation-retained-second-moment} and the relative-variance conclusion also hold conditionally on
\(\mathcal G_l\), almost surely.
\end{prop}

\supplementproof{supp.J.4}{Section J.4}

On a Gaussian gap event, the Poisson mass above the threshold is small
whenever smoothing succeeds. Proposition~\ref{prop:poisson-activation}
shows that discarding the smoothing failures preserves the first- and
second-moment scales of the count, with the dependence between grid
matrices retained.

The next proposition applies the Gaussian grid-moment estimate to the
shifted dimension path. The submatrix threshold is the full-matrix
threshold lowered by $2\widehat\ell_{l,j}$, so the deterministic shifts
match the full-matrix edge.

\begin{prop}[Gaussian left-tail counts]
\label{prop:corrected-gaussian-moments}
Let \(\gamma>0\). Let \((M_N)_{N\geq1}\) be a nondecreasing sequence of positive integers such that
\[
 \frac{M_N}{N}\longrightarrow\gamma,
 \qquad
 0\leq M_{N+1}-M_N\leq K
\]
for a fixed finite \(K\). Fix \(d>0\), and put \(k_0=0\).
For each fixed choice of the amplitudes, grid exponents, and smoothing
parameter satisfying the inequalities below, there are increasing integers
\(k_l\), allowed to depend on all these choices, with
\[
 k_l-k_{l-1}\geq2,
\]
such that, on putting
\[
 N_l=2^{k_l},
 \qquad
 P_l=2^{k_{l-1}+1},
 \qquad
 J_l=N_l-P_l,
\]
we have
\[
 P_l\leq(\log N_l)^d.
\]

For each \(l\), let \(Z_l\) be an independent infinite array of independent centered circular complex Gaussian variables of second absolute moment one. Assume that all \(Z_l\) are independent of the sigma-algebra
\[
 \mathcal F_l^{\mathrm{past}}
 =
 \sigma\set{x_{i,j}:1\leq i\leq M_{P_l},\ 1\leq j\leq P_l}
\]
generated by the original northwest array through row \(M_{P_l}\) and column \(P_l\). For \(j\geq1\), put
\[
 \widehat M_l(j)=M_{P_l+j}-M_{P_l},
\]
and let \(Z_l^{(j)}\) be the \(\widehat M_l(j)\)-by-\(j\) northwest rectangle of \(Z_l\).

For positive \(u,v\), define
\[
 a(u,v)=(\sqrt u+\sqrt v)^2,
\qquad
 b(u,v)
 =(\sqrt u+\sqrt v)(u^{-1/2}+v^{-1/2})^{1/3},
\]
and, for indices \(j\) such that \(\widehat M_l(j)>0\), put
\[
 \widehat\chi_{l,j}
 =
 \frac{\norm{Z_l^{(j)}}_{\op}^2-a(\widehat M_l(j),j)}
 {b(\widehat M_l(j),j)}.
\]

The scales may be chosen large enough that
\(\widehat M_l(j)\asymp j\) throughout \([J_l,2J_l)\), so every
standardization used on either grid has positive dimensions.
Indeed, \(P_l=o(N_l)\), \(M_{P_l}=O(P_l)\), and
\(M_n/n\to\gamma>0\) uniformly once \(n\) exceeds a fixed cutoff.
The soft-edge variables are defined on these positive-dimensional windows.

Fix \(a>0\), put
\[
 \kappa_+=\frac43a^{3/2},
\]
and choose \(\theta_+,\epsilon_+>0\) such that
\begin{equation}\label{eq:high-terminal-exponent-budget}
 \kappa_++\theta_++\epsilon_+<\frac13.
\end{equation}
For \(j\in[J_l,2J_l)\), put
\[
 p_{l,j}=\min\{\widehat M_l(j),j\},
 \qquad
 q_{l,j}=\max\{\widehat M_l(j),j\},
\]
\[
 \mu_{l,j}
 =
 \bigl(\sqrt{p_{l,j}+1/2}+\sqrt{q_{l,j}+1/2}\bigr)^2,
\]
and
\[
 \sigma_{l,j}
 =
 \bigl(\sqrt{p_{l,j}+1/2}+\sqrt{q_{l,j}+1/2}\bigr)
 \bigl((p_{l,j}+1/2)^{-1/2}+(q_{l,j}+1/2)^{-1/2}\bigr)^{1/3}.
\]
Then there are grids \(\widehat T_l^+\subset[J_l,2J_l)\), whose distinct points are separated by at least \(J_l^{2/3+\epsilon_+}\), such that
\[
 \abs{\widehat T_l^+}
 =N_l^{\kappa_++\theta_++o(1)}.
\]
The hard high count
\[
 R_l^+
 =
 \sum_{j\in\widehat T_l^+}
 \ind_{\{\norm{Z_l^{(j)}}_{\op}^2
 \geq\mu_{l,j}+a(\log j)^{2/3}\sigma_{l,j}\}}
\]
satisfies, conditionally on \(\mathcal F_l^{\mathrm{past}}\),
\[
 m_l^+
 :=
 \E[R_l^+\mid\mathcal F_l^{\mathrm{past}}]
 =N_l^{\theta_++o(1)},
\]
\begin{equation}\label{eq:high-terminal-second-moment}
 \E[(R_l^+)^2\mid\mathcal F_l^{\mathrm{past}}]
 \leq(1+o(1))(m_l^+)^2+m_l^+,
\end{equation}
and
\begin{equation}\label{eq:high-terminal-count-concentration}
 \Prob(R_l^+\geq N_l^{\theta_+/2}
       \mid\mathcal F_l^{\mathrm{past}})
 =1-o(1).
\end{equation}
All \(o(1)\) terms in \eqref{eq:high-terminal-second-moment}--\eqref{eq:high-terminal-count-concentration} are deterministic.

Fix \(b_0>0\), put
\[
 \kappa_-=\frac{b_0^3}{12},
\]
and choose \(\theta_-,\epsilon_->0\) such that
\begin{equation}\label{eq:low-terminal-exponent-budget}
 \kappa_-+\theta_-+\epsilon_-<\frac13.
\end{equation}
Fix \(0<\kappa<1/72\). For \(j\in[J_l,2J_l)\), put
\[
 n=P_l+j,
 \qquad
 e_n=a(M_n,n),
 \qquad
 d_n=b(M_n,n),
\]
\begin{equation}\label{eq:low-terminal-strengthened-wall}
 \widehat\ell_{l,j}=p_{l,j}^{1/3-6\kappa},
 \qquad
 t^\sharp_{l,j}
 =
 e_n-b_0d_n(\log n)^{1/3}-2\widehat\ell_{l,j}.
\end{equation}
Then there are grids \(\widehat T_l^-\subset[J_l,2J_l)\), whose distinct points are separated by at least \(J_l^{2/3+\epsilon_-}\), such that
\[
 \abs{\widehat T_l^-}
 =N_l^{\kappa_-+\theta_-+o(1)}.
\]
The strengthened hard low count
\[
 R_l^-
 =
 \sum_{j\in\widehat T_l^-}
 \ind_{\{\norm{Z_l^{(j)}}_{\op}^2\leq t^\sharp_{l,j}\}}
\]
satisfies, conditionally on \(\mathcal F_l^{\mathrm{past}}\),
\[
 m_l^-
 :=
 \E[R_l^-\mid\mathcal F_l^{\mathrm{past}}]
 =N_l^{\theta_-+o(1)},
\]
\begin{equation}\label{eq:low-terminal-second-moment}
 \E[(R_l^-)^2\mid\mathcal F_l^{\mathrm{past}}]
 \leq(1+o(1))(m_l^-)^2+m_l^-,
\end{equation}
and
\begin{equation}\label{eq:low-terminal-count-concentration}
 \Prob(R_l^-\geq N_l^{\theta_-/2}
       \mid\mathcal F_l^{\mathrm{past}})
 =1-o(1).
\end{equation}
Again, all \(o(1)\) terms are deterministic.

Translating either grid by \(P_l\) gives a physical grid contained in \([N_l,2N_l)\), with the same cardinality and separation. In particular, taking \(\theta_+=2\eta\) gives a high terminal count at least \(N_l^\eta\) with conditional probability \(1-o(1)\), while taking \(\theta_-=\eta\) gives a low terminal count at least \(N_l^{\eta/2}\) with conditional probability \(1-o(1)\), whenever the corresponding strict inequalities \eqref{eq:high-terminal-exponent-budget} and \eqref{eq:low-terminal-exponent-budget} hold.
\end{prop}

\supplementproof{supp.J.5}{Section J.5}

\subsection*{Full-matrix cutoff sum and conditional flow}

Proposition~\ref{prop:full-matrix-poisson-ladder} constructs the full-matrix
cutoff sum at $t_{l,n}^X$. At the block-diagonal reference endpoint,
the strengthened Gaussian threshold leaves enough room for the retained
event to activate level $0$ even with the northwest block present.
At the actual endpoint, positivity implies the strict left-tail event.
The later strip interpolation connects these two matrix families.

\begin{prop}[Smooth cutoff functions for the left tail]
\label{prop:full-matrix-poisson-ladder}
Let \(\gamma>0\), and let \((M_N)\) be a nondecreasing sequence of positive integers such that
\[
 \frac{M_N}{N}\longrightarrow\gamma,
 \qquad
 \sup_N(M_{N+1}-M_N)<\infty.
\]
Let \((x_{i,j})\) be independent centered complex random variables satisfying
\[
 \E\abs{x_{i,j}}^2=1,
 \qquad
 \E[x_{i,j}^2]=0,
\]
with uniform bounds on every fixed moment. Put
\[
 X^{(n)}=(x_{i,j})_{1\leq i\leq M_n,\ 1\leq j\leq n},
\]
\[
 e_n=(\sqrt{M_n}+\sqrt n)^2,
\]
\[
 d_n
 =
 (\sqrt{M_n}+\sqrt n)
 (M_n^{-1/2}+n^{-1/2})^{1/3},
 \qquad
 s_n=(\log n)^{1/3},
\]
and
\[
 \chi_n^+
 =
 \frac{\norm{X^{(n)}}_{\op}^2-e_n}{d_n}.
\]

Fix \(0<b<4^{1/3}\), put
\[
 c=\frac{b^3}{12},
\]
and choose \(\eta,\epsilon,d>0\) such that
\[
 c+\eta+\epsilon<\frac13.
\]
Fix \(0<\kappa<1/72\). Apply Proposition~\ref{prop:corrected-gaussian-moments} with high-side parameters
\[
 a=10^{-3},
 \qquad
 \theta_+=\frac1{100},
 \qquad
 \epsilon_+=\frac1{100},
\]
and low-side parameters
\[
 b_0=b,
 \qquad
 \theta_-=\eta,
 \qquad
 \epsilon_-=\epsilon,
\]
together with the present \(d\) and \(\kappa\). It supplies increasing integers \(k_l\), scales
\[
 N_l=2^{k_l},
 \qquad
 P_l=2^{k_{l-1}+1},
 \qquad
 J_l=N_l-P_l,
\]
with \(P_l\leq(\log N_l)^d\), independent circular complex Gaussian arrays \(Z_l\), and low shifted grids \(\widehat T_l^-\subset[J_l,2J_l)\).

Put
\[
 T_l=\set{P_l+j:j\in\widehat T_l^-},
 \qquad
 j_{l,n}=n-P_l,
\]
\[
 u_{l,n}=M_n-M_{P_l},
 \qquad
 v_{l,n}=j_{l,n},
\]
\begin{equation}\label{eq:low-full-tag-dimensions}
 p_{l,n}=\min\{u_{l,n},v_{l,n}\},
 \qquad
 q_{l,n}=\max\{u_{l,n},v_{l,n}\}.
\end{equation}
Then \(T_l\subset[N_l,2N_l)\), its distinct points are separated by at least \(J_l^{2/3+\epsilon}\), and
\[
 \abs{T_l}=N_l^{c+\eta+o(1)}.
\]
The arrays \(Z_l\) are independent of
\[
 \mathcal F_l^{\mathrm{past}}
 =
 \sigma\set{x_{i,j}:1\leq i\leq M_{P_l},\ 1\leq j\leq P_l}.
\]

Define the exact full threshold and strengthened independent-submatrix threshold by
\[
 t_{l,n}^X=e_n-bd_n(\log n)^{1/3},
\]
\begin{equation}\label{eq:low-full-strengthened-wall}
 \widehat\ell_{l,n}=p_{l,n}^{1/3-6\kappa},
 \qquad
 t^\sharp_{l,n}
 =
 t_{l,n}^X-2\widehat\ell_{l,n},
\end{equation}
and put
\[
 H^\sharp_{l,n}
 =
 \ind_{\{\norm{Z_l^{(j_{l,n})}}_{\op}^2\leq t^\sharp_{l,n}\}},
 \qquad
 R_l^\sharp
 =
 \sum_{n\in T_l}H^\sharp_{l,n}.
\]
Conditionally on \(\mathcal F_l^{\mathrm{past}}\),
\[
 m_l^\sharp
 :=
 \E[R_l^\sharp\mid\mathcal F_l^{\mathrm{past}}]
 =
 N_l^{\eta+o(1)}
 \geq N_l^{\eta/2},
\]
and
\begin{equation}\label{eq:low-fresh-count-second-moment}
 \E[(R_l^\sharp)^2\mid\mathcal F_l^{\mathrm{past}}]
 \leq(1+o(1))(m_l^\sharp)^2,
\end{equation}
with deterministic \(o(1)\).

Put
\[
 \widehat\eta_{l,n}
 =
 p_{l,n}^{1/3-12\kappa},
 \qquad
 a_{l,n}
 =
 t_{l,n}^X-\widehat\ell_{l,n},
\]
\begin{equation}\label{eq:low-detector-upper-limit}
 B_{l,n}
 =
 (\sqrt{p_{l,n}}+\sqrt{q_{l,n}})^2
 +4p_{l,n}^{1/3+\kappa}.
\end{equation}
For all sufficiently large \(l\), uniformly on \(T_l\),
\begin{equation}\label{eq:low-detector-window-separation}
 a_{l,n}<t_{l,n}^X\leq B_{l,n},
 \qquad
 t_{l,n}^X-a_{l,n}\geq4\widehat\eta_{l,n}.
\end{equation}

Use the Poisson-plus-upper smoothed statistic of Section~\ref{sec:detectors} with the four parameters in \eqref{eq:low-detector-upper-limit} and target threshold \(t_{l,n}^X\). Fix \(R_{\mathrm{c}}>1\), \(\zeta>0\), and
\[
 \rho_{\mathrm{c}}\in C^3(\R;[0,1]),
\]
with
\[
 \rho_{\mathrm{c}}(y)=1\quad\text{for }y\leq-1,
 \qquad
 \rho_{\mathrm{c}}(y)=0\quad\text{for }y\geq0.
\]
Set
\[
 J_{\mathrm{c},l}=\lceil\zeta\log N_l\rceil,
 \qquad
 \delta_l=\frac{1/16}{J_{\mathrm{c},l}+2},
\]
\[
 y_{l,n}(W)
 =
 \frac{Z_{l,n}(W)-1/16}{\delta_l},
\]
\[
 A_{l,n,r}(W)
 =
 R_{\mathrm{c}}^{-r}
 \rho_{\mathrm{c}}
 \bigl(y_{l,n}(W)+J_{\mathrm{c},l}-r\bigr),
\]
and
\begin{equation}\label{eq:low-global-ladder-count}
 S_l((W_n))
 =
 \sum_{n\in T_l}\sum_{r=0}^{J_{\mathrm{c},l}}
 A_{l,n,r}(W_n).
\end{equation}

Let \(X_l^{\mathrm{past}}=X^{(P_l)}\), define
\begin{equation}\label{eq:low-past-norm-event}
 \mathcal G_l^{\mathrm{past}}
 =
 \set{\norm{X_l^{\mathrm{past}}}_{\op}^2\leq N_l/4},
\end{equation}
and let \(W_{l,n}^T\) be the \(M_n\)-by-\(n\) block-diagonal matrix
\[
 \diag(X_l^{\mathrm{past}},Z_l^{(j_{l,n})})
\]
after adding the necessary zero rectangular blocks and applying row and column permutations.

For every prescribed \(D>0\), there are independent-submatrix events \(\Omega_{l,n}\), independent of \(\mathcal F_l^{\mathrm{past}}\), such that
\[
 \Prob(\Omega_{l,n}^c)\leq p_{l,n}^{-D},
\]
and, on \(\mathcal G_l^{\mathrm{past}}\),
\begin{equation}\label{eq:low-retained-hit-activation}
 H^\sharp_{l,n}\ind_{\Omega_{l,n}}=1
 \quad\Longrightarrow\quad
 Z_{l,n}(W_{l,n}^T)
 \leq C_{\mathrm{act}}N_l^{-\kappa}
\end{equation}
for a deterministic constant \(C_{\mathrm{act}}\) independent of \(l,n\). Consequently, every retained hit activates
\[
 A_{l,n,0}(W_{l,n}^T)=1
\]
for all sufficiently large \(l\).

Define
\[
 \widetilde H_{l,n}
 =
 H^\sharp_{l,n}\ind_{\Omega_{l,n}},
 \qquad
 \widetilde R_l=\sum_{n\in T_l}\widetilde H_{l,n}.
\]
If \(D>c+\eta+2\), then, conditionally on \(\mathcal F_l^{\mathrm{past}}\),
\begin{equation}\label{eq:low-retained-count-mean}
 \E[\widetilde R_l\mid\mathcal F_l^{\mathrm{past}}]
 =
 (1+o(1))m_l^\sharp
 \geq N_l^{\eta/2}(1-o(1)),
\end{equation}
\begin{equation}\label{eq:low-retained-count-second-moment}
 \E[\widetilde R_l^2\mid\mathcal F_l^{\mathrm{past}}]
 \leq
 (1+o(1))
 \E[\widetilde R_l\mid\mathcal F_l^{\mathrm{past}}]^2,
\end{equation}
and, on \(\mathcal G_l^{\mathrm{past}}\),
\begin{equation}\label{eq:low-reference-ladder-dominance}
 S_l((W_{l,n}^T)_{n\in T_l})\geq\widetilde R_l.
\end{equation}

At the target endpoint, put \(W_{l,n}^X=X^{(n)}\). Then
\[
 S_l((W_{l,n}^X)_{n\in T_l})>0
\]
implies
\begin{equation}\label{eq:low-target-event}
 L_l(b)
 :=
 \bigcup_{n\in T_l}
 \set{\frac{\chi_n^+}{s_n}<-b}.
\end{equation}
Finally,
\[
 \sum_l\Prob((\mathcal G_l^{\mathrm{past}})^c)<\infty.
\]
\end{prop}

\supplementproof{supp.J.6}{Section J.6}

In a wide matrix, subtracting the deterministic Poisson contribution of
the structural zero eigenvalues recovers the smaller-dimension
singular-value statistic. The correction has zero positive-order
raw-coordinate derivatives.

For the left-tail family we keep the northwest product corner frozen
and check the spectral heights before applying the common recursion.

\begin{cor}[Left-tail specialization]
\label{prop:four-orientation-recursion}
Adopt the full-array, moment, family, normalized-monomial, and smoothed statistic
assumptions of Theorem~\ref{thm:four-orientation-recursion}, with the same
$0<\kappa_0<1/288$ and $12\kappa_0<\nu<1/12$.
For the endpoint families below its conclusions also hold at
\[
 \eta'_k=p_k^{1/3-12\kappa_0}/n_k,
 \qquad p_k\asymp n_k\asymp N,
\]
with the stated real-part window. The scale is
$\Psi_N=N^{-1/3+\nu}$ and all traces use the full column and row
backgrounds of that theorem. Endpoint representation of derivatives is required through
$K_{\rm der}$; natural averages are converted to normalized free-label
sums by the bounded ratios $N/n_k$ and $N/m_k$.

For the high family of Section~\ref{sec:high-occurrence}, use the reindexed
unrevealed array and empty old label sets, as in
Corollary~\ref{prop:high-orientation-specialization}. For the full-matrix
low family of Proposition~\ref{prop:full-matrix-poisson-ladder}, keep
$(m_k,n_k)=(M_n,n)$ in its physical orientation and take
\[
 \mathcal R_{\rm old}=\{1,\ldots,M_{P_l}\},\qquad
 \mathcal C_{\rm old}=\{1,\ldots,P_l\},\qquad
 \mathcal G=\mathcal F_l^{\rm past}.
\]
All masks, separate conjugation signs, coefficient normalizations, and
cubic terms are exactly those of
Theorem~\ref{thm:four-orientation-recursion}, including the
 degree-lowering estimates
\eqref{eq:revision-cubic-drop}--\eqref{eq:revision-single-offdiag}.
Every moving low source has at least one unrevealed label. For this general application, $12\kappa_0<\nu<1/12$; the homogeneous
comparison below selects $\nu=1/48$.
\end{cor}
\begin{proof}
The full low matrices and the frozen product corner have precisely the raw
coordinates of Theorem~\ref{thm:four-orientation-recursion}. A moving
coordinate lies outside that product, so its row or its column is unrevealed.
The simultaneous event is supplied by
Theorem~\ref{thm:frozen-corner-local-law}. The height ratio
$\eta'_k/n_k^{-2/3-12\kappa_0}=(p_k/n_k)^{1/3-12\kappa_0}$ is bounded;
the height-transfer argument gives $O(\Psi_N)$ bounds throughout that
interval for every deleted and insertion-segment matrix. Its inverse
identities, bounded coefficients, and endpoint calculus are unchanged,
which verifies the spectral-height use of the common recursion.
The derivative identity \eqref{eq:recursion-mixed-endpoint} and the
cutoff derivative tensor bound supply the specified finite-order closure.
For a wide matrix the column smoothed statistic differs from the statistic
on the smaller side, including zero singular values with multiplicity, by
the deterministic structural-zero-mode term in
\eqref{eq:high-physical-detector}. Subtract this term before differentiating.
No separate transpose for each matrix is used. A simultaneous adjoint of an entire family
instead exchanges the two old label sets and pulls all derivatives back
to the same original coordinates.
\end{proof}

For the frozen-corner family, Theorem E.2 of the technical supplement
applies with the estimates in Section E.4.

\begin{prop}[Left-tail comparison]
\label{prop:homogeneous-derivative-bound}
Use Proposition~\ref{prop:full-matrix-poisson-ladder} at
$\kappa=\kappa_0=1/1000$ and $\nu_0=1/48$, with its specified heights,
physical $M_n$-by-$n$ matrices, and
$\mathcal G=\mathcal F_l^{\rm past}$. Keep the product corner
$\{1,\ldots,M_{P_l}\}\times\{1,\ldots,P_l\}$ fixed. Use one
$C^\infty$ cutoff, equal to one on $(-\infty,-1]$ and zero on
$[0,\infty)$, in the target, flow, and strip counts, and choose the depth
so that
\begin{equation}\label{eq:low-homogeneous-final-layer}
 |T_l|R_{\rm c}^{-J_l^{\rm c}}\le N^{-12}.
\end{equation}
Evolve every nonpast raw coordinate by \eqref{eq:recursion-flow}.
Divide each full matrix and its unnormalized smoothed statistic endpoints by the
common scales $\sqrt N$ and $N$, respectively. For the resulting count
$S_N(t)$ put $\Phi_N(t,\tau)=\E[\ee^{-\tau S_N(t)}\mid\mathcal G]$.
On good histories with summable complements,
\[
 |\partial_t\Phi_N(t,\tau)|\le e_{{\rm hom},N}(t)
 \le C_{\tau_*}\mathfrak e_N,
 \qquad
 \int_0^{8\log N}e_{{\rm hom},N}(t)\,\dd t\le C_{\tau_*}N^{-1/4},
\]
with $\mathfrak e_N$ from \eqref{eq:high-derivative-envelope} and the same
time and $\tau$ ranges as Proposition~\ref{prop:high-homogeneous-bound}.

\end{prop}

\supplementproof{supp.E.4}{Section E.4}

\subsection*{Two strips and conditional low occurrence}

At time $8\log N$, the homogeneous flow keeps the northwest corner fixed.
The terminal coupling replaces the moving array by its exact Gaussian
partner. We then remove the two cross strips to reach the block-diagonal
reference family.

With the northwest block and southeast core fixed, follow $(\sqrt v,1)$
from the full matrix to $(0,1)$, then $(0,\sqrt v)$ to $(0,0)$.
The normalization identities in Section~\ref{sec:comparison} match the
spectral endpoints and heights along both paths. The Gaussian heat
identity expresses each derivative through second coordinate derivatives
of the same cutoff sum, which the Schur estimates control. At $(0,0)$,
the nested family has the conditional joint law of the Gaussian
block-diagonal reference.

Adding the homogeneous, coupling, and strip errors gives the conditional
Laplace-transform comparison. The reference cutoff sum dominates the
retained Gaussian count, and the target cutoff sum vanishes outside the
left-tail event. The moment bounds and
Lemma~\ref{lem:conditional-occurrence-closure} give the required lower
bound for its conditional probability.

\begin{thm}[Conditional left-tail lower bound]
\label{thm:low-endpoint-occurrence}
Let \((x_{i,j})_{i,j\geq1}\) be mutually independent complex random variables satisfying
\[
 \E x_{i,j}=0,
 \qquad
 \E\abs{x_{i,j}}^2=1,
 \qquad
 \E[x_{i,j}^2]=0,
\]
with uniform bounds on every fixed absolute moment. Let \(\gamma>0\), let \((M_N)\) be nondecreasing, and assume
\[
 \frac{M_N}{N}\longrightarrow\gamma,
 \qquad
 \sup_N(M_{N+1}-M_N)<\infty.
\]
Put
\[
 X^{(n)}=(x_{i,j})_{1\leq i\leq M_n,\ 1\leq j\leq n},
\]
\[
 e_n=(\sqrt{M_n}+\sqrt n)^2,
\]
\[
 d_n
 =
 (\sqrt{M_n}+\sqrt n)
 (M_n^{-1/2}+n^{-1/2})^{1/3},
\]
\[
 \chi_n^+
 =
 \frac{\norm{X^{(n)}}_{\op}^2-e_n}{d_n},
 \qquad
 s_n=(\log n)^{1/3}.
\]

Fix \(0<b<4^{1/3}\), put
\[
 c=\frac{b^3}{12},
\]
and choose \(\eta,\epsilon,d>0\) such that
\[
 c+\eta+\epsilon<\frac13.
\]
Fix
\[
 \kappa_0=\frac1{1000},
 \qquad
 \nu_0=\frac1{48},
 \qquad
 R_{\mathrm{c}}>1,
\]
and
\[
 \zeta>\frac{c+\eta+12}{\log R_{\mathrm{c}}}.
\]
Choose one
\[
 \rho_{\mathrm{c}}\in C^\infty(\R;[0,1])
\]
equal to one on \((-\infty,-1]\) and zero on \([0,\infty)\).

Take the increasing dyadic scales \(N_l\), past cutoffs
\[
 P_l\leq(\log N_l)^d,
\]
sigma-algebras
\[
 \mathcal G_l
 =
 \mathcal F_l^{\mathrm{past}}
 =
 \sigma\set{x_{i,j}:1\leq i\leq M_{P_l},\ 1\leq j\leq P_l},
\]
low grids \(T_l\subset[N_l,2N_l)\), exact target thresholds, Poisson-plus-upper smoothed statistics, and Gaussian terminal data supplied by Proposition~\ref{prop:full-matrix-poisson-ladder}. Use the chosen \(\rho_{\mathrm{c}}\),
\[
 J_l^{\mathrm{c}}=\lceil\zeta\log N_l\rceil,
\]
the resulting target counts \(S_l^X\), Gaussian block-diagonal terminal counts \(S_l^T\), and retained Gaussian independent-block counts \(\widetilde R_l\).

Then
\[
 \abs{T_l}=N_l^{c+\eta+o(1)},
\]
\begin{equation}\label{eq:low-occurrence-target-support}
 S_l^X>0
 \quad\Longrightarrow\quad
 L_l(b)
 :=
 \bigcup_{n\in T_l}
 \set{\frac{\chi_n^+}{s_n}<-b}.
\end{equation}
There are \(\mathcal G_l\)-measurable events with summable complements on which
\[
 S_l^T\geq\widetilde R_l,
\]
\[
 \E[\widetilde R_l\mid\mathcal G_l]
 \geq
 N_l^{\eta/2}(1-o(1)),
\]
and
\begin{equation}\label{eq:low-occurrence-reference-second-moment}
 \E[\widetilde R_l^2\mid\mathcal G_l]
 \leq
 (1+o(1))
 \E[\widetilde R_l\mid\mathcal G_l]^2,
\end{equation}
where the \(o(1)\) terms are deterministic. In particular,
$\Var(\widetilde R_l\mid\mathcal G_l)=
o(\E[\widetilde R_l\mid\mathcal G_l]^2)$.

There are \(\mathcal G_l\)-measurable events
\[
 \mathcal H_l^{-}
\]
and deterministic numbers \(\epsilon_l\to0\) such that
\[
 \sum_l
 \Prob((\mathcal H_l^{-})^c)
 <\infty
\]
and, on \(\mathcal H_l^{-}\), almost surely, for every sufficiently large \(l\),
\begin{equation}\label{eq:low-occurrence-conditional-bound}
 \Prob(L_l(b)\mid\mathcal G_l)
 \geq
 1-\epsilon_l.
\end{equation}
\end{thm}

\begin{proof}
\noindent\textbf{Step 1. Matrix family and comparison.}
Fix the parameters in the statement, and write
\[
 N=N_l,
 \qquad
 P=P_l,
 \qquad
 \mathcal G=\mathcal G_l.
\]
Proposition~\ref{prop:full-matrix-poisson-ladder} gives
\[
 \abs{T_l}=N^{c+\eta+o(1)}.
\]
Since
\[
 \zeta\log R_{\mathrm{c}}>c+\eta+12,
\]
the prescribed depth \(J_l^{\mathrm{c}}=\lceil\zeta\log N\rceil\) satisfies
\begin{equation}\label{eq:low-occurrence-final-layer}
 \abs{T_l}R_{\mathrm{c}}^{-J_l^{\mathrm{c}}}\leq N^{-12}
\end{equation}
for every sufficiently large \(l\). Proposition~\ref{prop:homogeneous-derivative-bound} therefore applies to exactly \(S_l^X\), with the same physical matrices, endpoints, heights, cutoff levels, and cutoff.

We first account for structural zero eigenvalues in the smoothed count. Fix
\[
 g_0\in C^\infty(\R;[0,1])
\]
that is zero on \((-\infty,0]\) and one on \([1,\infty)\). For a matrix with normalized upper endpoint \(B_i\) and height \(\eta_i\), choose
\begin{equation}\label{eq:low-occurrence-upper-cutoff}
 g_i(x)=g_0\left(\frac{x-B_i}{\eta_i}\right).
\end{equation}
Then
\[
 \operatorname{Lip}(g_i)
 \leq
 \norm{g_0'}_\infty\eta_i^{-1}.
\]
In a fixed orientation, let \(m_i,n_i\) be the full row and column dimensions, put
\[
 h_i=(n_i-m_i)_+,
\]
and let \(P_i\) be the Poisson function of the smoothed statistic. Since the full column covariance has precisely \(h_i\) additional zero eigenvalues when \(n_i>m_i\), while \(g_i(0)=0\) for large \(N\), the resolvent-trace smoothed statistic
\begin{equation}\label{eq:low-occurrence-trace-detector}
 Z_i^{\mathrm{tr}}(L_i)
 =
 \frac1\pi
 \int_{a_i}^{B_i}
 \Im\Tr(L_i-y-\ii\eta_iI)^{-1}\,\dd y
 +
 \Tr g_i(L_i)
\end{equation}
and the min-dimension singular-value smoothed statistic satisfy
\begin{equation}\label{eq:low-occurrence-zero-mode-correction}
 Z_i^{\min}(L_i)
 =
 Z_i^{\mathrm{tr}}(L_i)-h_iP_i(0).
\end{equation}
The correction is deterministic and coordinate-independent, so all positive coordinate derivatives of the two sides agree. This permits one common fresh-smaller orientation throughout the terminal and strip comparisons, including when the limiting aspect is one.

On one product extension, choose a circular standard complex Gaussian raw array \(W\), independent of the original array and all other auxiliary randomness. Use the same array in the homogeneous flow, terminal approximation, and both strip paths. At the final block-diagonal endpoint, only equality of conditional laws is needed: given \(\mathcal G\), the nested independent-block restrictions of \(W\) and those of the Gaussian array in Proposition~\ref{prop:full-matrix-poisson-ladder} are both northwest restrictions of infinite arrays of independent centered circular complex Gaussian variables of second absolute moment one, independent of the same fixed northwest past. Thus the two nested families have the same conditional joint law.

Let \(S_{\mathrm{hom},l}(t)\) denote the homogeneous count and
\[
 \Phi_{\mathrm{hom},l}(t,\tau)
 =
 \E[
 \exp(-\tau S_{\mathrm{hom},l}(t))
 \mid\mathcal G].
\]
Then
\[
 S_{\mathrm{hom},l}(0)=S_l^X.
\]
By Proposition~\ref{prop:homogeneous-derivative-bound}, on histories with summable exceptional probabilities,
\begin{equation}\label{eq:low-occurrence-homogeneous-error}
 \left|
 \Phi_{\mathrm{hom},l}(8\log N,\tau)
 -
 \Phi_{\mathrm{hom},l}(0,\tau)
 \right|
 \leq
 C_\tau N^{-1/4}
\end{equation}
for every fixed \(\tau>0\). At time \(8\log N\), every moving normalized coordinate is
\begin{equation}\label{eq:low-occurrence-terminal-gaussianization}
 N^{-4}X_e+\sqrt{1-N^{-8}}\,W_e,
\end{equation}
while the northwest past remains fixed.

Choose a fixed \(\theta\) with
\[
 c+\eta<\theta<\frac13.
\]
Then \(\abs{T_l}\leq N^\theta\) for large \(l\), and in particular
\[
 \theta<\frac43-25\kappa_0.
\]
For \(n\in T_l\), put
\[
 u=M_n-M_P,
 \qquad
 v=n-P,
\]
\[
 p=\min\{u,v\},
 \qquad
 q=\max\{u,v\}.
\]
Orient the matrix for the natural-scale, terminal, and strip estimates so that \(p\) is the fresh column dimension and \(q\) is the fresh row dimension. The past dimensions are swapped simultaneously when conjugate transposition is used.

Bounded increments, \(M_N/N\to\gamma\), \(n\in[N,2N)\), and \(P\leq(\log N)^d\) imply
\begin{equation}\label{eq:low-occurrence-dimension-range}
 c_{\mathrm{dim}}N
 \leq p\leq q
 \leq C_{\mathrm{dim}}N.
\end{equation}
All matrices are restrictions of one raw array, and their frozen coordinates form the single product corner
\begin{equation}\label{eq:low-occurrence-frozen-corner}
 \{1,\ldots,M_P\}\times\{1,\ldots,P\}.
\end{equation}

Use the unnormalized spectral data
\begin{equation}\label{eq:low-occurrence-spectral-data}
 \widehat A_k=t_{l,n}^X-\widehat\ell_{l,n},
 \qquad
 \widehat T_k=t_{l,n}^X,
 \qquad
 \widehat B_k=B_{l,n},
 \qquad
 \widehat\eta_k=p^{1/3-12\kappa_0}.
\end{equation}
Proposition~\ref{prop:full-matrix-poisson-ladder} gives
\[
 \widehat A_k<\widehat T_k\leq\widehat B_k,
 \qquad
 \widehat T_k-\widehat A_k\geq4\widehat\eta_k.
\]
The difference between the full upper edge and independent-block upper edge is \(O(M_P+P)\). Consequently,
\begin{equation}\label{eq:low-occurrence-edge-window}
\begin{aligned}
 &\abs{\widehat A_k-(\sqrt{m_k}+\sqrt{n_k})^2}
 +
 \abs{\widehat B_k-(\sqrt{m_k}+\sqrt{n_k})^2}\\
 &\qquad
 \leq
 C\left[
 N^{1/3}(\log N)^{1/3}
 +(M_P+P)
 +N^{1/3+\kappa_0}
 \right].
\end{aligned}
\end{equation}
After division by the appropriate full dimension, all terms except the final endpoint displacement are
\[
 o(N^{-2/3+\kappa_0}),
\]
while the final displacement has leading coefficient four. Thus both normalized full endpoints lie within
\[
 5n_k^{-2/3+\kappa_0}
\]
of the full upper edge. After division by \(p\), the same argument places both endpoints within
\[
 5p^{-2/3+\kappa_0}
\]
of the independent-block upper edge.

For the full dimensions \(m_k,n_k\), put
\[
 h_k=(n_k-m_k)_+.
\]
In Proposition~\ref{prop:terminal-grid}, choose affine sign \(+1\) and shift
\begin{equation}\label{eq:low-occurrence-affine-shift}
 \alpha_k
 =
 -\delta_l^{-1}
 \left(\frac1{16}+h_kP_k(0)\right).
\end{equation}
By \eqref{eq:low-occurrence-zero-mode-correction}, its affine coordinate is exactly
\[
 \delta_l^{-1}(Z_k^{\min}-1/16),
\]
so its complete cutoff sum is the specified target cutoff sum even when the full column dimension is larger. Equation~\eqref{eq:low-occurrence-upper-cutoff} gives the required upper-cutoff Lipschitz bound. The interval width is
\[
 O(N^{-2/3+\kappa_0}),
\]
the height is comparable to
\[
 N^{-2/3-12\kappa_0},
\]
the depth satisfies \(J_l^{\mathrm{c}}=O(\log N)\), and the inverse transition width is \(O(\log N)\).

On the event \(\mathcal G_l^{\mathrm{past}}\) from Proposition~\ref{prop:full-matrix-poisson-ladder}, the fixed normalized northwest block has operator norm at most \(1/2\). Proposition~\ref{prop:terminal-grid}, applied to \eqref{eq:low-occurrence-terminal-gaussianization}, therefore gives a constant \(c_{\mathrm{term}}>0\) such that
\begin{equation}\label{eq:low-occurrence-terminal-error}
 \left|
 \Phi_{\mathrm{hom},l}(8\log N,\tau)
 -
 \Phi_{\mathrm G,l}(\tau)
 \right|
 \leq
 N^{-c_{\mathrm{term}}},
\end{equation}
where \(\Phi_{\mathrm G,l}\) is the conditional Laplace transform of the exact Gaussian all-moving count with the same matrix family, smoothed statistic, cutoff sum, and cutoff.

We next connect the all-moving point to the block-diagonal point by two exact strip paths. The first path has amplitudes
\[
 (\sqrt v,1),
\]
so \(v=1\) is the all-moving Gaussian point and \(v=0\) is the point \((0,1)\). The second path has amplitudes
\[
 (0,\sqrt v),
\]
so \(v=1\) is the same point \((0,1)\), while \(v=0\) is the block-diagonal matrix with the fixed northwest past and fresh core from \(W\), divided by \(\sqrt N\).

Define
\begin{equation}\label{eq:low-occurrence-corrected-coordinate}
 \overline Z_k
 =
 Z_k^{\mathrm{tr}}-h_kP_k(0)
 =
 Z_k^{\min},
 \qquad
 y_k=\frac{\overline Z_k-1/16}{\delta_l}.
\end{equation}
Use these coordinates in every cutoff sum summand on both strip paths.

Fix an active strip label \(e\), a real direction \(a\), and a path point \(v\). Let \(\mathcal A_{e,a,v}\) be the set of matrices containing \(e\) for which some cutoff level has nonzero \(\rho_{\mathrm{c}}'\) or \(\rho_{\mathrm{c}}''\). Because the term subtracted in \eqref{eq:low-occurrence-corrected-coordinate} is deterministic, the first and second strip-coordinate derivatives equal those of the trace smoothed statistic.

For \(k\in\mathcal A_{e,a,v}\), derivative support lies in the transition interval. Thus
\[
 y_k+J_l^{\mathrm{c}}-r<0
\]
for some \(0\leq r\leq J_l^{\mathrm{c}}\), which implies
\[
 y_k<0,
 \qquad
 \overline Z_k<\frac1{16}.
\]
The smoothed statistic support implication gives
\[
 \norm{W_k}_{\op}^2<t_k\leq B_k.
\]
Hence the upper spectral cutoff is locally constant. The strip-coordinate estimates in Theorem~\ref{thm:natural-tag-interpolation} give
\[
 \max_{k\in\mathcal A_{e,a,v}}
 \abs{\partial_{h_a}y_k}
 \leq CN^{-3/10},
\]
\begin{equation}\label{eq:low-occurrence-strip-second-derivative}
 \max_{k\in\mathcal A_{e,a,v}}
 \abs{\partial_{h_a}^2y_k}
 \leq CN^{-3/10}.
\end{equation}

Let \(S_{j,l}(v)\) be the target cutoff sum on strip \(j\), and put
\[
 \Phi_{j,l}(v,\tau)
 =
 \E[\exp(-\tau S_{j,l}(v))\mid\mathcal G],
\]
\[
 D_{j,l}(v,\tau)
 =
 \E[\tau S_{j,l}(v)\exp(-\tau S_{j,l}(v))\mid\mathcal G].
\]
For one active circular Gaussian raw coordinate, write its two real coordinates as \(Q_1,Q_2\), each of variance \(1/2\), and set
\[
 h_a(v)=\frac{\sqrt v\,Q_a}{\sqrt N}.
\]
Real Gaussian integration by parts gives the exact heat identity
\begin{equation}\label{eq:low-occurrence-heat-identity}
 \partial_v\Phi_{j,l}(v,\tau)
 =
 \frac1{4N}
 \sum_{e\in E_j}\sum_{a=1}^2
 \E\left[
 \partial_{h_a}^2
 \exp(-\tau S_{j,l,e}(h(v)))
 \,\middle|\,
 \mathcal G
 \right].
\end{equation}
This extends continuously to \(v=0\).

The chain rule gives
\[
 \partial_{h_a}^2\exp(-\tau S)
 =
 \exp(-\tau S)
 \left[
 -\tau\partial_{h_a}^2S
 +\tau^2(\partial_{h_a}S)^2
 \right].
\]
Inactive matrices contribute zero to both derivatives because all their cutoff derivatives vanish. Applying Lemma~\ref{lem:ladder-absorption}, using \eqref{eq:low-occurrence-final-layer} and \eqref{eq:low-occurrence-strip-second-derivative}, gives
\[
 \abs{\partial_{h_a}S}
 \leq
 CN^{-3/10}S+CN^{-12},
\]
\begin{equation}\label{eq:low-occurrence-ladder-second-derivative}
 \abs{\partial_{h_a}^2S}
 \leq
 C(N^{-3/10}+N^{-3/5})S+CN^{-12}.
\end{equation}

For \(x\geq0\), both \(x\ee^{-\tau x}\) and \(x^2\ee^{-\tau x}\) are bounded by a \(\tau\)-dependent constant times \(x\ee^{-\tau x/2}\). The number of strip labels is at most
\[
 C(M_P+P)N.
\]
Choosing the conditional local-law failure exponent sufficiently large pays for the polynomial resolvent bound, the coordinate truncation, and this label count. Substituting \eqref{eq:low-occurrence-ladder-second-derivative} into \eqref{eq:low-occurrence-heat-identity} therefore gives nonnegative functions \(a_{j,l}\) and \(e_{j,l}\) such that
\begin{equation}\label{eq:low-occurrence-strip-differential-bound}
 \abs{\partial_v\Phi_{j,l}(v,\tau)}
 \leq
 a_{j,l}(v)D_{j,l}(v,\tau/2)+e_{j,l}(v),
\end{equation}
with
\begin{equation}\label{eq:low-occurrence-main-coefficient-integral}
 \int_0^1a_{j,l}(v)\,\dd v
 \leq
 C(M_P+P)(N^{-3/10}+N^{-3/5})
 \leq N^{-1/4},
\end{equation}
and
\begin{equation}\label{eq:low-occurrence-error-integral}
 \int_0^1e_{j,l}(v)\,\dd v
 \leq C_\tau N^{-9}.
\end{equation}

Let \(S_l^{(T,W)}\) denote the final block-diagonal count built from \(W\), while \(S_l^T\) denotes the count built from the Gaussian array in Proposition~\ref{prop:full-matrix-poisson-ladder}. Conditionally on \(\mathcal G\), the two block-diagonal nested families have the same joint law. Hence
\begin{equation}\label{eq:low-occurrence-reference-law}
 \E[\exp(-\tau S_l^{(T,W)})\mid\mathcal G]
 =
 \E[\exp(-\tau S_l^T)\mid\mathcal G].
\end{equation}

\noindent\textbf{Step 2. Fixed-parameter Laplace telescoping.}
Fix
\[
 \tau=\frac12.
\]
For every nonnegative random variable \(S\) and every \(u>0\),
\begin{equation}\label{eq:low-occurrence-laplace-elementary-bound}
 \E[uS\exp(-uS)\mid\mathcal G]\leq\ee^{-1},
\end{equation}
because \(x\ee^{-x}\leq\ee^{-1}\) for \(x\geq0\).

Integrating \eqref{eq:low-occurrence-strip-differential-bound} and using \eqref{eq:low-occurrence-main-coefficient-integral}--\eqref{eq:low-occurrence-laplace-elementary-bound} gives
\begin{equation}\label{eq:low-occurrence-strip-error}
\begin{aligned}
 &\abs{\Phi_{1,l}(1,\tau)-\Phi_{1,l}(0,\tau)}
 +
 \abs{\Phi_{2,l}(1,\tau)-\Phi_{2,l}(0,\tau)}\\
 &\qquad
 \leq
 2\ee^{-1}N^{-1/4}
 +
 2C_\tau N^{-9}.
\end{aligned}
\end{equation}
The two paths meet at the same one-strip endpoint.

Combining \eqref{eq:low-occurrence-homogeneous-error}, \eqref{eq:low-occurrence-terminal-error}, \eqref{eq:low-occurrence-reference-law}, and \eqref{eq:low-occurrence-strip-error}, define
\begin{equation}\label{eq:low-occurrence-total-comparison-error}
 \beta_l
 =
 C_\tau N_l^{-1/4}
 +N_l^{-c_{\mathrm{term}}}
 +2\ee^{-1}N_l^{-1/4}
 +2C_\tau N_l^{-9}.
\end{equation}
Then \(\beta_l\to0\), and on the intersection of the good histories,
\begin{equation}\label{eq:low-occurrence-laplace-comparison}
 \left|
 \E[\exp(-\tau S_l^X)\mid\mathcal G_l]
 -
 \E[\exp(-\tau S_l^T)\mid\mathcal G_l]
 \right|
 \leq\beta_l.
\end{equation}

\noindent\textbf{Step 3. Second-moment argument.}
Put $R_l=\widetilde R_l$ and $m_l=\E[R_l\mid\mathcal G_l]$.
Proposition~\ref{prop:full-matrix-poisson-ladder} gives, on its good
histories,
\[
 S_l^T\ge R_l,\qquad m_l\ge N_l^{\eta/2}(1-o(1)),\qquad
 v_l^-:=\frac{\Var(R_l\mid\mathcal G_l)}{m_l^2}=o(1).
\]
The moments of this retained count are deterministic because it is a
function of the independent Gaussian array. Apply the variance estimate
in Lemma~\ref{lem:conditional-occurrence-closure} with $a=1/2$.
The support implication \eqref{eq:low-occurrence-target-support} and
comparison \eqref{eq:low-occurrence-laplace-comparison} give
\[
 \Prob(L_l(b)\mid\mathcal G_l)\ge1-\epsilon_l,
\]
where
\begin{equation}\label{eq:low-occurrence-final-error}
 \epsilon_l=4v_l^-+e^{-\tau m_l/2}+\beta_l\longrightarrow0.
\end{equation}
Let $\mathcal H_l^-$ be the intersection of the good histories used for
the past, the homogeneous flow, and the two strip paths. Their
complements have summable probabilities, and the preceding bound holds
on $\mathcal H_l^-$.
\end{proof}

The two occurrence estimates will be combined with the fixed-size tail
bounds and control of the intervening matrix sizes developed next.

\section{Tail estimates and eigenvalue increments}\label{sec:fixed-level}

\subsection*{Tail estimates and short-range interpolation}

We now bound excursions beyond the proposed endpoints. Gaussian Laguerre
tails and fixed-size comparison give polynomial probability bounds with
exponents strictly above \(1/3\). One-step row and column estimates,
combined with martingale control, bound excursions between consecutive
points of a polynomial grid. These two inputs yield the summable
pathwise exclusions used in the main proof.

\begin{prop}[Gaussian tail estimates]\label{prop:lfl-local-gaussian-supercritical-margins}
Let \(M_N/N\to\gamma>0\), and let \(G_N\) be an \(M_N\times N\) matrix with independent centered circular complex Gaussian entries of second absolute moment one.  With
\[
 \chi_{N,G}^+
 =
 \frac{\lambda_{\max}(N^{-1}G_NG_N^*)-(\sqrt{M_N}+\sqrt N)^2/N}
 {N^{-1}(\sqrt{M_N}+\sqrt N)(M_N^{-1/2}+N^{-1/2})^{1/3}},
\]
\(r_N=(\log N)^{2/3}\), \(s_N=(\log N)^{1/3}\),
\(A_2=(1/4)^{2/3}\), and \(B_2=4^{1/3}\), the following hold:
\begin{enumerate}
\item for every \(q>A_2\), there are \(\eta_q^{G,+}>1/3\) and \(N_q^{G,+}\) such that
\[
 \Prob(\chi_{N,G}^+\geq qr_N)\leq N^{-\eta_q^{G,+}}
 \quad(N\geq N_q^{G,+});
\]
\item for every \(q>B_2\), there are \(\eta_q^{G,-}>1/3\) and \(N_q^{G,-}\) such that
\[
 \Prob(\chi_{N,G}^+\leq-qs_N)\leq N^{-\eta_q^{G,-}}
 \quad(N\geq N_q^{G,-}).
\]
\end{enumerate}
\end{prop}

\begin{proof}
Put
\[
 d_N=\min(M_N,N),
 \qquad
 e_N=\max(M_N,N).
\]
After conjugate transposition when necessary, the nonzero squared singular values of \(G_N\) form the \(\beta=2\) Laguerre ensemble with ordered dimensions \((d_N,e_N)\).  Since \(M_N/N\to\gamma>0\),
\[
 1\leq\frac{e_N}{d_N}=O(1),
 \qquad
 \frac{\log d_N}{\log N}\longrightarrow1.
\]
Fix a deterministic constant \(M>1\) such that
\[
 \frac{e_N}{d_N}\leq M
\]
for all sufficiently large \(N\).  The unnormalized soft-edge center
\[
 (\sqrt{d_N}+\sqrt{e_N})^2
\]
and scale
\[
 (\sqrt{d_N}+\sqrt{e_N})
 (d_N^{-1/2}+e_N^{-1/2})^{1/3}
\]
are symmetric in the two dimensions, so this transposition leaves \(\chi_{N,G}^+\) unchanged.

For the right tail, fix \(q>A_2\) and choose \(\varepsilon_+>0\) so small that
\[
 \frac43(1-\varepsilon_+)q^{3/2}>\frac13.
\]
Apply the bounded-aspect Laguerre right-tail upper bound
\cite[Theorem~1.1(ii)]{BBBK24} with \(\beta=2\) and
\[
 t=q(\log N)^{2/3}.
\]
Because \(d_N\asymp N\), this value of \(t\) is eventually above the theorem's lower threshold and below its upper range
\(\gamma(\varepsilon_+,M)d_N^{2/3}\).  The normalization reduction therefore gives
\[
 \Prob(\chi_{N,G}^+\geq q r_N)
 \leq
 N^{-\frac43(1-\varepsilon_+)q^{3/2}}
\]
for all sufficiently large \(N\).  Choose
\[
 \frac13<\eta_q^{G,+}<\frac43(1-\varepsilon_+)q^{3/2}.
\]

For the left tail, fix \(q>B_2\), choose any sufficiently small
\(\delta>0\), and then choose \(\varepsilon_->0\) so small that
\[
 \frac1{12}(1-\varepsilon_-)q^3>\frac13.
\]
Apply the bounded-aspect Laguerre left-tail upper bound
\cite[Theorem~1.4(ii)]{BBBK24} with \(\beta=2\) and
\[
 t=q(\log N)^{1/3}.
\]
Again \(d_N\asymp N\), so this value is eventually in the permitted range
\(t\leq d_N^{1/6-\delta}\).  Hence
\[
 \Prob(\chi_{N,G}^+\leq-q s_N)
 \leq
 N^{-\frac1{12}(1-\varepsilon_-)q^3}
\]
for all sufficiently large \(N\).  Choose
\[
 \frac13<\eta_q^{G,-}<\frac1{12}(1-\varepsilon_-)q^3.
\]
The identities
\[
 \frac43A_2^{3/2}=\frac13,
 \qquad
 \frac1{12}B_2^3=\frac13
\]
ensure that the two choices are possible.  
\end{proof}

To control the matrix sizes between grid points, we estimate one-step
increments in the two directions of the northwest path. A new column
tests a left singular vector; a new row tests a right singular vector.
After subtraction of the edge-center increment, the leading terms have
conditional mean zero and provide the martingale part of the interpolation.

\begin{lem}[One-step eigenvalue increments]\label{lem:one-step-growth}
Let \(0<c<C<\infty\) be fixed. For each pair of positive integers \((m,n)\) with
\[
cn\leq m\leq Cn,
\]
let \(Y_{m,n}=(x_{ij})_{1\leq i\leq m,\,1\leq j\leq n}\), where the entries are independent centered complex random variables satisfying
\[
\E\abs{x_{ij}}^2=1,
\qquad
\E x_{ij}^2=0,
\]
and have uniformly bounded fixed absolute moments. Put
\[
L_{m,n}=\lambda_{\max}(Y_{m,n}Y_{m,n}^*),
\qquad
A(m,n)=(\sqrt m+\sqrt n)^2.
\]
Let
\[
g=(x_{1,n+1},\ldots,x_{m,n+1})^{\mathsf T},
\]
choose a measurable unit top eigenvector \(u\) of \(Y_{m,n}Y_{m,n}^*\), and put
\[
W_{\mathrm{col}}=\abs{u^*g}^2.
\]
Let
\[
h=(\overline{x_{m+1,1}},\ldots,\overline{x_{m+1,n}})^{\mathsf T},
\]
choose a measurable unit top eigenvector \(v\) of \(Y_{m,n}^*Y_{m,n}\), and put
\[
W_{\mathrm{row}}=\abs{v^*h}^2.
\]
For every \(\rho>0\) and \(D>0\), there is a constant \(C_{\rho,D}<\infty\) such that, for all sufficiently large \(n\), outside an event of probability at most \(C_{\rho,D}n^{-D}\),
\[
L_{m,n+1}-L_{m,n}
\geq
\left(1+\sqrt{\frac mn}\right)W_{\mathrm{col}}
-C_{\rho,D}n^{-1/3+\rho},
\]
and
\[
L_{m+1,n}-L_{m,n}
\geq
\left(1+\sqrt{\frac nm}\right)W_{\mathrm{row}}
-C_{\rho,D}n^{-1/3+\rho}.
\]
Consequently, after increasing \(C_{\rho,D}\) if necessary,
\[
\begin{split}
&[L_{m,n+1}-A(m,n+1)]-[L_{m,n}-A(m,n)]\\
&\qquad\geq
\left(1+\sqrt{\frac mn}\right)(W_{\mathrm{col}}-1)
-C_{\rho,D}n^{-1/3+\rho},
\end{split}
\]
and
\[
\begin{split}
&[L_{m+1,n}-A(m+1,n)]-[L_{m,n}-A(m,n)]\\
&\qquad\geq
\left(1+\sqrt{\frac nm}\right)(W_{\mathrm{row}}-1)
-C_{\rho,D}n^{-1/3+\rho}.
\end{split}
\]
Conditionally on \(Y_{m,n}\), both \(W_{\mathrm{col}}-1\) and \(W_{\mathrm{row}}-1\) have conditional mean zero and uniformly bounded fixed moments.
\end{lem}

\supplementproof{supp.I.1}{Section I.1}

Proposition~\ref{prop:fixed-level-derivative} compares the Gaussian and
original ensembles at a fixed matrix size. Along the Ornstein--Uhlenbeck
flow, its differential inequality passes the derivative term to the next
shifted cutoff. The remaining cumulant and Taylor terms are
\(O(N^{-\sigma})\), with \(\sigma>1/3\).
Here $\mathcal C_{r,K,j}$ denotes the order-$r$ cumulant contribution
and $\mathcal R_{6,K,j}$ the sixth-order Taylor remainder.

\begin{prop}[Derivative bound for shifted cutoffs]\label{prop:fixed-level-derivative}
Fix constants
\[
\gamma\geq1,\qquad q_R,q_L>0,\qquad
0<\kappa_0<\frac1{288},
\qquad
12\kappa_0<\nu<\frac1{12}.
\]
Let \(M_N\geq N\) be integers with \(M_N/N\to\gamma\). For
\(1\leq\alpha\leq M_N\) and \(1\leq b\leq N\), let \(x_{\alpha b}\) be independent complex random variables satisfying
\[
\E x_{\alpha b}=0,
\qquad
\E\abs{x_{\alpha b}}^2=1,
\qquad
\E x_{\alpha b}^2=0,
\]
and
\[
\sup_{N,\alpha,b}\E\abs{x_{\alpha b}}^p<\infty
\]
for every fixed \(p\). Let \(g_{\alpha b}\) be independent standard circular complex Gaussian variables, independent of the \(x_{\alpha b}\), and put
\[
\begin{aligned}
\xi_{\alpha b}(t)
&=\ee^{-t/2}x_{\alpha b}
+\sqrt{1-\ee^{-t}}\,g_{\alpha b},\\
X_N(t)_{\alpha b}&=N^{-1/2}\xi_{\alpha b}(t),
&0&\leq t\leq8\log N.
\end{aligned}
\]
Define
\[
H_N(t,z)=
\begin{pmatrix}
-zI_N&X_N(t)^*\\
X_N(t)&-I_{M_N}
\end{pmatrix},
\qquad
G_N(t,z)=H_N(t,z)^{-1},
\]
and
\[
E_{+,N}=\left(1+\sqrt{\frac{M_N}{N}}\right)^2,
\]
\[
\begin{aligned}
\sigma_{+,N}
&=\frac{\sqrt{M_N}+\sqrt N}{N}
\left(\frac1{\sqrt{M_N}}+\frac1{\sqrt N}\right)^{1/3},\\
r_N&=(\log N)^{2/3},
&s_N&=(\log N)^{1/3},
\end{aligned}
\]
\[
\eta_N=N^{-2/3-12\kappa_0},
\qquad
\ell_N=N^{-2/3-6\kappa_0},
\qquad
B_N=E_{+,N}+4N^{-2/3+\kappa_0},
\]
\[
A_{R,N}=E_{+,N}+q_Rr_N\sigma_{+,N},
\qquad
A_{L,N}=E_{+,N}-q_Ls_N\sigma_{+,N}.
\]

Fix an integer \(J\geq1\). For \(K\in\{R,L\}\) and \(0\leq j\leq J\), put
\[
a_{R,j}=A_{R,N}-(2j+1)\ell_N,
\qquad
a_{L,j}=A_{L,N}+(2j+1)\ell_N,
\]
\[
Z_{K,j}(t)
=
\frac1\pi
\int_{a_{K,j}}^{B_N}
\Im\Tr\!\left(P_NG_N(t,x+\ii\eta_N)\right)\dd x,
\qquad
\mathcal X_{K,j}(t)=\pi Z_{K,j}(t),
\]
where \(P_N\) projects onto the first \(N\) linearization coordinates.

Let \(F_0,\ldots,F_J\colon\R\to[0,1]\) be smooth, with every fixed derivative bounded uniformly in \(j\). In right mode, assume that for each \(j<J\) there is a Borel set \(S_j\) with finite infimum such that every positive derivative of \(F_j\) vanishes outside \(S_j\), and
\[
F_{j+1}(u)=1
\qquad\text{for }u\geq\inf S_j.
\]
In left mode, assume instead that \(S_j\) has finite supremum, every positive derivative of \(F_j\) vanishes outside \(S_j\), and
\[
F_{j+1}(u)=1
\qquad\text{for }u\leq\sup S_j.
\]
Put
\[
f_{K,j}(t)=\E F_j(Z_{K,j}(t)),
\qquad
\widehat F_j(u)=F_j(u/\pi).
\]

For a label \(e=(\alpha,b)\), define
\[
S_1=E_{\alpha b}+E_{b\alpha},
\qquad
S_2=\ii E_{\alpha b}-\ii E_{b\alpha},
\qquad
D_aG=-GS_aG,
\]
and
\[
J_1=G_{b\alpha}+G_{\alpha b},
\qquad
J_2=\ii(G_{b\alpha}-G_{\alpha b}).
\]
Let \(\Delta_{K,j}\) denote evaluation at \(B_N+\ii\eta_N\) minus evaluation at \(a_{K,j}+\ii\eta_N\). For \(a_1,\ldots,a_r\in\{1,2\}\), let
\[
\mathcal K_{e;a_1\cdots a_r}(t)
\]
be the joint cumulant of the indicated real coordinates of \(\xi_e(t)\). For \(r\geq3\),
\[
\mathcal K_{e;a_1\cdots a_r}(t)
=
\ee^{-rt/2}\mathcal K_{e;a_1\cdots a_r}(0).
\]
For \(r=3,4,5\), define
\[
\begin{split}
\mathcal C_{r,K,j}(t)
={}&
\frac1{2(r-1)!N^{r/2}}
\sum_e
\sum_{a_1,\ldots,a_r\in\{1,2\}}
\mathcal K_{e;a_1\cdots a_r}(t)\\
&\times
\E D_{a_2}\cdots D_{a_r}
\left\{
\widehat F_j'(\mathcal X_{K,j})
\Delta_{K,j}\Im J_{a_1}
\right\}.
\end{split}
\]
Then, for almost every \(t\in[0,8\log N]\),
\begin{equation}
f'_{K,j}(t)
=
\mathcal C_{3,K,j}(t)
+\mathcal C_{4,K,j}(t)
+\mathcal C_{5,K,j}(t)
+\mathcal R_{6,K,j}(t).
\label{eq:ou-source-expansion}
\end{equation}
There are constants \(C<\infty\), \(\sigma>1/3\), and \(N_0\) such that, simultaneously for \(K=R,L\), \(0\leq j<J\), \(0\leq t\leq8\log N\), and \(N\geq N_0\),
\begin{equation}
\abs{\mathcal C_{3,K,j}}
+
\abs{\mathcal C_{5,K,j}}
+
\abs{\mathcal R_{6,K,j}}
\leq CN^{-\sigma},
\label{eq:odd-source-bound}
\end{equation}
\begin{equation}
\abs{\mathcal C_{4,K,j}}
\leq
CN^{-1/3+\nu}f_{K,j+1}(t)+CN^{-\sigma},
\label{eq:fourth-source-bound}
\end{equation}
and hence
\begin{equation}
\abs{f'_{K,j}(t)}
\leq
CN^{-1/3+\nu}f_{K,j+1}(t)+CN^{-\sigma}.
\label{eq:shifted-derivative-bound}
\end{equation}

The scalar-cutoff recursion used in \eqref{eq:odd-source-bound} has the following form. Open every raw derivative of \(Z_{K,j}\) into endpoint Green-function monomials. A resolvent monomial with a scalar cutoff is a normalized free-label average whose coefficient is a bounded deterministic factor times one derivative
\[
\widehat F_j^{(k)}(\mathcal X_{K,j}),
\]
and whose explicit factors are possibly conjugated endpoint Green-function entries, possibly conjugated centered diagonal entries, or possibly conjugated centered partial traces. Its degree is the number of off-diagonal, centered-diagonal, and centered-partial-trace factors. It is unmatched if some free label has odd total multiplicity in the row and column positions. Put
\[
\Psi_N=N^{-1/3+\nu}.
\]
For every fixed \(D_*\geq4\), every fixed choice of initial bounds on the factor count and cutoff-derivative order, and every \(\epsilon>0\), each initial unmatched scalar-cutoff monomial \(Q_d\) of degree \(d<D_*\) satisfies
\begin{equation}
\abs{\E Q_d}
\leq
C_{\epsilon,D_*}N^\epsilon
\left\{N^{-1}+\Psi_N^{D_*}\right\}
+CN^{-20}.
\label{eq:scalar-cutoff-recursion}
\end{equation}
All eight ordered triples of real-coordinate indices
\[
111,\ 112,\ 121,\ 211,\ 122,\ 212,\ 221,\ 222
\]
are retained. No cancellation between non-identical third-cumulant coefficients is assumed.
\end{prop}

\supplementproof{supp.I.2}{Section I.2}

The derivative estimate is applied with the exact orientation and entry structure of each fixed-level clause. The first clause below is a tall-orientation triangular-array statement with \(M_N\geq N\). The other two clauses are common-array estimates valid across both tall and wide orientations. Their event conventions also remain distinct: the tall right event uses \(\geq\), the common-array right event uses \(>\), and the common-array left event uses \(\leq\).

\begin{thm}[Tail estimates at a fixed matrix size]\label{thm:fixed-level-margins}
The following three fixed-level estimates hold. The three clauses are proved in Supplementary Sections I.3--I.5, respectively.

\begin{enumerate}
\item\label{item:fixed-level-right-tall}
\textup{\textbf{Tall-orientation triangular-array right margin.}}
Let \(\gamma\geq1\), and let \((M_N)_{N\geq1}\) be positive integers such that
\[
M_N\geq N
\]
for all sufficiently large \(N\), and \(M_N/N\to\gamma\). For each \(N\), let
\[
x_{\alpha b}^{(N)},
\qquad
1\leq\alpha\leq M_N,
\quad
1\leq b\leq N,
\]
be independent complex random variables satisfying
\[
\E x_{\alpha b}^{(N)}=0,
\qquad
\E\abs{x_{\alpha b}^{(N)}}^2=1,
\qquad
\E[(x_{\alpha b}^{(N)})^2]=0,
\]
and
\[
\sup_{N,\alpha,b}\E\abs{x_{\alpha b}^{(N)}}^p<\infty
\]
for every fixed \(p\geq1\). Let \(X_N\) be the \(M_N\times N\) matrix with entries
\[
N^{-1/2}x_{\alpha b}^{(N)},
\]
let
\[
L_N=X_N^*X_N,
\qquad
\lambda_N=\lambda_{\max}(L_N),
\]
and define
\[
E_{+,N}=\left(1+\sqrt{\frac{M_N}{N}}\right)^2,
\]
\[
\sigma_{+,N}
=
\frac{\sqrt{M_N}+\sqrt N}{N}
\left(\frac1{\sqrt{M_N}}+\frac1{\sqrt N}\right)^{1/3},
\qquad
r_N=(\log N)^{2/3},
\]
\[
A_2=\left(\frac14\right)^{2/3},
\qquad
\chi_N=\frac{\lambda_N-E_{+,N}}{\sigma_{+,N}}.
\]
For every fixed \(q>A_2\), there are \(\theta_q>1/3\) and \(N_q\) such that
\[
\Prob(\chi_N\geq qr_N)\leq N^{-\theta_q}
\]
for every \(N\geq N_q\).

For the remaining two clauses, let \(\gamma>0\), let
\((M_N)_{N\geq1}\) be positive integers with \(M_N/N\to\gamma\), and let
\((x_{ij})_{i,j\geq1}\) be independent complex random variables such
that, for every fixed \(p\geq1\),
\[
\E x_{ij}=0,
\qquad
\E\abs{x_{ij}}^2=1,
\qquad
\E x_{ij}^2=0,
\qquad
\sup_{i,j}\E\abs{x_{ij}}^p<\infty.
\]
Set
\[
X^{(N)}=(x_{ij})_{1\leq i\leq M_N,\,1\leq j\leq N},
\qquad
Q_N=N^{-1}X^{(N)}(X^{(N)})^*,
\]
\[
\mu_{+,N}=\frac{(\sqrt{M_N}+\sqrt N)^2}{N},
\qquad
\sigma_{+,N}
=
\frac{\sqrt{M_N}+\sqrt N}{N}
\left(\frac1{\sqrt{M_N}}+\frac1{\sqrt N}\right)^{1/3},
\]
and
\[
\chi_N^+
=
\frac{\lambda_{\max}(Q_N)-\mu_{+,N}}{\sigma_{+,N}}.
\]

\item\label{item:fixed-level-right-all}
\textup{\textbf{Both-orientation common-array right margin.}}
Put
\[
r_N=(\log N)^{2/3},
\qquad
A_2=\left(\frac14\right)^{2/3}.
\]
For every fixed \(q>A_2\), there are \(\theta_q>1/3\) and \(N_q\) such that
\[
\Prob(\chi_N^+>qr_N)\leq N^{-\theta_q}
\]
for every \(N\geq N_q\).

\item\label{item:fixed-level-left-all}
\textup{\textbf{Both-orientation common-array left margin.}}
Put
\[
s_N=(\log N)^{1/3},
\qquad
B_2=4^{1/3}.
\]
For every fixed \(q>B_2\), there are \(\theta_q>1/3\) and \(N_q\) such that
\[
\Prob(\chi_N^+\leq-qs_N)\leq N^{-\theta_q}
\]
for every \(N\geq N_q\).
\end{enumerate}
\end{thm}

Choose the grid \(N_k=\lceil k^\alpha\rceil\) with
\(\alpha\theta>1\) and \(\alpha<3\). The first condition makes the
fixed-size probability bounds summable; the second makes the grid gaps
\(o(N^{2/3})\). The increment and martingale estimates control normalized
excursions over these gaps with overwhelming probability, transferring
the grid bounds to the intervening matrix sizes.

\begin{prop}[Control between consecutive grid points]\label{prop:mesh-path-control}
Assume a northwest-nested sample-covariance array with independent centered complex variance-one entries, uniformly bounded fixed moments, and zero complex second moment. Let \((M_N)_{N\geq1}\) be nondecreasing, satisfy
\[
\frac{M_N}{N}\longrightarrow\gamma\in(0,\infty),
\qquad
K:=\sup_N(M_{N+1}-M_N)<\infty,
\]
and put
\[
L_N=\lambda_{\max}(X^{(N)}(X^{(N)})^*),
\]
\[
a_N=(\sqrt{M_N}+\sqrt N)^2,
\qquad
b_N=(\sqrt{M_N}+\sqrt N)
\left(M_N^{-1/2}+N^{-1/2}\right)^{1/3},
\]
\[
\chi_N=\frac{L_N-a_N}{b_N},
\qquad
r_N=(\log N)^{2/3},
\qquad
s_N=(\log N)^{1/3}.
\]
Fix
\[
1<\alpha<3,
\qquad
N_k=\lceil k^\alpha\rceil.
\]
For arbitrary \(q\geq0\), \(\delta>0\), and sufficiently large \(k\), let \(E_k^+(q,\delta)\) be the event that
\[
\chi_{N_k}\leq qr_{N_k},
\]
but
\[
\chi_n\geq(q+\delta)r_n
\]
for at least one integer \(n\in[N_{k-1},N_k]\). Let \(E_k^-(q,\delta)\) be the event that
\[
\chi_{N_{k-1}}\geq-qs_{N_{k-1}},
\]
but
\[
\chi_n\leq-(q+\delta)s_n
\]
for at least one integer \(n\in[N_{k-1},N_k]\). For every \(D>0\), there is \(C_D<\infty\) such that
\[
\Prob(E_k^+(q,\delta))+\Prob(E_k^-(q,\delta))
\leq C_DN_k^{-D}
\]
for every sufficiently large \(k\).

Consequently, almost surely,
\[
\limsup_{k\to\infty}\frac{\chi_{N_k}}{r_{N_k}}\leq q
\quad\Longrightarrow\quad
\limsup_{N\to\infty}\frac{\chi_N}{r_N}\leq q,
\]
and
\[
\liminf_{k\to\infty}\frac{\chi_{N_k}}{s_{N_k}}\geq-q
\quad\Longrightarrow\quad
\liminf_{N\to\infty}\frac{\chi_N}{s_N}\geq-q.
\]
\end{prop}

\supplementproof{supp.I.6}{Section I.6}

The fixed-size tail bounds and the control between grid points give the
dyadic exclusions in the next section.

\section{Tail bounds on dyadic intervals}\label{sec:recurrence}

\subsection*{From the mesh to a dyadic block}

We pass from the summable polynomial-grid bounds to the two dyadic
tail estimates used in the main proof.

\begin{lem}[Dyadic interpolation]
\label{lem:mesh-dyadic-bridge}
Let \((X_N)_{N\geq3}\) be real random variables, let \(t_N>0\), and put
\(I_k=[2^k,2^{k+1})\cap\N\).  Fix \(c,\epsilon>0\), set
\[
 q=c+\frac\epsilon2,
 \qquad
 \delta=\frac\epsilon2,
\]
and suppose that, for some \(\theta>1/3\),
\[
 \max\left\{1,\frac1\theta\right\}<\alpha<3.
\]
Write \(n_j=\lceil j^\alpha\rceil\) and
\(J_j=[n_{j-1},n_j]\cap\N\).

For the upper sign, assume
\[
 \Prob(X_N>q t_N)\leq N^{-\theta}
\]
for all sufficiently large \(N\), and, for every \(D>0\),
\[
 \Prob\!\left(
 X_{n_j}\leq q t_{n_j},\ 
 \max_{n\in J_j}\frac{X_n}{t_n}\geq q+\delta
 \right)\leq C_Dn_j^{-D}.
\]
Then
\[
 \sum_k\Prob\!\left(
 \sup_{N\in I_k}\frac{X_N}{t_N}>c+\epsilon
 \right)<\infty.
\]
For the lower sign, the analogous assumptions
\[
 \Prob(X_N<-q t_N)\leq N^{-\theta}
\]
and
\[
 \Prob\!\left(
 X_{n_{j-1}}\geq-q t_{n_{j-1}},\ 
 \min_{n\in J_j}\frac{X_n}{t_n}\leq-(q+\delta)
 \right)\leq C_Dn_j^{-D}
\]
imply
\[
 \sum_k\Prob\!\left(
 \inf_{N\in I_k}\frac{X_N}{t_N}<-c-\epsilon
 \right)<\infty.
\]
\end{lem}

\begin{proof}
We prove the upper assertion; the lower one is identical with inequalities
reversed and the left endpoint of \(J_j\) as anchor.  Since
\(\alpha\theta>1\), the fixed-level bound is summable on \((n_j)\).
Choose \(D>1/\alpha\); then the excursion bound is also summable.
Moreover, \(n_j/n_{j-1}\to1\), so every sufficiently late \(J_j\)
intersects at most two dyadic blocks.

If the upper dyadic event occurs, choose a witness \(N\) and a mesh interval
\(J_j\) containing it.  Either \(X_{n_j}>qt_{n_j}\), or the corresponding
excursion event occurs.  After the union bound, reversing the nonnegative
double sum costs a factor at most two.  The two summable mesh series therefore
give the asserted dyadic series.  The same argument anchored at \(n_{j-1}\)
proves the lower assertion.
\end{proof}

\begin{prop}[Dyadic tail bounds]
\label{prop:right-dyadic-exclusion}\label{prop:left-dyadic-exclusion}
Under the assumptions of Theorem~\ref{thm:main}, let
$I_k=[2^k,2^{k+1})\cap\N$. For every $\epsilon>0$,
\begin{align}
 \sum_{k\ge2}\Prob\!\left(\sup_{N\in I_k}
       \frac{\chi_N^+}{r_N}>A_2+\epsilon\right)&<\infty,
 \label{eq:main-high-exclusion}\\
 \sum_{k\ge2}\Prob\!\left(\inf_{N\in I_k}
       \frac{\chi_N^+}{s_N}<-B_2-\epsilon\right)&<\infty.
 \label{eq:main-low-exclusion}
\end{align}
\end{prop}
\begin{proof}
For the upper assertion, put $q=A_2+\epsilon/2$ and
$\delta=\epsilon/2$. Theorem~\ref{thm:fixed-level-margins} gives a
fixed-level exponent $\theta>1/3$. Choose
$\max(1,1/\theta)<\alpha<3$; this is possible precisely because
$\theta>1/3$. Proposition~\ref{prop:mesh-path-control} gives the
between-mesh estimate, anchored at the right endpoint. Apply the upper
clause of Lemma~\ref{lem:mesh-dyadic-bridge} with $t_N=r_N$.
For the lower assertion use $q=B_2+\epsilon/2$, $t_N=s_N$, and the
left endpoint as anchor, with the exponent furnished for that tail.
The lower clause of the same lemma applies. The two choices of mesh
exponent need not coincide.
\end{proof}

\section{Proof of the main theorem}\label{sec:main}

\subsection*{Endpoint and cluster-set assembly}

We combine the full-matrix block occurrences from
Corollary~\ref{cor:full-grid-occurrence} with the dyadic exclusions,
then apply the deterministic cluster-set criterion.

\begin{proof}[Proof of Theorem~\ref{thm:main}]
For every integer $k\geq 2$, put
\[
 I_k=\set{N\in\N\mid 2^k\leq N<2^{k+1}}.
\]

\noindent\textbf{Step 1. Supercritical exclusions.}
Proposition~\ref{prop:right-dyadic-exclusion} gives
\eqref{eq:main-high-exclusion}--\eqref{eq:main-low-exclusion} under the
present model assumptions.

\noindent\textbf{Step 2. Subcritical occurrences and endpoints.}
Fix rational $0<a<A_2$ and $0<b<B_2$. Apply
Corollary~\ref{cor:full-grid-occurrence} with $N=2^l$, and denote its
right- and left-tail target unions by $E_l^+$ and $E_l^-$.
For either sign, $\Prob(E_l^\pm)\to1$. For every fixed $L$,
\[
 \Prob\left(\bigcup_{l\ge L}E_l^\pm\right)
 \ge \sup_{l\ge L}\Prob(E_l^\pm)=1.
\]
Taking the intersection over $L$ shows that these events occur infinitely
often almost surely. Since their indices lie in $[2^l,2^{l+1})$, this
gives arbitrarily large $n$ with $\chi_n^+/r_n\ge a$ and arbitrarily
large $n$ with $\chi_n^+/s_n<-b$.

Intersect these probability-one statements over rational $a,b$.
Borel--Cantelli applied to the summable supercritical probabilities
\eqref{eq:main-high-exclusion}--\eqref{eq:main-low-exclusion}, followed
by a countable choice of the threshold buffers, gives on the same event
\begin{equation}\label{eq:main-endpoints}
 \limsup_{N\to\infty}\frac{\chi_N^+}{r_N}=A_2,
 \qquad
 \liminf_{N\to\infty}\frac{\chi_N^+}{s_N}=-B_2.
\end{equation}

\noindent\textbf{Step 3. Cluster sets.}
Define
\[
 L_N=\lambda_{\max}\bigl(X^{(N)}(X^{(N)})^*\bigr),
 \qquad
 a_N=(\sqrt{M_N}+\sqrt N)^2,
\]
\[
 b_N=(\sqrt{M_N}+\sqrt N)
      \left(M_N^{-1/2}+N^{-1/2}\right)^{1/3},
 \qquad
 \xi_N=\frac{L_N-a_N}{b_N}.
\]
The definitions give the exact identities
\begin{equation}\label{eq:main-normalization-identities}L_N=N\lambda_1^{(N)},
 \qquad
 a_N=N\mu_{+,N},
 \qquad
 b_N=N\sigma_{+,N},
 \qquad
 \xi_N=\chi_N^+.
\end{equation}

The matrices $X^{(N)}$ are northwest nested because each is the upper-left $M_N$ by $N$ submatrix of the same infinite array and $(M_N)$ is nondecreasing. The dimension path also satisfies
\[
 \frac{M_N}{N}\longrightarrow\gamma>0,
 \qquad
 0\leq M_{N+1}-M_N\leq K.
\]
Extend $r_N$ and $s_N$ to $N=2$ by the same formulas. Applying Proposition~\ref{prop:tail-cluster-closure} to
\eqref{eq:main-endpoints} and
\eqref{eq:main-normalization-identities}, with $A=A_2$ and $B=B_2$, yields
\begin{equation}\label{eq:main-high-cluster}
 \bigcap_{m=2}^{\infty}
 \cl_{\R}\set{\frac{\chi_N^+}{r_N}\mid N\geq m}
 =[0,A_2]
\end{equation}
and
\begin{equation}\label{eq:main-low-cluster}
 \bigcap_{m=2}^{\infty}
 \cl_{\R}\set{\frac{\chi_N^+}{s_N}\mid N\geq m}
 =[-B_2,\infty).
\end{equation}
Since the tail closures decrease with $m$, beginning the intersections at $m=3$ gives the same sets. This proves all four assertions.
\end{proof}

\appendix
\section{Auxiliary Gaussian estimates}
\label{app:companion-estimates}

We give an Airy-kernel resolvent bound and a restricted Laguerre-function
estimate used in the Gaussian input \cite{YangWishart}. The latter
combines the approximations on the positive and negative half-axes.

\subsection{Airy kernel resolvent}

Let $K_s$ be the Airy-kernel operator on $L^2((s,\infty))$ and let
$\lambda_0(s)$ be its largest eigenvalue. The fixed-index expansion in
\cite[Corollary~1.7, arXiv version~2]{BothnerV2}, with index zero, gives
\[
 1-\lambda_0(-R)
 =\sqrt\pi\,2^{9/4}R^{3/4}
   \exp\!\left(-\frac{2\sqrt2}{3}R^{3/2}\right)(1+o(1)).
\]
The integral kernel $\int_0^\infty\Ai(x+u)\Ai(y+u)\,\dd u$
agrees with the Christoffel--Darboux form in that reference: integrate
the derivative of
$\Ai'(x+u)\Ai(y+u)-\Ai(x+u)\Ai'(y+u)$ and use the Airy
differential equation and the decay at positive infinity.
Since $K_s$ is a positive self-adjoint contraction, the spectral theorem
then gives, for sufficiently large $R$,
\[
 \|(I-K_{-R})^{-1}\|_{\op}
 =(1-\lambda_0(-R))^{-1}\le \exp(CR^{3/2}).
\]
For $R\le D(\log N)^{1/3}$ this is $N^{o(1)}$. The fixed-index estimate above supplies the input for Proposition~3.5 of
\cite{YangWishart}.

\subsection{Restricted Laguerre functions}

Use the normalized Laguerre functions $\phi_{n,m}$ and the parameters
$\mu_{n,m},\sigma_{n,m},\delta_{n,m}$ of
\cite[Propositions~3.2--3.3]{YangWishart}. Put
\[
 G_{n,m}(s)=(-1)^n\sqrt{\frac{\sigma_{n,m}}{\delta_{n,m}}}
       \phi_{n,m}(\mu_{n,m}+\sigma_{n,m}s).
\]
When $cN\le n\le m\le CN$, these inputs give
$\delta_{n,m}\asymp N^{-1/3}$ and
\[
 |G_{n,m}(s)-\Ai(s)|\le C N^{-2/3}e^{-s/4},\qquad s\ge0,
\]
whereas for $-R\le s\le0$, $R\le D(\log N)^{1/3}$, they give
\[
 |G_{n,m}(s)-\Ai(s)|\le
 E_N:=C N^{-2/3}(1+R)^3\exp(C(1+R)^2).
\]
Suppose a restriction interval $B$ has its lower endpoint no smaller
than $\mu_{n,m}-\sigma_{n,m}R$. The exact change of variables yields
\[
 \|\ind_B\phi_{n,m}\|_2^2
 \le\delta_{n,m}\int_{-R}^{\infty}|G_{n,m}(s)|^2\,\dd s.
\]
Split the integral at zero. The two displayed approximation bounds give
\begin{align*}
 \int_{-R}^{\infty}|G_{n,m}(s)|^2\,\dd s
 &\le 2\int_{-R}^{\infty}|\Ai(s)|^2\,\dd s
       +2R E_N^2+C N^{-4/3}\int_0^\infty e^{-s/2}\,\dd s\\
 &\le C(1+R)^{1/2}+2R E_N^2+C N^{-4/3}
 =N^{o(1)}.
\end{align*}
Here the negative-axis Airy envelope gives the square-integral bound,
and $(1+R)^2=o(\log N)$ controls $E_N$.
It follows that
\[
 \|\ind_B\phi_{n,m}\|_2^2\le N^{-1/3+o(1)},\qquad
 \|\ind_B\phi_{n,m}\|_2\le N^{-1/6+o(1)}.
\]
For the central modes in the proof of \cite[Equation~(154)]{YangWishart},
write $n=p_i-\ell$, $m=q_i-\ell$, with
$|\ell|\le L_N=\lfloor N^{1/3-\epsilon/2}\rfloor$.
The dimensions remain comparable with $N$. The derivatives of
$\mu_{n,m}$ in these dimensions are bounded and
$\sigma_{n,m}\asymp N^{1/3}$, so the center shift in mode coordinates is
$O(L_NN^{-1/3})=O(N^{-\epsilon/2})$; the scale ratio changes by
$O(L_N/N)$. Thus the same lower threshold is covered uniformly by
$R=O((\log N)^{1/3})$.
The restricted-mode bound is therefore uniform over the stated central modes.

\section{Recurrence on a common filtration}\label{app:recurrence-refinements}

This appendix treats occurrence estimates with prescribed histories.
The conditional bounds of Sections~\ref{sec:high-occurrence} and
\ref{sec:low-occurrence} can be arranged on a common filtration,
including the threshold shifts below. We record that stronger conditional
organization together with its abstract probability argument.

\begin{lem}[Construction of a common filtration]\label{lem:lacunary-scheduling}
Let $(\mathcal R_j)_{j\ge0}$ be increasing sigma-algebras and let
$\mathcal Q$ be countable. Suppose integers and labels satisfy
\[
 0\le u_1<v_1\le u_2<v_2\le\cdots,
 \qquad q_t\in\mathcal Q,
\]
with every label occurring infinitely often. Let
$E_t\in\mathcal R_{v_t}$, $G_t\in\mathcal R_{u_t}$, and deterministic
$p_t\in[0,1]$ satisfy
\begin{equation}\label{eq:scheduling-conditional}
 \Prob(E_t\mid\mathcal R_{u_t})\ge p_t\quad\hbox{on }G_t.
\end{equation}
For every $q$ assume
\begin{equation}\label{eq:scheduling-series}
 \sum_{t:q_t=q}\Prob(G_t^c)<\infty,\qquad
 \sum_{t:q_t=q}p_t=\infty.
\end{equation}
Define $\mathcal F_k=\mathcal R_0$ for $k<u_1$, and thereafter
\[
 \mathcal F_k=
 \begin{cases}
 \mathcal R_{u_t},&u_t\le k<v_t,\\
 \mathcal R_{v_t},&v_t\le k<u_{t+1}.
 \end{cases}
\]
At $(k,q)=(v_t,q_t)$ put
$(E_k(q),G_k(q),p_k(q))=(E_t,G_t,p_t)$; use
$(\varnothing,\Omega,0)$ at all other pairs. Then the filtration is
increasing, $\mathcal F_{v_t-1}=\mathcal R_{u_t}$,
$E_k(q)\in\mathcal F_k$, $G_k(q)\in\mathcal F_{k-1}$, and the
conditional lower bounds and both series in
\eqref{eq:scheduling-series} are preserved.

For a nested array with $M_N$ nondecreasing and
$C_M=\sup_N M_N/N<\infty$, take
$\mathcal R_j=\sigma\{x_{ab}:a\le M_{2^{j+1}},\ b\le2^{j+1}\}$.
When $u_t=v_{t-1}$, the target indices can be chosen recursively so that,
for each prescribed $d>0$ and $N_t=2^{v_t}$,
\begin{equation}\label{eq:scheduling-polylog}
 2^{u_t+1}\le(\log N_t)^d.
\end{equation}
The revealed rectangle then has at most $(\log N_t)^d$ columns,
$C_M(\log N_t)^d$ rows, and $C_M(\log N_t)^{2d}$ entries.
\end{lem}
\begin{proof}
The inequalities between the paired indices make the displayed filtration
increasing. Since $u_t<v_t$ are integers, its value at $v_t-1$ is
exactly $\mathcal R_{u_t}$. Every nonempty event therefore retains its
original conditional probability; all other slots have lower bound zero.
The two series are reindexed. Finally, for fixed $u_t$, the right
side of \eqref{eq:scheduling-polylog} tends to infinity with $v_t$.
Choose $v_t\ge u_t+2$ sufficiently large and use the bound on $M_N$.

\end{proof}

\begin{prop}[Borel--Cantelli argument]
\label{prop:dyadic-borel-cantelli}
Let \((\Omega,\mathcal F,\Prob)\) be a probability space, and let
\((X_N)_{N\geq3}\) be real-valued random variables. Let \(\log\) denote the natural logarithm, and put
\[
 r_N=(\log N)^{2/3},\qquad
 s_N=(\log N)^{1/3}.
\]
Let \(A,B>0\). For \(k\geq2\), define
\[
 I_k=\set{N\in\N\mid2^k\leq N<2^{k+1}},
\]
and
\[
 K_+
 =
 \bigcap_{m\geq3}
 \cl_{\R}\set{\frac{X_N}{r_N}\mid N\geq m},
 \qquad
 K_-
 =
 \bigcap_{m\geq3}
 \cl_{\R}\set{\frac{X_N}{s_N}\mid N\geq m}.
\]
Assume the following hypotheses.

\begin{enumerate}
\item[\({\rm(H1)}\)] For every rational \(\epsilon>0\),
\[
 \sum_{k=2}^{\infty}
 \Prob\left(
 \sup_{N\in I_k}\frac{X_N}{r_N}>A+\epsilon
 \right)<\infty
\]
and
\[
 \sum_{k=2}^{\infty}
 \Prob\left(
 \inf_{N\in I_k}\frac{X_N}{s_N}<-B-\epsilon
 \right)<\infty.
\]

\item[\({\rm(H2+)}\)] There is an increasing sequence
\((\mathcal F_k)_{k\geq1}\) of sub-sigma-algebras such that, for every rational \(a\) with \(0<a<A\), there are, for all sufficiently large \(k\),
\[
 U_k(a)\in\mathcal F_k,\qquad
 G_k^+(a)\in\mathcal F_{k-1},\qquad
 p_k^+(a)\in[0,1],
\]
with the following properties:
on \(U_k(a)\), there exists \(N\in I_k\) such that
\[
 \frac{X_N}{r_N}\geq a;
\]
moreover,
\[
 \sum_k\Prob((G_k^+(a))^c)<\infty,
 \qquad
 \sum_kp_k^+(a)=\infty,
\]
and
\[
 \Prob(U_k(a)\mid\mathcal F_{k-1})
 \geq p_k^+(a)
 \quad\text{on }G_k^+(a)\text{ almost surely}.
\]

\item[\({\rm(H2-)}\)] For every rational \(b\) with \(0<b<B\), there are, for all sufficiently large \(k\),
\[
 L_k(b)\in\mathcal F_k,\qquad
 G_k^-(b)\in\mathcal F_{k-1},\qquad
 p_k^-(b)\in[0,1],
\]
such that, on \(L_k(b)\), there exists \(N\in I_k\) satisfying
\[
 \frac{X_N}{s_N}\leq-b;
\]
moreover,
\[
 \sum_k\Prob((G_k^-(b))^c)<\infty,
 \qquad
 \sum_kp_k^-(b)=\infty,
\]
and
\[
 \Prob(L_k(b)\mid\mathcal F_{k-1})
 \geq p_k^-(b)
 \quad\text{on }G_k^-(b)\text{ almost surely}.
\]
\end{enumerate}
Then, with probability one,
\[
 \limsup_{N\to\infty}\frac{X_N}{r_N}=A,
 \qquad
 \liminf_{N\to\infty}\frac{X_N}{s_N}=-B.
\]
On the same probability-one event,
\[
 K_+\subseteq[0,A],\qquad 0,A\in K_+,
\]
and
\[
 K_-\subseteq[-B,\infty),\qquad -B\in K_-.
\]
Furthermore, for every \(R>0\) and \(m\geq3\), there exists \(N\geq m\) such that
\[
 \frac{X_N}{s_N}>R.
\]
\end{prop}
\begin{proof}
For each rational subcritical threshold, apply
Lemma~\ref{lem:conditional-recurrence} to its stated event, good history,
and lower-bound sequence. The two rational threshold sets are countable,
so all these occurrences hold on one probability-one event. Ordinary
Borel--Cantelli applied to (H1) excludes every rational supercritical
threshold. Letting the thresholds approach $A$ and $B$ proves the two
endpoint equalities.

For every $\epsilon>0$, eventually
$X_N\ge-(B+\epsilon)s_N$ and $X_N\le(A+\epsilon)r_N$.
Since $s_N/r_N\to0$, the first inequality shows
$K_+\subseteq[0,A]$. A subsequence realizing the lower endpoint yields
$X_N/r_N\to0$, and a subsequence realizing the upper endpoint yields
$X_N/r_N\to A$. The lower bound gives
$K_-\subseteq[-B,\infty)$ and $-B\in K_-$. Along the upper-endpoint
subsequence, $X_N/s_N=(X_N/r_N)(r_N/s_N)\to\infty$, which proves the
last assertion.
\end{proof}

\begin{thm}[Recurrence on a common filtration]
\label{thm:common-filtration-recurrence}
Work under Theorem~\ref{thm:main}, and let
$\mathcal F=\sigma(x_{ij}:i,j\ge1)$ and
$I_k=[2^k,2^{k+1})\cap\N$. Let
\[
 \mathcal Q=
 \{(+,a,h):a,h\in\mathbb Q,\ 0<a<A_2\}
 \cup\{(-,b,h):b,h\in\mathbb Q,\ 0<b<B_2\}.
\]
There is one increasing filtration $(\mathcal F_k^*)_{k\ge1}$ in
$\mathcal F$ and, for every $q\in\mathcal Q$, events
$E_k^*(q)\in\mathcal F_k^*$, $G_k^*(q)\in\mathcal F_{k-1}^*$,
and deterministic $p_k^*(q)\in[0,1]$ such that
\begin{equation}\label{eq:global-conditional}
 \Prob(E_k^*(q)\mid\mathcal F_{k-1}^*)\ge p_k^*(q)
 \quad\hbox{on }G_k^*(q),
\end{equation}
\begin{equation}\label{eq:global-sums}
 \sum_k\Prob(G_k^*(q)^c)<\infty,\qquad\sum_k p_k^*(q)=\infty.
\end{equation}
Every $E_k^*((+,a,h))$ has a witness $N\in I_k$ with
\begin{equation}\label{eq:global-high-witness}
 \chi_N^+\ge a r_N+h,
\end{equation}
and every $E_k^*((-,b,h))$ has a witness $N\in I_k$ with
\begin{equation}\label{eq:global-low-witness}
 \chi_N^+\le-b s_N+h.
\end{equation}
The unshifted events $U_k^*(a)$ and $L_k^*(b)$ are the cases $h=0$.
\end{thm}
\begin{proof}
Let
$\mathcal R_j=\sigma\{x_{in}:i\le M_{2^{j+1}},\ n\le2^{j+1}\}$.
For a high label $q=(+,a,h)$ choose $a<a_0<a_1<a_2<A_2$ and
positive parameters satisfying
\begin{equation}\label{eq:high-parameter-choice}
 \frac43a_2^{3/2}+\epsilon+2\eta<\frac13.
\end{equation}
For a low label $q=(-,b,h)$ choose $b<b_0<B_2$ and
\begin{equation}\label{eq:low-parameter-choice}
 b_0^3/12+\eta+\epsilon<\frac13.
\end{equation}
The respective occurrence theorem, with $d=1$, supplies a complete
label-specific construction on exponents $e_q(m)$, with
$e_q(m)-e_q(m-1)\ge2$, target scale $N_q(m)=2^{e_q(m)}$, and
past cutoff $P_q(m)=2^{e_q(m-1)+1}$. Its past is exactly
\begin{equation}\label{eq:literal-past}
 \sigma\{x_{in}:i\le M_{P_q(m)},\ n\le P_q(m)\}
 =\mathcal R_{e_q(m-1)}.
\end{equation}
Write $C_q(m)$ and $D_q(m)$ for its event and good history. For each label,
\begin{equation}\label{eq:high-good-sum}\sum_m\Prob(D_q(m)^c)<\infty,
\end{equation}
and, after discarding finitely many $m$,
\begin{equation}\label{eq:high-conditional}\Prob(C_q(m)\mid\mathcal R_{e_q(m-1)})\ge1/2
 \quad\hbox{on }D_q(m).
\end{equation}
The event is $\mathcal R_{e_q(m)}$-measurable and has a witness in
$I_{e_q(m)}$, with
\begin{align}
 \chi_N^+&\ge a_0r_N &&\text{for a high label},\label{eq:high-witness}\\
 \chi_N^+/s_N&<-b_0 &&\text{for a low label}.\label{eq:low-witness}
\end{align}
The normalization follows from \eqref{eq:main-normalization-identities}.
Further finite discards ensure, throughout every remaining target block,
\begin{align}
 a_0r_N&\ge a r_N+h &&\text{for a high label},\label{eq:high-shift}\\
 -b_0s_N&\le-b s_N+h &&\text{for a low label}.\label{eq:low-shift}
\end{align}
This follows from $r_N,s_N\to\infty$ and the strict amplitude gaps.

Choose a deterministic sequence $(q_t)$ visiting every label infinitely
often. Starting with $v_0=0$, choose a previously unused index $m_t$ for
$q_t$, beyond its discarded prefix, such that
\[
 u_t=e_{q_t}(m_t-1)\ge v_{t-1}+2,\qquad
 v_t=e_{q_t}(m_t)\ge u_t+2.
\]
Such a choice exists since each past exponent tends to infinity.
Apply Lemma~\ref{lem:lacunary-scheduling} to
$(C_{q_t}(m_t),D_{q_t}(m_t),1/2)$ at these pairs.
Its filtration obeys $\mathcal F_{v_t-1}^*=\mathcal R_{u_t}$, so every
conditional lower bound remains conditioned on its original past. Each
label's bad-history sum is a subseries of its original convergent sum,
while its lower-bound series contains infinitely many terms $1/2$.
The target exponent and witness block are unchanged. The shifted
inequalities above establish \eqref{eq:global-high-witness} and
\eqref{eq:global-low-witness}; $h=0$ gives the named unshifted events.
\end{proof}

\makeatletter
\enddoc@text
\let\enddoc@text\relax
\makeatother
\input{supplement_pages.tex}

\end{document}

%% file: supplement_navigation.tex
\providecommand{\NCLdestination}[3]{%
  \put(#1,-#2){\pdfdest name {#3} xyz}%
}
\providecommand{\NCLlink}[5]{%
  \put(#1,-#2){\pdfannot width #3bp height #4bp depth 0bp%
    {/Subtype /Link /Border [0 0 0] /A << /S /GoTo /D (#5) >>}}%
}
\providecommand{\NCLurl}[5]{%
  \put(#1,-#2){\pdfannot width #3bp height #4bp depth 0bp%
    {/Subtype /Link /Border [0 0 0] /A << /S /URI /URI (#5) >>}}%
}

\expandafter\def\csname NCLsupppage1\endcsname{%
\NCLdestination{0.000000}{0.000000}{arxiv.supplement}%
\NCLlink{87.843002}{256.390991}{209.387009}{16.490967}{supp.A}%
\NCLlink{87.843002}{270.888000}{332.668991}{16.489990}{supp.B}%
\NCLlink{87.843002}{285.385986}{176.961990}{16.489990}{supp.B.1}%
\NCLlink{87.843002}{299.884003}{188.962997}{16.489990}{supp.B.2}%
\NCLlink{87.843002}{314.381989}{266.992996}{16.490997}{supp.B.3}%
\NCLlink{87.843002}{328.878998}{190.175003}{16.489990}{supp.B.4}%
\NCLlink{87.843002}{343.377014}{232.206009}{16.490021}{supp.C}%
\NCLlink{87.843002}{357.875000}{204.038010}{16.490997}{supp.C.1}%
\NCLlink{87.843002}{372.372009}{159.780991}{16.490021}{supp.C.2}%
\NCLlink{87.843002}{386.869995}{389.259995}{16.489990}{supp.D}%
\NCLlink{87.843002}{401.368011}{245.370987}{16.490021}{supp.E}%
\NCLlink{87.843002}{415.865997}{150.235001}{16.490997}{supp.E.1}%
\NCLlink{87.843002}{430.363007}{140.264999}{16.490021}{supp.E.2}%
\NCLlink{87.843002}{444.860992}{142.961990}{16.489990}{supp.E.3}%
\NCLlink{87.843002}{459.359009}{132.992996}{16.490997}{supp.E.4}%
\NCLlink{87.843002}{473.855988}{172.090988}{16.489990}{supp.E.4}%
\NCLlink{87.843002}{488.354004}{279.578995}{16.489990}{supp.G}%
\NCLlink{87.843002}{502.851990}{179.887009}{16.489990}{supp.G.1}%
\NCLlink{87.843002}{517.348999}{195.401993}{16.489990}{supp.G.2}%
\NCLlink{87.843002}{531.846985}{157.069000}{16.489990}{supp.G.3}%
\NCLlink{87.843002}{546.344971}{180.917007}{16.489990}{supp.G.4}%
\NCLlink{87.843002}{560.843018}{135.128998}{16.491028}{supp.G.5}%
\NCLlink{87.843002}{575.339966}{154.705002}{16.489990}{supp.G.6}%
\NCLlink{87.843002}{589.838013}{180.643997}{16.489990}{supp.G.7}%
\NCLlink{87.843002}{604.335999}{160.007996}{16.489990}{supp.G.8}%
\NCLlink{87.843002}{618.833008}{171.826004}{16.489990}{supp.G.9}%
\NCLlink{87.843002}{633.330994}{139.946991}{16.489990}{supp.G.10}%
\NCLlink{87.843002}{647.828979}{278.155991}{16.489990}{supp.H}%
\NCLlink{87.843002}{662.327026}{129.175003}{16.491028}{supp.H.1}%
\NCLlink{87.843002}{676.823975}{136.446991}{16.489990}{supp.H.2}%
\NCLlink{87.843002}{691.322021}{208.143997}{16.489990}{supp.H.3}%
\NCLlink{87.843002}{705.820007}{130.416992}{16.489990}{supp.H.4}%
\NCLlink{87.843002}{720.317017}{177.598984}{16.489990}{supp.H.5}%
}
\expandafter\def\csname NCLsupppage2\endcsname{%
\NCLlink{87.843002}{110.502991}{173.052994}{16.489990}{supp.H.6}%
\NCLlink{87.843002}{125.000977}{124.311005}{16.489990}{supp.H.7}%
\NCLlink{87.843002}{139.499023}{284.835007}{16.490051}{supp.I}%
\NCLlink{87.843002}{153.997009}{182.749985}{16.491028}{supp.I.1}%
\NCLlink{87.843002}{168.494019}{202.628983}{16.489990}{supp.I.2}%
\NCLlink{87.843002}{182.992004}{192.568985}{16.489990}{supp.I.3}%
\NCLlink{87.843002}{197.489990}{112.750000}{16.490967}{supp.I.4}%
\NCLlink{87.843002}{211.987000}{105.477997}{16.489990}{supp.I.5}%
\NCLlink{87.843002}{226.484985}{224.296005}{16.489990}{supp.I.6}%
\NCLlink{87.843002}{240.982971}{185.486984}{16.489990}{supp.J}%
\NCLlink{87.843002}{255.479980}{158.190002}{16.489990}{supp.J.1}%
\NCLlink{87.843002}{269.978027}{230.447006}{16.490051}{supp.J.2}%
\NCLlink{87.843002}{284.476013}{152.175003}{16.489990}{supp.J.3}%
\NCLlink{87.843002}{298.973999}{237.584000}{16.490997}{supp.J.4}%
\NCLlink{87.843002}{313.471008}{150.886993}{16.490021}{supp.J.5}%
\NCLlink{87.843002}{327.968994}{223.143997}{16.489990}{supp.J.6}%
\NCLlink{87.843002}{342.467010}{181.022995}{16.490021}{supp.K}%
\NCLlink{87.843002}{356.963989}{117.296005}{16.489990}{supp.K.1}%
\NCLlink{87.843002}{371.462006}{110.022995}{16.490021}{supp.K.2}%
\NCLlink{87.843002}{385.959991}{60.895004}{16.489990}{supp.references}%
}
\expandafter\def\csname NCLsupppage3\endcsname{%
\NCLdestination{77.469002}{84.825806}{supp.A}%
\NCLdestination{77.469002}{89.212830}{supp.A.1}%
\NCLdestination{77.469002}{89.212830}{supp.A.2}%
\NCLdestination{77.469002}{89.212830}{supp.A.3}%
\NCLdestination{77.760002}{135.904053}{supp.B}%
\NCLdestination{77.760002}{151.272949}{supp.B.1}%
\NCLdestination{77.469002}{655.265381}{supp.B.2}%
\NCLdestination{77.469002}{713.870483}{supp.B.3}%
\NCLdestination{77.469002}{346.043701}{supp.proof.B_delta}%
\NCLlink{248.591003}{118.237854}{15.932007}{15.334045}{thm.2.1}%
\NCLlink{157.496994}{186.270935}{16.488998}{11.690002}{thm.2.2}%
\NCLlink{128.783005}{686.551453}{15.931992}{11.688965}{thm.3.4}%
\NCLdestination{77.469000}{84.825809}{supp.outline.001}%
\NCLdestination{77.760000}{135.904060}{supp.outline.002}%
\NCLdestination{77.760000}{151.272950}{supp.outline.003}%
\NCLdestination{77.469000}{655.265400}{supp.outline.004}%
\NCLdestination{77.469000}{713.870500}{supp.outline.005}%
}
\expandafter\def\csname NCLsupppage4\endcsname{%
\NCLdestination{77.469002}{117.116272}{supp.B.4}%
\NCLdestination{77.760002}{334.647583}{supp.C}%
\NCLdestination{77.760002}{350.015472}{supp.C.1}%
\NCLdestination{77.469002}{119.382263}{supp.lemma.B.1}%
\NCLlink{332.989014}{102.749268}{15.931976}{11.690002}{thm.3.4}%
\NCLlink{164.175003}{385.013458}{16.489990}{11.689026}{thm.4.7}%
\NCLlink{147.934998}{477.729462}{12.884293}{12.901001}{equation.66}%
\NCLlink{356.710815}{477.729462}{12.901978}{12.901001}{equation.64}%
\NCLlink{182.529999}{555.150024}{12.902008}{12.902039}{equation.68}%
\NCLdestination{77.469000}{117.116270}{supp.outline.006}%
\NCLdestination{77.760000}{334.647590}{supp.outline.007}%
\NCLdestination{77.760000}{350.015480}{supp.outline.008}%
}
\expandafter\def\csname NCLsupppage5\endcsname{%
\NCLlink{104.420998}{101.272217}{15.931999}{11.690002}{thm.4.5}%
\NCLlink{81.005997}{114.829224}{12.902000}{12.901001}{equation.62}%
\NCLlink{145.945999}{166.248230}{12.901001}{12.902039}{equation.69}%
\NCLlink{431.808014}{429.255188}{12.901978}{12.901001}{equation.71}%
\NCLlink{145.945999}{541.379211}{12.901001}{12.902039}{equation.72}%
\NCLlink{165.457001}{554.330200}{15.932007}{12.901978}{thm.4.5}%
\NCLlink{471.644012}{554.330200}{12.901978}{12.901978}{equation.63}%
\NCLlink{390.735992}{673.920227}{12.902008}{12.902039}{equation.49}%
}
\expandafter\def\csname NCLsupppage6\endcsname{%
\NCLlink{520.025024}{100.654968}{15.931946}{14.815002}{thm.4.6}%
\NCLlink{102.676003}{113.148010}{12.902000}{12.902039}{equation.56}%
\NCLlink{146.248993}{127.070984}{12.901001}{12.901001}{equation.73}%
\NCLlink{171.097000}{127.070984}{12.901993}{12.901001}{equation.74}%
\NCLlink{256.703003}{398.935974}{12.902008}{12.901978}{equation.80}%
\NCLlink{129.089996}{474.707001}{12.902008}{12.902008}{equation.81}%
\NCLlink{165.612000}{667.148987}{12.902008}{12.901001}{equation.80}%
\NCLlink{209.854996}{667.148987}{12.902008}{12.901001}{equation.81}%
}
\expandafter\def\csname NCLsupppage7\endcsname{%
\NCLdestination{77.469002}{187.298218}{supp.C.2}%
\NCLlink{442.044006}{172.268188}{12.902008}{12.901978}{equation.83}%
\NCLlink{466.893005}{172.268188}{12.901001}{12.901978}{equation.84}%
\NCLlink{394.447998}{206.316223}{15.932007}{11.688965}{thm.4.8}%
\NCLlink{343.083008}{372.441223}{15.932007}{11.690002}{thm.3.3}%
\NCLlink{103.078003}{571.049255}{21.385994}{12.901001}{equation.91}%
\NCLdestination{77.469000}{187.298220}{supp.outline.009}%
}
\expandafter\def\csname NCLsupppage8\endcsname{%
\NCLdestination{77.469002}{602.837891}{supp.lemma.C.1}%
\NCLlink{116.663002}{142.884705}{21.387001}{18.526001}{equation.89}%
\NCLlink{255.136002}{393.471893}{21.387009}{12.902008}{equation.91}%
\NCLlink{300.226990}{561.942932}{26.841003}{12.902039}{equation.103}%
\NCLlink{330.529999}{561.942932}{26.841003}{12.902039}{equation.104}%
\NCLlink{349.265015}{713.156921}{27.640991}{12.902039}{equation.100}%
\NCLlink{380.489014}{713.156921}{27.640991}{12.902039}{equation.101}%
}
\expandafter\def\csname NCLsupppage9\endcsname{%
}
\expandafter\def\csname NCLsupppage10\endcsname{%
\NCLlink{141.731995}{141.075928}{21.387009}{12.901001}{equation.96}%
\NCLlink{166.580994}{141.075928}{21.386002}{12.901001}{equation.97}%
\NCLlink{469.187012}{368.128876}{26.841003}{12.901978}{equation.103}%
\NCLlink{466.454987}{394.031891}{26.841003}{12.901978}{equation.104}%
}
\expandafter\def\csname NCLsupppage11\endcsname{%
\NCLlink{413.407013}{193.059448}{15.931976}{14.125000}{thm.4.5}%
\NCLlink{168.442993}{351.020294}{26.841003}{12.901978}{equation.100}%
\NCLlink{198.746002}{351.020294}{26.841003}{12.901978}{equation.101}%
\NCLlink{156.188995}{709.992310}{26.841003}{15.185974}{equation.102}%
}
\expandafter\def\csname NCLsupppage12\endcsname{%
\NCLdestination{77.760002}{76.000000}{supp.D}%
\NCLdestination{77.469002}{93.920166}{supp.prose.D_moment_intro}%
\NCLlink{346.024994}{105.429321}{15.932007}{11.255310}{thm.7.5}%
\NCLdestination{178.946000}{275.715200}{supp.target.085}%
\NCLlink{254.147003}{345.093201}{12.902008}{12.901978}{supp.target.085}%
\NCLdestination{252.805000}{514.157200}{supp.target.086}%
\NCLdestination{77.760000}{76.000000}{supp.outline.010}%
}
\expandafter\def\csname NCLsupppage13\endcsname{%
\NCLdestination{77.469002}{263.939392}{supp.E}%
\NCLdestination{77.469002}{292.350403}{supp.theorem.E.1}%
\NCLlink{263.325012}{249.113586}{12.901978}{12.205261}{supp.target.086}%
\NCLlink{313.083008}{249.113586}{18.355988}{12.205261}{supp.target.087}%
\NCLlink{519.036987}{315.778442}{16.489990}{12.902039}{thm.4.8}%
\NCLlink{343.332001}{344.190430}{16.489014}{12.901978}{thm.5.4}%
\NCLlink{250.942001}{461.609436}{16.488998}{11.689026}{thm.4.8}%
\NCLlink{355.436005}{475.167419}{16.488983}{12.902008}{thm.6.4}%
\NCLlink{377.128998}{651.237427}{15.932007}{12.901978}{thm.4.8}%
\NCLdestination{77.469000}{263.939400}{supp.outline.011}%
}
\expandafter\def\csname NCLsupppage14\endcsname{%
\NCLdestination{77.469002}{444.441681}{supp.proof.E2_intro}%
\NCLlink{366.321991}{211.127808}{15.932007}{11.689941}{thm.4.8}%
\NCLlink{265.934998}{224.684814}{26.841003}{12.901001}{equation.105}%
\NCLlink{296.238007}{224.684814}{26.841003}{12.901001}{equation.106}%
\NCLlink{335.498993}{349.046387}{15.932007}{13.290253}{thm.5.3}%
\NCLlink{171.059006}{425.924377}{15.931992}{11.689026}{thm.4.8}%
\NCLdestination{164.672000}{668.359770}{supp.target.088}%
}
\expandafter\def\csname NCLsupppage15\endcsname{%
\NCLdestination{77.469002}{683.615906}{supp.E_joint_iteration}%
\NCLlink{473.872406}{187.858948}{18.313019}{12.902039}{supp.theorem.E.1}%
\NCLdestination{150.498000}{307.176980}{supp.target.089}%
\NCLdestination{235.517000}{350.130930}{supp.target.090}%
}
\expandafter\def\csname NCLsupppage16\endcsname{%
\NCLdestination{77.469002}{640.908569}{supp.prose.E_activation_proof}%
\NCLlink{275.066986}{182.360168}{12.902008}{12.901978}{supp.target.088}%
\NCLlink{461.187012}{267.757202}{12.901978}{12.902039}{supp.target.089}%
\NCLlink{374.140991}{293.660217}{12.901001}{12.902039}{supp.target.090}%
\NCLlink{446.700012}{609.096191}{16.488983}{11.690002}{thm.3.1}%
\NCLlink{457.255005}{634.908569}{16.489990}{11.599365}{thm.3.3}%
\NCLlink{473.368988}{691.705383}{15.931000}{11.688965}{thm.3.1}%
\NCLlink{516.565002}{705.263428}{15.930969}{14.125000}{thm.3.3}%
\NCLdestination{77.760000}{296.412180}{supp.target.128}%
}
\expandafter\def\csname NCLsupppage17\endcsname{%
\NCLdestination{77.469002}{84.825806}{supp.E.1}%
\NCLdestination{77.469002}{143.454895}{supp.E.2}%
\NCLdestination{77.469002}{431.061951}{supp.E.3}%
\NCLdestination{77.469002}{502.643097}{supp.E.4}%
\NCLdestination{77.469002}{502.643097}{supp.F}%
\NCLlink{318.692993}{116.111816}{15.932007}{11.688965}{thm.5.4}%
\NCLlink{122.564003}{128.629089}{17.902000}{11.255249}{supp.theorem.E.1}%
\NCLlink{176.582001}{168.142944}{16.488998}{11.689026}{thm.5.5}%
\NCLlink{392.320007}{169.544922}{15.932007}{15.724976}{thm.5.3}%
\NCLlink{136.406006}{219.948914}{15.931992}{11.690002}{thm.5.4}%
\NCLlink{273.726013}{247.994934}{21.385986}{11.690002}{thm.4.10}%
\NCLlink{317.791992}{260.945923}{15.932007}{11.689026}{thm.3.4}%
\NCLlink{130.429001}{273.896912}{15.931992}{11.688965}{thm.3.3}%
\NCLlink{136.162003}{449.396973}{15.931992}{11.688965}{thm.6.4}%
\NCLlink{123.865997}{475.299988}{15.932007}{11.690002}{thm.4.8}%
\NCLlink{517.625000}{475.299988}{17.901978}{11.690002}{supp.theorem.E.1}%
\NCLlink{176.582001}{527.331116}{16.488998}{11.689026}{thm.6.5}%
\NCLlink{320.437988}{553.234070}{15.931000}{11.688965}{thm.6.3}%
\NCLlink{123.737000}{647.533081}{15.932007}{11.688965}{thm.6.4}%
\NCLlink{380.746002}{660.485107}{21.386993}{11.690002}{thm.4.10}%
\NCLdestination{77.469000}{84.825809}{supp.outline.012}%
\NCLdestination{77.469000}{143.454900}{supp.outline.013}%
\NCLdestination{77.469000}{431.061960}{supp.outline.014}%
\NCLdestination{77.469000}{502.643100}{supp.outline.015}%
\NCLdestination{77.469000}{502.643100}{supp.outline.016}%
}
\expandafter\def\csname NCLsupppage18\endcsname{%
\NCLlink{222.822006}{118.109802}{15.931992}{15.239014}{thm.3.3}%
\NCLlink{261.539001}{118.109802}{15.932007}{15.239014}{thm.3.4}%
\NCLlink{167.882004}{142.611816}{21.386993}{11.690063}{thm.4.10}%
}
\expandafter\def\csname NCLsupppage19\endcsname{%
\NCLdestination{77.469002}{175.763367}{supp.proof.F_top}%
\NCLlink{148.371002}{711.707275}{15.931992}{12.901978}{thm.4.2}%
\NCLdestination{213.895000}{725.585290}{supp.target.093}%
}
\expandafter\def\csname NCLsupppage20\endcsname{%
\NCLlink{147.692001}{481.580139}{15.931992}{11.690002}{thm.7.3}%
\NCLdestination{215.882000}{496.038160}{supp.target.091}%
\NCLlink{317.558990}{530.671143}{12.902008}{12.901978}{supp.target.091}%
\NCLlink{290.431000}{544.081177}{15.932007}{13.360046}{thm.4.2}%
\NCLdestination{188.456000}{545.613180}{supp.target.092}%
\NCLlink{439.135986}{631.661621}{12.902008}{12.901001}{supp.target.092}%
\NCLlink{242.300995}{681.248596}{15.932007}{12.901001}{thm.7.1}%
\NCLdestination{238.495000}{695.101600}{supp.target.094}%
}
\expandafter\def\csname NCLsupppage21\endcsname{%
\NCLlink{180.173996}{108.790161}{15.932007}{11.688965}{thm.2.7}%
\NCLlink{132.431000}{179.450195}{12.901001}{12.901001}{supp.target.093}%
\NCLlink{157.279007}{179.450195}{12.901993}{12.901001}{supp.target.094}%
\NCLdestination{236.766000}{224.932250}{supp.target.095}%
\NCLlink{94.400002}{294.060211}{12.902000}{12.902008}{supp.target.095}%
\NCLdestination{77.760000}{296.812240}{supp.target.125}%
}
\expandafter\def\csname NCLsupppage22\endcsname{%
\NCLdestination{221.018000}{126.561520}{supp.target.096}%
\NCLlink{147.552002}{261.933533}{12.901001}{12.901978}{supp.target.096}%
\NCLlink{518.812988}{550.373535}{12.901001}{12.901001}{supp.target.096}%
\NCLlink{282.390015}{678.081543}{12.901978}{12.901978}{supp.target.096}%
}
\expandafter\def\csname NCLsupppage23\endcsname{%
\NCLdestination{77.760002}{630.911255}{supp.G}%
\NCLdestination{77.760002}{678.658142}{supp.G.1}%
\NCLlink{166.253006}{576.459045}{18.203995}{13.360962}{thm.B.1}%
\NCLdestination{77.760000}{630.911280}{supp.outline.017}%
\NCLdestination{77.760000}{678.658200}{supp.outline.018}%
}
\expandafter\def\csname NCLsupppage24\endcsname{%
\NCLlink{118.873001}{92.186829}{16.488998}{11.690002}{thm.2.3}%
\NCLlink{297.320007}{315.364746}{7.446991}{12.901978}{equation.1}%
\NCLlink{165.718994}{328.315796}{7.447006}{12.901001}{equation.2}%
}
\expandafter\def\csname NCLsupppage25\endcsname{%
\NCLdestination{77.760002}{612.590088}{supp.G.2}%
\NCLlink{174.067001}{232.491943}{7.447006}{12.902039}{equation.3}%
\NCLlink{128.128006}{348.125916}{7.446991}{12.901001}{equation.3}%
\NCLlink{164.067001}{483.852905}{7.447006}{12.901001}{equation.4}%
\NCLlink{296.066986}{483.852905}{7.447021}{12.901001}{equation.1}%
\NCLlink{114.734001}{532.331177}{7.446999}{12.205261}{equation.4}%
\NCLlink{168.378998}{547.929993}{15.932007}{11.598816}{thm.2.1}%
\NCLlink{215.311996}{547.929993}{7.447006}{11.598816}{Item.1}%
\NCLlink{118.873001}{647.588013}{16.488998}{11.690002}{thm.2.4}%
\NCLdestination{77.760000}{612.590100}{supp.outline.019}%
}
\expandafter\def\csname NCLsupppage26\endcsname{%
\NCLdestination{153.112000}{447.159190}{supp.target.097}%
\NCLlink{287.203003}{575.248291}{12.902008}{12.902039}{supp.target.097}%
\NCLdestination{234.849000}{578.906280}{supp.target.098}%
}
\expandafter\def\csname NCLsupppage27\endcsname{%
\NCLlink{111.823997}{161.571472}{12.902000}{12.901978}{supp.target.098}%
\NCLdestination{241.900000}{258.651500}{supp.target.099}%
\NCLlink{227.278000}{541.436462}{12.901001}{12.901978}{supp.target.099}%
}
\expandafter\def\csname NCLsupppage28\endcsname{%
\NCLdestination{77.760002}{122.711731}{supp.G.3}%
\NCLdestination{77.760002}{598.840454}{supp.G.4}%
\NCLlink{285.638000}{104.280762}{12.901001}{12.901001}{supp.target.099}%
\NCLlink{118.873001}{154.982727}{16.488998}{11.690002}{thm.2.5}%
\NCLlink{118.873001}{629.661926}{16.488998}{11.688965}{thm.2.7}%
\NCLdestination{77.760000}{122.711730}{supp.outline.020}%
\NCLdestination{77.760000}{598.840480}{supp.outline.021}%
}
\expandafter\def\csname NCLsupppage29\endcsname{%
}
\expandafter\def\csname NCLsupppage30\endcsname{%
\NCLdestination{77.469002}{206.410950}{supp.G.5}%
\NCLdestination{77.469002}{252.063965}{supp.G.6}%
\NCLdestination{77.469002}{310.694092}{supp.G.7}%
\NCLdestination{77.760002}{345.546417}{supp.G.8}%
\NCLlink{127.847000}{224.745972}{18.203995}{11.688965}{thm.B.1}%
\NCLlink{372.783997}{270.399963}{18.205017}{11.690002}{thm.B.2}%
\NCLlink{76.473000}{294.180969}{18.205002}{9.446960}{thm.B.3}%
\NCLlink{128.526001}{329.029144}{15.931992}{11.688965}{thm.2.9}%
\NCLlink{118.873001}{376.367828}{16.488998}{11.688965}{thm.3.1}%
\NCLlink{267.919006}{507.148712}{7.446991}{12.205872}{equation.6}%
\NCLdestination{77.469000}{206.410950}{supp.outline.022}%
\NCLdestination{77.469000}{252.063970}{supp.outline.023}%
\NCLdestination{77.469000}{310.694130}{supp.outline.024}%
\NCLdestination{77.760000}{345.546430}{supp.outline.025}%
}
\expandafter\def\csname NCLsupppage31\endcsname{%
\NCLdestination{77.760000}{111.895999}{supp.target.100}%
\NCLlink{234.485992}{154.313965}{12.902008}{12.901978}{supp.target.100}%
\NCLdestination{182.778000}{281.876000}{supp.target.101}%
\NCLlink{129.106995}{528.540955}{12.902008}{12.901978}{supp.target.101}%
\NCLlink{342.980011}{528.540955}{12.901978}{12.901978}{equation.7}%
\NCLlink{357.372986}{541.492004}{12.902008}{12.901001}{equation.8}%
\NCLlink{331.125000}{554.443970}{12.902008}{12.901978}{supp.target.101}%
\NCLlink{307.181000}{567.394958}{12.901001}{12.901978}{equation.9}%
\NCLlink{165.718994}{580.346008}{12.902008}{12.901001}{equation.10}%
\NCLdestination{205.515000}{179.019960}{supp.target.102}%
}
\expandafter\def\csname NCLsupppage32\endcsname{%
\NCLdestination{77.760002}{558.525818}{supp.G.9}%
\NCLdestination{77.469002}{347.208862}{supp.prose.G_expectation_end}%
\NCLlink{138.955994}{101.156860}{12.901001}{12.902039}{supp.target.102}%
\NCLlink{325.207001}{394.726685}{21.385986}{12.901978}{equation.10}%
\NCLlink{118.873001}{591.402832}{16.488998}{11.690002}{thm.3.2}%
\NCLlink{220.817993}{592.008850}{12.901001}{12.901978}{equation.12}%
\NCLdestination{185.704000}{599.408900}{supp.target.103}%
\NCLlink{388.440002}{640.157837}{12.901001}{12.901001}{supp.target.103}%
\NCLdestination{77.760000}{558.525800}{supp.outline.026}%
}
\expandafter\def\csname NCLsupppage33\endcsname{%
\NCLlink{114.734001}{181.624634}{12.901001}{12.902100}{equation.13}%
\NCLlink{162.688995}{324.737579}{12.901001}{12.902008}{equation.14}%
\NCLlink{178.279007}{430.499573}{12.901993}{12.902008}{equation.13}%
\NCLlink{235.037003}{430.499573}{12.901001}{12.902008}{equation.15}%
\NCLlink{145.945999}{525.605591}{12.901001}{12.901001}{equation.16}%
\NCLlink{160.070999}{538.557617}{15.932007}{13.494019}{thm.3.1}%
\NCLlink{297.080994}{551.508606}{12.901001}{12.901001}{equation.14}%
\NCLlink{291.546997}{564.637634}{12.902008}{12.902039}{equation.7}%
\NCLlink{317.516998}{564.637634}{12.901001}{12.902039}{equation.8}%
\NCLlink{364.584015}{564.637634}{12.901978}{12.902039}{equation.9}%
\NCLlink{81.005997}{577.588623}{12.902000}{12.901001}{equation.17}%
\NCLlink{105.855003}{577.588623}{12.900993}{12.901001}{equation.19}%
\NCLlink{156.460999}{713.616943}{12.902008}{12.205322}{equation.16}%
}
\expandafter\def\csname NCLsupppage34\endcsname{%
\NCLdestination{77.760002}{76.000000}{supp.G.10}%
\NCLlink{118.873001}{106.821472}{16.488998}{11.688965}{thm.3.3}%
\NCLdestination{190.925000}{617.489470}{supp.target.105}%
\NCLdestination{254.349000}{455.076490}{supp.target.109}%
\NCLdestination{77.760000}{76.000000}{supp.outline.027}%
}
\expandafter\def\csname NCLsupppage35\endcsname{%
\NCLdestination{202.429000}{218.711000}{supp.target.104}%
\NCLlink{465.212006}{277.809998}{12.902008}{12.902039}{supp.target.104}%
\NCLlink{208.007004}{329.615997}{12.901993}{12.902008}{supp.target.104}%
\NCLlink{340.885010}{342.566986}{12.901978}{12.901001}{supp.target.105}%
\NCLdestination{182.203000}{346.694000}{supp.target.106}%
\NCLlink{518.812988}{393.898987}{12.901001}{12.902008}{supp.target.106}%
\NCLdestination{156.010000}{111.344970}{supp.target.107}%
\NCLlink{346.132996}{537.285950}{12.901001}{12.901978}{supp.target.107}%
\NCLdestination{261.200000}{418.088000}{supp.target.108}%
\NCLlink{390.375000}{537.285950}{12.902008}{12.901978}{supp.target.108}%
\NCLlink{98.278999}{713.299988}{12.902000}{12.672974}{supp.target.109}%
\NCLdestination{77.760000}{590.094000}{supp.target.110}%
}
\expandafter\def\csname NCLsupppage36\endcsname{%
\NCLdestination{77.760002}{378.947693}{supp.H}%
\NCLdestination{77.469002}{440.009033}{supp.H.1}%
\NCLdestination{77.469002}{485.662048}{supp.H.2}%
\NCLdestination{77.760002}{520.515381}{supp.H.3}%
\NCLdestination{77.469002}{76.000000}{supp.proof.G_coefficient}%
\NCLdestination{77.469002}{113.026001}{supp.proof.G_exponential}%
\NCLlink{247.246002}{157.165710}{12.901001}{12.901001}{supp.target.110}%
\NCLlink{202.369995}{170.117737}{12.902008}{12.901978}{equation.20}%
\NCLlink{283.687012}{261.239746}{12.901978}{12.902039}{equation.20}%
\NCLlink{340.445007}{261.239746}{12.901978}{12.902039}{equation.21}%
\NCLlink{301.273010}{458.344055}{15.931976}{11.688965}{thm.4.2}%
\NCLlink{406.713989}{458.344055}{15.932007}{11.688965}{thm.4.1}%
\NCLlink{411.346985}{503.998047}{15.932007}{11.690002}{thm.4.2}%
\NCLlink{118.873001}{551.337830}{16.488998}{11.690002}{thm.4.1}%
\NCLlink{177.453995}{549.216797}{27.144012}{9.446960}{supp.target.111}%
\NCLdestination{77.760000}{378.947700}{supp.outline.028}%
\NCLdestination{77.469000}{440.009040}{supp.outline.029}%
\NCLdestination{77.469000}{485.662060}{supp.outline.030}%
\NCLdestination{77.760000}{520.515400}{supp.outline.031}%
}
\expandafter\def\csname NCLsupppage37\endcsname{%
\NCLdestination{77.469002}{394.618561}{supp.H.4}%
\NCLdestination{77.469002}{440.271576}{supp.H.5}%
\NCLdestination{77.760002}{488.375916}{supp.H.6}%
\NCLlink{224.455994}{341.634613}{12.902008}{12.204834}{equation.33}%
\NCLlink{419.524994}{341.634613}{12.902008}{12.204834}{equation.33}%
\NCLlink{158.535004}{354.586639}{12.901993}{12.902039}{equation.34}%
\NCLlink{515.781982}{354.586639}{12.902039}{12.902039}{equation.35}%
\NCLlink{103.776001}{423.783600}{15.931999}{9.567993}{thm.4.3}%
\NCLlink{127.690002}{458.607574}{15.931992}{11.690002}{thm.4.4}%
\NCLlink{281.928009}{484.075897}{15.931976}{11.255310}{thm.4.5}%
\NCLlink{118.873001}{520.384888}{16.488998}{11.688965}{thm.4.6}%
\NCLdestination{77.469000}{394.618570}{supp.outline.032}%
\NCLdestination{77.469000}{440.271590}{supp.outline.033}%
\NCLdestination{77.760000}{488.375930}{supp.outline.034}%
}
\expandafter\def\csname NCLsupppage38\endcsname{%
\NCLlink{349.710999}{434.740723}{12.902008}{12.901978}{equation.49}%
\NCLdestination{165.425000}{330.566700}{supp.target.112}%
\NCLlink{375.772003}{434.740723}{12.901001}{12.901978}{supp.target.112}%
\NCLdestination{254.278000}{197.845700}{supp.target.113}%
\NCLdestination{248.125000}{437.262740}{supp.target.114}%
}
\expandafter\def\csname NCLsupppage39\endcsname{%
\NCLdestination{77.469002}{575.446960}{supp.H.7}%
\NCLdestination{77.760002}{636.202271}{supp.I}%
\NCLdestination{77.760002}{686.072144}{supp.I.1}%
\NCLlink{148.914993}{100.195984}{12.902008}{12.901001}{supp.target.113}%
\NCLlink{173.764008}{100.195984}{12.901993}{12.901001}{supp.target.114}%
\NCLlink{139.643005}{216.758972}{12.901001}{12.901001}{equation.52}%
\NCLlink{164.490997}{216.758972}{12.902008}{12.901001}{equation.53}%
\NCLlink{286.688995}{229.710999}{12.902008}{12.902039}{equation.45}%
\NCLlink{145.945999}{289.305969}{12.901001}{12.901001}{equation.54}%
\NCLlink{191.162003}{326.367981}{12.901001}{12.902008}{equation.55}%
\NCLlink{204.203995}{339.319000}{12.902008}{12.901001}{equation.56}%
\NCLlink{179.369995}{477.522980}{12.902008}{12.901978}{equation.51}%
\NCLdestination{227.697000}{364.491000}{supp.target.115}%
\NCLlink{255.279007}{477.522980}{12.901993}{12.901978}{supp.target.115}%
\NCLlink{147.729996}{593.781982}{15.932007}{11.688965}{thm.4.9}%
\NCLdestination{77.469000}{575.447000}{supp.outline.035}%
\NCLdestination{77.760000}{636.202300}{supp.outline.036}%
\NCLdestination{77.760000}{686.072170}{supp.outline.037}%
}
\expandafter\def\csname NCLsupppage40\endcsname{%
\NCLlink{118.873001}{91.310181}{16.488998}{11.690002}{thm.7.2}%
\NCLlink{350.121002}{91.310181}{15.932007}{11.690002}{thm.2.2}%
}
\expandafter\def\csname NCLsupppage41\endcsname{%
\NCLdestination{246.896000}{683.646000}{supp.target.116}%
}
\expandafter\def\csname NCLsupppage42\endcsname{%
\NCLlink{457.071991}{267.259644}{18.356018}{12.901978}{supp.I.1}%
\NCLlink{81.005997}{349.500641}{18.356003}{12.901001}{supp.target.116}%
\NCLlink{378.583008}{689.910583}{7.446991}{11.688965}{section.2}%
}
\expandafter\def\csname NCLsupppage43\endcsname{%
\NCLdestination{77.760002}{127.505859}{supp.I.2}%
\NCLdestination{77.469002}{426.224518}{supp.I_joint_scalar}%
\NCLlink{118.873001}{160.726807}{16.488998}{11.689941}{thm.7.3}%
\NCLdestination{189.881000}{233.108890}{supp.target.117}%
\NCLlink{151.253006}{289.578857}{18.356995}{12.901978}{supp.target.117}%
\NCLlink{81.005997}{399.172852}{18.356003}{12.901001}{equation.187}%
\NCLlink{318.169006}{544.931458}{15.932007}{11.690002}{thm.4.8}%
\NCLlink{200.873993}{669.954773}{26.841003}{12.583313}{equation.103}%
\NCLlink{231.177002}{669.954773}{26.842010}{12.583313}{equation.104}%
\NCLdestination{77.760000}{127.505860}{supp.outline.038}%
}
\expandafter\def\csname NCLsupppage44\endcsname{%
\NCLlink{200.873993}{76.318726}{26.841003}{0.318726}{equation.103}%
\NCLlink{231.177002}{76.318726}{26.842010}{0.318726}{equation.104}%
\NCLlink{122.564003}{88.663696}{15.931999}{11.688965}{thm.4.5}%
\NCLlink{465.716003}{169.495728}{26.841003}{12.901978}{equation.191}%
\NCLlink{385.295013}{327.328705}{15.931976}{12.901978}{thm.3.4}%
\NCLlink{233.975998}{415.414032}{21.387009}{11.688965}{thm.2.9}%
\NCLlink{335.996002}{601.729065}{7.446991}{11.689026}{section.4}%
\NCLlink{316.226990}{676.571045}{18.356018}{12.902039}{supp.I_joint_scalar}%
\NCLdestination{244.328000}{555.224090}{supp.target.118}%
\NCLlink{348.527008}{676.571045}{18.355988}{12.902039}{supp.target.118}%
\NCLdestination{253.350000}{629.460000}{supp.target.119}%
\NCLlink{402.665985}{676.571045}{18.356018}{12.902039}{supp.target.119}%
\NCLlink{129.854996}{718.946045}{18.356003}{12.902039}{equation.187}%
\NCLdestination{211.444000}{492.996040}{supp.target.120}%
\NCLlink{161.369995}{718.946045}{18.356003}{12.902039}{supp.target.120}%
\NCLlink{192.884995}{718.946045}{18.356003}{12.902039}{supp.target.118}%
\NCLlink{245.613007}{718.946045}{18.355988}{12.902039}{supp.target.119}%
\NCLlink{303.522003}{718.946045}{18.355988}{12.902039}{equation.188}%
\NCLlink{333.825012}{718.946045}{18.355988}{12.902039}{equation.190}%
}
\expandafter\def\csname NCLsupppage45\endcsname{%
\NCLdestination{77.760002}{78.363708}{supp.I.3}%
\NCLlink{164.175003}{110.634705}{16.489990}{11.690002}{thm.7.4}%
\NCLlink{218.539001}{110.634705}{7.567993}{11.690002}{Item.34}%
\NCLlink{168.865005}{574.413940}{15.931992}{11.689026}{thm.7.3}%
\NCLdestination{220.518000}{637.991970}{supp.target.121}%
\NCLlink{452.459991}{723.956909}{18.356018}{12.901978}{supp.target.121}%
\NCLdestination{77.760000}{78.363710}{supp.outline.039}%
}
\expandafter\def\csname NCLsupppage46\endcsname{%
\NCLdestination{77.469002}{328.423767}{supp.I_scalar_terminal}%
\NCLdestination{77.469002}{426.705750}{supp.proof.I_ref}%
\NCLlink{147.445999}{101.272217}{15.932007}{11.690002}{thm.4.2}%
\NCLlink{360.522003}{254.696228}{21.385986}{11.689026}{thm.4.9}%
\NCLlink{341.928009}{341.431580}{15.931976}{12.753967}{thm.4.9}%
\NCLlink{302.173004}{439.260773}{15.932007}{12.301178}{thm.4.2}%
\NCLlink{116.361000}{480.809631}{15.931999}{11.598846}{thm.2.1}%
\NCLlink{163.296005}{480.809631}{7.446991}{11.598846}{Item.3}%
\NCLlink{214.365997}{529.335632}{18.356003}{12.901001}{supp.proof.I_ref}%
\NCLdestination{238.209000}{495.176650}{supp.target.122}%
\NCLlink{244.669006}{529.335632}{18.355988}{12.901001}{supp.target.122}%
\NCLlink{135.490997}{667.514648}{15.932007}{15.714966}{thm.7.1}%
\NCLdestination{203.044000}{102.687198}{supp.target.124}%
}
\expandafter\def\csname NCLsupppage47\endcsname{%
\NCLdestination{77.760002}{472.789978}{supp.I.4}%
\NCLdestination{77.760002}{652.493469}{supp.I.5}%
\NCLlink{168.218994}{320.046265}{15.931000}{11.689026}{thm.2.7}%
\NCLlink{432.009003}{355.737305}{18.355988}{12.901001}{supp.target.121}%
\NCLdestination{238.711000}{252.764280}{supp.target.123}%
\NCLlink{106.954002}{368.689270}{18.355995}{12.901978}{supp.target.123}%
\NCLlink{81.005997}{381.640259}{18.356003}{12.901001}{supp.target.124}%
\NCLlink{112.521004}{381.640259}{18.357002}{12.901001}{supp.I_scalar_terminal}%
\NCLlink{144.037003}{381.640259}{18.356003}{12.901001}{supp.target.121}%
\NCLlink{196.764008}{381.640259}{18.355988}{12.901001}{supp.target.123}%
\NCLlink{164.175003}{503.611450}{16.489990}{11.688965}{thm.7.4}%
\NCLlink{218.539001}{503.611450}{7.567993}{11.688965}{Item.35}%
\NCLlink{304.108002}{503.611450}{17.597992}{11.688965}{supp.target.125}%
\NCLlink{273.996002}{546.260437}{15.932007}{11.688965}{thm.7.4}%
\NCLlink{329.342010}{546.260437}{7.446991}{11.688965}{Item.34}%
\NCLlink{164.175003}{687.491455}{16.489990}{11.690002}{thm.7.4}%
\NCLlink{218.539001}{687.491455}{7.567993}{11.690002}{Item.36}%
\NCLlink{304.108002}{687.491455}{17.597992}{11.690002}{supp.target.125}%
\NCLdestination{77.760000}{472.789990}{supp.outline.040}%
\NCLdestination{77.760000}{652.493490}{supp.outline.041}%
}
\expandafter\def\csname NCLsupppage48\endcsname{%
\NCLdestination{77.760002}{157.860352}{supp.I.6}%
\NCLdestination{77.469002}{76.000000}{supp.prose.I_orientation_input}%
\NCLlink{118.873001}{189.423401}{16.488998}{11.689026}{thm.7.5}%
\NCLlink{149.570007}{391.250366}{15.931992}{12.901978}{thm.7.2}%
\NCLdestination{77.760000}{157.860350}{supp.outline.042}%
}
\expandafter\def\csname NCLsupppage49\endcsname{%
\NCLdestination{211.962000}{325.930640}{supp.target.087}%
\NCLlink{209.067001}{168.588623}{15.931992}{11.690002}{thm.7.2}%
\NCLlink{372.278992}{318.254639}{10.326019}{11.690002}{supp.D}%
\NCLlink{81.005997}{383.178650}{18.356003}{12.902008}{supp.target.087}%
}
\expandafter\def\csname NCLsupppage50\endcsname{%
\NCLdestination{77.760002}{666.790405}{supp.J}%
\NCLdestination{209.538000}{92.370910}{supp.target.126}%
\NCLlink{202.612000}{164.437927}{18.356995}{12.901978}{supp.target.126}%
\NCLlink{117.088997}{649.928955}{15.931999}{11.690002}{thm.2.1}%
\NCLlink{165.662003}{649.928955}{7.446991}{11.690002}{Item.2}%
\NCLdestination{77.760000}{666.790400}{supp.outline.043}%
}
\expandafter\def\csname NCLsupppage51\endcsname{%
\NCLdestination{77.760002}{77.757568}{supp.J.1}%
\NCLdestination{77.469002}{643.409546}{supp.J_right_moments}%
\NCLlink{118.873001}{110.634583}{16.488998}{11.690002}{thm.5.1}%
\NCLlink{165.307999}{121.464539}{37.143997}{9.446960}{supp.target.127}%
\NCLdestination{77.760000}{77.757570}{supp.outline.044}%
}
\expandafter\def\csname NCLsupppage52\endcsname{%
\NCLdestination{77.760002}{223.409546}{supp.J.2}%
\NCLdestination{77.469002}{472.051544}{supp.J_right_input}%
\NCLlink{118.873001}{256.286499}{16.488998}{11.688965}{thm.5.2}%
\NCLlink{168.453995}{351.870544}{15.932007}{11.690002}{thm.5.1}%
\NCLlink{391.238007}{484.135742}{15.932007}{11.688965}{thm.5.1}%
\NCLdestination{77.760000}{223.409550}{supp.outline.045}%
}
\expandafter\def\csname NCLsupppage53\endcsname{%
\NCLdestination{77.469002}{77.626892}{supp.J_right_conditional}%
\NCLlink{185.531006}{133.207703}{15.931992}{16.545593}{thm.2.5}%
\NCLlink{135.490997}{543.660706}{15.932007}{11.690002}{thm.2.3}%
}
\expandafter\def\csname NCLsupppage54\endcsname{%
\NCLlink{422.829010}{184.593018}{7.446991}{11.689026}{section.3}%
\NCLlink{147.755005}{290.440979}{15.931992}{11.688965}{thm.4.1}%
\NCLlink{381.299011}{390.658997}{15.931000}{11.690002}{thm.4.1}%
}
\expandafter\def\csname NCLsupppage55\endcsname{%
\NCLdestination{77.469002}{559.808838}{supp.J_right_retention}%
\NCLlink{470.008484}{88.204834}{17.812378}{12.204834}{supp.target.128}%
\NCLlink{424.226990}{572.999084}{15.931000}{12.901001}{thm.2.6}%
}
\expandafter\def\csname NCLsupppage56\endcsname{%
\NCLdestination{77.760002}{288.749817}{supp.J.3}%
\NCLdestination{77.469002}{167.707153}{supp.proof.J_coefficient}%
\NCLdestination{77.469002}{259.154846}{supp.prose.J_cubic_scope}%
\NCLlink{160.386002}{87.598816}{17.901001}{11.598816}{supp.target.128}%
\NCLlink{117.865166}{319.572266}{16.084785}{11.690002}{thm.5.3}%
\NCLlink{286.796936}{319.572266}{15.541443}{11.690002}{thm.5.2}%
\NCLlink{160.772003}{346.081268}{18.356003}{12.901978}{equation.109}%
\NCLlink{378.268127}{466.218262}{17.788483}{12.901978}{equation.110}%
\NCLlink{149.966995}{492.121277}{18.357010}{12.902008}{equation.111}%
\NCLlink{399.734375}{505.072266}{17.906525}{12.901001}{equation.112}%
\NCLlink{216.477005}{556.878296}{18.355988}{12.902039}{equation.113}%
\NCLlink{279.002014}{593.642273}{18.355988}{12.901978}{equation.114}%
\NCLlink{241.742996}{632.496277}{18.356003}{12.901001}{equation.115}%
\NCLlink{372.440002}{632.496277}{18.355988}{12.901001}{equation.114}%
\NCLlink{275.127991}{667.878296}{18.356018}{12.902039}{equation.115}%
\NCLdestination{77.760000}{288.749830}{supp.outline.046}%
}
\expandafter\def\csname NCLsupppage57\endcsname{%
\NCLdestination{77.760002}{156.272278}{supp.J.4}%
\NCLlink{202.673004}{98.766296}{18.356003}{12.901001}{equation.116}%
\NCLlink{118.873001}{189.149292}{16.488998}{11.690002}{thm.6.1}%
\NCLlink{516.994995}{189.149292}{15.931030}{11.690002}{thm.4.1}%
\NCLdestination{251.691000}{421.319300}{supp.target.129}%
\NCLdestination{150.479000}{623.991300}{supp.target.130}%
\NCLdestination{77.760000}{156.272280}{supp.outline.047}%
}
\expandafter\def\csname NCLsupppage58\endcsname{%
\NCLlink{272.315002}{88.204834}{18.356995}{12.204834}{supp.target.129}%
\NCLlink{344.606995}{101.156860}{18.356018}{12.902039}{supp.target.129}%
\NCLlink{297.492004}{443.389832}{15.932007}{13.208984}{thm.4.1}%
\NCLlink{224.985992}{549.873840}{18.356003}{12.902039}{supp.target.130}%
\NCLlink{484.565002}{585.564819}{18.355988}{12.901001}{supp.target.129}%
\NCLlink{476.832001}{598.516846}{18.355988}{12.901978}{supp.target.129}%
\NCLlink{102.218002}{659.387817}{18.356995}{12.901001}{supp.target.129}%
}
\expandafter\def\csname NCLsupppage59\endcsname{%
\NCLdestination{77.760002}{418.074097}{supp.J.5}%
\NCLdestination{77.469002}{76.000000}{supp.proof.J_family}%
\NCLdestination{77.469002}{504.809082}{supp.prose.J_fixed_prefix}%
\NCLlink{147.065002}{87.509094}{26.841003}{11.509094}{equation.135}%
\NCLlink{177.367996}{87.509094}{26.841003}{11.509094}{equation.136}%
\NCLdestination{218.046000}{209.686100}{supp.target.131}%
\NCLlink{94.278999}{272.668091}{18.356003}{12.901001}{supp.target.131}%
\NCLlink{139.643005}{358.781067}{18.355988}{12.901978}{equation.137}%
\NCLlink{169.945999}{358.781067}{18.356003}{12.901978}{equation.138}%
\NCLlink{290.096985}{412.595093}{18.357025}{12.902008}{equation.137}%
\NCLlink{320.399994}{412.595093}{18.356995}{12.902008}{equation.138}%
\NCLlink{118.228210}{450.345093}{16.230392}{11.689026}{thm.6.2}%
\NCLdestination{77.760000}{418.074100}{supp.outline.048}%
}
\expandafter\def\csname NCLsupppage60\endcsname{%
\NCLdestination{77.760002}{328.848724}{supp.J.6}%
\NCLdestination{77.469002}{533.452637}{supp.J_left_sharp}%
\NCLlink{118.873001}{361.119720}{16.488998}{11.690002}{thm.6.3}%
\NCLlink{194.177002}{374.070740}{15.931992}{11.690002}{thm.6.2}%
\NCLdestination{77.760000}{328.848730}{supp.outline.049}%
}
\expandafter\def\csname NCLsupppage61\endcsname{%
\NCLdestination{210.976000}{317.281230}{supp.target.132}%
\NCLlink{419.664001}{356.513245}{18.355988}{12.902008}{supp.target.132}%
\NCLlink{195.097000}{457.142242}{18.356003}{12.902008}{equation.150}%
\NCLlink{448.664001}{469.487244}{15.932007}{11.688995}{thm.6.1}%
\NCLlink{125.552002}{670.793213}{18.356003}{12.901001}{supp.target.132}%
\NCLlink{446.074127}{724.276184}{15.931976}{11.688965}{thm.6.1}%
}
\expandafter\def\csname NCLsupppage62\endcsname{%
\NCLdestination{77.469002}{279.385742}{supp.J_left_retention}%
\NCLlink{136.123993}{185.524170}{15.932007}{12.895996}{thm.6.1}%
\NCLlink{115.413002}{377.939667}{15.931999}{11.688965}{thm.2.6}%
\NCLlink{135.110001}{422.403076}{26.841003}{12.361359}{equation.154}%
\NCLlink{184.807007}{422.403076}{26.840988}{12.361359}{equation.155}%
\NCLdestination{225.334000}{237.530280}{supp.target.133}%
\NCLlink{341.421814}{438.608887}{18.355988}{12.205811}{supp.target.133}%
\NCLlink{128.128006}{512.894897}{18.355988}{12.902008}{supp.target.132}%
\NCLdestination{193.587000}{582.142900}{supp.target.134}%
}
\expandafter\def\csname NCLsupppage63\endcsname{%
\NCLdestination{77.469002}{579.365723}{supp.prose.J_terminal_end}%
\NCLlink{309.646820}{90.448364}{18.355988}{12.901001}{supp.target.133}%
\NCLlink{208.580994}{171.698364}{18.356003}{12.901001}{supp.target.134}%
\NCLlink{135.878006}{225.661682}{18.355988}{12.477539}{supp.target.133}%
\NCLdestination{217.405000}{185.459350}{supp.target.135}%
\NCLlink{184.106995}{225.661682}{18.356003}{12.477539}{supp.target.135}%
\NCLlink{372.855011}{338.163727}{15.931976}{11.690002}{thm.3.2}%
\NCLlink{297.985992}{506.917725}{17.902008}{12.902008}{supp.target.128}%
\NCLlink{270.226990}{542.867676}{18.356018}{12.901001}{equation.154}%
\NCLlink{304.295990}{542.867676}{18.356018}{12.901001}{equation.155}%
\NCLlink{327.054993}{591.914551}{26.841003}{12.548828}{equation.156}%
\NCLlink{377.628998}{591.914551}{26.841003}{12.548828}{equation.157}%
}
\expandafter\def\csname NCLsupppage64\endcsname{%
\NCLdestination{77.469002}{206.244934}{supp.K}%
\NCLdestination{77.469002}{225.995972}{supp.K.1}%
\NCLdestination{77.469002}{284.600952}{supp.K.2}%
\NCLdestination{274.553986}{319.556213}{supp.references}%
\NCLdestination{77.760000}{510.267070}{supp.target.111}%
\NCLdestination{77.760000}{532.185080}{supp.target.127}%
\NCLlink{315.806000}{244.330994}{15.932007}{11.690002}{thm.8.2}%
\NCLlink{115.413002}{268.112976}{15.931999}{9.447021}{thm.8.1}%
\NCLurl{464.829987}{546.382080}{71.126984}{10.958984}{https://arxiv.org/abs/2608.15137}%
\NCLdestination{77.469000}{206.244940}{supp.outline.050}%
\NCLdestination{77.469000}{225.995970}{supp.outline.051}%
\NCLdestination{77.469000}{284.600960}{supp.outline.052}%
\NCLdestination{274.554000}{319.556200}{supp.outline.053}%
}

%% file: external_inputs.tex
\makeatletter
\expandafter\gdef\csname r@cor:gaussian-grid-moments\endcsname{{3.8}{76}{Separated Gaussian grid moments}{thm.3.8}{}}
\makeatother

%% file: supplement_pages.tex
\clearpage
\pagenumbering{arabic}
\renewcommand{\thepage}{S-\arabic{page}}
\bookmark[level=0,dest={arxiv.supplement}]{Technical supplement}

\bookmark[level=1,dest={supp.outline.001}]{Appendix A. Normalization identities}
\bookmark[level=1,dest={supp.outline.002}]{Appendix B. Standard resolvent and statistic differentiation}
\bookmark[level=2,dest={supp.outline.003}]{B.1. Edge-resolvent differentiation}
\bookmark[level=2,dest={supp.outline.004}]{B.2. Covariance direction derivatives}
\bookmark[level=2,dest={supp.outline.005}]{B.3. Coordinate derivatives of covariance observables}
\bookmark[level=2,dest={supp.outline.006}]{B.4. Norm of the linearized resolvent}
\bookmark[level=1,dest={supp.outline.007}]{Appendix C. Row and column expansions}
\bookmark[level=2,dest={supp.outline.008}]{C.1. Scaling and Gaussian interpolation}
\bookmark[level=2,dest={supp.outline.009}]{C.2. Joint cumulant expansion}
\bookmark[level=1,dest={supp.outline.010}]{Appendix D. A maximal moment estimate for the exposed increments}
\bookmark[level=1,dest={supp.outline.011}]{Appendix E. Right- and left-tail comparison}
\bookmark[level=2,dest={supp.outline.012}]{E.1. Right-tail specialization}
\bookmark[level=2,dest={supp.outline.013}]{E.2. Right-tail comparison}
\bookmark[level=2,dest={supp.outline.014}]{E.3. Left-tail specialization}
\bookmark[level=2,dest={supp.outline.015}]{E.4. Left-tail comparison}
\bookmark[level=1,dest={supp.outline.016}]{Appendix F. Auxiliary lemmas}
\bookmark[level=1,dest={supp.outline.017}]{Appendix G. Probability and smoothing estimates}
\bookmark[level=2,dest={supp.outline.018}]{G.1. Comparison with submatrices}
\bookmark[level=2,dest={supp.outline.019}]{G.2. Cluster sets from endpoint limits}
\bookmark[level=2,dest={supp.outline.020}]{G.3. Independent submatrices}
\bookmark[level=2,dest={supp.outline.021}]{G.4. Iterated differential inequality}
\bookmark[level=2,dest={supp.outline.022}]{G.5. A common filtration}
\bookmark[level=2,dest={supp.outline.023}]{G.6. Borel--Cantelli argument}
\bookmark[level=2,dest={supp.outline.024}]{G.7. Monotonicity in the threshold}
\bookmark[level=2,dest={supp.outline.025}]{G.8. Weighted cutoff functions}
\bookmark[level=2,dest={supp.outline.026}]{G.9. Smoothed eigenvalue counts}
\bookmark[level=2,dest={supp.outline.027}]{G.10. Derivative estimates}
\bookmark[level=1,dest={supp.outline.028}]{Appendix H. Local laws and Gaussian comparison}
\bookmark[level=2,dest={supp.outline.029}]{H.1. Left-tail smoothing}
\bookmark[level=2,dest={supp.outline.030}]{H.2. Right-tail smoothing}
\bookmark[level=2,dest={supp.outline.031}]{H.3. Approximation of eigenvalue counts}
\bookmark[level=2,dest={supp.outline.032}]{H.4. Entrywise local law}
\bookmark[level=2,dest={supp.outline.033}]{H.5. Uniformity in the parameters}
\bookmark[level=2,dest={supp.outline.034}]{H.6. Schur complement estimates}
\bookmark[level=2,dest={supp.outline.035}]{H.7. Gaussian coupling}
\bookmark[level=1,dest={supp.outline.036}]{Appendix I. Tail bounds and eigenvalue increments}
\bookmark[level=2,dest={supp.outline.037}]{I.1. One-step eigenvalue increments}
\bookmark[level=2,dest={supp.outline.038}]{I.2. Derivative bound for shifted cutoffs}
\bookmark[level=2,dest={supp.outline.039}]{I.3. Right-tail bound for tall matrices}
\bookmark[level=2,dest={supp.outline.040}]{I.4. Right-tail bound}
\bookmark[level=2,dest={supp.outline.041}]{I.5. Left-tail bound}
\bookmark[level=2,dest={supp.outline.042}]{I.6. Control between consecutive grid points}
\bookmark[level=1,dest={supp.outline.043}]{Appendix J. Cutoff constructions}
\bookmark[level=2,dest={supp.outline.044}]{J.1. Gaussian right-tail counts}
\bookmark[level=2,dest={supp.outline.045}]{J.2. Smooth cutoff functions for the right tail}
\bookmark[level=2,dest={supp.outline.046}]{J.3. Change of normalization}
\bookmark[level=2,dest={supp.outline.047}]{J.4. Approximation of Gaussian left-tail events}
\bookmark[level=2,dest={supp.outline.048}]{J.5. Gaussian left-tail counts}
\bookmark[level=2,dest={supp.outline.049}]{J.6. Smooth cutoff functions for the left tail}
\bookmark[level=1,dest={supp.outline.050}]{Appendix K. Dyadic tail bounds}
\bookmark[level=2,dest={supp.outline.051}]{K.1. Right-tail bound}
\bookmark[level=2,dest={supp.outline.052}]{K.2. Left-tail bound}
\bookmark[level=1,dest={supp.outline.053}]{References}


\section*{Supplementary material}

Technical supplement to

# A Law of Fractional Logarithm for Nested Complex Sample Covariance Matrices

Xiufan Yang

This supplement contains the supporting normalization identities, resolvent estimates, comparison arguments and probability estimates used in the main article. References to results in the main article link to that text. Sections and formulas within this supplement are numbered separately.

Contents

Contents (continued)

## Appendix A. Normalization identities

The singular-value normalization, its covariance form, and the finite threshold shift are given with their common proof in Lemma 2.1 of the main article. The proof uses $\lambda_{\max}(XX^*) = \|X\|_{\mathrm{op}}^2$ and $\sigma_{+,N} \asymp N^{-2/3}$.

## Appendix B. Standard resolvent and statistic differentiation

### B.1. Edge-resolvent differentiation.

*Proof of Lemma 2.2.* At $z = b_\gamma$, the factor $z - b_\gamma$ vanishes. Moreover,

$$b_\gamma + \gamma - 1 = 1 + 2\sqrt{\gamma} + \gamma + \gamma - 1 = 2\sqrt{\gamma}(1 + \sqrt{\gamma}).$$

Consequently,

$$g_\gamma(b_\gamma) = \frac{2\sqrt{\gamma}(1+\sqrt{\gamma})}{2\gamma(1+\sqrt{\gamma})^2} = \frac{1}{\sqrt{\gamma}(1+\sqrt{\gamma})}.$$

It follows that

$$\gamma g_\gamma(b_\gamma) = \frac{\sqrt{\gamma}}{1+\sqrt{\gamma}}$$

and

$$1 - \gamma g_\gamma(b_\gamma) = \frac{1}{1+\sqrt{\gamma}}.$$

This number is positive, and its reciprocal is $1 + \sqrt{\gamma}$.

Set $\delta = \gamma^{-1} > 0$. The identity just proved yields

$$\delta g_\delta(b_\delta) = \frac{\sqrt{\delta}}{1+\sqrt{\delta}}.$$

Substitution gives

$$\gamma^{-1} g_{\gamma^{-1}}(b_{\gamma^{-1}}) = \frac{\gamma^{-1/2}}{1+\gamma^{-1/2}} = \frac{1}{1+\sqrt{\gamma}}.$$

Therefore

$$1 - \gamma^{-1} g_{\gamma^{-1}}(b_{\gamma^{-1}}) = \frac{\sqrt{\gamma}}{1+\sqrt{\gamma}},$$

whose reciprocal is

$$\frac{1+\sqrt{\gamma}}{\sqrt{\gamma}} = 1 + \gamma^{-1/2}.$$

Finally,

$$A(m,n) = m + n + 2\sqrt{mn}.$$

Differentiation with the other variable fixed gives

$$\partial_n A(m,n) = 1 + \frac{\sqrt{m}}{\sqrt{n}} = 1 + \sqrt{\gamma}$$

and

$$\partial_m A(m,n) = 1 + \frac{\sqrt{n}}{\sqrt{m}} = 1 + \gamma^{-1/2}.$$

□

### B.2. Covariance direction derivatives.

The endpoint trace identity and the estimates for bounded complex-entry directions are proved in Lemma 3.4 of the main article. The proof uses the inverse at the spectral parameters under consideration and applies to rectangular matrices of arbitrary rank.

### B.3. Coordinate derivatives of covariance observables.

The real- and imaginary-coordinate formulas, together with their first- and second-derivative bounds, are the coordinate cases of the same Lemma 3.4.

**B.4. Norm of the linearized resolvent.**

**Lemma B.1.** *Fix $0 < a < b < \infty$. For any rectangular complex matrix $X$, any $a \leq \Re z \leq b$, and $0 < \eta = \Im z \leq 1$, put*

$$H_X(z) = \begin{pmatrix} -zI & X^* \\ X & -I \end{pmatrix}, \qquad G_X(z) = H_X(z)^{-1}.$$

*There is a constant $C_{a,b}$, independent of $X$ and its dimensions, such that*

$$\|G_X(z)\|_{\mathrm{op}} \leq C_{a,b}\eta^{-1}.$$

*Proof.* Apply a singular-value decomposition of $X$. Each nontrivial singular-value block of $G_X$ is

$$\frac{1}{u^2 - z}\begin{pmatrix} 1 & u \\ u & z \end{pmatrix}, \qquad u \geq 0,$$

and the remaining blocks are $-1/z$ or $-1$. If $u^2 \leq 2(b+2)$, the numerator has uniformly bounded norm and $|u^2 - z| \geq \eta$. If $u^2 > 2(b+2)$, then $|u^2 - z| \geq u^2 - b \geq u^2/2$, so the block norm is bounded by a constant. The unmatched blocks are bounded by $\max(1, a^{-1})$. Since $\eta \leq 1$, these estimates give the claim. □

## Appendix C. Row and column expansions

C.1. **Scaling and Gaussian interpolation.**

*Proof of Theorem 4.7.* For the first scaling, put

$$D = \mathrm{diag}(a_k I_{n_k}, I_{m_k}).$$

Direct multiplication gives

$$\widetilde{H}_k(a_k^2 z) = DH_k(z)D.$$

Therefore

$$G_k(z) = D\widetilde{G}_k(a_k^2 z)D,$$

which proves (66). The change of variables $\widetilde{x} = a_k^2 x$ in (64), together with the column-block factor $a_k^2$, proves statistic invariance. The singular values scale by $a_k$, and

$$\widetilde{g}_k(a_k^2 u) = g_k(u).$$

The chain rule gives (68), since $a_k/\sqrt{N} = 1/\sqrt{n_k}$. The deterministic background scales by the same block conjugation.

In the natural normalization, every matrix is a frozen-northwest partial flow of the form in Theorem 4.5. Restrictions, conjugation, and permutations preserve all entry hypotheses. Equation (62) gives the required edge domain. Moreover,

$$\frac{\widetilde{\eta}_k}{\eta_{0,k}} = \left(\frac{p_k}{n_k}\right)^{1/3-12\kappa_0} = 1 + O(P_N/N),$$

which proves (69). Resolvent continuity between the two comparable heights gives an error

$$O\left(\frac{P_N}{N\eta_{0,k}}\right) = O\left(P_N N^{-1/3+12\kappa_0}\right) = O(N^{-5/16}).$$

The same estimate applies to the deterministic backgrounds.

On $|Q_e| \leq N^{1/16}$, the natural inserted entry has size at most $CN^{-7/16}$. The rank-two Woodbury formula based at the deleted matrix has a uniformly bounded two-by-two inverse. It changes cross entries by

$$O(N^{-7/16}) + O(N^{-7/8}) = O(N^{-5/16}),$$

and changes imaginary or centered diagonal parts by

$$O(N^{-7/16}N^{-5/16} + N^{-7/8}) = O(N^{-5/16}).$$

Thus the endpoint estimates hold along the entire insertion segment. Repeated inverse differentiation gives, for every bounded entry direction $S$ and every fixed $r \geq 1$,

$$\partial_\theta^r Y(0) = \frac{(-1)^r(r-1)!}{\pi} \Delta \Im \operatorname{Tr}((SG)^r). \tag{13}$$

The endpoint bounds therefore give every directional derivative through order six as $O(N^{-5/16})$. Finite-dimensional polarization gives every mixed derivative. Multiplication by $O(\log N)$ yields $O(N^{-3/10})$.

At the Gaussian reference point, restrictions, conjugation, and permutations preserve circular Gaussian laws and independence of disjoint blocks. The second scaling in (71) makes every fresh entry have variance $p_k^{-1}$. The scaling argument already proved applies again with $b_k$.

The raw past contains at most $P_N^2$ entries. For a fixed large integer $L$,

$$\mathbb{P}\left(\max_{e\in\mathcal{E}_{\text{past}}} |x_e| > N^{1/16}\right) \leq C_L P_N^2 N^{-L/16}.$$

Choose $L$ so that this is summable on dyadic subsequences. On the complementary event,

$$\|A_k^{\text{raw}}\|_{\text{op}} \leq \|A_k^{\text{raw}}\|_F \leq P_N N^{1/16} \leq p_k^{3/16},$$

which proves (72).

Apply Theorem 4.5 to the Gaussian core with no frozen block at the endpoints in (63). Reinsert one deleted core coordinate by Woodbury. The normalized trace changes are at most

$$\frac{C}{p_k\overline{\eta}_k} = Cp_k^{-1/3+12\kappa_0} = O(p_k^{-5/16}).$$

If $b = \sqrt{q_k/p_k}$, then the quadratic formula gives

$$|m(z) + \frac{1}{1+b}| \leq C|z-(1+b)^2|^{1/2} = O(p_k^{-5/16}).$$

The row background differs from $-(1+b)/b$ by the same order. After the row trace is normalized by $p_k$, its edge background is $-b(1+b)$. Hence the core satisfies (49).

A direct Gaussian net gives $\|\overline{D}_k\|_{\mathrm{op}} \le C$ with exponentially high probability. Proposition 4.6 and (56) then supply all remaining norm hypotheses and the strip estimates. The exact Schur identities are (73)–(74). Since $b = \sqrt{q_k/p_k} \ge 1$,

$$\Re s_{\mathrm{col}} \le -(1+b)^2 + \beta^2 b(1+b) + o(1) \le -(1+b) + o(1),$$

and

$$s_{\mathrm{row}} \le -1 + \frac{1}{1+b} = -\frac{b}{1+b} \le -\frac{1}{2}.$$

Thus no boundary amplitude is omitted.

When a cutoff sum cutoff derivative is nonzero, its argument lies in $(-1, 0)$, so $Z_k < 1/16$. The fixed-window statistic implication gives $\|X_k\|_{\mathrm{op}}^2 < \widehat{T}_k/N \le B_k$. Hence $\operatorname{Tr} g_k(X_k^* X_k)$ vanishes in a neighborhood and contributes no first or second insertion derivative.

For the heat expansion, the strip sets are disjoint and

$$|\mathcal{E}_1 \cup \mathcal{E}_2| \le |\mathcal{R}_{\mathrm{past}}||\mathcal{C}_N| + |\mathcal{R}_N||\mathcal{C}_{\mathrm{past}}| \le C_{\mathrm{amb}} P_N N.$$

A possible conjugate transpose merely interchanges the two strip amplitudes.

For one circular Gaussian raw coordinate $Q_1 + \mathrm{i}Q_2$, Gaussian integration by parts gives, for $0 < v < 1$,

$$\begin{aligned}\partial_v \mathbb{E}[F(h(v), \tau) \mid \mathcal{G}] &= \sum_{a=1}^{2} \mathbb{E}\left[\partial_{h_a} F \, \frac{Q_a}{2\sqrt{vN}} \,\middle|\, \mathcal{G}\right] \\ &= \frac{1}{4N} \sum_{a=1}^{2} \mathbb{E}[\partial_{h_a}^2 F \mid \mathcal{G}].\end{aligned}$$

Summing over distinct labels proves (80). Since

$$\partial_{h_a}^2 F = F(-\tau S_{aa} + \tau^2 S_a^2),$$

and

$$S_{aa} = \sum_{k,r} \lambda_r (\rho' y_{k,a,a} + \rho'' y_{k,a}^2), \qquad S_a = \sum_{k,r} \lambda_r \rho' y_{k,a},$$

we obtain (81). The heat generator has order two, so every third-order or cubic row has coefficient zero.

Set $B = N^{-3/10}$. Only cutoff active matrices contribute. For such a matrix,

$$|y_{k,a}| + |y_{k,a,a}| \le CB.$$

If a cutoff derivative is nonzero at level $r < J_N$, the cutoff at level $r+1$ equals one and

$$\lambda_r \le R_{\mathrm{c}} \lambda_{r+1} \rho(y_k + J_N - (r+1)).$$

The final layer is charged to $|\mathcal{I}_N| \lambda_{J_N} \le N^{-12}$. Therefore

$$\begin{aligned}|S_a| &\le CBS + CN^{-12}B, \\ |S_{aa}| &\le C(B+B^2)S + CN^{-12}(B+B^2).\end{aligned}$$

For $u \ge 0$ and $0 < \tau \le \tau_*$,

$$\begin{aligned}\tau u \mathrm{e}^{-\tau u} &\le 2(\tau/2) u \mathrm{e}^{-(\tau/2)u}, \\ \tau^2 u^2 \mathrm{e}^{-\tau u} &\le C_{\tau_*} (\tau/2) u \mathrm{e}^{-(\tau/2)u}.\end{aligned}$$

Substitution into (80) and (81) bounds the good-event contribution of each label by

$$\frac{C}{N}(B+B^2)\mathcal{D}_j(v,\tau/2)+\frac{C_{\tau_*}}{N}N^{-12}.$$

There are at most $C_{\mathrm{amb}}P_N N$ labels. Thus

$$a_j(v)\le CP_N(N^{-3/10}+N^{-3/5}).$$

Choosing all failure exponents after the polynomial unions makes the bad-event and final-layer contributions at most $C_{\tau_*}N^{-9}$. Integrating over paths of length one proves (83)–(84). □

## C.2. Joint cumulant expansion

We use the array, endpoint, and cutoff assumptions of Theorem 4.8. Throughout this section, $\mathcal{G}=\mathcal{F}_{\mathrm{old}}$ and every conditional expectation contains the complete spectral coefficient and all its index sums.

Spectral coefficients. Write $q=(q_i)_{i\in I}$ for the formal spectral variables of one cutoff sum. At evaluation, $q_i$ is the specified bounded deterministic affine combination of the integrated statistics, together with the specified smooth upper-spectral term when present. The affine constant has zero raw-coordinate derivatives; each bounded slope stays in the coefficients. With $y_i=(q_i-a_i)/\delta$, set

$$S(q)=\sum_{i\in I}\sum_{r=0}^{J}R_{\mathrm{c}}^{-r}\rho(y_i+J-r),\qquad V(q)=e^{-\tau S(q)},$$

$$T^{(p)}_{i_0,r_0;i_1,\dots,i_p}(q)=\partial_{q_{i_1}}\cdots\partial_{q_{i_p}}\left[\tau V(q)R_{\mathrm{c}}^{-r_0}\delta^{-1}\rho'(y_{i_0}+J-r_0)\right],\qquad p\ge 0.$$

The total spectral order of this tensor is $p+1$. Lemma 3.3 gives

$$\sum_{i_0,r_0,i_1,\dots,i_p}|T^{(p)}_{i_0,r_0;i_1,\dots,i_p}(q)|\le C_p\delta^{-(p+1)}=N^{o(1)}. \tag{C.1}$$

Two exact identities describe differentiation:

$$\partial_{q_j}V=-\sum_{r=0}^{J}T^{(0)}_{j,r},\qquad \partial_{q_j}T^{(p)}_{i_0,r_0;i_1,\dots,i_p}=T^{(p+1)}_{i_0,r_0;i_1,\dots,i_p,j}. \tag{C.2}$$

For two cutoff sums, their formal coordinate sets are disjoint. A coefficient is a product with one factor for each copy: either its $V$, or one tensor $T^{(p)}$. The Leibniz rule and (C.2) generate all subsequent coefficients. Thus there is exactly one Laplace factor from each copy. The tensor norm is bounded by the product of (C.1) and $|V|\le 1$. Evaluation on matrices with shared raw coordinates takes place after these formal derivatives. For the fixed-size comparisons, the coefficient family consists instead of $f^{(p)}(q)$, where $f$ is one of the fixed smooth scalar cutoffs and $p\le K_{\mathrm{der}}$.

In (91), $\mathcal{M}$ is one of these coefficient tensors, with all its spectral and distinguished cutoff slots summed. Finite sums of Leibniz allocations are kept as finite sums of such terms. The support of a differentiated cutoff stays in its tensor. Matrix membership and endpoint choices are encoded by the mask $\chi$; it and $C$ are $\mathcal{G}$-measurable for each discrete index tuple and have zero moving-coordinate derivatives. The absolute coefficient normalization is the one in the theorem. Every additional matrix index is either fixed or is a slot of $\mathcal{M}$.

Joint differentiation. Let $e=(a,b)$ be a moving raw coordinate. A real-coordinate derivative of a coefficient is the exact contraction

$$\partial_{e,\omega}\mathcal{M}_{\boldsymbol{i},\boldsymbol{r}}(Y)=\sum_j(\partial_{q_j}\mathcal{M}_{\boldsymbol{i},\boldsymbol{r}})(Y)\partial_{e,\omega}Y_j,\qquad \omega\in\{1,2\}. \tag{C.3}$$

For a single integrated statistic in matrix $j$, put $S_{e,1}=E_{ab}+E_{ba}$ and $S_{e,2}=\mathrm{i}E_{ab}-\mathrm{i}E_{ba}$ in its physical coordinates. Then

$$\partial_{e,\omega} Y_j = -\frac{n_j^{-1/2}}{\pi} \Delta \operatorname{Im} \operatorname{Tr}(S_{e,\omega} G^j).$$

The membership mask is zero when $e$ is absent from matrix $j$. Higher derivatives are the endpoint formula (89) with one factor $n_j^{-1/2}$ for each raw derivative and the bounded affine coefficient of the statistic. A statistic may contain a fixed number of such integrals; a uniformly bounded absolute sum of their deterministic coefficients gives the same estimates. Each endpoint and conjugation choice produces only a fixed finite number of terms.

Equation (C.3) appends a spectral slot, and the derivative of $Y_j$ creates the corresponding endpoint Green factors. In particular, $\sum_j$ stays inside the same conditional expectation. Iterating this rule gives ordered partitions of the raw derivatives among the spectral slots and the explicit Green factors. The number of partitions depends on the derivative order. The values of the spectral and cutoff indices are summed by (C.1), at the stage where a pointwise estimate is made.

In (C.3), a derivative of $Y_j$ is split into its integral derivative and its upper-spectral derivative. A term containing a positive derivative of an upper-spectral part is retained as an exceptional term; subsequent product rules preserve at least one such differentiated factor. The upper-spectral cutoff terms in the statistics of Sections 5–6 have vanishing derivatives on the common rigidity event, up to every fixed order used here. Their derivatives elsewhere have polynomial bounds: the divided-difference formula bounds a derivative of order $r$ of $\operatorname{Tr} g(L)$ by a fixed dimension power times $\max_{j \le r} \|g^{(j)}\|_\infty$ and the matrix-direction norms. They therefore contribute to the exceptional-event estimate below. Deterministic zero-mode corrections have zero raw derivatives.

Joint term classes. The summation in (91) is first resolved into index patterns. Indices of different physical types remain distinct. An identification between free indices of the same type, or with a fixed index of that type, removes one free sum and retains its factor $N^{-1}$ in $N^{-\ell}$. Natural averages contribute the bounded ratios $N/n_j$ or $N/m_j$. A centered partial trace can be kept as one centered factor until it is differentiated; it is then written as its normalized diagonal sum. Its trace index occurs twice, so opening it preserves every existing parity. Coincidences with source indices are resolved by the same rule.

The unmarked class consists of these contractions with a fixed global coefficient family and at least one unrevealed unmatched index. To generate a marked term, apply one third-cumulant branch at a selected unrevealed index, retain the full ordered real triple and coordinate cumulant, and factor its normalization as in the theorem. Retain the bounded dimension ratios and cumulant in $C$, and record the factor $N^{-1/2}$ externally. Further covariance expansions retain this mark. A further third-cumulant branch is estimated as a remainder. Products whose initial cubic differentiation lowers degree are evaluated by (103)–(104) before this mark is assigned. These operations define membership without dividing by a cumulant, including when that cumulant vanishes.

**Lemma C.1 (Joint expansion of unmatched terms).**

*For the joint classes just defined, a covariance expansion at an unrevealed unmatched index has a bounded invertible self coefficient. Every other same-degree branch reduces the multiplicity of that same index by two. After these branches terminate, the joint conditional expectation is a finite sum of higher-degree joint terms, together with first-cubic marked terms multiplied by* $N^{-1/2}$ *and an* $O(N^{-1+\epsilon})$ *remainder. For a marked term the cubic remainder is* $O(N^{-1/2+\epsilon})$*. The output budgets can be taken as* $F_+ = 4F + 10$ *and* $k_+ = k + 4F + 20$*. The constants depend on the fixed budgets and coefficient-family bounds, and are independent of the numbers of matrices and cutoff levels. Consequently the joint suprema satisfy (100)–(101).*

*Proof.* Fix the budgets for the whole finite iteration, and then the moment and probability exponents as specified below. All algebraic identities are applied to full conditional expectations before exceptional events are separated.
Inverse identities. For a row index $a$ and a column index $b$, the four identities are

$$G^k_{ay} = n_k^{-1/2} \sum_{c\in\mathcal{C}_k} \zeta_{ac} G^k_{cy} - \delta_{ay}, \tag{14}$$

$$G^k_{ya} = n_k^{-1/2} \sum_{c\in\mathcal{C}_k} G^k_{yc} \overline{\zeta_{ac}} - \delta_{ya}, \tag{15}$$

$$zG^k_{by} = n_k^{-1/2} \sum_{a\in\mathcal{R}_k} \overline{\zeta_{ab}} G^k_{ay} - \delta_{by}, \tag{16}$$

$$zG^k_{yb} = n_k^{-1/2} \sum_{a\in\mathcal{R}_k} G^k_{ya} \zeta_{ab} - \delta_{yb}. \tag{17}$$

Conjugation gives the identities for $G^-$. For an entry of a matrix with column normalization $n^{-1/2}$, Wirtinger differentiation gives

$$\begin{aligned} D_\zeta G^+_{uv} &= -n^{-1/2} G^+_{ua} G^+_{bv}, \quad D_{\bar\zeta} G^+_{uv} = -n^{-1/2} G^+_{ub} G^+_{av}, \\ D_\zeta G^-_{uv} &= -n^{-1/2} G^-_{ub} G^-_{av}, \quad D_{\bar\zeta} G^-_{uv} = -n^{-1/2} G^-_{ua} G^-_{bv}. \end{aligned} \tag{18}$$

For each fixed value of the discrete indices, use the real two-component cumulant formula on the complete product $\mathcal{M}\prod Z_u$. Writing $q = (\operatorname{Re}\zeta, \operatorname{Im}\zeta)$,

$$\begin{aligned} \mathbb{E}[q_a B(q) \mid \mathcal{G}] = \frac{1}{2}\mathbb{E}[\partial_a B(q) \mid \mathcal{G}] + \frac{1}{2}\sum_{b,c=1}^{2} K_{e,abc} \mathbb{E}[\partial_b \partial_c B(q) \mid \mathcal{G}] \\ + \frac{1}{6}\sum_{b,c,d=1}^{2} K_{e,abcd} \mathbb{E}[\partial_b\partial_c\partial_d B(q) \mid \mathcal{G}] + \mathcal{E}_e. \end{aligned}$$

The remainder is bounded by moments of $q$ and fourth derivatives along the coordinate-insertion segment. The coordinates of $q$ have covariance $I_2/2$; independence is used between different complex coordinates. Sum the resulting identities over the spectral and free indices with their original masks and coefficients.
Covariance branches. Choose an unrevealed index $a$ of odd multiplicity and an off-diagonal factor containing it. Such a factor exists because diagonal factors contribute two occurrences. Consider first the row-side factor $G^k_{ay}$. The self hit in (14) produces $-n_k^{-1}\sum_c G^k_{cc} G^k_{ay}$. Its deterministic trace part moves to the left with coefficient $1+\widetilde{m}^k$; the centered trace part has one higher degree. With $\rho_k = m_k/n_k$, the column identity has coefficient $z - \rho_k/(1+\widetilde{m}^k) = -1/\widetilde{m}^k$. For variable matrix and endpoint slots, divide each component by its own deterministic self coefficient before summing. The inverse coefficient stays in $C$ and is uniformly bounded. The self trace is completed over the whole matrix column or row set; the finitely many excluded index coincidences retain their factor $N^{-1}$ and are estimated with the collision terms.

Every remaining covariance derivative hits exactly one of the following objects. If it hits another off-diagonal Green factor containing the selected index in the slot prescribed by (18), the leading source diagonal can be replaced by its deterministic background. This removes the selected occurrence and the hit occurrence of $a$. A hit on a centered diagonal at $a$ obeys the same rule: the derivative removes its centering and leaves one off-diagonal factor containing $a$, so the total multiplicity again decreases by two. The centered remainder of the source diagonal has higher degree. A hit on an ordinary diagonal, on a Green factor with no eligible occurrence, or on a centered diagonal at another index introduces an additional off-diagonal or centered factor.

A hit on a centered trace has this form after its trace sum is opened, apart from the separately counted collisions. Finally, a hit on $\mathcal{M}$ obeys (C.3); its first endpoint derivative is off-diagonal and raises the degree by one.

For selected slot $s_0$ and conjugation signs $\varsigma_0, \varsigma_1$, (18) gives $s_1 = -s_0\varsigma_0\varsigma_1$. The two coefficient formulas are (96)–(97), with their separate conjugation signs and dimension ratio. The new opposite-type index occurs twice in each same-degree branch. Thus the selected index stays odd. Keep that same index throughout the same-degree subtree. Its positive odd multiplicity $\mu$ is replaced by $\mu - 2$ at each such step; when $\mu = 1$ there is no eligible second occurrence. Initially $\mu \leq 2F - 1$, so there are at most $F - 1$ such steps on a branch. The number of branches and their total absolute coefficient weight are bounded in terms of $F$ and the fixed derivative orders. This finite expansion bounds the total coefficient weight using the uniform bound for each neutral coefficient.

Third cumulants. A first third-cumulant term applies two derivatives to the transported product. Each derivative falls either on an explicit Green or trace factor, or on the spectral tensor. If it falls on the tensor, (C.3) and its second derivative give respectively one appended spectral slot with a second endpoint derivative, or two appended slots with two first endpoint derivatives. These are all ordered allocations. Their raw factors are $n_k^{-1/2} n_h^{-1/2} n_j^{-1/2}$, including when two derivatives act on the same matrix. The new opposite-type index has odd multiplicity; its coincident and old values are separated before marking the remaining terms.

A derivative of an off-diagonal factor preserves at least one off-diagonal occurrence away from coincidences. A single derivative of a centered diagonal makes an off-diagonal factor. Two derivatives on that same centered diagonal have the pure and mixed forms in (103) and its subsequent displayed identities. The uncentered term in the mixed derivative is the one degree-lowering branch. At initial degree two it has exactly one off-diagonal factor, and (104) gives its complete joint expectation bound. In that argument, a derivative on $\mathcal{M}$ is precisely (C.3), so the new off-diagonal endpoint and the appended spectral index stay in one contraction. If the new index is old, there are at most $P_N$ summands and the bound is $P_N N^{-3/2+\epsilon}$. At degree at least three the bound $N^{-1/2+\epsilon}\Psi_N^{d-1} = O(N^{-1+\epsilon'})$ suffices. A centered trace loses all off-diagonal factors only when its normalized trace index coincides with a source index, contributing an additional $N^{-1}$. These cases exhaust two-derivative allocations.

The other first-cubic terms retain the required degree. Factor out $N^{-1/2}$ and store their bounded normalization ratios and indexed cumulant in the marked coefficient. All eight ordered real triples are retained. For a marked term, a second cubic can be estimated pointwise: its new sum and three raw factors give $N^{-1/2}$, while (C.1) sums all spectral indices. The original exterior factor then gives $O(N^{-1+\epsilon})$ in the unmarked expectation.

Higher orders and exceptional events. For a cumulant of order $r \geq 4$, the selected resolvent identity and $r - 1$ derivatives contribute $N^{-r/2}$ and one new free sum. On the local-law event its complete contraction is bounded by $CN^{1-r/2+\epsilon}$ using (C.1), so $r = 4$ already gives $O(N^{-1+\epsilon})$. The fourth-derivative remainder above has $N^{-5/2}$ and one new sum and is smaller on the insertion event. An index collision has a factor $N^{-1}$, and a normalized sum restricted to an old set has $P_N/N = N^{-1+o(1)}$. They are estimated with bounded Green entries and the same tensor norm. The tensor and its Green factors remain inside the same conditional expectation throughout these estimates.

The insertion estimates also hold with a fixed finite number of coordinate deletions. Indeed, their linearized perturbation is $UEU^*$ with a fixed number of columns and $\|E\| = O(N^{-1/2+\varepsilon})$. The relevant entries of $U^*GU$ are bounded on the original event. Woodbury therefore gives a uniformly bounded inverse of $I + EU^*GU$ and an entrywise difference $O(N^{-1/2+\varepsilon}) = o(\Psi_N)$. This includes the extra deleted coordinate in a Taylor remainder about an already deleted matrix. Normalized diagonal sums inherit the same estimate.

Choose the finite moment orders after fixing all derivative budgets. The norm bound of Supplementary Lemma B.1 and (C.1) give a finite polynomial $L^2$ bound $N^{B_*}$ for every final expression, including moment factors in the remainders. Choose the conditional failure exponent $D > 2(B_* + A_* + 10)$, where $N^{A_*}$ bounds any remaining external family or source sum. Choose the additional truncation moments to give unconditional failure $O(N^{-D-6})$. Include in $\Omega_N(D)$ the bound $\max_{e\in\mathcal{P}_{\mathrm{old}}} |\xi_e| \le N^{\varepsilon}$, whose failure has any prescribed polynomial order by the uniform moments. On a history with positive conditional probability of $\Omega_N(D)$ this old-entry bound is deterministic; together with the conditional moving-coordinate moments it gives the conditional $L^2$ bound $CN^{B_*}$ used below. Obtain $\Omega_N(D)$ from Theorem 4.5 and the just described deterministic insertion argument. Conditional Markov gives

$$\mathcal{H}_N(D) = \{\mathbb{P}(\Omega_N(D)^c \mid \mathcal{G}) \le N^{-D}\}, \qquad \mathbb{P}(\mathcal{H}_N(D)^c) \le C_D N^{-6}.$$

Conditional Cauchy–Schwarz bounds every exceptional contribution by $CN^{B_*}N^{-D/2}$ on this same history. The event indicator is introduced at this final estimation stage. Thus the exact cumulant identities use the original smooth global observable.

Each same-degree step retains the Green-factor count. At an exit, centering a new diagonal, opening a trace, or differentiating at most four times adds a fixed number of factors. The additional expansion for the degree-two cubic branch uses at most $F + 5$ factors and spectral order $k + 5$. Resolving the complete same-degree subtree and these exits is consequently covered by $F_+ = 4F + 10$, $k_+ = k + 4F + 20$. Combining its finitely many identities, with their bounded coefficients, proves (100)–(101) for the joint conditional suprema. ☐

Finite iteration and the original source. Set $F_{j+1} = 4F_j + 10$ and $k_{j+1} = k_j + 4F_j + 20$. Explicitly,

$$F_j = 4^j(F_0 + 10/3) - 10/3, \qquad k_j = k_0 + \frac{4}{3}(F_0 + 10/3)(4^j - 1) + \frac{20}{3}j.$$

Take $K_{\mathrm{der}} \ge k_4 + 6$, then choose the moments and $D$ as in the proof. A degree-increasing step and a change from $h = 0$ to $h = 1$ each use the next budget. There are at most three degree increases from degree one to degree four, and at most one change of mark. Thus every branch has been covered when it reaches the fourth budget. At degree four, pointwise tensor summation followed by conditional expectation gives $O(N^{\epsilon}\Psi_N^4) + O(N^{-10})$. Backward substitution first in the marked and then in the unmarked identities yields

$$\mathcal{R}_{d,F_0,k_0,1} \le CN^{-1/2+\epsilon}, \qquad \mathcal{R}_{d,F_0,k_0,0} \le CN^{-1+\epsilon}, \qquad d \ge 1.$$

Here $\Psi_N^4 = N^{-4/3+4\nu} = o(N^{-1})$.

To initialize, differentiate the fixed global observable three times. Formula (C.2), the product rule, and the endpoint identity give exactly three raw source insertions in every term. The source row and source column each have three displayed occurrences. At least one is unrevealed, and it stays odd after any deterministic diagonal backgrounds are removed. Each resulting term therefore has an unrevealed unmatched index and degree at least one. Its spectral indices are slots of the tensors already defined. Three insertions use $F_0 = k_0 = 3$; the source sum has the form $N^{1/2}$ times a normalized two-index joint contraction with bounded coefficients. Consequently

$$|\mathcal{L}_3^{\mathrm{mov}}| \le CN^{-1/2+\epsilon}.$$

For these starting budgets, $F_4 = 1618$, $k_4 = 2183$, so $K_{\mathrm{der}} = 2189$ is one fixed sufficient choice. The cutoff is $C^\infty$ and all fixed moments are available. Larger initial families use the same four-step formula with their own fixed budgets.

Finally, the Ornstein–Uhlenbeck generator in the two real coordinates is $\frac{1}{4}(\partial_1^2 + \partial_2^2) - \frac{1}{2}(q_1\partial_1 + q_2\partial_2)$. The covariance term cancels the diffusion term. The ordered cubic coefficient is $-\frac{1}{2} \cdot \frac{1}{2!} = -\frac{1}{4}$, which gives (102) with its full joint observable.

## Appendix D. A maximal moment estimate for the exposed increments

We prove the maximal-moment bound in Proposition 7.5, with the notation used in its proof. Fix an even integer $p \geq 2$, and write

$$\xi_t = c_t d_t.$$

The conditional-moment bound established immediately before the definition of $M_j$ in that proof, together with boundedness of $c_t$, gives constants $C_{\mathrm{mom}}(r)$, depending only on $r$, such that

$$\mathbb{E}(|\xi_t|^r \mid \mathcal{F}_t) \leq C_{\mathrm{mom}}(r), \qquad 2 \leq r \leq p,$$

uniformly in $t$ and $k$. Put

$$a_h^{(\mathrm{mom})} = \mathbb{E}|M_h|^p.$$

The variable $M_{h-1}$ is $\mathcal{F}_h$-measurable because $\mathcal{F}_h$ is the information available before update $h$, and

$$\mathbb{E}(\xi_h \mid \mathcal{F}_h) = 0.$$

Conditional binomial expansion and the conditional moment bounds give

$$\mathbb{E}(|M_h|^p \mid \mathcal{F}_h) \leq |M_{h-1}|^p + C_{\mathrm{mom}}(p)\left(1 + |M_{h-1}|^{p-2}\right). \tag{19}$$

Indeed, the linear binomial term vanishes. Every term of degree $\ell$ with $2 \leq \ell \leq p$ is bounded by

$$C_{\mathrm{mom}}(p)|M_{h-1}|^{p-\ell} \leq C_{\mathrm{mom}}(p)\left(1 + |M_{h-1}|^{p-2}\right).$$

For $p \geq 4$, taking expectations in (19) and using Lyapunov's inequality gives

$$a_h^{(\mathrm{mom})} \leq a_{h-1}^{(\mathrm{mom})} + C_{\mathrm{mom}}(p)\left(1 + (a_{h-1}^{(\mathrm{mom})})^{1-2/p}\right).$$

For $p = 2$, the martingale-difference identity gives

$$a_h^{(\mathrm{mom})} = \sum_{t=1}^{h} \mathbb{E}|\xi_t|^2 \leq C_{\mathrm{mom}}(2)h.$$

For $p \geq 4$, choose $D_{\mathrm{mom}}(p)$ so large that

$$C_{\mathrm{mom}}(p)\left(1 + D_{\mathrm{mom}}(p)^{1-2/p}\right) \leq D_{\mathrm{mom}}(p).$$

Using

$$h^{p/2} - (h-1)^{p/2} \geq h^{p/2-1},$$

induction on $h$ gives

$$a_h^{(\mathrm{mom})} \leq D_{\mathrm{mom}}(p) h^{p/2} \tag{20}$$

for every $h \geq 1$.

To pass from the terminal moment to the maximum, put

$$M_k^{\max} = \max_{0 \leq j \leq H_k} |M_j|.$$

For $\lambda > 0$, let $T_\lambda$ be the first $j \le H_k$ such that

$$|M_j| \ge \lambda$$

on the event $\{M_k^{\max} \ge \lambda\}$. Since $|M_j|$ is a submartingale for its natural after-update filtration and $T_\lambda$ is bounded,

$$\lambda \mathbb{P}(M_k^{\max} \ge \lambda) \le \mathbb{E}\Big[|M_{H_k}| \mathbf{1}_{\{M_k^{\max} \ge \lambda\}}\Big].$$

Multiply by $p\lambda^{p-2}$, integrate over $\lambda > 0$, and apply Hölder's inequality. This gives

$$\mathbb{E}(M_k^{\max})^p \le \frac{p}{p-1}\mathbb{E}\big[|M_{H_k}|(M_k^{\max})^{p-1}\big]$$

and hence

$$\mathbb{E}(M_k^{\max})^p \le \frac{p}{p-1}(\mathbb{E}|M_{H_k}|^p)^{1/p}(\mathbb{E}(M_k^{\max})^p)^{(p-1)/p}.$$

Combining the last inequality with (20) gives (123).

## Appendix E. Right- and left-tail comparison

We verify the normalizations and conditional estimates used in the two occurrence arguments.

**Theorem E.1** (Applications of the unmatched-term recursion)**.** *The recursion of Theorem 4.8 has the following endpoint specializations.*

(1) *Under the hypotheses and notation of Corollary 5.4, reindex the southeast array and use empty old row and column sets. Retain the external conditioning sigma-algebra $\mathcal{G} = \mathcal{F}_l^{\text{past}}$, which is independent of the complete reindexed array. In the physical normalization $v_n^{-1/2}$,*

$$Q_{l,n} = Y_n - Z_{0,n}.$$

*The quantity $Z_{0,n}$ is deterministic and has zero positive moving-coordinate derivatives. The local-law, deletion, insertion-segment and centered-trace estimates hold at the specified endpoints, on the stated good histories, with $\Psi_N = N^{-5/16}$. Both labels of every moving source are unrevealed. The coefficient formulas, two recurrences and exact third source are those of Theorem 4.8, conditionally on $\mathcal{G}$.*

(2) *Under the hypotheses and notation of Corollary 6.4, retain the polylogarithmic product corner and condition on $\mathcal{G} = \mathcal{F}_{\text{old}}$. The physical row–column orientation and the full centered-trace convention are preserved by the common-to-natural normalization, the bounded-ratio height change and the deterministic structural-zero-mode correction. The deletion and insertion-segment estimates hold with*

$$\Psi_N = N^{-1/3+\nu}, \qquad 12\kappa_0 < \nu < 1/12.$$

*Every moving source has at least one unrevealed label. The coefficients, third-cumulant terms and recurrences therefore hold at the specified high- and low-endpoint smoothing heights.*

*The fixed-parameter comparisons in Sections E.2 and E.4 use $\kappa_0 = 1/1000$ and $\nu = 1/48$, for which $\Psi_N = N^{-5/16}$.*

*Proof.* Fix the derivative budgets and a conditional failure exponent $D$ sufficient for all estimates in the chosen application. Use the event $\mathcal{H}_N(D)$ of Theorem 4.8 throughout; in the high case

the conditioning is the independent external past. For a matrix indexed by $(l,n)$, let $p_{l,n}$ be its smaller fresh dimension and $n_k$ its physical column normalization. Its normalized height is

$$\eta'_{l,n} = \frac{p_{l,n}^{1/3-12\kappa_0}}{n_k} = \left(\frac{p_{l,n}}{n_k}\right)^{1/3-12\kappa_0} n_k^{-2/3-12\kappa_0}. \tag{21}$$

All these dimensions are comparable, so the height ratio stays between positive constants. The high half-integer and logarithmic displacements are $O(n_k^{-2/3+\kappa_0})$. For the low family the fresh-core edge differs from the full raw edge by $O(P_l)$, the logarithmic displacement is $O(N_l^{1/3}(\log N_l)^{1/3})$, and the upper endpoint adds at most $4p_{l,n}^{1/3+\kappa_0}$. Since $P_l$ is polylogarithmic, the normalized real endpoints lie in the window of Theorem 4.8.

Apply the height-transfer argument (105)–(106) of the main article at the scale of the corresponding clause. It starts from the local-law event at $\eta_{0,k} = n_k^{-2/3-12\kappa_0}$ and gives $O(\Psi_N)$ changes in Green entries, normalized partial traces and their Marchenko–Pastur backgrounds throughout the bounded-ratio interval. The argument applies to every specified deleted and insertion-segment matrix; the leading coefficients remain invertible and the same-degree coefficients remain bounded.

For the right-tail family, write $A = v^{-1/2}W$ for a physical $u$-by-$v$ matrix. The eigenvalues of $A^*A$ consist of the $\min(u,v)$ squared singular values of $W$ divided by $v$, including zeros with multiplicity, together with $v - \min(u,v)$ structural zeros. Changing variables in the integrated column trace gives $Q_{l,n}(W) = Y_n - Z_{0,n}$. Lemma 5.3 transports this normalization to the common one with bounded matrix speeds. Independence of the complete reindexed array from $\mathcal{F}_l^{\mathrm{past}}$ permits the same conditional estimates while both old label sets are empty.

For the left-tail family, use the two sides of the frozen product corner as the old label sets. A coordinate outside their product has an unrevealed row or column. The same height argument and the deterministic zero-mode subtraction preserve the hypotheses of the common recursion. Applying Theorem 4.8 proves the stated coefficient, source and recurrence conclusions in both cases. □

**Theorem E.2 (Green function comparison).** *Let $\mathcal{G}$ be a conditioning sigma-algebra and let $\mathcal{E}_N^{\mathrm{mov}}$ contain at most $CN^2$ moving complex coordinates. Conditionally on $\mathcal{G}$, suppose the vectors of real and imaginary components are independent across coordinates, centered, have covariance $I_2/2$ and uniformly bounded fixed moments, and follow the Ornstein–Uhlenbeck interpolation on $[0, 8\log N]$. Let $S_N(t)$ be the specified global smooth cutoff observable, put*

$$\Phi_N(t,\tau) = \mathbb{E}[\exp(-\tau S_N(t)) \mid \mathcal{G}],$$

*and use the specified one-coordinate sections $F_e(h;t,\tau)$.*

Fix all derivative and moment orders before choosing the conditional failure exponent. The good histories in the assumptions below are those for this one exponent, as in Supplementary Lemma C.1.

*Assume that the high- or low-endpoint specialization supplies, on $\mathcal{G}$-measurable histories with summable complements:*

1. *the physical/common/natural normalization dictionary with bounded matrix speeds, specified endpoints and heights, the deterministic zero-mode correction, and deleted and insertion-segment local laws in the centered-trace convention at scale $\Psi_N = N^{-5/16}$;*
2. *for every fixed $1 \le r \le 5$,*

$$N^{-2} \sum_{e \in \mathcal{E}_N^{\mathrm{mov}}} \mathbb{E}[\max|\partial^r F_e(h_e;t,\tau)| \mid \mathcal{G}] \le C_{r,\tau_*} N^{-3/10} + N^{-10}; \tag{22}$$

3. *a grid exponent $\theta < 1/3$,*

$$d_{\text{det}} = \max(\theta - 3/10, 0) \le 1/30,$$

*the weighted-cutoff final-layer and transition-width bounds, and enough endpoint closure so that*

$$N^{-2}\sum_e \mathbb{E}\left[(1+|q_e(t)|^6)\max_{\mathbf{a}\in\{1,2\}^6}\sup_{0\le s\le 1}|\partial_{\mathbf{a}}F_e(sh_e;t,\tau)|\middle|\mathcal{G}\right] \le C_{\tau_*}N^{6d_{\text{det}}+1/100}+N^{-10}; \tag{23}$$

(4) *the two unmatched recurrences of the applicable specialization in Theorem E.1, the terminal bound*

$$\mathcal{R}(4,F,k,h) \le CN^{\epsilon}\Psi_N^4 + CN^{-20},$$

*source normalization* $N^{1/2}[N^{-2}\sum_e]$, *and the clause-specific unrevealed-label property; and*

(5) *conditional moment and truncation bounds making every bad-event or final layer contribution at most* $CN^{-9}$.

*Then, uniformly for almost every* $0 \le t \le 8\log N$ *and every* $0 < \tau \le \tau_*$,

$$|\partial_t \Phi_N(t,\tau)| \le e_N(t),$$

*where*

$$e_N(t) \le C_{\tau_*}\left[N^{-3/10} + N^{-1/2+1/100} + N^{-4/5} + N^{-79/100} + N^{-9}\right], \tag{24}$$

*and*

$$\int_0^{8\log N} e_N(t)\,\mathrm{d}t \le C_{\tau_*}N^{-1/4}. \tag{25}$$

*The shifted-transform coefficient is zero.*

*Proof.* Let $K_{e,a_1\cdots a_r}(t)$ denote the conditional joint cumulants of the two raw real coordinates. Differentiating the conditional interpolation and using Gaussian integration by parts cancels the covariance term. The moment-cumulant expansion through order five gives

$$\partial_t\Phi_N = L_3 + L_4 + L_5 + \mathcal{R}_6,$$

with

$$L_3 = -\frac{1}{4N^{3/2}}\sum_e \mathbb{E}\left[\sum_{a,b,c=1}^{2} K_{e,abc}\partial_{abc}F_e\middle|\mathcal{G}\right],$$

$$L_4 = -\frac{1}{12N^2}\sum_e \mathbb{E}\left[\sum_{a,b,c,d=1}^{2} K_{e,abcd}\partial_{abcd}F_e\middle|\mathcal{G}\right],$$

and

$$L_5 = -\frac{1}{48N^{5/2}}\sum_e \mathbb{E}\left[\sum_{\mathbf{a}\in\{1,2\}^5} K_{e,\mathbf{a}}\partial_{\mathbf{a}}F_e\middle|\mathcal{G}\right].$$

Thus the coefficients are exactly $-1/4, -1/12, -1/48$, and all eight ordered index triples are retained. Taylor's integral remainder satisfies

$$|\mathcal{R}_6| \le CN^{-3}\sum_e \mathbb{E}\left[(1+|q_e|^6)\max_{\mathbf{a}}\sup_{0\le s\le 1}|\partial_{\mathbf{a}}F_e(sh_e)|\ \mathcal{G}\right].$$

Open every statistic, product-rule, and cutoff sum index in $L_3$ before taking a good-event bound.

The source row and column each occur three times. The high family has both labels unrevealed, and the low family has at least one. Use the joint contraction classes of Supplementary Lemma C.1. Starting with $F_0 = k_0 = 3$, choose $K_{\text{der}} \ge 2189$ by its four-step budget recursion, then

choose the moment orders and conditional failure exponent. That lemma yields

$$\mathcal{R}_{1,F_0,k_0,0} \le CN^{-1+3\epsilon}.$$

The spectral and cutoff-level sums remain inside the conditional expectation at every expansion. Restoring the outer $N^{1/2}$ source normalization and choosing the finite loss to be at most $1/100$ yields

$$|L_3| \le CN^{-1/2+1/100} + CN^{-9}.$$

Uniform fourth and fifth cumulants and (22) give

$$|L_4| \le CN^{-3/10} + CN^{-9}, \qquad |L_5| \le CN^{-4/5} + CN^{-9}.$$

The sixth-order bound and $d_{\text{det}} \le 1/30$ give

$$|\mathcal{R}_6| \le CN^{-1+6/30+1/100} + CN^{-9} = CN^{-79/100} + CN^{-9}.$$

The conditional moment and truncation package justifies the same bounds on the complement of the simultaneous insertion-segment event. Adding the four estimates proves (24). No shifted count or Laplace derivative has been introduced, so the shifted-transform coefficient is zero. Finally, multiplying the five negative powers by $8\log N$ proves (25) for all sufficiently large $N$. □

**Lemma E.3** (Conditional comparison)**.** *Let $\mathcal{G}$ be a sigma-algebra and let $\mathcal{H}$ be a $\mathcal{G}$-measurable event. Fix a finite set $I$, a number $R > 1$, an integer $J \ge 1$, and a cutoff $\rho \in C^3(\mathbb{R};[0,1])$ such that*

$$\rho = 1 \quad \text{on } (-\infty,-1], \qquad \rho = 0 \quad \text{on } [0,\infty).$$

*For smooth scalar coordinates $y_i$, define*

$$A_{i,r}(W) = R^{-r}\rho\big(y_i(W) + J - r\big), \qquad S((W_i)) = \sum_{i\in I}\sum_{r=0}^{J} A_{i,r}(W_i).$$

*Suppose that a specified target event $\mathcal{L}$ satisfies*

$$y_i(W_i^X) < 0 \quad \Longrightarrow \quad \mathcal{L} \qquad (i \in I),$$

*and that retained indicators $\widetilde{H}_i$ satisfy, on $\mathcal{H}$,*

$$\widetilde{H}_i = 1 \quad \Longrightarrow \quad y_i(W_i^{\text{ref}}) \le -J - 1.$$

*Put $\widetilde{R} = \sum_i \widetilde{H}_i$, and assume that its supplied conditional moment package is*

$$\widetilde{m} = \mathbb{E}[\widetilde{R} \mid \mathcal{G}] \ge m_0, \qquad \mathbb{E}[\widetilde{R}^2 \mid \mathcal{G}] \le C\widetilde{m}^2.$$

*If*

$$|I|R^{-J} \le \varepsilon_{\text{fin}},$$

*then*

$$S((W_i^X)) > 0 \quad \Longrightarrow \quad \mathcal{L}, \qquad S((W_i^{\text{ref}})) \ge \widetilde{R} \quad \text{on } \mathcal{H}.$$

*The displayed conditional moment package is unchanged. Moreover, the first, second, and non-cubic third derivative sums satisfy the support-absorption estimates of Lemma 3.1, and their only unabsorbed cutoff sum contribution carries the factor $\varepsilon_{\text{fin}}$. Higher fixed orders obey the corresponding conclusion whenever the additional smoothness and tensor hypotheses of Lemma 3.3 are supplied.*

*Proof.* If the target cutoff sum is positive, some summand has a negative cutoff argument. Since $r \le J$, this gives $y_i(W_i^X) < 0$ and hence $\mathcal{L}$. On $\mathcal{H}$, each retained hit satisfies $y_i(W_i^{\text{ref}}) + J \le -1$, so its $r = 0$ summand equals one. Summing these contributions proves the terminal inequality. The retained indicators have the conditional moments in the hypotheses. Lemma 3.1 gives the derivative bounds with final-layer factor $|I|R^{-J}$; under the additional hypotheses of Lemma 3.3, the same construction gives the higher-order bounds. □

### E.1. Right-tail specialization.

The right-tail application, with its external conditioning sigma-algebra, normalization, height and zero-mode identities, is proved in Corollary 5.4 of the main article. The first clause of Theorem E.1 records the same estimates for use in the comparison below.

### E.2. Right-tail comparison

*Proof of Proposition 5.5.* Fix $N = N_l$ and $\mathcal{G} = \mathcal{F}_l^{\text{past}}$. Lemma 5.3 realizes the count as one function of the complete reindexed southeast array. This array is independent of $\mathcal{G}$. Its matrices retain their physical row-column indices; a common normalized coordinate has matrix speeds bounded above and below because all column counts are comparable to $N$.

Corollary 5.4 gives the endpoint, height, zero-mode, deletion and insertion estimates at $\Psi_N = N^{-5/16}$. The cutoff parameters satisfy $|T_l| R_c^{-J_l^c} \le N^{-12}$ and $\delta_l^{-1} = O(\log N)$. Apply the pointwise chain-rule argument in Proposition 4.10 with the right-tail affine sign. Its ingredients are precisely the endpoint derivative bound of Lemma 3.4 and the full spectral and cutoff-index sum in Lemma 3.3. Every positive-order coordinate derivative contains at least one differentiated statistic, of size $O(\Psi_N)$, and the complete coefficient contraction costs $N^{o(1)}$. Consequently (22) holds. For the sixth derivative the same insertion bound, the fixed raw moments and a sufficiently large failure exponent give the left side of (23) bounded by $CN^{1/100}$. This implies (23) with $d_{\text{det}} = \max(\theta_h - 3/10, 0) \le 1/30$.

Fix $K_{\text{der}} \ge 2189$. The three source insertions give $F_0 = k_0 = 3$ in Supplementary Lemma C.1, and both source indices are unrevealed. That lemma supplies the joint recurrences, their degree-four terminal estimate, and the $N^{1/2}[N^{-2}\sum_e]$ normalization. Choose the moment and conditional failure orders after this finite budget. The final cutoff layer and the exceptional terms are $O(N^{-9})$. Thus all five hypotheses of Theorem E.2 hold for this specified count. Its estimates (24)–(25) give the asserted differential and integrated bounds, with shifted-transform coefficient zero. □

### E.3. Left-tail specialization.

Corollary 6.4 of the main article proves the application to the frozen product corner, including the specified-height and full-trace conventions. Its coefficients and source terms are those of Theorem 4.8, with each factor retaining its own conjugation sign. Clause 2 of Theorem E.1 provides the corresponding local-law scale.

### E.4. Left-tail comparison

*Proof of Proposition 6.5.* Fix $N = N_l$, $P = P_l$, $\kappa_0 = 1/1000$, $\nu_0 = 1/48$, and $\mathcal{G} = \mathcal{F}_l^{\text{past}}$. The old block is $\{1, \dots, M_P\} \times \{1, \dots, P\}$; all other coordinates follow the homogeneous flow. The observable is the full-matrix count of Proposition 6.3, with the same thresholds, upper cutoff, transition width and cutoff function at every time. A coordinate section changes that one raw coordinate in every matrix containing it and retains every summand of the count.

For $n \in T_l$, let $p_{l,n}$ be the smaller southeast dimension. Then $p_{l,n} \asymp n \asymp N$, and the natural height is

$$\eta'_{l,n} = \frac{p_{l,n}^{1/3-12\kappa_0}}{n} = \left(\frac{p_{l,n}}{n}\right)^{1/3-12\kappa_0} n^{-2/3-12\kappa_0}.$$

Corollary 6.4 verifies the endpoint window and transfers the conditional local law to these heights. The common-to-natural congruence in the proof of Proposition 4.10 has bounded matrix speeds here as well. All matrices remain in their $M_n$-by-$n$ orientation. The statistic uses the smaller-side singular-value list, including zeros with multiplicity; subtracting $(n-M_n)_+$ times its Poisson mass

at zero from the full column trace gives exactly this statistic. The subtraction is deterministic.

Choose a fixed $\theta < 1/3$ with $|T_l| \le N^{\theta}$ for large $l$. The specified depth gives $|T_l| R_c^{-J_l^c} \le N^{-12}$, and $\delta_l^{-1} = O(\log N)$. Lemmas 3.3 and 3.4 bound the complete coefficient contractions and their endpoint derivatives. The left cutoff can be differentiated only where its statistic is below $1/16$. As in Proposition 4.10, the Poisson mass of an eigenvalue in the upper transition interval is greater than $1/8$. Thus derivatives of the smooth upper-spectral term vanish on this support. The pointwise bounds in that proposition therefore apply to the present coordinate sections on the simultaneous conditional local-law event.

Fix $K_{\mathrm{der}} \ge 2189$ and the moment orders, then choose the one failure exponent $D$ and the history $\mathcal{H}_N(D)$ as in Supplementary Lemma C.1. On this history the old-entry bound is fixed, the moving-coordinate moments are uniformly bounded, and conditional Cauchy–Schwarz pays every exceptional term. The pointwise derivative bounds give (22) after conditional expectation. Their sixth-order insertion form gives the left side of (23) bounded by $CN^{1/100}$, which is sufficient for $d_{\det} = \max(\theta - 3/10, 0) \le 1/30$. The spectral tensors remain in the same conditional contraction.

Every moving source lies outside the old product, so at least one of its row and column indices is unrevealed. Supplementary Lemma C.1 with $F_0 = k_0 = 3$ therefore supplies both joint recurrences, the degree-four terminal estimate, and the third-source normalization. Its old-index, collision and degree-lowering branches are included in this application. The final cutoff layer and the exceptional-event bounds are $O(N^{-9})$. All hypotheses of Theorem E.2 are now verified on the same good history. Equations (24)–(25) give the required comparison and its integrated error. The complements of these histories are summable on the selected lacunary scales. □

## Appendix F. Auxiliary lemmas

The first lemma records the left-tail estimate in the tall orientation. The next two lemmas package the exact orientation lifting and countable occurrence scheduling used in the main text.

**Lemma F.1** (Left-tail comparison)**.** *Let $\gamma \ge 1$, and let $(M_N)_{N\ge 1}$ be positive integers such that*

$$M_N \ge N$$

*for all sufficiently large $N$, and $M_N/N \to \gamma$. For each $N$, let*

$$x_{ij}^{(N)}, \qquad 1 \le i \le M_N, \quad 1 \le j \le N,$$

*be independent complex random variables satisfying*

$$\mathbb{E}x_{ij}^{(N)} = 0, \qquad \mathbb{E}|x_{ij}^{(N)}|^2 = 1, \qquad \mathbb{E}[(x_{ij}^{(N)})^2] = 0,$$

*and*

$$\sup_{N,i,j} \mathbb{E}|x_{ij}^{(N)}|^p < \infty$$

*for every fixed $p \ge 1$. Let $X^{(N)}$ be the $M_N$-by-$N$ matrix with entries $x_{ij}^{(N)}$, put*

$$Q_N = N^{-1} X^{(N)} (X^{(N)})^*,$$

*and define*

$$\mu_{+,N} = \frac{(\sqrt{M_N} + \sqrt{N})^2}{N},$$

$$\sigma_{+,N} = \frac{\sqrt{M_N} + \sqrt{N}}{N} \left( \frac{1}{\sqrt{M_N}} + \frac{1}{\sqrt{N}} \right)^{1/3},$$

$$\chi_N^+ = \frac{\lambda_{\max}(Q_N) - \mu_{+,N}}{\sigma_{+,N}}, \qquad s_N = (\log N)^{1/3}, \qquad B_2 = 4^{1/3}.$$

*For every fixed $q > B_2$, there are $\theta_q > 1/3$ and $N_q$ such that*

$$\mathbb{P}(\chi_N^+ \le -q s_N) \le N^{-\theta_q}$$

*for every* $N \ge N_q$.

*Proof.* Fix $q > B_2$. We first treat indices for which $M_N \ge N$. On those indices put

$$Y_N = N^{-1/2} X^{(N)}, \qquad L_N = Y_N^* Y_N, \qquad E_{+,N} = \left(1 + \sqrt{\frac{M_N}{N}}\right)^2.$$

The largest eigenvalues of $L_N$ and $Q_N$ agree, and $E_{+,N} = \mu_{+,N}$.

Choose

$$B_2 < q_G < q_0 < q.$$

Choose parameters

$$0 < \kappa_0 < \frac{1}{288}, \qquad 12\kappa_0 < \nu < \frac{1}{12},$$

put

$$\rho = \frac{1}{3} - \nu,$$

and choose a positive integer $J$ such that

$$J\rho > \frac{1}{3}.$$

Define

$$\eta_N = N^{-2/3-12\kappa_0}, \qquad \ell_N = N^{-2/3-6\kappa_0},$$
$$E_N = E_{+,N} - q_0 s_N \sigma_{+,N}, \qquad B_N = E_{+,N} + 4N^{-2/3+\kappa_0},$$

and

$$E_{N,j} = E_N + (2j+1)\ell_N, \qquad 0 \le j \le J.$$

Choose a smooth nonincreasing function $\phi\colon \mathbb{R} \to [0,1]$ that equals one on $(-\infty, 0]$ and zero on $[1,\infty)$. With

$$h = \frac{1}{4J+8}, \qquad a_j = (j+1)h, \qquad b_j = (j+2)h,$$

put

$$F_j(u) = \phi\left(\frac{u - a_j}{b_j - a_j}\right).$$

Every positive derivative of $F_j$ is supported on $(a_j, b_j)$, and

$$F_{j+1}(u) = 1 \qquad \text{for } u \le b_j$$

when $j < J$.

Let $W_N$ be an independent $M_N \times N$ matrix with independent centered circular complex Gaussian entries of variance $1/N$. For

$$0 \le t \le T_N := 8 \log N,$$

put

$$Y_N(t) = \mathrm{e}^{-t/2} Y_N + \sqrt{1 - \mathrm{e}^{-t}}\, W_N, \qquad L_N(t) = Y_N(t)^* Y_N(t),$$
$$Z_{N,j}(t) = \frac{1}{\pi} \int_{E_{N,j}}^{B_N} \Im \operatorname{Tr}((L_N(t) - (y + \mathrm{i}\eta_N) I_N)^{-1})\, \mathrm{d}y,$$

and

$$f_{N,j}(t) = \mathbb{E} F_j(Z_{N,j}(t)).$$

Corollary 4.2 applies to these parameters and cutoffs. Thus, for every $D_0 > 1/3$ and all sufficiently large $N$,

$$\mathbb{P}(\lambda_{\max}(L_N) < E_N) \le f_{N,0}(0) + N^{-D_0}. \tag{52}$$

Put

$$L_N^G = W_N^* W_N,$$

and let $g_{N,j}$ be the expectation of $F_j$ applied to the same statistic with $L_N^G$ in place of $L_N(T_N)$. For any fixed $D_* > 0$ and sufficiently small $\epsilon_0 > 0$, uniform moment bounds and Markov's inequality give an event with complement probability at most $N^{-D_*}$ on which

$$\|Y_N\|_{\mathrm{HS}} + \|W_N\|_{\mathrm{HS}} \le 2N^{1/2+\epsilon_0}.$$

Since

$$\mathrm{e}^{-T_N/2} = N^{-4},$$

on this event

$$\|L_N(T_N) - L_N^G\|_{\mathrm{op}} \le CN^{-3+2\epsilon_0}.$$

The resolvent identity gives, at imaginary part $\eta_N$, a trace difference bounded by

$$N\eta_N^{-2}\|L_N(T_N) - L_N^G\|_{\mathrm{op}}.$$

Since $M_N/N$ stays in a compact subset of $(0,\infty)$, the scale $\sigma_{+,N}$ is comparable to $N^{-2/3}$, and the integration interval has length at most

$$CN^{-2/3+\kappa_0}.$$

Therefore the two statistics differ by at most

$$CN^{-4/3+25\kappa_0+2\epsilon_0}.$$

Choose $\epsilon_0$ so that

$$\frac{4}{3} - 25\kappa_0 - 2\epsilon_0 > \frac{1}{3},$$

and then choose

$$\frac{1}{3} < D_T < \frac{4}{3} - 25\kappa_0 - 2\epsilon_0, \qquad D_* > D_T.$$

The Lipschitz bounds on the finitely many $F_j$ yield

$$|f_{N,j}(T_N) - g_{N,j}| \le N^{-D_T}, \qquad 0 \le j \le J. \tag{53}$$

Proposition 7.3 applies to the same statistics and left cutoffs. It gives constants $C < \infty$ and $\sigma_{\mathrm{src}} > 1/3$ such that, for $j < J$ and almost every $t$,

$$|f'_{N,j}(t)| \le CN^{-\rho} f_{N,j+1}(t) + CN^{-\sigma_{\mathrm{src}}}. \tag{54}$$

Its generator identity and the integrable bound in (54) also give absolute continuity of every $f_{N,j}$.

Apply the Gaussian clause of Corollary 4.2 to each $F_j$. For every $D_1 > 1/3$,

$$g_{N,j} \le \mathbb{P}\big(\lambda_{\max}(L_N^G) < E_N + (2j+2)\ell_N\big) + N^{-D_1}. \tag{55}$$

The ratio

$$\frac{\ell_N}{\sigma_{+,N} s_N}$$

tends to zero. Since $q_G < q_0$, for all sufficiently large $N$, the probability in (55) is bounded by

$$\mathbb{P}\left(\frac{\lambda_{\max}(L_N^G) - E_{+,N}}{\sigma_{+,N}} \le -q_G s_N\right).$$

The left-tail clause of Proposition 7.1 supplies an exponent $\eta_G > 1/3$ for this probability. Consequently, there is $c_G > 1/3$ such that

(56) $$g_{N,j} \le N^{-c_G}, \qquad 0 \le j \le J.$$

Apply Proposition 2.7 with

$$H_N^X = \{\lambda_{\max}(L_N) < E_N\}, \qquad H_N^G = \varnothing, \qquad T_N = 8 \log N,$$

derivative exponent $\rho$, additive exponent $\sigma_{\text{src}}$, terminal exponent $c_G$, and

$$D = \min\{D_0, D_T\}.$$

Equations (52)–(56) verify all its hypotheses, including

$$J\rho > \frac{1}{3}.$$

It gives $\theta_{\text{tall}} > 1/3$ such that

(57) $$\mathbb{P}(\lambda_{\max}(L_N) < E_N) \le N^{-\theta_{\text{tall}}}$$

for all sufficiently large tall indices. Since $q_0 < q$,

$$\{\chi_N^+ \le -q s_N\} \subseteq \{\lambda_{\max}(L_N) < E_N\},$$

so (57) proves the tall bound. □

**Lemma F.2** (Transposition and edge scaling)**.** *Fix $\beta > 0$, $C_* > 0$, and $\varepsilon \in \{-1, 1\}$, and put*

$$u_n = (\log n)^\beta.$$

*For positive integers $m, n$ and an $m$-by-$n$ matrix $X$, define*

$$Q_{m,n} = n^{-1} X X^*, \qquad \mu_{m,n} = \frac{(\sqrt{m} + \sqrt{n})^2}{n},$$

$$\sigma_{m,n} = \frac{\sqrt{m} + \sqrt{n}}{n} \left(m^{-1/2} + n^{-1/2}\right)^{1/3}, \qquad \chi_{m,n} = \frac{\lambda_{\max}(Q_{m,n}) - \mu_{m,n}}{\sigma_{m,n}}.$$

*Assume the following tall-sequence estimate. All matrix sequences below may be placed on an arbitrary common probability space, enlarged by adjoining independent variables when a completion is formed.*

*For every $\Gamma \ge 1$, every integer sequence $m_n$ satisfying*

$$m_n \ge n \quad \textit{eventually}, \qquad \frac{m_n}{n} \longrightarrow \Gamma,$$

*and every sequence of $m_n$-by-$n$ random matrices on this probability space whose entries are independent within each matrix and for which every entry $x$ satisfies*

$$\mathbb{E}x = 0, \qquad \mathbb{E}|x|^2 = 1, \qquad \mathbb{E}x^2 = 0,$$

*with uniform fixed-moment bounds, the following holds: for every $q > C_*$ there are $\theta > 1/3$ and $n_0$, allowed to depend on the complete sequence, such that*

(58) $$\mathbb{P}(\varepsilon \chi_{m_n,n} \ge q u_n) \le n^{-\theta}, \qquad n \ge n_0.$$

*No independence between different values of $n$ is required.*

*Then, for every $\gamma > 0$, every positive-integer sequence $M_N$ with*

$$\frac{M_N}{N} \longrightarrow \gamma,$$

*and every admissible sequence of $M_N$-by-$N$ random matrices as above, for every $q > C_*$ there are $\theta_q > 1/3$ and $N_q$ such that*

$$\mathbb{P}(\varepsilon\chi_{M_N,N} \geq qu_N) \leq N^{-\theta_q}, \qquad N \geq N_q.$$

*Proof.* Suppose the conclusion fails for some $q > C_*$. Then one can choose a strictly increasing sequence $N_s$ such that

$$\mathbb{P}\big(\varepsilon\chi_{M_{N_s},N_s} \geq qu_{N_s}\big) > N_s^{-(1/3+1/s)}. \tag{59}$$

After passing to a subsequence, either every selected index is tall or every selected index is wide.

In the tall case, put $n = N_s$ and $m_n = M_{N_s}$ at the selected source sizes, using the original data matrix there. Infinitely many tall indices force $\gamma \geq 1$. Complete the unselected source sizes with admissible Gaussian matrices and dimensions $m_n \geq n$ satisfying $m_n/n \to \gamma$. The tall-sequence estimate at level $q$ gives an exponent $\theta > 1/3$ on the completed sequence. For all sufficiently large $s$,

$$\frac{1}{3} + \frac{1}{s} < \theta,$$

contradicting (59).

In the wide case, put

$$n_s = M_{N_s}, \qquad m_{n_s} = N_s.$$

Since $M_N/N \to \gamma > 0$, the source sizes $n_s$ tend to infinity; pass to a further subsequence so that they are strictly increasing. Infinitely many wide indices force $\gamma \leq 1$. At the selected source size use the transposed data matrix without an additional normalization; the definition of $Q_{m,n}$ already supplies the source factor $n_s^{-1}$. Complete the other source sizes with admissible Gaussian matrices. The completed tall sequence satisfies

$$\frac{m_n}{n} \longrightarrow \gamma^{-1} \geq 1.$$

At a selected source size, the nonzero eigenvalues and the deterministic center and scale obey

$$\lambda_{\max}(Q^{\mathrm{src}}_{m_{n_s},n_s}) = \frac{N_s}{M_{N_s}}\lambda_{\max}(Q_{M_{N_s},N_s}),$$

$$\mu^{\mathrm{src}}_{m_{n_s},n_s} = \frac{N_s}{M_{N_s}}\mu_{M_{N_s},N_s}, \qquad \sigma^{\mathrm{src}}_{m_{n_s},n_s} = \frac{N_s}{M_{N_s}}\sigma_{M_{N_s},N_s}.$$

Consequently the two standardized variables are exactly equal.

Choose $q_0$ with

$$C_* < q_0 < q.$$

Because $M_N/N \to \gamma > 0$,

$$\frac{\log M_N}{\log N} \longrightarrow 1, \qquad \frac{u_{M_N}}{u_N} \longrightarrow 1,$$

and hence $qu_{N_s} \geq q_0 u_{M_{N_s}}$ for all sufficiently large selected indices. Thus the bad event in (59) is contained in the source event at level $q_0$. The tall-sequence estimate gives $\theta_0 > 1/3$ such that its probability is at most

$$M_{N_s}^{-\theta_0}.$$

Choose $\theta_1$ with

$$\frac{1}{3} < \theta_1 < \theta_0.$$

Again using $M_N/N \to \gamma > 0$, for all sufficiently large selected indices,

$$M_{N_s}^{-\theta_0} \leq N_s^{-\theta_1}.$$

For large $s$, $1/3 + 1/s < \theta_1$, contradicting (59). This proves the lemma. ∎

**Lemma F.3** (Construction of a common filtration). *Let $(\Omega, \mathcal{F}, \mathbb{P})$ be a probability space, let $(\mathcal{R}_j)_{j\geq 0}$ be an increasing sequence of sub-sigma-algebras of $\mathcal{F}$, and let $\mathcal{Q}$ be countable. For every $q \in \mathcal{Q}$, suppose there are integer pairs*

$$p_q(m) < e_q(m), \qquad m \geq 1,$$

*such that*

$$e_q(m) - p_q(m) \geq 2, \qquad p_q(m) \longrightarrow \infty,$$

*events and constants satisfying*

$$C_q(m) \in \mathcal{R}_{e_q(m)}, \qquad D_q(m) \in \mathcal{R}_{p_q(m)}, \qquad c_q > 0,$$

*and*

$$\mathbb{P}\big(C_q(m) \mid \mathcal{R}_{p_q(m)}\big) \geq c_q \quad \text{on } D_q(m) \tag{60}$$

*for all sufficiently large $m$. Assume also*

$$\sum_m \mathbb{P}(D_q(m)^c) < \infty$$

*for every $q$. Finally, suppose that, after discarding finitely many values of $m$ separately for each label if necessary, every $C_q(m)$ carries a witness in the specified block $I_{e_q(m)}$ satisfying a predicate $W_q$.*

*Then there is one increasing filtration $(\mathcal{F}^*_k)_{k\geq 0}$ and, for every $q \in \mathcal{Q}$ and $k \geq 1$, events*

$$E^*_k(q) \in \mathcal{F}^*_k, \qquad G^*_k(q) \in \mathcal{F}^*_{k-1},$$

*and deterministic $p^*_k(q) \in [0,1]$ such that*

$$\mathbb{P}\big(E^*_k(q) \mid \mathcal{F}^*_{k-1}\big) \geq p^*_k(q) \quad \text{on } G^*_k(q),$$

$$\sum_k \mathbb{P}(G^*_k(q)^c) < \infty, \qquad \sum_k p^*_k(q) = \infty,$$

*and every nonempty $E^*_k(q)$ is one of the original events $C_q(m)$, placed at the unchanged target index $k = e_q(m)$. At every such selected actual index,*

$$\mathcal{F}^*_{e_q(m)-1} = \mathcal{R}_{p_q(m)},$$

*so the conditioning past is literally unchanged, and the witness remains in the unchanged block $I_k$.*

*Proof.* Choose a deterministic sequence $(q_t)_{t\geq 1}$ in which every label occurs infinitely often. Recursively choose unused $m_t$ so that, with $u_t = p_{q_t}(m_t)$ and $v_t = e_{q_t}(m_t)$,

$$u_t \geq v_{t-1} + 2, \qquad v_t \geq u_t + 2,$$

and (60) and all cofinite witness conditions hold. This is possible because $p_q(m) \to \infty$ for each $q$. Apply Lemma B.1 of the main article to these paired indices, the events $C_{q_t}(m_t)$ and the histories $D_{q_t}(m_t)$. Its filtration satisfies $\mathcal{F}^*_{v_t-1} = \mathcal{R}_{u_t}$. The conditional probability is therefore exactly the one in (60). For a fixed label, the selected bad-history series is a subseries of its original summable series, whereas the positive lower bound $c_q$ occurs infinitely often. Every selected event stays at its original target index, so its witness remains in the same block. ☐

## Appendix G. Probability and smoothing estimates

This section contains the closure, submatrix, scheduling, and statistic proofs cited from the main article.

### G.1. Comparison with submatrices.

*Proof of 2.3.* Let $P_N^{\mathrm{del}}$ be the coordinate row-selection map onto the retained rows, and let $\iota_N^{\mathrm{del}}$ be the coordinate isometric inclusion from the retained column space into the full column space. After ordering the retained coordinates,

$$Y_N = P_N^{\mathrm{del}} X_N \iota_N^{\mathrm{del}}.$$

Both maps have operator norm one. Hence

$$\|Y_N\|_{\mathrm{op}} \le \|P_N^{\mathrm{del}}\|_{\mathrm{op}} \|X_N\|_{\mathrm{op}} \|\iota_N^{\mathrm{del}}\|_{\mathrm{op}} = \|X_N\|_{\mathrm{op}}.$$

Fix $N$, and abbreviate

$$m = m_N, \quad n = n_N, \quad r = r_N, \quad s = s_N, \quad u = m - r, \quad v = n - s.$$

The definitions give

$$b(m,n)\chi_{X,N} = \|X_N\|_{\mathrm{op}}^2 - a(m,n),$$

and

$$b(u,v)\chi_{Y,N} = \|Y_N\|_{\mathrm{op}}^2 - a(u,v).$$

The norm inequality implies

$$b(m,n)\chi_{X,N} \ge b(u,v)\chi_{Y,N} - \big[a(m,n) - a(u,v)\big].$$

Since $b(m,n) > 0$, division by $b(m,n)$ proves (1).

We next prove (2). Expanding the square gives

$$a(x,y) = x + y + 2\sqrt{xy}.$$

Therefore

$$a(m,n) - a(u,v) = r + s + 2\big[\sqrt{mn} - \sqrt{uv}\big].$$

This quantity is nonnegative. Rationalization gives

$$\sqrt{mn} - \sqrt{uv} = \frac{mn - (m-r)(n-s)}{\sqrt{mn} + \sqrt{(m-r)(n-s)}} = \frac{ms + nr - rs}{\sqrt{mn} + \sqrt{(m-r)(n-s)}}.$$

The numerator is nonnegative and at most $ms + nr$, while

$$ms + nr \le C_0 N(r+s), \qquad \sqrt{mn} + \sqrt{(m-r)(n-s)} \ge \sqrt{mn} \ge c_0 N.$$

Consequently,

$$0 \le \sqrt{mn} - \sqrt{uv} \le \frac{C_0}{c_0}(r+s),$$

and hence

$$0 \le a(m,n) - a(u,v) \le \left(1 + \frac{2C_0}{c_0}\right)(r+s).$$

For the scale estimates, define

$$f_{\mathrm{del}}(x,y) = (\sqrt{x} + \sqrt{y})(x^{-1/2} + y^{-1/2})^{1/3}.$$

Then

$$b(Nx, Ny) = N^{1/3} f_{\mathrm{del}}(x,y).$$

The function $f_{\mathrm{del}}$ is positive and continuous on $[c_0, C_0]^2$. It therefore has a positive minimum and a finite maximum there. This proves the asserted upper and lower bounds for $b(m_N, n_N)$.

Since $r_N + s_N = o(N)$, for all sufficiently large $N$,

$$\frac{m_N - r_N}{N}, \frac{n_N - s_N}{N} \in [c_0/2, C_0].$$

The first partial derivatives of $f_{\mathrm{del}}$ have a common finite bound on $[c_0/2, C_0]^2$. The line segment joining

$$\left(\frac{m_N}{N}, \frac{n_N}{N}\right) \quad \text{and} \quad \left(\frac{m_N - r_N}{N}, \frac{n_N - s_N}{N}\right)$$

lies in this square. The two-variable mean-value estimate gives

$$\left| f_{\mathrm{del}}\left(\frac{m_N - r_N}{N}, \frac{n_N - s_N}{N}\right) - f_{\mathrm{del}}\left(\frac{m_N}{N}, \frac{n_N}{N}\right) \right| \le C \frac{r_N + s_N}{N}.$$

The denominator $f_{\mathrm{del}}(m_N/N, n_N/N)$ has a positive uniform lower bound. Division by it and homogeneity prove (3).

Now suppose

$$d_N^{\mathrm{del}} := r_N + s_N = o\big(N^{1/3}(\log N)^{2/3}\big).$$

Then $d_N^{\mathrm{del}} = o(N)$. Since $n_N$ is comparable to $N$ and $s_N = o(N)$,

$$\frac{\log(n_N - s_N)}{\log n_N} \longrightarrow 1,$$

so

$$\big(\log(n_N - s_N)\big)^{2/3} = (1 + o(1))(\log n_N)^{2/3}.$$

Equation (3) gives

$$\frac{b(m_N - r_N, n_N - s_N)}{b(m_N, n_N)} = 1 + o(1).$$

The center estimate and the scale lower bound give

$$\begin{aligned} 0 &\le \frac{a(m_N, n_N) - a(m_N - r_N, n_N - s_N)}{b(m_N, n_N)} \\ &\le C \frac{d_N^{\mathrm{del}}}{N^{1/3}} = o\big((\log N)^{2/3}\big) \\ &= o\big((\log n_N)^{2/3}\big). \end{aligned}$$

If the premise of (4) holds, substitution into (1) yields

$$\chi_{X,N} \ge [a_+ + o(1)](\log n_N)^{2/3}.$$

Since $a_+ - a_- > 0$, the coefficient on the right is at least $a_-$ for all sufficiently large $N$. This proves (4).

Finally, Lemma 2.1, item 1, applied with $m = m_N$, $n = n_N$, and $X = X_N$, shows that the edge-normalized largest eigenvalue of $Q_N$ equals

$$\frac{\|X_N\|_{\mathrm{op}}^2 - a(m_N, n_N)}{b(m_N, n_N)} = \chi_{X,N}.$$

□

G.2. **Cluster sets from endpoint limits.**

*Proof of 2.4.* Write $m = M_N$ and $d = M_{N+1} - M_N$. First,

$$L_{N+1} \ge L_N.$$

To prove this, extend a unit vector $v$ in the $N$-column space by a zero last coordinate. The first $M_N$ coordinates of $X^{(N+1)}(v, 0)$ equal $X^{(N)}v$, while the additional coordinates can only increase its Euclidean norm. Taking the supremum over $v$ gives

$$\|X^{(N+1)}\|_{\mathrm{op}} \geq \|X^{(N)}\|_{\mathrm{op}},$$

and squaring proves the assertion.

Since $M_N/N \to \gamma > 0$, there are constants $c_0, c_1 > 0$ such that

$$c_0 N \leq M_N \leq c_1 N$$

for all sufficiently large $N$. Since $0 \leq d \leq K$,

$$a_{N+1} - a_N = d + 1 + 2\frac{d(N+1) + M_N}{\sqrt{(M_N + d)(N+1)} + \sqrt{M_N N}}.$$

The numerator in the fraction is $O(N)$, while its denominator is bounded below by a positive constant times $N$. Hence

$$0 \leq a_{N+1} - a_N \leq C.$$

Define

$$g(x) = (1 + \sqrt{x})^{4/3} x^{-1/6}.$$

Direct algebra gives

$$b_N = N^{1/3} g(M_N/N).$$

Moreover,

$$\left| \frac{M_{N+1}}{N+1} - \frac{M_N}{N} \right| = \frac{|dN - M_N|}{N(N+1)} \leq \frac{C}{N}.$$

The function $g$ is positive and continuously differentiable on a compact interval containing $M_N/N$ for all sufficiently large $N$. Consequently, $b_N$ is bounded above and below by positive constants times $N^{1/3}$, and

$$\left| \frac{b_N}{b_{N+1}} - 1 \right| \leq \frac{C}{N}.$$

Since

$$L_N = a_N + b_N \xi_N, \qquad L_{N+1} = a_{N+1} + b_{N+1} \xi_{N+1},$$

monotonicity yields

$$\xi_N - \xi_{N+1} \leq \left(1 - \frac{b_N}{b_{N+1}}\right) \xi_N + \frac{a_{N+1} - a_N}{b_{N+1}} \leq C\frac{|\xi_N|}{N} + CN^{-1/3}. \tag{61}$$

The finite limsup and liminf assumptions imply that eventually

$$\frac{\xi_N}{r_N} \leq A + 1, \qquad \frac{\xi_N}{s_N} \geq -B - 1.$$

If $\xi_N \geq 0$, use the first bound. If $\xi_N < 0$, use the second. Since $s_N \leq r_N$ eventually, in both cases

$$|\xi_N| \leq C r_N$$

for all sufficiently large $N$. Substitution in (61) gives

$$\xi_N - \xi_{N+1} \leq C\frac{r_N}{N} + CN^{-1/3}. \tag{62}$$

For $t_N = r_N$ or $t_N = s_N$, the mean-value theorem gives

$$0 \le \frac{t_{N+1}}{t_N} - 1 \le \frac{C}{N \log N}.$$

With $y_N = \xi_N / t_N$,

$$y_N - y_{N+1} = \frac{\xi_N - \xi_{N+1}}{t_N} + \xi_{N+1}\left(\frac{1}{t_N} - \frac{1}{t_{N+1}}\right).$$

Using (62) and $|\xi_{N+1}| \le C r_{N+1}$, we obtain

$$y_N - y_{N+1} \le C\frac{r_N}{N t_N} + C\frac{N^{-1/3}}{t_N} + C\frac{r_N/t_N}{N \log N}.$$

The right-hand side tends to zero for $t_N = r_N$. It also tends to zero for $t_N = s_N$, because

$$\frac{r_N}{s_N} = (\log N)^{1/3}.$$

Thus, for both normalizations,

$$\limsup_{N\to\infty}\left(\frac{\xi_N}{t_N} - \frac{\xi_{N+1}}{t_{N+1}}\right) \le 0. \tag{63}$$

Set

$$u_N = \frac{\xi_N}{r_N}, \qquad v_N = \frac{\xi_N}{s_N}.$$

The endpoint assumptions give eventual upper bounds $u_N \le A + o(1)$ and lower bounds $v_N \ge -B - o(1)$. Since

$$u_N = v_N\frac{s_N}{r_N} \quad \text{and} \quad \frac{s_N}{r_N} \longrightarrow 0,$$

the lower bound on $v_N$ also gives $u_N \ge -o(1)$. Hence every finite cluster point of $u_N$ lies in $[0, A]$. A subsequence of $u_N$ converges to $A$. A subsequence of $v_N$ converges to $-B$, and along that subsequence

$$u_N = v_N(\log N)^{-1/3} \longrightarrow 0.$$

Thus both 0 and $A$ are cluster points of $u_N$.

Fix $c \in (0, A)$. Choose arbitrarily large indices $h$ with $u_h > c$. After each such $h$, choose a later index $\ell$ along a subsequence for which $v_\ell \to -B$. Then $u_\ell \to 0$, so $u_\ell < c$ for sufficiently late choices. There is an index $n$ between $h$ and $\ell$ such that

$$u_n \ge c > u_{n+1}.$$

Therefore

$$0 < c - u_{n+1} \le u_n - u_{n+1}.$$

As $h \to \infty$, also $n \to \infty$, and (63) with $t_N = r_N$ shows that the right-hand side tends to zero. Hence $u_{n+1} \to c$. Every $c \in (0, A)$ is a cluster point, proving that the first cluster set is $[0, A]$.

Every finite cluster point of $v_N$ is at least $-B$, and $-B$ is a cluster point. Along a subsequence for which $u_N \to A > 0$,

$$v_N = u_N\frac{r_N}{s_N} = u_N(\log N)^{1/3} \longrightarrow \infty.$$

Fix $q > -B$. Choose arbitrarily large indices $h$ from this high subsequence with $v_h > q$. After each such $h$, choose a later index $\ell$ from a subsequence along which $v_\ell \to -B$, so that $v_\ell < q$. There is an index $n$ between $h$ and $\ell$ satisfying

$$v_n \ge q > v_{n+1}.$$

Then

$$0 < q - v_{n+1} \le v_n - v_{n+1},$$

and the right-hand side tends to zero by (63) with $t_N = s_N$. Hence $v_{n+1} \to q$. Every finite $q > -B$ is a cluster point, so the second cluster set is $[-B, \infty)$. □

### G.3. **Independent submatrices.**

*Proof of 2.5.* Fix $j_1 < j_2$. Monotonicity of $(M_N)$ gives

$$\widehat{M}_{j_1} \le \widehat{M}_{j_2}.$$

The row labels of $Y^{(j_1)}$ form an initial subset of those of $Y^{(j_2)}$, and the same holds for the column labels. Hence $Y^{(j_1)}$ is the corresponding northwest coordinate submatrix of $Y^{(j_2)}$.

For fixed $j$, deleting rows $1, \dots, M_P$ and columns $1, \dots, P$ from $X^{(P+j)}$ leaves precisely the coordinates

$$(M_P + i, P + h), \qquad 1 \le i \le M_{P+j} - M_P = \widehat{M}_j, \quad 1 \le h \le j.$$

These are exactly the entries of $Y^{(j)}$. This proves part (a).

Let

$$I_{\text{past}} = \mathcal{P}_P, \qquad I_{\text{fresh}} = \{(M_P + i, P + h) \mid i, h \ge 1\}.$$

A coordinate in $I_{\text{past}}$ has row index at most $M_P$ and column index at most $P$. A coordinate in $I_{\text{fresh}}$ has both indices strictly larger than these bounds. Thus the two index sets are disjoint. Mutual independence implies that the sigma-algebras generated by the corresponding subfamilies are independent. Hence

$$\sigma(x_{M_P+i,P+h} \mid i, h \ge 1)$$

is independent of $\mathcal{F}_P$.

For completeness, let $U$ be the random vector associated with any finite list of fresh coordinates, and let $f$ be a bounded Borel function on its finite-dimensional real or complex state space. Independence gives

$$\mathbb{E}[f(U) \mid \mathcal{F}_P] = \mathbb{E}[f(U)] \quad \text{almost surely.}$$

The right-hand side is integration against the product of the marginal laws. These identities for every finite list and every such $f$ give the asserted finite-dimensional conditional law for the countable fresh array and hence for all entries used by $(Y^{(j)})_{j \ge 1}$. This proves part (b).

Finally,

$$\frac{\widehat{M}_j}{j} = \frac{M_{P+j}}{P+j}\frac{P+j}{j} - \frac{M_P}{j} = \frac{M_{P+j}}{P+j}\left(1 + \frac{P}{j}\right) - \frac{M_P}{j}.$$

The three terms converge respectively to $\gamma$, 1, and 0, proving

$$\widehat{M}_j / j \to \gamma.$$

Moreover,

$$\widehat{M}_{j+1} - \widehat{M}_j = M_{P+j+1} - M_{P+j}.$$

Monotonicity and the definition of $K$ place this difference in $[0, K]$, proving part (c). □

### G.4. **Iterated differential inequality.**

*Proof of 2.7.* Since $J\rho > 1/3$, choose $a$ such that

$$\frac{1}{3J} < a < \rho.$$

Choose $\theta_*$ satisfying

$$\frac{1}{3} < \theta_* < \min\{Ja, \sigma, c_G, D\}.$$

For all sufficiently large $N$, define

$$u_N = C_d T_N N^{-\rho}, \qquad v_N = C_d T_N N^{-\sigma}.$$

The time bound and $a < \rho$ imply

$$u_N \le N^{-a}$$

for all sufficiently large $N$. The same time bound and $\theta_* < \sigma$ give

$$v_N \le N^{-\theta_*}.$$

Increase the lower threshold for $N$ so that

$$N^{-a} \le \frac{1}{2}, \qquad N^{-D} \le N^{-\theta_*}, \qquad N^{-c_G} \le N^{-\theta_*}, \qquad N^{-Ja} \le N^{-\theta_*}.$$

For $0 \le j \le J$, set

$$G_{N,j} = \sup_{0 \le t \le T_N} f_{N,j}(t).$$

Then $0 \le G_{N,j} \le 1$. For $0 \le j \le J-1$, absolute continuity and the differential inequality give, for every $t \in [0, T_N]$,

$$\begin{aligned} f_{N,j}(t) &\le f_{N,j}(T_N) + \int_t^{T_N} |f'_{N,j}(s)| \, \mathrm{d}s \\ &\le f_{N,j}(T_N) + u_N G_{N,j+1} + v_N. \end{aligned}$$

Taking the supremum gives

$$G_{N,j} \le f_{N,j}(T_N) + u_N G_{N,j+1} + v_N. \tag{64}$$

Iteration from $j = 0$ through $j = J-1$ yields

$$G_{N,0} \le \sum_{r=0}^{J-1} u_N^r f_{N,r}(T_N) + v_N \sum_{r=0}^{J-1} u_N^r + u_N^J G_{N,J}.$$

Since $u_N \le 1/2$, the finite geometric sum is at most 2. The endpoint matching gives

$$f_{N,r}(T_N) \le g_{N,r} + N^{-D},$$

and $G_{N,J} \le 1$. Therefore

$$f_{N,0}(0) \le g_{N,0} + \sum_{r=1}^{J-1} u_N^r g_{N,r} + 2N^{-D} + 2v_N + u_N^J.$$

When $J = 1$, the sum is empty. When $J \ge 2$,

$$\sum_{r=1}^{J-1} u_N^r g_{N,r} \le N^{-c_G} \sum_{r=1}^{J-1} u_N^r \le 2N^{-c_G}.$$

Thus this bound, with the empty sum interpreted as zero, is valid for every positive integer $J$. Using the terminal bounds and $u_N^J \le N^{-Ja}$, we obtain

$$\begin{aligned} f_{N,0}(0) &\le \mathbb{P}_N(H_N^G) + N^{-c_G} + 2N^{-c_G} + 2N^{-D} + 2v_N + N^{-Ja} \\ &\le \mathbb{P}_N(H_N^G) + 8N^{-\theta_*}. \end{aligned}$$

The event comparison then gives

$$\mathbb{P}_N(H_N^X) \le \mathbb{P}_N(H_N^G) + 9N^{-\theta_*}.$$

Choose $\theta$ with

$$\frac{1}{3} < \theta < \theta_*.$$

Since $\theta_* - \theta > 0$, there is $N_0$ such that

$$9N^{-\theta_*} \le N^{-\theta} \qquad (N \ge N_0).$$

The desired conclusion follows. □

### G.5. A common filtration.

Lemma B.1 of the main article gives the paired-index filtration and proves its measurability and conditional-probability identities.

### G.6. Borel–Cantelli argument.

The abstract endpoint argument is proved in Proposition B.2 of the main article. The application to a common filtration, including rational thresholds and shifts, is proved in Theorem B.3.

### G.7. Monotonicity in the threshold

Lemma 2.9 follows from nonnegativity of the integrand. Its proof and the resulting cutoff support inequality are given together in the main article.

### G.8. Weighted cutoff functions.

*Proof of 3.1.* Since $R > 1$ and $0 \le \rho \le 1$, we have $0 < \lambda_r \le 1$. Therefore

$$0 \le A_{i,r}(h) \le 1$$

whenever $A_{i,r}(h)$ is defined.

Fix $i \in I$, $r < J$, and $k \in \{1,2,3\}$. The function $\rho$ is constant on each of $(-\infty, -1]$ and $[0, \infty)$. Hence $\rho^{(k)}$ vanishes on those two closed half-lines, including their endpoints by continuity. Therefore

$$\rho^{(k)}\big(y_i(0) + J - r\big) \ne 0 \quad \Longrightarrow \quad -1 < y_i(0) + J - r < 0.$$

Subtracting one gives

$$-2 < y_i(0) + J - (r+1) < -1.$$

The plateau condition on $\rho$ now yields (6).

For $r < J$, the same implication gives

$$\lambda_r|\rho^{(k)}\big(y_i(0)+J-r\big)| \le \lambda_r C_k \rho\big(y_i(0)+J-(r+1)\big) = RC_k\lambda_{r+1}\rho\big(y_i(0)+J-(r+1)\big), \tag{65}$$

because $\lambda_r = R\lambda_{r+1}$. Summing (65) over $i \in I$ and $0 \le r < J$, and observing that the resulting terms form a subsum of $S(0)$, gives

$$\sum_{i\in I}\sum_{r=0}^{J-1} \lambda_r|\rho^{(k)}\big(y_i(0)+J-r\big)| \le RC_k S(0). \tag{66}$$

At the last level,

$$\sum_{i\in I} \lambda_J|\rho^{(k)}\big(y_i(0)\big)| \le L\lambda_J C_k. \tag{67}$$

Consequently,

$$\sum_{i\in I}\sum_{r=0}^{J} \lambda_r|\rho^{(k)}\big(y_i(0)+J-r\big)| \le RC_k S(0) + L\lambda_J C_k. \tag{68}$$

All sums are finite and may be differentiated term by term. The chain rule at $h = 0$ gives

$$S'(0) = \sum_{i\in I}\sum_{r=0}^{J} \lambda_r \rho'\big(y_i(0)+J-r\big) y_i'(0),$$

$$S''(0) = \sum_{i\in I}\sum_{r=0}^{J} \lambda_r \big[\rho''\big(y_i(0)+J-r\big)(y_i'(0))^2 + \rho'\big(y_i(0)+J-r\big) y_i''(0)\big],$$

and

$$S'''(0) = \sum_{i\in I}\sum_{r=0}^{J} \lambda_r \big[\rho'''\big(y_i(0)+J-r\big)(y_i'(0))^3 + 3\rho''\big(y_i(0)+J-r\big) y_i'(0) y_i''(0) + \rho'\big(y_i(0)+J-r\big) y_i'''(0)\big].$$

Applying (68) with $k = 1$ to the first identity proves ( 7 ). Applying it separately with $k = 2$ and $k = 1$ to the two terms in the second identity proves ( 8 ). Subtracting $C_{\text{cub}}$ from the third identity deletes its last summand exactly. Applying (68) with $k = 3$ to the cubic first-derivative term and with $k = 2$ to the mixed term proves ( 9 ).

We next prove (10). The set $I$ and the cutoff sum $\{0, \dots, J\}$ are finite, and

$$0 \le S(0) \le L(J+1).$$

Thus $V_\tau \le 1$. The assumed bounds on $Q_{i,r}$ make $T_k$ integrable. Bounded dominated differentiation of the conditional expectation gives

$$\partial_\tau \Phi(\tau) = \mathbb{E}[-S(0)\exp(-\tau S(0)) \mid \mathcal{G}],$$

and hence

$$D(\tau) = \mathbb{E}[\tau S(0) V_\tau \mid \mathcal{G}].$$

Split $T_k$ into the levels $r < J$ and $r = J$, and split each conditional expectation over $\Omega_0$ and $\Omega_0^c$. On $\Omega_0$, (66) and $|Q_{i,r}| \le K_{\mathrm{loc}}$ give

$$\sum_{i\in I}\sum_{r=0}^{J-1} \lambda_r |\rho^{(k)}\big(y_i(0)+J-r\big)Q_{i,r}| \le RC_k K_{\mathrm{loc}} S(0).$$

The contribution from $r = J$ is at most

$$C_k L \lambda_J K_{\mathrm{loc}}.$$

On $\Omega_0^c$, use $V_\tau \le 1$, $|Q_{i,r}| \le K_{\mathrm{abs}}$, and

$$\sum_{r=0}^{J} \lambda_r = \sum_{r=0}^{J} R^{-r} \le (1-R^{-1})^{-1} = \frac{R}{R-1}.$$

The entire absolute sum on $\Omega_0^c$ is therefore at most

$$C_k L K_{\mathrm{abs}} \frac{R}{R-1}.$$

Conditional triangle inequalities yield

$$\begin{aligned}|\mathbb{E}[V_\tau T_k \mid \mathcal{G}]| &\le RC_k K_{\mathrm{loc}} \mathbb{E}[V_\tau S(0) \mid \mathcal{G}] + C_k L \lambda_J K_{\mathrm{loc}} \\ &\quad + C_k L K_{\mathrm{abs}} \frac{R}{R-1} \mathbb{P}(\Omega_0^c \mid \mathcal{G}).\end{aligned}$$

Multiplication by $\tau$, the identity

$$\tau \mathbb{E}[V_\tau S(0) \mid \mathcal{G}] = D(\tau),$$

and the bound $\tau \le \tau_*$ in the last two terms give (10). For variables and bounds depending measurably on an interpolation time, the argument applies at each time.

It remains to prove the endpoint assertions. If $S(0) > 0$, nonnegativity of all $A_{i,r}(0)$ implies that $A_{i,r}(0) > 0$ for some pair $(i, r)$. Since $\rho(x) = 0$ for $x \ge 0$, positivity implies

$$y_i(0) + J - r < 0.$$

Since $J - r \ge 0$, it follows that $y_i(0) < 0$, so the target event occurs at $i$.

At terminal time, the hard-event premise $y_i(T) \le -J - 1$ gives

$$y_i(T) + J \le -1.$$

The unit plateau of $\rho$ and $\lambda_0 = 1$ then imply

$$A_{i,0}(T) = \rho\big(y_i(T) + J\big) = 1.$$

This proves every assertion. □

G.9. **Smoothed eigenvalue counts.**

*Proof of 3.2.* The integral in (12) is

$$P_i(x) = \frac{1}{\pi}\left[\arctan\left(\frac{B_i - x}{\eta_i}\right) - \arctan\left(\frac{a_i - x}{\eta_i}\right)\right]. \tag{69}$$

Suppose that $x \ge t_i$. If $t_i \le x \le B_i$, then the first arctangent in (69) is nonnegative, while

$$\frac{a_i - x}{\eta_i} \le -4.$$

Therefore

$$P_i(x) \ge \frac{\arctan(4)}{\pi} > \frac{1}{8}.$$

If $B_i \le x \le B_i + \eta_i$, then the first arctangent is at least $-\pi/4$, and again

$$\frac{a_i - x}{\eta_i} \le -4.$$

It follows that

$$P_i(x) > \frac{-\pi/4 + \arctan(4)}{\pi} > \frac{1}{8}.$$

If $x \ge B_i + \eta_i$, then $g_i(x) = 1$. Applying these alternatives to the largest squared singular value proves (13).

Let

$$F_i(x) = P_i(x) + g_i(x).$$

This function is $C^\infty$ on $\mathbb{R}$. If $n_i \le m_i$, then

$$Z_i(W) = \operatorname{Tr} F_i(W^*W).$$

If $m_i < n_i$, then

$$Z_i(W) = \operatorname{Tr} F_i(WW^*).$$

The trace of a smooth scalar spectral function of a finite-dimensional Hermitian matrix is smooth in the matrix entries, including when eigenvalues collide. Since $W^*W$ and $WW^*$ are polynomial functions of the real and imaginary entry coordinates, $Z_i$ is $C^\infty$.

If one term in (14) is positive, then

$$\rho\big(y_i(W_i^X) + J - r\big) > 0$$

for some $i$ and $r$. Since $\rho$ vanishes on $[0, \infty)$,

$$y_i(W_i^X) < r - J \le 0.$$

Thus

$$Z_i(W_i^X) < a_* = \frac{1}{16}.$$

The consequence of (13) proves (15).

If $Z_i(W_i^T) \le \delta$, then

$$y_i(W_i^T) \le \frac{\delta - a_*}{\delta} = 1 - (J+2) = -J - 1.$$

Therefore $y_i(W_i^T) + J \le -1$. The plateau condition on $\rho$ and the identity $R^0 = 1$ give

$$A_{i,0}(W_i^T) = 1,$$

which proves (16).

Apply Lemma 3.1 with $L = |I|$, the same $J$, $R$, and $\rho$, with $\lambda_r = R^{-r}$, and with the functions $y_i(h)$. Its definitions agree term by term with (14) along the perturbations $W_i(h)$. The quantities $B_1$, $B_2$, $C_k$, and $C_{\mathrm{cub}}$ also agree. Equations ( 7 ), ( 8 ), and ( 9 ), together with $\lambda_J = R^{-J}$, prove (17)–(19).

For the last assertion,

$$\begin{aligned} |I|R^{-J} &\le N^\theta \exp\big(-\zeta \log(R) \log N\big) \\ &\le N^{\theta - \zeta \log R} \le N^{-A} \end{aligned}$$

for all sufficiently large $N$. Moreover,

$$\delta^{-1} = 16(J+2) = O(\log N).$$

Finally, $N^{-\kappa} \log N$ tends to zero. Hence

$$2N^{-\kappa} \le \delta$$

eventually, and (16) applies. □

G.10. **Derivative estimates.**

*Proof of 3.3.* Fix a nonnegative integer $m$. All estimates are pointwise in $y$ and uniform over $0 < \tau \le \tau_*$.

**Step 1. Support of positive-order derivatives.** For $i \in I$ and $0 \le r \le J$, put

$$u_{i,r} = y_i + J - r.$$

Every derivative $\rho^{(k)}$, $k \ge 1$, vanishes unless

$$-1 < u_{i,r} < 0.$$

It also vanishes at $u_{i,r} = -1$ and $u_{i,r} = 0$. For fixed $i$, the numbers $u_{i,r}$ differ by nonzero integers as $r$ varies. Hence at most one $r$ satisfies $-1 < u_{i,r} < 0$.

Define

$$W(y) = \sum_{i\in I} \sum_{r=0}^{J} \lambda_r \mathbf{1}_{\{-1<u_{i,r}<0\}}.$$

If $-1 < u_{i,r} < 0$ and $r < J$, then

$$u_{i,r+1} = u_{i,r} - 1 < -1,$$

so $\rho(u_{i,r+1}) = 1$. The nonnegative sum $S(y)$ therefore contains the summand

$$\lambda_{r+1} = \frac{\lambda_r}{R_{\mathrm{c}}}.$$

Summing over all active nonterminal pairs gives

$$\sum_{i\in I} \sum_{r=0}^{J-1} \lambda_r \mathbf{1}_{\{-1<u_{i,r}<0\}} \le R_{\mathrm{c}} S(y).$$

For terminal pairs,

$$\sum_{i\in I} \lambda_J \mathbf{1}_{\{-1<u_{i,J}<0\}} \le |I|\lambda_J = \epsilon_{\mathrm{fin}} \le 1.$$

Consequently,

$$W(y) \le R_{\mathrm{c}} S(y) + 1. \tag{70}$$

**Step 2. Ordered $\ell^1$-derivative bounds.** For $1 \le k \le m+1$, put

$$M_k = \|\rho^{(k)}\|_\infty, \qquad M = \max_{1\le k\le m+1} M_k.$$

For $k \ge 1$, a mixed derivative

$$\partial_{y_{i_1}} \cdots \partial_{y_{i_k}} S(y)$$

vanishes unless $i_1 = \cdots = i_k$. When all indices equal $i$, it is

$$\sum_{r=0}^{J} \lambda_r \rho^{(k)}(u_{i,r}).$$

The support observation in Step 1 therefore gives, for $1 \le k \le m$,

$$\sum_{i_1,\ldots,i_k\in I} \left| \partial_{y_{i_1}} \cdots \partial_{y_{i_k}} S(y) \right| \le M_k W(y) \le M W(y). \tag{71}$$

For an output pair $(i_0, r_0)$, define

$$H_{i_0,r_0}(y) = \lambda_{r_0} \rho'(u_{i_0,r_0}).$$

Its $k$-th ordered derivative vanishes unless every derivative index equals $i_0$. In that case it equals

$$\lambda_{r_0} \rho^{(k+1)}(u_{i_0,r_0}).$$

Thus, for $0 \le k \le m$,

$$\sum_{i_0\in I}\sum_{r_0=0}^{J}\sum_{i_1,\dots,i_k\in I}\left|\partial_{y_{i_1}}\cdots\partial_{y_{i_k}}H_{i_0,r_0}(y)\right| \le M_{k+1}W(y) \le MW(y). \tag{72}$$

For $k = 0$, the last index sum is absent.

**Step 3. Derivatives of the exponential.** Put

$$E(y) = \exp(-\tau S(y)).$$

For $k \ge 1$, repeated application of the chain and product rules gives the exact set-partition formula

$$\partial_{y_{i_1}}\cdots\partial_{y_{i_k}}E = E\sum_{\pi\in\mathcal{P}_k}(-\tau)^{|\pi|}\prod_{B\in\pi}\partial_{y_{i_j}:j\in B}S. \tag{73}$$

Here $\mathcal{P}_k$ is the finite set of partitions of $\{1,\dots,k\}$, and within each block $B$ the indicated derivative is taken once in every coordinate indexed by that block. Formula (73) follows by induction. Differentiating $E$ creates a new singleton block, while differentiating one existing derivative of $S$ inserts the new position into that block. These operations produce every partition of $\{1,\dots,k+1\}$ exactly once.

Take absolute values in (73) and sum over all ordered $i_1,\dots,i_k$. For each fixed partition, the sums belonging to its disjoint blocks factor. Equation (71) gives

$$\sum_{i_1,\dots,i_k\in I}\left|\partial_{y_{i_1}}\cdots\partial_{y_{i_k}}E(y)\right| \le E(y)\sum_{\pi\in\mathcal{P}_k}\big(\tau MW(y)\big)^{|\pi|}. \tag{74}$$

For $k = 0$, the corresponding $\ell^1$-quantity is $E(y)$. Since $k \le m$, the finite partition sums in (74) are bounded by a constant depending only on $m$ and $M$, multiplied by

$$E(y)\sum_{p=0}^{k}\big(\tau W(y)\big)^p. \tag{75}$$

**Step 4. Derivatives of $G$.** The definition of $G$ is

$$G_{i_0,r_0,\tau} = \tau E H_{i_0,r_0}.$$

Apply the product rule for $m$ ordered derivatives. For each subset $B$ of the $m$ derivative positions, put the derivatives indexed by $B$ on $H$ and the remaining derivatives on $E$. For that term, the sum of absolute values over the output pair, the indices in $B$, and the remaining derivative indices is at most the product of the corresponding $\ell^1$-sums in (72) and (75). Hence

$$\begin{aligned}\sum_{i_0\in I}\sum_{r_0=0}^{J}\sum_{i_1,\dots,i_m\in I}&\left|\partial_{y_{i_1}}\cdots\partial_{y_{i_m}}G_{i_0,r_0,\tau}(y)\right|\\ &\le C_m\tau W(y)E(y)\sum_{p=0}^{m}\big(\tau W(y)\big)^p\\ &\le C_m E(y)\sum_{p=1}^{m+1}\big(\tau W(y)\big)^p,\end{aligned} \tag{76}$$

where $C_m$ depends only on $m$ and $M$.

Set

$$x = \tau S(y) \ge 0.$$

By (70) and $\tau \le \tau_*$,

$$\tau W(y) \le R_c x + \tau \le R_c x + \tau_*.$$

For every integer $1 \le p \le m+1$, the function

$$x \longmapsto \exp(-x)(R_c x + \tau_*)^p$$

is bounded on $[0, \infty)$. It is continuous on every compact interval and tends to zero as $x$ tends to infinity. Thus the right side of (76) is bounded by a finite constant $C_{m,\tau_*}$ with the asserted dependence. This proves (20), including the case $m = 0$.

**Step 5. Statistic scaling.** If

$$y_i = \frac{q_i - a_i}{\delta},$$

then

$$\partial_{q_i} = \delta^{-1} \partial_{y_i}.$$

Since $\widetilde{G} = \delta^{-1} G$, every $m$-th ordered $q$-derivative of $\widetilde{G}$ is exactly $\delta^{-(m+1)}$ times the corresponding $y$-derivative of $G$. Summing and applying (20) proves (21).

**Step 6. Two disjoint copies.** Suppose that $a \ge 0$ derivatives are taken in $q$ and $b \ge 0$ derivatives are taken in $q'$. Disjointness gives

$$\partial_q^a \partial_{q'}^b \big(\widetilde{G}(q)\widetilde{G}'(q')\big) = \big(\partial_q^a \widetilde{G}(q)\big)\big(\partial_{q'}^b \widetilde{G}'(q')\big).$$

After summing absolute values over all output and ordered derivative indices, the two sums factor. The one-copy result bounds them by their respective constants and by

$$\delta^{-(a+1)} \qquad \text{and} \qquad (\delta')^{-(b+1)}.$$

Their product is the asserted bound for every mixed derivative tensor. □

## Appendix H. Local laws and Gaussian comparison

This section contains the smoothing, local-law, deletion, two-strip, and terminal-comparison proofs cited from the main article.

### H.1. Left-tail smoothing.

The left-tail inequality is proved in Corollary 4.2, using Proposition 4.1 and the finite collection of shifted thresholds. The proof of the count approximation is in Section H.3.

### H.2. Right-tail smoothing.

The right-tail inequality is the second part of the same Corollary 4.2. Its proof retains the right-tail threshold shifts and the strict and non-strict event conventions.

### H.3. Approximation of eigenvalue counts.

*Proof of 4.1.* Apply [SX21, Lemmas 3.1 and 3.2] with

$$\epsilon = \kappa, \qquad E = E_N, \qquad E_L = B_N, \qquad \eta = \eta_N,$$
$$l_1 = N^{3\kappa}\eta_N = N^{-2/3-9\kappa}, \qquad l = N^{3\kappa} l_1 = N^{-2/3-6\kappa} = \ell_N.$$

The scale relations are

$$\frac{\eta_N}{N^{-1}} = N^{1/3-12\kappa} \to \infty, \qquad \frac{\eta_N}{l_1} = N^{-3\kappa} \to 0, \qquad \frac{l_1}{\ell_N} = N^{-3\kappa} \to 0.$$

Because $(\log N)^{2/3} = o(N^\kappa)$, for all sufficiently large $N$,

$$C_E(\log N)^{2/3} \le N^\kappa.$$

Hence

$$3N^{-2/3+\kappa} \le B_N - E_N \le 5N^{-2/3+\kappa}.$$

In particular,

$$\frac{\ell_N}{B_N - E_N} \le \frac{1}{3} N^{-7\kappa} \longrightarrow 0.$$

All scale assumptions of the smoothing comparison hold with a uniform constant.

For completeness, diagonalizing $L_N$ gives, for every $a \le B_N$,

$$\begin{aligned} \operatorname{Tr}\big((\chi_a * \theta_{\eta_N})(L_N)\big) &= \sum_{j=1}^N \frac{1}{\pi} \int_a^{B_N} \frac{\eta_N}{(\lambda_j(L_N) - y)^2 + \eta_N^2} \, \mathrm{d}y \\ &= \frac{N}{\pi} \int_a^{B_N} \Im m_N(y + \mathrm{i}\eta_N) \, \mathrm{d}y. \end{aligned} \tag{81}$$

Thus the two traces in the smoothing comparison are exactly $Z_N^+$ and $Z_N^-$.

Invoke the comparison with failure exponent $D+1$ and call its good event $\Omega_N$. After increasing $N_0$,

$$\mathbb{P}(\Omega_N^c) \le N^{-D},$$

and the comparison becomes (33). If $\lambda_N < E_N$, the integer count in (33) is zero, and the left inequality gives (34). If $\lambda_N \ge E_N$, the count is at least one, and the right inequality gives (35). The scale estimates depend on the threshold only through $C_E$, which proves uniformity. The matching complex white Wishart matrices satisfy the same hypotheses. □

### H.4. Entrywise local law.

The fixed-parameter covariance local law and its extension along the flow are proved in Theorem 4.3 of the main article.

### H.5. Uniformity in the parameters.

Lemma 4.4 of the main article proves the passage from a deterministic parameter net to the continuous family and then to conditional good histories. The deletion and complex-coordinate insertion estimates are proved in Theorem 4.5.

### H.6. Schur complement estimates.

*Proof of 4.6.* Order the exceptional indices before the core indices. Then

$$H_p = \begin{pmatrix} H_E & U \\ U^* & H_K \end{pmatrix}, \quad H_E = \begin{pmatrix} -zI_{c_p} & A_p^* \\ A_p & -I_{r_p} \end{pmatrix}, \quad U = \begin{pmatrix} 0 & \beta C_p^* \\ \alpha B_p & 0 \end{pmatrix}.$$

The exceptional Schur complement is

$$S = H_E - U R_p U^*.$$

Put

$$m_c = \frac{1}{p} \operatorname{Tr} R_{cc}, \qquad m_r = \frac{1}{p} \operatorname{Tr} R_{rr},$$

$$m_c^0 = -\frac{1}{1+b_p}, \qquad m_r^0 = -b_p(1+b_p),$$

and

$$D_{\mathrm{ref}} = \operatorname{diag}(s_{\mathrm{col}} I_{c_p}, s_{\mathrm{row}} I_{r_p}),$$

where

$$s_{\text{col}} = -z + \beta^2 b_p(1+b_p), \qquad s_{\text{row}} = -1 + \frac{\alpha^2}{1+b_p}. \tag{90}$$

At $x = (1+b_p)^2$,

$$|\Re s_{\text{col}}| = (1+b_p)\big(1+(1-\beta^2)b_p\big) \geq 1+b_p,$$

and

$$|s_{\text{row}}| = \frac{b_p+1-\alpha^2}{1+b_p} \geq \frac{1}{2}.$$

The allowed displacement and $\eta_p$ tend to zero, so for all sufficiently large $p$,

$$\min(|s_{\text{col}}|, |s_{\text{row}}|) \geq \frac{1}{3}. \tag{91}$$

Only $s_{\text{col}}$ has nonzero imaginary part, and therefore

$$\max(|\Im s_{\text{col}}^{-1}|, |\Im s_{\text{row}}^{-1}|) \leq C\eta_p \leq Cp^{-5/16}. \tag{92}$$

Condition on $D_p$ and set

$$K_C = C_p^* R_{rr} C_p, \quad K_B = B_p R_{cc} B_p^*, \quad K_{CB} = C_p^* R_{rc} B_p^*, \quad K_{BC} = B_p R_{cr} C_p.$$

Gaussian linear, bilinear, and quadratic-form concentration applied to these blocks gives, for every fixed $\epsilon > 0$,

$$\begin{aligned} &\|K_C - m_r I_{c_p}\|_{\text{op}} + \|K_B - m_c I_{r_p}\|_{\text{op}} + \|K_{CB}\|_{\text{op}} + \|K_{BC}\|_{\text{op}} \\ &\qquad \leq C P_p p^{\epsilon} \left(p^{-1/3+\omega_0} + p^{-1/3+12\kappa_0}\right). \end{aligned} \tag{93}$$

Since $\omega_0 < 1/48$ and $P_p$ is polylogarithmic, the right side is at most $Cp^{-5/16}$ for sufficiently small $\epsilon$ and large $p$. The same concentration estimates give

$$\max_{e\in E, k\in K} \left(|(UR_p)_{e,k}| + |(R_p U^*)_{k,e}|\right) \leq Cp^{-5/16}. \tag{94}$$

These estimates are uniform in $\alpha, \beta \in [0,1]$. Combining (49), (93), and $\|A_p\|_{\text{op}} \leq p^{-5/16}$ yields

$$\|S - D_{\text{ref}}\|_{\text{op}} \leq Cp^{-5/16}. \tag{95}$$

The block inverse

$$\begin{aligned} &G_{EE} = S^{-1}, \quad G_{EK} = -S^{-1}UR_p, \quad G_{KE} = -R_p U^* S^{-1}, \\ &G_{KK} = R_p + R_p U^* S^{-1} U R_p \end{aligned} \tag{96}$$

together with (91)–(95) gives

$$\max_{e\in E,k\in K}\big(|G_{e,k}|+|G_{k,e}|\big)\le C\sqrt{2P_p}\,p^{-5/16}\le Cp^{-61/200},$$

because

$$\frac{5}{16}-\frac{61}{200}=\frac{3}{400}.$$

The diagonal and imaginary-diagonal estimates are

$$C\left(p^{-5/16}+P_p p^{-5/8}\right)\le Cp^{-61/200}.$$

This proves (52)–(53).

The exact endpoint derivative identities (45) give

$$|Y_p'(0)|\le Cp^{-61/200},$$
$$|Y_p''(0)|\le C\left(p^{-122/200}+p^{-61/200}\right)\le Cp^{-61/200},$$

which proves (54). Since

$$\frac{61}{200}-\frac{3}{10}=\frac{1}{200},\qquad \log p\le p^{1/200}$$

for sufficiently large $p$, (55) follows.

For the final assertion, (56) is the Schur-complement inverse. The Ward identities give

$$\|Q_pe_k\|_2^2=\eta_p^{-1}\Im(Q_p)_{k,k},$$
$$\|Q_pD_p^*e_a\|_2^2=\eta_p^{-1}\Im(D_pQ_pD_p^*)_{a,a}. \tag{97}$$

Moreover,

$$\|D_pQ_pe_k\|_2\le C_0\|Q_pe_k\|_2,$$
$$\|(-I_q+D_pQ_pD_p^*)e_a\|_2\le 1+C_0\|Q_pD_p^*e_a\|_2.$$

Because

$$\frac{2/3+12\kappa_0-5/16}{2}=\frac{1}{6}+\omega_0,$$

these estimates give (51). Summing (97) over at most $(1+C_0)p$ indices gives

$$\|R_p\|_F^2\le Cp+C\eta_p^{-1}\sum_k|\Im R_p(z)_{k,k}|\le Cp^{4/3+\nu_0+12\kappa_0},$$

which is the Frobenius estimate. Finally, $\|Q_p\|_{\mathrm{op}}\le\eta_p^{-1}$ and $\|D_p\|_{\mathrm{op}}\le C_0$ give

$$\|R_p\|_{\mathrm{op}}\le C(1+\eta_p^{-1})\le Cp^{2/3+12\kappa_0}.$$

□

### H.7. Gaussian coupling.

Proposition 4.9 of the main article proves the joint coupling on the complete matrix family. The argument controls the largest underlying arrays in Frobenius norm, cancels the common fixed block in the difference and uses the effective-rank trace-norm estimate for the covariance matrices. It includes all specified restrictions, transposes, zero paddings and single-entry deletions.

## Appendix I. Tail bounds and eigenvalue increments

This section contains the one-step growth, fixed-level margin, and polynomial-mesh path-control proofs cited from the main article.

### I.1. One-step eigenvalue increments.

*Proof of 7.2.* The edge-resolvent computation in Lemma 2.2 gives, for

$$\gamma > 0, \qquad b_\gamma = (1+\sqrt{\gamma})^2,$$

the identities

$$r_\gamma := \gamma g_\gamma(b_\gamma) = \frac{\sqrt{\gamma}}{1+\sqrt{\gamma}}, \qquad \alpha_\gamma := (1-r_\gamma)^{-1} = 1+\sqrt{\gamma}.$$

It also gives the reciprocal-aspect identity and the two partial derivatives of $A$.

We use upper-edge rigidity and the averaged local law [BEKYY14, Theorems 2.10 and 2.4, equation (2.13), and Remark 2.11]. Let $Y = Y_{m,n}$ be the raw matrix and put

$$Q = YY^*/n, \qquad \gamma = m/n \in [c, C], \qquad b_\gamma = (1+\sqrt{\gamma})^2.$$

Write $q_1 \geq \cdots \geq q_m \geq 0$ for all eigenvalues, including zeros, and let $\nu_\gamma$ be the row-side Marchenko–Pastur probability law, including its atom at zero. We use

$$g_\gamma(z) = \int \frac{\nu_\gamma(\mathrm{d}x)}{z-x}.$$

For $0 < \varepsilon < 1/24$, set $s = n^{-2/3+4\varepsilon}$. For every fixed failure exponent $D$, outside an event of probability $O(n^{-D})$,

$$q_1 \in [b_\gamma - s/4, b_\gamma + s/4]. \tag{101}$$

On the same event, the following bounds hold:

$$T_s := \sum_{q_i \leq b_\gamma - s} \frac{1}{n(b_\gamma + 2s - q_i)} \geq \gamma g_\gamma(b_\gamma) - Cn^{-1/3+3\varepsilon}, \tag{102}$$

$$U_s := \sum_{q_i \leq b_\gamma - s} \frac{1}{n^2(b_\gamma + 2s - q_i)^2} \leq Cn^{-2/3+3\varepsilon}. \tag{103}$$

We first match the normalization of the local law, then estimate $U_s$ and $T_s$.

Set $X = (mn)^{-1/4}Y$. Its entries have variance $(mn)^{-1/2}$ and satisfy the independent-entry and normalized fixed-moment assumptions of the cited results. Since

$$Q = \sqrt{\gamma}\, XX^*, \qquad w = z/\sqrt{\gamma},$$

Remark 2.11 gives the row-side estimate after interchanging the dimensions. If $m_{\gamma^{-1}}$ denotes the row-side transform in that reference, then

$$\frac{1}{n}\operatorname{Tr}(z-Q)^{-1} = -\sqrt{\gamma}\,\frac{1}{m}\operatorname{Tr}(XX^* - w)^{-1}, \qquad \gamma g_\gamma(z) = -\sqrt{\gamma}\, m_{\gamma^{-1}}(w).$$

The upper-edge rigidity estimate proves (101). Applying the same normalization to $Y^*$ treats row addition. The positive upper-edge spectral domain stays uniformly away from zero for $\gamma \in [c, C]$.

The square-root density and rigidity give, simultaneously for $s \leq r \leq r_0$,

$$\#\{i : b_\gamma - r < q_i \leq b_\gamma + s/4\} \leq C(nr^{3/2} + n^\varepsilon),$$

where $r_0 > 0$ is fixed and small. A dyadic union contributes a logarithmic factor, which is absorbed by choosing the initial stochastic loss smaller. At $z = b_\gamma + 2s + is$, the averaged local law gives

$$\left|\frac{1}{n}\operatorname{Tr}(z-Q)^{-1} - \gamma g_\gamma(z)\right| \leq \frac{Cn^\varepsilon}{ns}.$$

These identities include all structural zero eigenvalues.

Split the upper-edge region into shells $2^k s \le b_\gamma - q_i < 2^{k+1}s$, up to $r_0$. The shell count yields

$$\begin{aligned} U_s &\le C \sum_{2^k s \le r_0} \frac{n(2^k s)^{3/2} + n^\varepsilon}{n^2 (2^k s)^2} + \frac{C}{n} \\ &\le C\left(n^{-1}s^{-1/2} + n^{-2+\varepsilon}s^{-2} + n^{-1}\right) \\ &= C\left(n^{-2/3-2\varepsilon} + n^{-2/3-7\varepsilon} + n^{-1}\right) \le Cn^{-2/3+3\varepsilon}. \end{aligned}$$

The spectrum away from the upper edge contributes $O(n^{-1})$, since $m = O(n)$ and the denominators there stay bounded away from zero. This proves (103).

For (102), put $E = b_\gamma + 2s$. By (101), all $E - q_i$ are positive, and

$$\frac{1}{E - q_i} \ge \Re \frac{1}{E + \mathrm{i}s - q_i}.$$

At most $C(ns^{3/2} + n^\varepsilon)$ eigenvalues exceed $b_\gamma - s$; each corresponding real resolvent summand is at most $C/s$. Thus

$$\begin{aligned} T_s &\ge \Re\left(\frac{1}{n}\operatorname{Tr}(E + \mathrm{i}s - Q)^{-1}\right) - C\left(s^{1/2} + \frac{n^\varepsilon}{ns}\right) \\ &\ge \gamma \Re g_\gamma(E + \mathrm{i}s) - C\left(s^{1/2} + \frac{n^\varepsilon}{ns}\right). \end{aligned}$$

The explicit Marchenko–Pastur transform has a uniform square-root expansion at the positive upper edge:

$$|g_\gamma(b_\gamma + 2s + \mathrm{i}s) - g_\gamma(b_\gamma)| \le Cs^{1/2}.$$

Consequently,

$$T_s \ge \gamma g_\gamma(b_\gamma) - C\left(n^{-1/3+2\varepsilon} + n^{-1/3-3\varepsilon}\right) \ge \gamma g_\gamma(b_\gamma) - Cn^{-1/3+3\varepsilon},$$

which proves (102).

Condition on $Y_{m,n}$. If $B$ is deterministic and positive semidefinite, and $z$ has independent centered variance-one coordinates with uniformly bounded moments, then direct expansion into centered monomials gives, for each fixed even $p$,

$$\mathbb{E}_z|z^* B z - \operatorname{Tr} B|^p \le C_p\left[(\operatorname{Tr} B^2)^{p/2} + \|B\|_{\mathrm{op}}^p\right].$$

Every surviving coordinate index occurs at least twice. Hölder's inequality groups paired indices into $\operatorname{Tr} B^2$ and fully repeated indices into $\|B\|_{\mathrm{op}}$. Markov's inequality, with $p$ chosen in terms of $D$, yields

$$|z^* B z - \operatorname{Tr} B| \le n^\epsilon \left[(\operatorname{Tr} B^2)^{1/2} + \|B\|_{\mathrm{op}}\right]$$

outside an event of conditional probability at most $C_D n^{-D}$.

Take $z = g$ and

$$B_s = \mathbf{1}_{\{Q \le b_\gamma - s\}}\left[n(b_\gamma + 2s)I - Y_{m,n}Y_{m,n}^*\right]^{-1}.$$

Then

$$\operatorname{Tr} B_s = T_s, \qquad \operatorname{Tr} B_s^2 = U_s, \qquad \|B_s\|_{\mathrm{op}} \le \frac{1}{3ns}.$$

It follows that, with overwhelming probability,

$$g^* B_s g \ge r_\gamma - Cn^{-1/3+4\epsilon}. \tag{104}$$

The same expansion applied to $u^* g$ gives

$$W_{\mathrm{col}} \le n^{\epsilon}$$

with overwhelming probability.

Write

$$L = L_{m,n}, \qquad \delta = L_{m,n+1} - L, \qquad W = W_{\mathrm{col}}, \qquad r = r_\gamma, \qquad \alpha = \alpha_\gamma.$$

Choose an eigenbasis $(u_i)$ of $Y_{m,n}Y_{m,n}^*$ beginning with $u$. For $t > 0$, the decreasing rank-one secular function

$$F(t) = \sum_i \frac{|\langle u_i, g\rangle|^2}{L + t - \lambda_i}$$

crosses 1 at $t = \delta$. In a multiple top eigenspace, its other summands remain. If the whole top projection vanishes, then $W = 0$, and the desired lower bound follows from $\delta \ge 0$.

Let

$$E_n = K_0 n^{-1/3+\rho}, \qquad t_0 = \alpha W - E_n.$$

If $t_0 \le 0$, then $\delta \ge 0$ suffices. Suppose that $t_0 > 0$, and choose $\epsilon < \rho/10$. By (101), the bound $W \le n^{\epsilon}$, and

$$t_0/n = o(s),$$

we have

$$L + t_0 \le n(b_\gamma + 2s).$$

Moreover, $q_1 > b_\gamma - s$ makes the subspace defining $B_s$ orthogonal to the top eigenspace. Hence (104) gives

$$F(t_0) \ge \frac{W}{t_0} + r - e_n, \qquad e_n = C n^{-1/3+4\epsilon}.$$

Since $1 - r = 1/\alpha$,

$$\frac{W}{t_0} - (1-r) = \frac{E_n}{\alpha t_0} \ge \frac{E_n}{\alpha^2 W}.$$

The aspect box bounds $\alpha$, and $W \le n^{\epsilon}$. Since

$$\rho - \epsilon > 4\epsilon,$$

we may choose $K_0$ uniformly so that the last quantity is at least $e_n$. Therefore $F(t_0) \ge 1$, so $\delta \ge t_0$, and

$$L_{m,n+1} - L_{m,n} \ge \left(1 + \sqrt{\frac{m}{n}}\right) W_{\mathrm{col}} - C_{\rho,D} n^{-1/3+\rho}.$$

Apply the same argument to $Y_{m,n}^* Y_{m,n}$, viewing the fresh row as a column of $Y_{m,n}^*$, and replace $\gamma$ by $1/\gamma$. The reciprocal-aspect endpoint identity gives the coefficient

$$1 + \sqrt{\frac{n}{m}}.$$

Comparability of the dimensions converts the error to the displayed $n$-scale. A union bound gives both inequalities simultaneously.

Finally,

$$A(m, n+1) - A(m,n) = 1 + \frac{2\sqrt{m}}{\sqrt{n+1} + \sqrt{n}} = 1 + \sqrt{\frac{m}{n}} + O(n^{-1}),$$

and

$$A(m+1, n) - A(m,n) = 1 + \frac{2\sqrt{n}}{\sqrt{m+1} + \sqrt{m}} = 1 + \sqrt{\frac{n}{m}} + O(n^{-1}).$$

These coefficients agree with the partial derivatives from Section 2 up to a uniform finite-difference error. Subtraction proves the centered bounds. Conditionally on $Y_{m,n}$, the vectors $u, v$ are fixed and $g, h$ are fresh covariance-identity vectors. Hence

$$\mathbb{E}(W_{\mathrm{col}} \mid Y_{m,n}) = \mathbb{E}(W_{\mathrm{row}} \mid Y_{m,n}) = 1.$$

The same centered-monomial expansion gives all conditional moments uniformly, including those of $W_{\mathrm{col}} - 1$ and $W_{\mathrm{row}} - 1$. □

### I.2. **Derivative bound for shifted cutoffs.**

*Proof of 7.3.* **Step 1. Exact two-coordinate Ornstein–Uhlenbeck source.** Write

$$\xi_e = q_{e,1} + \mathrm{i}q_{e,2}.$$

The assumptions

$$\mathbb{E}|\xi_e|^2 = 1, \qquad \mathbb{E}\xi_e^2 = 0$$

give covariance $\frac{1}{2}I_2$ for $(q_{e,1}, q_{e,2})$. The generator of the interpolation is

$$\sum_e \left\{ \frac{1}{4}\left(\partial^2_{q_{e,1}} + \partial^2_{q_{e,2}}\right) - \frac{1}{2}\left(q_{e,1}\partial_{q_{e,1}} + q_{e,2}\partial_{q_{e,2}}\right) \right\}. \tag{105}$$

In the multivariate cumulant expansion of the drift, the second-cumulant term cancels the diffusion term in (105). A raw-coordinate derivative is $N^{-1/2}D_a$. Resolvent differentiation and the fundamental theorem of calculus at the two spectral endpoints give

$$D_a\mathcal{X}_{K,j} = -\Delta_{K,j}\Im J_a. \tag{106}$$

Consequently, extracting the first derivative of

$$\widehat{F}_j(\mathcal{X}_{K,j}) = F_j(Z_{K,j})$$

shows that the cumulant term of order $r = 3, 4, 5$ is exactly $\mathcal{C}_{r,K,j}$. Expanding the drift through total cumulant order five leaves the sixth-order Taylor remainder $\mathcal{R}_{6,K,j}$. This proves (187). Independence of the Gaussian part and the vanishing of its cumulants above order two give the displayed formula for $\mathcal{K}(t)$.

*Step 2. Scalar-cutoff joint expansion.* For each fixed $j$, the coefficient family is $\widehat{F}_j^{(p)}(\mathcal{X}_{K,j})$, $p \le K_{\mathrm{der}}$, where $\widehat{F}_j(u) = F_j(u/\pi)$. Every fixed derivative is uniformly bounded. A coordinate derivative appends a derivative of $\mathcal{X}_{K,j}$ and increases $p$; the endpoint formula opens the new factor while the scalar coefficient stays in the same expectation. These are the scalar joint classes of Supplementary Lemma C.1.

Use one matrix with empty old row and column sets. Since $J, q_R, q_L$ are fixed,

$$r_N\sigma_{+,N} + s_N\sigma_{+,N} + J\ell_N = o(N^{-2/3+\kappa_0}),$$

so every endpoint is in the window of Theorem 4.8. Choose the initial factor and spectral-derivative budgets for this finite family, then take $K_{\mathrm{der}} \ge k_4 + 6$ from Lemma C.1, and finally choose the moment orders and failure exponent. That lemma gives, for the complete scalar-coefficient contractions,

$$\mathcal{R}_{d,F,k,0} \le C_F\left(\mathcal{R}_{d+1,F_+,k_+,0} + N^{-1/2}\mathcal{R}_{d,F_+,k_+,1}\right) + C_{F,\epsilon}N^{-1+\epsilon}, \tag{107}$$

$$\mathcal{R}_{d,F,k,1} \le C_F\mathcal{R}_{d+1,F_+,k_+,1} + C_{F,\epsilon}N^{-1/2+\epsilon}, \tag{108}$$

where $F_+ = 4F + 10$ and $k_+ = k + 4F + 20$. The scalar coefficient is kept inside the expectation; the same-degree multiplicity decreases by two exactly as in C.2. The centered-diagonal mixed derivative is estimated by (103)–(104). The deleted-matrix and insertion bounds are supplied by

Theorem 4.5, using the full linearized inverse bound in Supplementary Lemma B.1.

At degree four the local-law bound is $O(N^\epsilon \Psi_N^4)$, with $\Psi_N^4 = o(N^{-1})$. Finite substitution in (107)–(108) gives $|\mathbb{E}Q_d| \le CN^{-1+\epsilon}$ for each initial degree $d \ge 1$. In particular, for every fixed $D_* \ge 4$ it gives the stated form

$$|\mathbb{E}Q_d| \le CN^\epsilon \left(\Psi_N^{D_*} + N^{-1/2}\Psi_N^{D_*-1} + N^{-1}\right) + CN^{-20}. \tag{109}$$

The failure exponent is chosen after all the finite budgets, as in C.2. This proves (191).

*Step 3. Third cumulants.* Every term of $\mathcal{C}_{3,K,j}$ has three rank-two source insertions, including terms in which a derivative hits the scalar coefficient. The source row and column each occur three times, so all eight ordered real triples start the joint unmatched recursion. Their coordinate-dependent cumulants remain in the coefficients. The normalization is $N^{1/2}$ times a normalized two-index sum. Consequently

$$|\mathcal{C}_{3,K,j}| \le CN^{-1/2+\epsilon} + CN^{1/2+\epsilon}\Psi_N^{D_*} + CN^{-20}. \tag{110}$$

Take $D_* = 4$. Since $\nu < 1/12$, $1/2 - 4/3 + 4\nu < -1/2$, and a sufficiently small $\epsilon$ makes (110) an $O(N^{-\sigma_3})$ bound with $\sigma_3 > 1/3$.

*Step 4. Fourth and fifth cumulants and the remainder.* On the simultaneous local-law event from Section 4, endpoint cross entries, imaginary diagonal parts, centered diagonals and centered partial traces are $O(\Psi_N)$, and all diagonal entries are $O(1)$. Lemma 3.4 gives $O(\Psi_N)$ for every fixed positive coordinate derivative of $Z_{K,j}$; mixed real-coordinate derivatives follow by polarization of its bounded-direction estimate. The Faà di Bruno formula and the product rule therefore bound every order-four label integrand by

$$O(\Psi_N)\mathbf{1}_{\{Z_{K,j}\in S_j\}}. \tag{111}$$

The source normalization $N^{-2}$ cancels the $O(N^2)$ source labels. In right mode, the shifted-threshold domination in Lemma 2.9 gives

$$\mathbf{1}_{\{Z_{R,j}\in S_j\}} \le F_{j+1}(Z_{R,j+1}).$$

In left mode, its lower-threshold counterpart gives

$$\mathbf{1}_{\{Z_{L,j}\in S_j\}} \le F_{j+1}(Z_{L,j+1}).$$

The local-law complement is absorbed by an arbitrary polynomial failure exponent. Therefore

$$|\mathcal{C}_{4,K,j}| \le C\Psi_N f_{K,j+1}(t) + CN^{-D_0} \tag{112}$$

for every fixed $D_0$.

For order five, every fixed derivative product is $N^\epsilon$-bounded on the good event. Its normalization $N^{-5/2}$, together with $O(N^2)$ source labels, gives

$$|\mathcal{C}_{5,K,j}| \le CN^{-1/2+\epsilon}. \tag{113}$$

For the sixth-order cumulant remainder, truncate one raw coordinate, start from the coordinate-deleted matrix supplied by the local law in Section 4, and use the same Woodbury-segment argument as in Step 2. Six raw-coordinate derivatives contribute $N^{-3}$. After summing over $O(N^2)$ labels,

$$|\mathcal{R}_{6,K,j}| \le CN^{-1+\epsilon}. \tag{114}$$

The deterministic resolvent norm and an arbitrary failure exponent control the complement.

Choose $\epsilon > 0$ small and $D_0$ large. Equations (110), (113), and (114) have a common decay exponent $\sigma > 1/3$. Since

$$\Psi_N = N^{-1/3+\nu},$$

equations (187), (112), (113), and (114) prove (188)–(190). □

### I.3. **Right-tail bound for tall matrices.**

*Proof of Theorem 7.4, clause 1.* Fix $q > A_2$, and choose

$$A_2 < q_* < q.$$

Set

$$\kappa_0 = \frac{1}{1000}, \qquad \nu = \frac{1}{48}, \qquad \rho = \frac{1}{3} - \nu = \frac{5}{16}, \qquad J = 2.$$

Then

$$0 < \kappa_0 < \frac{1}{288}, \qquad 12\kappa_0 < \nu < \frac{1}{12}, \qquad J\rho = \frac{5}{8} > \frac{1}{3}.$$

Put

$$\eta_N = N^{-2/3-12\kappa_0}, \qquad \ell_N = N^{-2/3-6\kappa_0},$$
$$E_N = E_{+,N} + q r_N \sigma_{+,N}, \qquad B_N = E_{+,N} + 4N^{-2/3+\kappa_0},$$

and

$$E_{N,j} = E_N - (2j+1)\ell_N, \qquad j = 0, 1, 2.$$

Choose smooth nondecreasing functions $F_0, F_1, F_2 \colon \mathbb{R} \to [0,1]$, with every fixed derivative bounded, such that

$$F_0(u) = 0 \text{ for } u \le \frac{3}{4}, \qquad F_0(u) = 1 \text{ for } u \ge \frac{7}{8},$$
$$F_1(u) = 0 \text{ for } u \le \frac{1}{2}, \qquad F_1(u) = 1 \text{ for } u \ge \frac{3}{4},$$

and

$$F_2(u) = 0 \text{ for } u \le \frac{1}{4}, \qquad F_2(u) = 1 \text{ for } u \ge \frac{1}{2}.$$

Such functions exist by smooth interpolation on the displayed compact intervals. For $j = 0, 1$, every positive derivative of $F_j$ vanishes outside

$$S_0 = [3/4, 7/8], \qquad S_1 = [1/2, 3/4],$$

respectively, while

$$F_{j+1}(u) = 1 \qquad \text{for } u \ge \inf S_j.$$

Let $W_N$ be an independent $M_N \times N$ matrix with independent centered circular complex Gaussian entries of variance $1/N$. Put

$$X_N(t) = \mathrm{e}^{-t/2} X_N + \sqrt{1 - \mathrm{e}^{-t}}\, W_N, \qquad 0 \le t \le 8 \log N,$$
$$L_N(t) = X_N(t)^* X_N(t), \qquad m_N(t,z) = \frac{1}{N} \operatorname{Tr}((L_N(t) - zI)^{-1}),$$
$$Z_{N,j}(t) = \frac{N}{\pi} \int_{E_{N,j}}^{B_N} \Im m_N(t, y + \mathrm{i}\eta_N)\, \mathrm{d}y, \qquad f_{N,j}(t) = \mathbb{E} F_j(Z_{N,j}(t)).$$

Apply Proposition 7.3 in right mode with the displayed parameters. Its Gaussian variables are $\sqrt{N}$ times the entries of $W_N$. Its statistic equals $Z_{N,j}$, since its projected trace is

$$\operatorname{Tr}((L_N(t) - zI)^{-1}) = N m_N(t,z).$$

Therefore there are constants $C < \infty$, $\sigma > 1/3$, and $N_1$ such that, for $N \ge N_1$, $j = 0, 1$, and almost every $t \in [0, 8\log N]$,

$$|f'_{N,j}(t)| \le C N^{-\rho} f_{N,j+1}(t) + C N^{-\sigma}. \tag{115}$$

The functions used below are absolutely continuous. For fixed $N$, the condition $\eta_N > 0$ makes every resolvent entry and integrated statistic smooth in the finitely many real and imaginary matrix-entry coordinates, with derivatives of polynomial growth. All entry moments are finite. The finite-dimensional Ornstein–Uhlenbeck semigroup formula makes $f_{N,j}$ locally absolutely continuous on $(0, 8\log N]$ and continuous at zero. The uniform integrable bound (115) extends this

absolute continuity to the closed interval.

Corollary 4.2, with smoothing exponent $D_S = 1$, gives

$$\mathbb{P}(\lambda_N \geq E_{+,N} + q r_N \sigma_{+,N}) \leq f_{N,0}(0) + N^{-1} \tag{116}$$

for all sufficiently large $N$.

For $j = 0, 1, 2$, define

$$Z^W_{N,j} = \frac{N}{\pi} \int_{E_{N,j}}^{B_N} \Im m^W_N(y + \mathrm{i}\eta_N)\, \mathrm{d}y,$$

where

$$m^W_N(z) = \frac{1}{N} \operatorname{Tr}((W^*_N W_N - zI)^{-1}),$$

and put

$$g_{N,j} = \mathbb{E} F_j(Z^W_{N,j}).$$

The endpoints satisfy the window condition of Proposition 4.9 . Indeed,

$$\sigma_{+,N} = O(N^{-2/3}),$$

and the fixed logarithmic factor $r_N$ is at most $N^{\kappa_0}$ for all sufficiently large $N$. The shifts $\ell_N$ are smaller. Hence every $E_{N,j}$ and $B_N$ differs from $E_{+,N}$ by

$$O(N^{-2/3+\kappa_0}).$$

Apply the scalar Lipschitz conclusion of Proposition 4.9 to each fixed $F_j$. The common Lipschitz bound and $\kappa_0 = 1/1000$ give

$$|f_{N,j}(8 \log N) - g_{N,j}| \leq C_\epsilon N^{-4/3+25\kappa_0+\epsilon}.$$

Since $4/3 - 25\kappa_0 = 157/120$, choose $\epsilon < 37/120$ and absorb the fixed constant by taking $N$ large. Thus one may take $D_T = 1$ in

$$|f_{N,j}(8 \log N) - g_{N,j}| \leq N^{-D_T}, \qquad 0 \leq j \leq 2. \tag{117}$$

The second right-tail inequality in Corollary 4.2, applied separately to $W_N$ and each $F_j$, gives

$$g_{N,j} \leq \mathbb{P}(\lambda_{\max}(W^*_N W_N) \geq E_N - (2j+2)\ell_N) + N^{-1}. \tag{118}$$

Lemma 2.1, item 3, with shift exponent $6\kappa_0$ and fixed shift integer 6, gives, for all sufficiently large $N$,

$$q r_N \sigma_{+,N} - 6\ell_N \geq q_* r_N \sigma_{+,N}. \tag{119}$$

Since $2j + 2 \leq 6$, equations (118)–(119) imply

$$g_{N,j} \leq \mathbb{P}(\chi_{N,W} \geq q_* r_N) + N^{-1},$$

where

$$\chi_{N,W} = \frac{\lambda_{\max}(W^*_N W_N) - E_{+,N}}{\sigma_{+,N}}.$$

Let $G_N = \sqrt{N}\, W_N$. Its entries are independent standard circular complex Gaussian variables. The nonzero eigenvalues of

$$N^{-1} G_N G^*_N \quad \text{and} \quad N^{-1} G^*_N G_N = W^*_N W_N$$

coincide, including multiplicities, so their largest eigenvalues agree. The right-tail clause of Proposition 7.1, applied with $q_*$, supplies an exponent $\eta^{G,+}_{q_*} > 1/3$ such that

$$\mathbb{P}(\chi_{N,W} \ge q_* r_N) \le N^{-\eta^{G,+}_{q_*}}$$

for all sufficiently large $N$. Choose $c_0$ with

$$\frac{1}{3} < c_0 < \min\left\{1, \eta^{G,+}_{q_*}\right\}.$$

After increasing the deterministic size threshold,

$$\mathbb{P}(\chi_{N,W} \ge q_* r_N) \le N^{-c_0}.$$

Hence

$$g_{N,j} \le 2N^{-c_0}, \qquad 0 \le j \le 2.$$

Choose

$$\frac{1}{3} < c_G < c_0.$$

Increasing the deterministic size threshold absorbs the factor 2 and gives

$$g_{N,j} \le N^{-c_G}, \qquad 0 \le j \le 2. \tag{120}$$

Set

$$D = \min\{1, D_T\} > \frac{1}{3}.$$

Apply Proposition 2.7 with

$$T_N = 8 \log N, \qquad H_N^X = \{\chi_N \ge q r_N\}, \qquad H_N^G = \varnothing,$$

the functions $f_{N,j}$, the terminal values $g_{N,j}$, $\rho = 5/16$, the exponent $\sigma$ from (115), the exponent $c_G$ from (120), and approximation exponent $D$. The time condition holds with $C_T = 8$. Equations (116), (117), (115), and (120) verify the remaining hypotheses. Moreover,

$$J\rho = \frac{5}{8} > \frac{1}{3}.$$

The closure therefore gives $\theta_q > 1/3$ and $N_q$ such that

$$\mathbb{P}(\chi_N \ge q r_N) \le N^{-\theta_q}$$

for every $N \ge N_q$. □

### I.4. **Right-tail bound.**

*Proof of Theorem 7.4, clause 2.* Apply Lemma F.2 with

$$\beta = \frac{2}{3}, \qquad C_* = A_2, \qquad \varepsilon = 1.$$

Its tall-sequence hypothesis is Theorem 7.4, clause 1. Indeed, for a tall source sequence, the nonzero eigenvalues of

$$n^{-1} X X^* \quad \text{and} \quad n^{-1} X^* X$$

coincide, while their upper-edge center and scale are the same quantities in the two formulations. Thus their standardized largest eigenvalues agree exactly. The orientation-lifting conclusion gives an exponent $\theta_q > 1/3$ for the event $\{\chi_N^+ \ge q r_N\}$. Since

$$\{\chi_N^+ > q r_N\} \subseteq \{\chi_N^+ \ge q r_N\},$$

the asserted bound follows. □

### I.5. **Left-tail bound.**

*Proof of Theorem 7.4, clause 3.* Apply Lemma F.2 with

$$\beta = \frac{1}{3}, \qquad C_* = B_2, \qquad \varepsilon = -1.$$

Lemma F.1 supplies the estimate for each complete tall sequence, allowing both sequence-dependent exponents and arbitrary dependence between matrix sizes. The orientation-lifting lemma therefore gives $\theta_q > 1/3$ and $N_q$ such that

$$\mathbb{P}(-\chi_N^+ \geq q s_N) \leq N^{-\theta_q}, \qquad N \geq N_q.$$

This event is exactly $\{\chi_N^+ \leq -q s_N\}$, proving the theorem. □

I.6. **Control between consecutive grid points.**

*Proof of 7.5.* Put

$$\theta = 1 - \frac{1}{\alpha}, \qquad \epsilon = \frac{2}{3} - \theta = \frac{1}{\alpha} - \frac{1}{3} > 0.$$

The mean value theorem applied to $x^\alpha$ gives

$$N_k - N_{k-1} = O(N_k^\theta) = O(N_k^{2/3-\epsilon}).$$

Between the project indices $N_{k-1}$ and $N_k$, expose the nested array by elementary steps. For each $N$, first add the new column at the old row count, and then add $M_{N+1} - M_N$ rows one at a time. The number $H_k$ of elementary steps satisfies

$$H_k \leq (K+1)(N_k - N_{k-1}) = O(N_k^{2/3-\epsilon}).$$

Every intermediate rectangle has both dimensions comparable to $N_k$ for large $k$.

At an elementary state $(m, n)$, put

$$Z = L_{m,n} - A(m,n),$$
$$L_{m,n} = \lambda_{\max}(Y_{m,n} Y_{m,n}^*), \quad A(m,n) = (\sqrt{m} + \sqrt{n})^2.$$

Apply Lemma 7.2 with $\rho = \epsilon/4$ and an arbitrarily large failure exponent. At every column or row update,

$$Z_{t+1} - Z_t \geq c_t d_t - e_t. \tag{121}$$

Here $c_t$ is respectively

$$1 + \sqrt{\frac{m}{n}} \qquad \text{or} \qquad 1 + \sqrt{\frac{n}{m}}.$$

It is deterministic and bounded above and below by positive constants. The variable

$$d_t = W_t - 1$$

is adapted to the newly exposed vector and satisfies

$$\mathbb{E}(d_t \mid \mathcal{F}_t) = 0,$$

with uniformly bounded conditional moments of every fixed order. Outside an event of probability $O(N_k^{-D})$ for any prescribed $D$,

$$0 \le e_t \le CN_k^{-1/3+\epsilon/4}.$$

A union bound over $H_k$ steps preserves an $O(N_k^{-D})$ exceptional probability after increasing the failure exponent in Lemma 7.2. On the simultaneous good event,

$$\sum_{t=1}^{H_k} e_t \le CN_k^{1/3-3\epsilon/4}. \tag{122}$$

Let

$$M_j = \sum_{t=1}^{j} c_t d_t.$$

This is a real martingale. For every fixed even integer $p \ge 2$, the conditional-moment statement above and the boundedness of $c_t$ give uniform bounds on

$$\mathbb{E}(|c_t d_t|^r \mid \mathcal{F}_t), \qquad 2 \le r \le p.$$

The self-contained maximal-moment calculation in Appendix D therefore gives

$$\mathbb{E} \max_{j \le H_k} |M_j|^p \le \left(\frac{p}{p-1}\right)^p D_{\mathrm{mom}}(p) H_k^{p/2}. \tag{123}$$

For $N \in [N_{k-1}, N_k]$, the scale $b_N$ is uniformly comparable to $N_k^{1/3}$. Markov's inequality and (123) give, for every fixed $x > 0$,

$$\begin{aligned}\mathbb{P}\left(\max_{j \le H_k} |M_j| \ge x b_{N_k}\right) &\le C_{p,x} H_k^{p/2} b_{N_k}^{-p} \\ &\le C_{p,x} N_k^{-p\epsilon/2}.\end{aligned}$$

Choosing $p$ arbitrarily large makes this $O(N_k^{-D})$ for every $D$. The same conclusion holds with $b_{N_{k-1}}$ in place of $b_{N_k}$.

We next translate the excursion events into martingale deviations. The function

$$b(m,n) = (\sqrt{m} + \sqrt{n})(m^{-1/2} + n^{-1/2})^{1/3}$$

is smooth for positive $m, n$, homogeneous of degree $1/3$, and has logarithmic partial derivatives

$$O((m+n)^{-1})$$

when $m/n$ remains in a compact subset of $(0, \infty)$. Hence, uniformly for $n \in [N_{k-1}, N_k]$,

$$\frac{b_n}{b_{N_k}} = 1 + O\left(\frac{N_k - N_{k-1}}{N_k}\right) = 1 + o(1),$$

and likewise

$$\frac{b_n}{b_{N_{k-1}}} = 1 + o(1).$$

The same ratio conclusions hold for $r_n/r_{N_k}$, $s_n/s_{N_{k-1}}$, and the other endpoint choices because

$$N_k - N_{k-1} = o(N_k).$$

On $E_k^+(q, \delta)$, choose an index $n$ witnessing the event. For all sufficiently large $k$,

$$\frac{Z_{N_k} - Z_n}{b_{N_k}} = \chi_{N_k} - \frac{b_n}{b_{N_k}} \chi_n \le -\frac{\delta}{2} r_{N_k}. \tag{124}$$

Summing the elementary lower increment inequalities from $n$ to $N_k$ gives

$$Z_{N_k} - Z_n \geq M_{H_k} - M_{j(n)} - CN_k^{1/3-3\epsilon/4}, \tag{125}$$

where $j(n)$ is the elementary time corresponding to the project index $n$. Since

$$b_{N_k} \asymp N_k^{1/3}, \qquad r_{N_k} \longrightarrow \infty,$$

the accumulated error in (125) is smaller than

$$\frac{\delta}{4} b_{N_k} r_{N_k}$$

for all sufficiently large $k$. Therefore $E_k^+(q,\delta)$, apart from the simultaneous exceptional event, implies

$$\max_j |M_{H_k} - M_j| \geq \frac{\delta}{4} b_{N_k} r_{N_k},$$

and hence

$$\max_j |M_j| \geq \frac{\delta}{8} b_{N_k} r_{N_k}.$$

The martingale estimate, even with the logarithmic factor omitted, makes this probability $O(N_k^{-D})$ for every $D$.

On $E_k^-(q,\delta)$, choose a witnessing $n$. For all sufficiently large $k$,

$$\frac{Z_n - Z_{N_{k-1}}}{b_{N_{k-1}}} = \frac{b_n}{b_{N_{k-1}}}\chi_n - \chi_{N_{k-1}} \leq -\frac{\delta}{2} s_{N_{k-1}}. \tag{126}$$

Summing the lower increments forward from $N_{k-1}$ to $n$ shows that, outside the same negligible exceptional event, this event implies

$$\max_j |M_j| \geq \frac{\delta}{8} b_{N_{k-1}} s_{N_{k-1}}.$$

The martingale estimate again gives probability $O(N_k^{-D})$. This proves both displayed probability bounds.

The resulting bounds are summable over $k$. For the first consequence, assume that the mesh limsup is at most $q$. Given any $\delta > 0$, almost surely

$$\chi_{N_k} \leq (q + \delta/2) r_{N_k}$$

eventually. The first Borel–Cantelli lemma, applied to

$$E_k^+(q + \delta/2, \delta/2),$$

then gives

$$\chi_n < (q+\delta) r_n$$

throughout every sufficiently late mesh interval. Taking a countable sequence $\delta \downarrow 0$ proves

$$\limsup_{N\to\infty} \frac{\chi_N}{r_N} \leq q.$$

The liminf consequence follows identically from

$$E_k^-(q + \delta/2, \delta/2).$$

Lemma 2.1, item 2, identifies the variable $\chi_N$ used here with the normalized covariance top eigenvalue used elsewhere in the paper. □

## Appendix J. Cutoff constructions

This section contains the Gaussian-count, high- and low-cutoff sum, and endpoint-construction proofs cited from the main article.

## J.1. **Gaussian right-tail counts.**

*Proof of 5.1.* The Gaussian part is an application of the sparse-grid moment estimate from the companion paper [Yan26a, Corollary 3.8], followed by the deterministic submatrix embedding.

We choose the scales recursively. After $k_{l-1}$ has been fixed, set

$$P_l = 2^{k_{l-1}+1}, \qquad A_{P_l}(j) = \max\{1, M_{P_l+j} - M_{P_l}\}.$$

For this fixed prefix $P_l$, the path $A_{P_l}$ is positive and nondecreasing, its increments are bounded by $\max\{K, 1\}$, and

$$\frac{A_{P_l}(j)}{j} \longrightarrow \gamma.$$

Indeed, once $M_{P_l+j} - M_{P_l} > 0$,

$$\frac{A_{P_l}(j)}{j} = \frac{M_{P_l+j}}{P_l + j}\left(1 + \frac{P_l}{j}\right) - \frac{M_{P_l}}{j}.$$

We then choose $k_l \geq k_{l-1} + 2$ sufficiently large, with $N_l = 2^{k_l}$ and $J_l = N_l - P_l$, so that $P_l \leq (\log N_l)^d$, $A_{P_l}(j) = M_{P_l+j} - M_{P_l}$ on $j \in [J_l, 2J_l)$, and every numerical threshold used below, including the threshold in Corollary 3.8, holds for this fixed path at scale $J_l$. This requirement is imposed at every stage, in particular when $l = 1$ and $P_1 = 2$. This is a pointwise-in-prefix diagonal choice; no uniform assertion over all polylogarithmic prefixes is used.

For $j \in [J_l, 2J_l)$, set

$$p_{l,j} = \min\{A_{P_l}(j), j\}, \qquad q_{l,j} = \max\{A_{P_l}(j), j\},$$

and write

$$d_{l,j} = (\sqrt{p_{l,j}} + \sqrt{q_{l,j}})(p_{l,j}^{-1/2} + q_{l,j}^{-1/2})^{1/3}.$$

Apply Corollary 3.8 at the high amplitude $a_1$, with $\theta_+ = 2\eta$, the prescribed $\epsilon$, and the deterministic shifts

$$h_{l,j,+} = \frac{\mu_{p_{l,j},q_{l,j}} - (\sqrt{p_{l,j}} + \sqrt{q_{l,j}})^2}{d_{l,j}} + a_1(\log j)^{2/3}\left(\frac{\sigma_{p_{l,j},q_{l,j}}}{d_{l,j}} - 1\right). \tag{127}$$

Uniformly on this window, $p_{l,j} \asymp q_{l,j} \asymp J_l$, while Taylor expansion gives

$$\mu_{p_{l,j},q_{l,j}} - (\sqrt{p_{l,j}} + \sqrt{q_{l,j}})^2 = O(1), \qquad \frac{\sigma_{p_{l,j},q_{l,j}}}{d_{l,j}} = 1 + O(J_l^{-1}).$$

Hence $h_{l,j,+} = o(1)$ uniformly and is, in particular, an allowed bounded shift. Moreover, the shifted edge-normalized event is exactly

$$\left\{\chi_j \quad \geq a_1(\log j)^{2/3} + h_{l,j,+}\right\} = \left\{\overline{\chi}(Z_l^{(j)}) \geq a_1(\log j)^{2/3}\right\}.$$

Since $\kappa + \epsilon + 2\eta < 1/3$, the corollary gives a grid in $[J_l, 2J_l)$ of exact cardinality

$$|\mathcal{T}_{J_l}^+| = \left\lfloor J_l^{\kappa+2\eta} \right\rfloor.$$

Its translated physical grid $T_l = P_l + \mathcal{T}_{J_l}^+$ lies in $[N_l, 2N_l)$ and has the exact cardinality

$$|T_l| = \left\lfloor J_l^{\kappa+2\eta} \right\rfloor \leq N_l^{\kappa+2\eta} \leq 2N_l^{\kappa+2\eta}.$$

The full second-moment conclusion of the corollary gives deterministic $u_l \geq 0$, tending to zero, with

$$m_l = \mathbb{E}S_l^G \geq N_l^\eta, \qquad \mathbb{E}[(S_l^G)^2] \leq (1 + u_l)m_l^2 + m_l.$$

At step $l$, the prefix $P_l$ is fixed before $k_l$ is chosen. Increase $k_l$ until the corollary's error for that path is at most $1/l$ and the stated mean bound holds. The Gaussian array is independent of $\mathcal{F}_l^{\text{past}}$, so its conditional moments are exactly these same deterministic quantities.

It remains only to pass from the fresh actual corner to the full matrix. For $j = n - P_l$, the matrix $Y_l^{(j)}$ is a row-and-column submatrix of $X^{(n)}$, so

$$\|X^{(n)}\|_{\mathrm{op}} \geq \|Y_l^{(j)}\|_{\mathrm{op}}.$$

Uniformly for $n \in [N_l, 2N_l)$, deletion of $P_l$ columns and $M_{P_l} = O(P_l)$ rows changes the squared soft-edge center by $O(P_l + M_{P_l})$, changes the scale by a factor $1 + o(1)$, and

$$\frac{\log(n - P_l)}{\log n} \longrightarrow 1.$$

Since $a_1 > a_0$ and $P_l \leq (\log N_l)^d = o(N_l^{1/3})$, the submatrix threshold at amplitude $a_1$ is eventually above the full-matrix threshold at amplitude $a_0$. The asserted union containment follows. □

## J.2. Smooth cutoff functions for the right tail.

*Proof of 5.2.* **Step 1. Scales, fresh corners, and conditioning.** Choose a deterministic integer $N_*$ large enough for all uniform estimates used below. Define an auxiliary path by

$$\widetilde{M}_1 = M_1, \qquad \widetilde{M}_N = M_2 \quad (2 \leq N \leq N_*), \qquad \widetilde{M}_N = M_N \quad (N > N_*).$$

This path is nondecreasing, has finite maximum increment, and satisfies

$$\frac{\widetilde{M}_N}{N} \longrightarrow \gamma.$$

Apply Proposition 5.1 to this path with its lower and upper amplitudes equal to the present $a_1$ and $a_2$. Its exponent is

$$\frac{4}{3} a_2^{3/2} = \kappa_2,$$

and its hypothesis is precisely

$$\kappa_2 + \epsilon + 2\eta < \frac{1}{3}.$$

It supplies $k_l, N_l, P_l$, the shifted paths, the mutually independent Gaussian arrays, and grids $T_l$ satisfying

$$|T_l| \leq 2N_l^{\kappa_2 + 2\eta}.$$

For the hard Gaussian indicators at amplitude $a_2$, Proposition 5.1 gives deterministic $u_l \to 0$ and

$$m_l \geq N_l^{\eta}, \qquad \mathbb{E}[(R_l^G)^2] \leq (1 + u_l) m_l^2 + m_l.$$

For $l = 1$, remove from the returned grid every $n \leq N_*$. Such a point has fresh row dimension

$$\widetilde{M}_n - \widetilde{M}_2 = 0.$$

Its Gaussian matrix is a zero-row matrix, its half-shifted edge variable is negative, and its positive-amplitude indicator is identically zero. Removing these points changes neither $R_1^G$ nor either of its moments. Since $m_1 > 0$, at least one retained point satisfies $n > N_*$. Every point of $T_1$ is below $2N_1$, so

$$P_2 = 2N_1 > N_*.$$

For every $l \geq 2$, both $P_l$ and all $n \in T_l$ exceed $N_*$. For retained points with $l = 1$, $P_1 = 2$ and $n > N_*$. In every case,

$$\widetilde{M}_{P_l} = M_{P_l}, \qquad \widetilde{M}_n = M_n.$$

Thus all retained shifted dimensions, fresh corners, centers, scales, and hard indicators are those of the original dimension path. The grid-size and moment bounds remain unchanged.

The Gaussian arrays are independent of $\mathcal{F}_l^{\rm past}$, hence

$$\mathbb{E}[R_l^G \mid \mathcal{F}_l^{\rm past}] = m_l, \qquad \mathbb{E}[(R_l^G)^2 \mid \mathcal{F}_l^{\rm past}] \le (1+u_l)m_l^2 + m_l.$$

almost surely. Lemma 2.5 also shows that the matrices $Y_l^{(j)}$ form a nested coordinate-submatrix path, that their complete fresh array is independent of $\mathcal{F}_l^{\rm past}$, and that its conditional law is the original product law.

Uniformly for $n \in T_l$,

$$j_{l,n} = (1+o(1))n, \qquad \frac{u_{l,n}}{v_{l,n}} = \gamma + o(1),$$

and hence

$$p_{l,n} \asymp q_{l,n} \asymp N_l.$$

Indeed,

$$P_l \le (\log N_l)^d = o(N_l),$$

while $n \in [N_l, 2N_l)$. The convergence $M_N/N \to \gamma$ and the estimate $M_{P_l} = O(P_l)$ yield these relations.

**Step 2. Actual submatrix transfer.** Fix $n \in T_l$ and abbreviate

$$p = p_{l,n}, \qquad q = q_{l,n}, \qquad j = j_{l,n}.$$

Put

$$a_{\rm raw}(p,q) = (\sqrt{p} + \sqrt{q})^2,$$

and let $b_{\rm raw}(p,q)$ denote the corresponding raw edge scale. Since $p$ and $q$ are comparable, mean-value estimates give

$$\mu_{p,q} - a_{\rm raw}(p,q) = O(1), \qquad \frac{\sigma_{p,q}}{b_{\rm raw}(p,q)} = 1 + O(p^{-1})$$

uniformly over the present matrix family. Hence, after increasing $N_*$ if necessary,

$$\overline{\chi}(Y_l^{(j)}) \ge a_1(\log j)^{2/3}$$

implies that the raw normalized edge variable is at least

$$a_{\rm mid}(\log j)^{2/3}, \qquad a_{\rm mid} = \frac{a_0 + a_1}{2}.$$

The matrix $Y_l^{(j)}$ is obtained from $X^{(n)}$ by deleting $M_{P_l}$ rows and $P_l$ columns. The number deleted is $O(P_l)$ because $M_{P_l} = O(P_l)$, and hence is

$$o\big(n^{1/3}(\log n)^{2/3}\big).$$

Proposition 2.3, applied with upper amplitude $a_{\rm mid}$ and lower amplitude $a_0$, gives

$$\overline{\chi}(Y_l^{(j)}) \ge a_1(\log j)^{2/3} \quad \Longrightarrow \quad \chi_n^+ \ge a_0(\log n)^{2/3}.$$

This implication is uniform over $T_l$ and proves the claimed containment.

**Step 3. Exact normalization of the Poisson windows.** Fix one matrix $W$ under consideration and orient it as a $q$-by-$p$ matrix, replacing it by its conjugate transpose when necessary. Set

$$\widehat{X} = p^{-1/2}W.$$

Then $\widehat{X}$ has entry variance $p^{-1}$, and the eigenvalues of its covariance matrix are the squared singular values of $W$, divided by $p$. Put

$$E_r = \frac{t_{r,l,n}}{p}, \qquad r = 1,2.$$

The raw normalized covariance edge is

$$E_+ = \left(1 + \sqrt{\frac{q}{p}}\right)^2 = \frac{(\sqrt{p} + \sqrt{q})^2}{p}.$$

The half-shifted definitions and comparability of $p, q$ give

$$|E_r - E_+| = O\big(p^{-2/3}(\log p)^{2/3}\big).$$

After multiplication by $p$, the three smoothing parameters from Section 3 become

$$p\eta_p = \widehat{\eta}_{l,n}, \qquad p\ell_p = \widehat{\ell}_{l,n}, \qquad pB_p = \widehat{B}_{l,n}.$$

Spectral calculus and the substitution $x = py$ therefore identify the positive-window Poisson count at $E_1$ exactly with $Q_{l,n}(W)$. At $E_2$, the negative-window count is the same positive Poisson integrand integrated from

$$t_{2,l,n} - \widehat{\ell}_{l,n}$$

to $\widehat{B}_{l,n}$.

Proposition 4.1 applies uniformly to this triangular family. If uniformity failed, select failing pairs with strictly increasing $p$. Define the dimension and threshold sequences by these pairs on the selected indices, and fill the remaining indices with

$$q_p = \mathrm{round}(\Gamma p), \qquad E_p = E_+, \qquad \Gamma = \max(\gamma, \gamma^{-1}).$$

The resulting sequences satisfy $q_p/p \to \Gamma$ and the required threshold condition, contradicting that dichotomy.

**Step 4. The actual smoothing event.** For each actual fresh corner, let $G^X_{l,n}$ be the matrix-measurable event on which the two inequalities of Proposition 4.1 hold at $E_1$. Its complement satisfies the probability estimate in that proposition. Choose the failure exponent so large that

$$\mathbb{P}\left(\bigcup_{n \in T_l} (G^X_{l,n})^c\right) \le N_l^{-10},$$

and set

$$\Omega^X_l = \bigcap_{n \in T_l} G^X_{l,n}.$$

Uniform comparability $p \asymp N_l$ and the grid bound give this estimate. The event $\Omega^X_l$ is measurable with respect to the actual submatrix array. Since this array is independent of $\mathcal{F}^{\mathrm{past}}_l$,

$$\mathbb{P}((\Omega^X_l)^c \mid \mathcal{F}^{\mathrm{past}}_l) \le N_l^{-10}$$

almost surely.

For an actual fresh corner, the statistic coordinate satisfies

$$y_{l,n}(Y_l^{(j_{l,n})}) < 0$$

only if

$$Q_{l,n}(Y_l^{(j_{l,n})}) > a_* = \frac{1}{16}.$$

After increasing $N_*$, this is greater than $p^{-\kappa_0}$ at every retained point. On $\Omega^X_l$, the contrapositive of the upper Poisson-count implication gives

$$\lambda_{\max}\left((Y_l^{(j_{l,n})})^* Y_l^{(j_{l,n})}\right) \ge t_{1,l,n}.$$

Equivalently,

$$\overline{\chi}(Y_l^{(j_{l,n})}) \ge a_1(\log j_{l,n})^{2/3}.$$

Step 2 then yields $U_l(a_0)$. Thus the specified target-support hypothesis in Lemma E.3 holds with

$$\mathcal{L} = U_l(a_0) \cup (\Omega_l^X)^c.$$

**Step 5. The retained Gaussian eigenvalue count.** For each Gaussian corner, let $\Omega_{l,n}^G$ be the matrix-measurable success event for the two Poisson inequalities at $E_2$. On this event, the hard hit $H_{l,n}^G$ gives

$$Q_{2,l,n}^- \geq 1 - p^{-\kappa_0},$$

where $Q_{2,l,n}^-$ is the Poisson integral beginning at

$$t_{2,l,n} - \widehat{\ell}_{l,n}.$$

Now

$$t_{2,l,n} - t_{1,l,n} = (a_2 - a_1)(\log j_{l,n})^{2/3} \sigma_{l,n}.$$

Since

$$\sigma_{l,n} \asymp p^{1/3}, \qquad \widehat{\ell}_{l,n} = p^{1/3 - 6\kappa_0},$$

the choice of $N_*$ makes this difference greater than $2\widehat{\ell}_{l,n}$ at every retained point. Positivity of the Poisson integrand therefore gives

$$Q_{l,n}(Z_l^{(j_{l,n})}) \geq Q_{2,l,n}^- \geq 1 - p^{-\kappa_0}.$$

This exceeds $1/2$ at every retained point. Since

$$\delta_l = \frac{a_*}{J_{\mathrm{c},l} + 2}, \qquad a_* = \frac{1}{16},$$

we obtain

$$y_{l,n}(Z_l^{(j_{l,n})}) + J_{\mathrm{c},l} \leq -1.$$

Hence every retained indicator obeys

$$\widetilde{H}_{l,n}^G = 1 \quad \Longrightarrow \quad y_{l,n}(Z_l^{(j_{l,n})}) \leq -J_{\mathrm{c},l} - 1.$$

Choose the Gaussian failure exponent sufficiently large that

$$\sum_{n \in T_l} \mathbb{P}((\Omega_{l,n}^G)^c) = o(N_l^\eta).$$

Since $\widetilde{R}_l^G \leq R_l^G$,

$$0 \leq m_l - \widetilde{m}_l \leq \sum_{n \in T_l} \mathbb{P}((\Omega_{l,n}^G)^c) = o(m_l).$$

Thus

$$\widetilde{m}_l = (1 + o(1)) m_l.$$

Set $q_l = (m_l - \widetilde{m}_l)/m_l = o(1)$. The retention bound following Lemma 2.6 gives

$$\frac{\mathrm{Var}(\widetilde{R}_l^G \mid \mathcal{F}_l^{\mathrm{past}})}{\widetilde{m}_l^2} \leq \frac{1 + u_l + m_l^{-1}}{(1 - q_l)^2} - 1 = o(1).$$

The Gaussian array and the smoothing events are independent of the past, so all errors are deterministic. Thus $\widetilde{m}_l \geq N_l^{\eta/2}$ eventually and the retained second moment is $(1 + o(1))\widetilde{m}_l^2$.

**Step 6. Smoothness and cutoff sum absorption.** For $\widehat{\eta}_{l,n} > 0$, the trace

$$\mathrm{Tr}\,(W^*W - (x + \mathrm{i}\widehat{\eta}_{l,n})I)^{-1}$$

is a rational real-analytic function of the real and imaginary entry coordinates. Its denominator never vanishes because the spectral parameter has positive imaginary part. Integration over a fixed finite interval preserves smoothness. Hence every $Q_{l,n}$, $y_{l,n}$, and $A_{l,n,r}$ is $C^\infty$.

Apply Lemma E.3 with

$$\mathcal{G} = \mathcal{F}_l^{\mathrm{past}}, \quad \mathcal{H} = \Omega, \quad I = T_l, \quad J = J_{\mathrm{c},l}, \quad R = R_{\mathrm{c}},$$

where $\Omega$ is the whole sample space, with the target and terminal hypotheses above and the retained-count moment package from Step 5. It gives the proposition's target-support and terminal inequalities, preserves the conditional moments, and gives the asserted first-, second-, and noncubic third-derivative estimates. Its final-layer hypothesis follows from

$$\begin{aligned} |T_l| R_{\mathrm{c}}^{-J_{\mathrm{c},l}} &\le 2N_l^{\kappa_2 + 2\eta - \zeta \log R_{\mathrm{c}}} \\ &\le 2N_l^{-A_{\mathrm{fin}} - \epsilon/2} \le N_l^{-A_{\mathrm{fin}}} \end{aligned}$$

for all sufficiently large $l$, and

$$\delta_l^{-1} = \frac{J_{\mathrm{c},l} + 2}{a_*} = O(\log N_l).$$

The support-absorption identity separates the term containing the third coordinate derivative of $y_{l,n}$; the displayed derivative bound applies to the remaining terms. □

## J.3. **Change of normalization.**

*Proof of 5.3.* The fresh family in Proposition 5.2 is northwest nested. Therefore every matrix in $T_l$ is a deterministic restriction, possibly followed by conjugate transpose, of the single array indexed by $\mathcal{E}_N$. Formula (109) is consequently one function of the complete array.

If $\lambda_s(W)$ is an unnormalized squared singular value, then

$$\lambda_s(N^{-1/2}W) = \frac{\lambda_s(W)}{N}.$$

The change of variables $x = N\widetilde{x}$ gives

$$\Im\left(\frac{\lambda_s(W)}{N} - \widetilde{x} - \frac{\mathrm{i}\widehat{\eta}}{N}\right)^{-1} \mathrm{d}\widetilde{x} = \frac{\widehat{\eta}}{(\lambda_s(W) - x)^2 + \widehat{\eta}^2}\, \mathrm{d}x.$$

After summing the $p_{l,n}$ singular values and multiplying by $1/\pi$, this proves that $Z_n$ equals $Q_{l,n}$. Substitution into the definitions of $y_n$ and the cutoff sum proves (110) and shows that every physical and cutoff sum term occurs exactly once.

For fixed $e$, (111) restricts the single function $S_N$ to the affine coordinate plane on which only entry $e$ varies. Restoring $h_e(t)$ restores the complete array, proving (112). If $e$ is outside a matrix, the restriction to that matrix is unchanged. If $e$ belongs to a matrix, that matrix is obtained by a fixed coordinate restriction and possible conjugate transpose. Differentiating the composite gives the corresponding coordinate derivative of the global section.

The variance assertion in (113) follows from

$$a_n^2 N^{-1} = p_{l,n}^{-1}.$$

Direct Schur-complement inversion gives (114). In the conjugate-transposed orientation, conjugation sends $h_1 + \mathrm{i}h_2$ to $h_1 - \mathrm{i}h_2$, changing only the sign of the second real insertion direction. Multiplication by $a_n$ gives the stated speed.

Use $\widetilde{x} = a_n^2 x$ in the left side of (115). The first identity in (114) and

$$\mathrm{d}\widetilde{x} = a_n^2\, \mathrm{d}x$$

cancel the two scaling factors and prove (115). Since

$$a_n^2 = \frac{N}{p_{l,n}},$$

the first two identities in (116) follow from the definitions of $A_n$ and $B_n$. The identity

$$\widehat{\eta}_{l,n} = p_{l,n}^{1/3-12\kappa_0}$$

gives the final relation. The fresh dimensions are uniformly comparable with $N_l$, so $a_n$ and $a_n^{-1}$ are uniformly bounded. □

J.4. **Approximation of Gaussian left-tail events.**

*Proof of 6.1.* Fix $D > 0$. We first prove a uniform one-matrix consequence of Proposition 4.1. Put

$$K = \frac{C_0}{c_0}.$$

Let $p, q$ be integers satisfying $p \le q \le Kp$, let $E \in \mathbb{R}$ satisfy

$$\left| \frac{E}{p} - \left(1 + \sqrt{\frac{q}{p}}\right)^2 \right| \le C_E p^{-2/3} (\log p)^{1/3},$$

and let $Z$ be a $q$-by-$p$ standard circular complex Gaussian matrix. If $\lambda_s$ are the eigenvalues of $Z^*Z$, define

$$\eta = p^{1/3-12\kappa}, \qquad \ell = p^{1/3-6\kappa},$$
$$a = E + \ell, \qquad B = (\sqrt{p} + \sqrt{q})^2 + 4p^{1/3+\kappa},$$

and

$$Q = \sum_s \frac{1}{\pi} \int_a^B \frac{\eta}{(\lambda_s - u)^2 + \eta^2} \, \mathrm{d}u.$$

Define

$$\mathrm{Bad}(p, q, E) = \left\{ \lambda_{\max}(Z^*Z) \le E, \ Q > p^{-\kappa} \right\}.$$

We claim that, for all sufficiently large $p$, uniformly over the admissible $q$ and $E$,

$$\mathbb{P}(\mathrm{Bad}(p, q, E)) \le p^{-D}. \tag{128}$$

To identify the normalized statistic, set

$$X = p^{-1/2} Z, \qquad L = X^*X = p^{-1} Z^*Z.$$

Use matrix-size index $N = p$, row size $M_N = q$, normalized threshold $e = E/p$, and

$$\eta_N = p^{-2/3-12\kappa}, \qquad \ell_N = p^{-2/3-6\kappa}, \qquad B_N = \left(1 + \sqrt{\frac{q}{p}}\right)^2 + 4p^{-2/3+\kappa}.$$

If

$$m_N(z) = p^{-1} \operatorname{Tr}((L - zI_p)^{-1}),$$

then

$$\Im m_N(y + \mathrm{i}\eta_N) = \sum_s \frac{p\eta_N}{(\lambda_s - py)^2 + (p\eta_N)^2}.$$

Consequently, after the change of variables $u = py$,

$$\begin{aligned} \frac{p}{\pi} \int_{e+\ell_N}^{B_N} \Im m_N(y + \mathrm{i}\eta_N) \, \mathrm{d}y &= \sum_s \frac{1}{\pi} \int_{E+p\ell_N}^{pB_N} \frac{p\eta_N}{(\lambda_s - u)^2 + (p\eta_N)^2} \, \mathrm{d}u \\ &= Q, \end{aligned} \tag{129}$$

because $p\eta_N = \eta$, $p\ell_N = \ell$, and

$$pB_N = (\sqrt{p} + \sqrt{q})^2 + 4p^{1/3+\kappa}.$$

It remains to justify the uniformity in (128), since the source dichotomy is stated for dimension sequences whose aspect ratios converge. Suppose that (128) were false. Then there would be a sequence $p_r \to \infty$, integers

$$p_r \le q_r \le K p_r,$$

and admissible thresholds $E_r$ such that

$$\mathbb{P}(\mathrm{Bad}(p_r, q_r, E_r)) > p_r^{-D}.$$

After passing to a subsequence, we may assume that the $p_r$ are strictly increasing and

$$\frac{q_r}{p_r} \longrightarrow \gamma$$

for some $\gamma \in [1, K]$.

Construct positive integers $M_N$ such that $M_N \ge N$ for every sufficiently large $N$,

$$\frac{M_N}{N} \longrightarrow \gamma, \qquad M_{p_r} = q_r.$$

At the special indices $p_r$, use $M_{p_r} = q_r$, and at all remaining indices use

$$M_N = \max\{N, \lfloor \gamma N \rfloor\}.$$

Def ne deterministic normalized thresholds by

$$e_{p_r} = \frac{E_r}{p_r}$$

at the special indices and

$$e_N = \left(1 + \sqrt{\frac{M_N}{N}}\right)^2$$

elsewhere. These thresholds satisfy the required logarithmic window condition with the same constant $C_E$.

Let $X_N$ be an $M_N$-by-$N$ matrix of independent centered circular complex Gaussian variables of second absolute moment $N^{-1}$. Proposition 4.1 applies to this sequence.

For a standard circular Gaussian matrix, the event that a prescribed real number $E$ is an eigenvalue of $Z^*Z$ has probability zero. If $E \ge 0$, the equality event lies in the zero set of the nonzero real polynomial

$$\det(Z^*Z - EI_p)$$

in the real and imaginary entry coordinates. The Gaussian law has a Lebesgue density, so this zero set has probability zero. If $E < 0$, equality is impossible because $Z^*Z$ is positive semidefinite.

The strict inequality supplied by the dichotomy, together with its exceptional event of probability at most $N^{-D}$, equation (129), and deletion of the null equality event, therefore yields

$$\mathbb{P}(\mathrm{Bad}(p_r, q_r, E_r)) \le p_r^{-D}$$

for every sufficiently large $r$. This contradicts the choice of the sequence and proves (128).

Return to the arrays in the statement. Since $p_{l,i} \ge c_0 N_l$ and $N_l \to \infty$, estimate (128) applies simultaneously to every $i \in I_l$ for all sufficiently large $l$. Define

$$\Omega_{l,i} = \mathrm{Bad}(p_{l,i}, q_{l,i}, E_{l,i})^c,$$

using $Z_{l,i}$ in the definition of the bad event. The event $\Omega_{l,i}$ is measurable with respect to $Z_{l,i}$, and (128) gives

$$\mathbb{P}(\Omega_{l,i}^c) \le p_{l,i}^{-D}.$$

If $H_{l,i}\mathbf{1}_{\Omega_{l,i}} = 1$, then the defining bad conjunction cannot hold. Hence

$$Q_{l,i} \le p_{l,i}^{-\kappa}.$$

This proves (135)–(136) for the given, possibly dependent, matrix family.

Now assume $D > \theta + 2$. Pointwise,

$$0 \le R_l - \widetilde{R}_l = \sum_{i \in I_l} H_{l,i}(1 - \mathbf{1}_{\Omega_{l,i}}) \le \sum_{i \in I_l} \mathbf{1}_{\Omega_{l,i}^c}.$$

Therefore

$$0 \le m_l - \widetilde{m}_l \le \sum_{i \in I_l} p_{l,i}^{-D} \le c_0^{-D} N_l^{\theta - D} = O(N_l^{-2}). \tag{130}$$

The final exponent is at most $-2$ because $D > \theta + 2$. Since $m_l \ge N_l^\alpha$ with $\alpha > 0$, the right side of (130) is $o(m_l)$. Thus

$$\widetilde{m}_l = (1 + o(1)) m_l \ge N_l^\alpha (1 - o(1)).$$

Moreover,

$$0 \le \widetilde{R}_l \le R_l$$

pointwise, and hence

$$\mathbb{E}[\widetilde{R}_l^2] \le \mathbb{E}[R_l^2] \le C_R m_l^2 = (C_R + o(1)) \widetilde{m}_l^2.$$

This proves (137)–(138).

Finally, suppose that the entire Gaussian family at level $l$ is independent of $\mathcal{G}_l$. Each $\Omega_{l,i}$ is a measurable function of $Z_{l,i}$ and deterministic parameters. Therefore $R_l$ and $\widetilde{R}_l$ are independent of $\mathcal{G}_l$. Their conditional first and second moments given $\mathcal{G}_l$ equal their unconditional moments almost surely. The same calculations prove (137)–(138) conditionally on $\mathcal{G}_l$. □

## J.5. Gaussian left-tail counts.

*Proof of 6.2.* We apply the same sparse-grid moment estimate to the specified high and low fresh-corner thresholds, after recording their exact bounded shifts.

Suppose $k_{l-1}$ is fixed and put $P = 2^{k_{l-1}+1}$. Define

$$A_P(j) = \max\{1, M_{P+j} - M_P\}, \qquad j \ge 1.$$

The path $A_P$ has bounded increments and $A_P(j)/j \to \gamma$. With $P$ and $A_P$ fixed, choose $k_l \ge k_{l-1} + 2$ sufficiently large that, for $N_l = 2^{k_l}$ and $J_l = N_l - P$,

$$P_l = P \le (\log N_l)^d, \qquad A_P(j) = M_{P+j} - M_P \quad (j \in [J_l, 2J_l)),$$

and Corollary 3.8 applies to this path at scale $J_l$ for both parameter triples. The lower threshold on $J_l$ may depend on the fixed prefix $P$.

Fix $l$, write $P = P_l$, $J = J_l$, and $A_j = \widehat{M_l}(j)$. Uniformly for $j \in [J, 2J)$, $A_j/j = \gamma + o(1)$, $M_P + P = O(P)$, and the project scale $b(A_j, j) \asymp j^{1/3}$. The half-integer center and scale in the statement therefore def ine an exact high-threshold shift

$$h_{l,j,+} = \frac{\mu_{l,j} - a(A_j, j)}{b(A_j, j)} + a(\log j)^{2/3} \left( \frac{\sigma_{l,j}}{b(A_j, j)} - 1 \right) = o(1)$$

uniformly, and

$$\{\|Z_l^{(j)}\|_{\mathrm{op}}^2 \ge \mu_{l,j} + a(\log j)^{2/3} \sigma_{l,j}\} = \{\widehat{\chi}_{l,j} \ge a(\log j)^{2/3} + h_{l,j,+}\}.$$

For the low threshold, put $n = P + j$. Smoothness and homogeneity of the center and scale, together with $M_P + P = O(P) = o(j)$, give

$$0 \le e_n - a(A_j, j) = O(P), \qquad \frac{d_n}{b(A_j, j)} = 1 + O(P/j).$$

Consequently the exact shift

$$h_{l,j,-} = \frac{e_n - a(A_j, j)}{b(A_j, j)} - b_0 \left[ \frac{d_n}{b(A_j, j)} (\log n)^{1/3} - (\log j)^{1/3} \right] - \frac{2\widehat{\ell}_{l,j}}{b(A_j, j)}$$

also satisfies $h_{l,j,-} = o(1)$ uniformly, and

$$\{\|Z_l^{(j)}\|_{\mathrm{op}}^2 \le t_{l,j}^\sharp\} = \{\widehat{\chi}_{l,j} \le -b_0 (\log j)^{1/3} + h_{l,j,-}\}.$$

Both shift families are eventually bounded by one. Apply Corollary 3.8 with $H_{\mathrm{shift}} = 1$, first at $(a, \theta_+, \epsilon_+)$ and then at $(b_0, \theta_-, \epsilon_-)$. Its grids, cardinalities, first moments, and second moments give all displayed assertions after replacing powers of $J_l$ by powers of $N_l$, since $J_l/N_l \to 1$. Independence of the entire array $Z_l$ from $\mathcal{F}_l^{\mathrm{past}}$ invokes the conditional clause of the corollary. Conditional Chebyshev applied to

$$\mathrm{Var}(R_l^\pm \mid \mathcal{F}_l^{\mathrm{past}}) \le o(1)(m_l^\pm)^2 + m_l^\pm$$

yields the two concentration statements. Translating either grid by $P_l$ gives the physical grids in $[N_l, 2N_l)$, without changing their cardinality or separation. □

J.6. **Smooth cutoff functions for the left tail.**

*Proof of 6.3.* **Step 1. Fresh grid and geometric comparison.** The high-side numerical condition in Proposition 6.2 holds for the fixed auxiliary parameters, while the low-side condition is

$$c + \eta + \epsilon < \frac{1}{3}$$

because $c = b^3/12$. The proposition gives the scales, arrays, and grids in the statement. Translation by $P_l$ gives $T_l \subset [N_l, 2N_l)$, the asserted separation, and

$$|T_l| = N_l^{c+\eta+o(1)}.$$

It also gives

$$m_l^\sharp = \mathbb{E}[R_l^\sharp \mid \mathcal{F}_l^{\mathrm{past}}] = N_l^{\eta+o(1)} \ge N_l^{\eta/2}$$

for all sufficiently large $l$, together with

$$\mathbb{E}[(R_l^\sharp)^2 \mid \mathcal{F}_l^{\mathrm{past}}] \le (1 + o(1))(m_l^\sharp)^2 + m_l^\sharp = (1 + o(1))(m_l^\sharp)^2.$$

In particular, $\mathrm{Var}(R_l^\sharp \mid \mathcal{F}_l^{\mathrm{past}}) = o((m_l^\sharp)^2)$, with a deterministic error.

Fix $l$ and $n \in T_l$, and abbreviate

$$j = j_{l,n}, \quad u = u_{l,n}, \quad v = v_{l,n}, \quad p = p_{l,n}, \quad q = q_{l,n}.$$

Since $J_l \le j < 2J_l$, $P_l \le (\log N_l)^d$, and $J_l/N_l \to 1$, we have

$$n = P_l + j \in [N_l, 2N_l).$$

The convergence $M_n/n \to \gamma$, uniform for $n \ge N_l$, and

$$M_{P_l} = O(P_l)$$

imply that $u, v, p, q$ are bounded above and below by fixed positive multiples of $N_l$, uniformly for $n \in T_l$.

Put

$$e_{l,n}^{\mathrm{core}} = (\sqrt{u} + \sqrt{v})^2.$$

Since

$$M_n = u + M_{P_l}, \qquad n = v + P_l,$$

we have

$$e_n - e_{l,n}^{\mathrm{core}} = M_{P_l} + P_l + 2\left[\sqrt{(u + M_{P_l})(v + P_l)} - \sqrt{uv}\right].$$

For positive $u, v, A, B$,

$$\sqrt{(u + A)(v + B)} - \sqrt{uv} = \frac{uB + vA + AB}{\sqrt{(u + A)(v + B)} + \sqrt{uv}}.$$

Here $u, v \asymp N_l$ and

$$A + B = M_{P_l} + P_l = O(P_l),$$

so the quotient is $O(P_l)$. Consequently,

$$|e_n - e_{l,n}^{\mathrm{core}}| \le C P_l.$$

Also,

$$d_n \le C N_l^{1/3}$$

uniformly on the grid. Hence

$$|t_{l,n}^X - e_{l,n}^{\mathrm{core}}| \le C\left[P_l + N_l^{1/3}(\log N_l)^{1/3}\right]. \tag{131}$$

Because $p \asymp N_l$, the excess $4p^{1/3+\kappa}$ in $B_{l,n}$ dominates the right side of (131). Therefore

$$t_{l,n}^X \le B_{l,n}$$

for all sufficiently large $l$. Moreover,

$$t_{l,n}^X - a_{l,n} = p^{1/3-6\kappa}, \qquad \widehat{\eta}_{l,n} = p^{1/3-12\kappa},$$

so

$$a_{l,n} < t_{l,n}^X, \qquad t_{l,n}^X - a_{l,n} \ge 4\widehat{\eta}_{l,n}$$

eventually. This proves (150).

**Step 2. Poisson retention and terminal activation.** Apply Proposition 6.1 with index set $I_l = T_l$, matrix dimensions $(p_{l,n}, q_{l,n})$, and hard threshold

$$E_{l,n} = t_{l,n}^{\sharp} = t_{l,n}^X - 2p_{l,n}^{1/3-6\kappa}.$$

When the physical fresh rectangle has the opposite orientation, use its conjugate transpose as the $q_{l,n}$-by-$p_{l,n}$ input. Conjugate transposition preserves every squared singular value, hard indicator, Poisson count, dependency among matrices, and independence from $\mathcal{F}_l^{\mathrm{past}}$.

The Poisson lower endpoint is

$$E_{l,n} + p_{l,n}^{1/3-6\kappa} = a_{l,n},$$

and its width and upper endpoint are exactly

$$\widehat{\eta}_{l,n} \quad \text{and} \quad B_{l,n}.$$

Since

$$\frac{e_{l,n}^{\mathrm{core}}}{p} = \left(1 + \sqrt{\frac{q}{p}}\right)^2,$$

equation (131), subtraction of $2p^{1/3-6\kappa}$, and comparability of $p$ and $N_l$ give

$$\left|\frac{E_{l,n}}{p} - \left(1 + \sqrt{\frac{q}{p}}\right)^2\right| \le C_E p^{-2/3}(\log p)^{1/3}$$

uniformly on the grid. Thus the threshold-window hypothesis of Proposition 6.1 holds.

Fix an arbitrary prescribed $D > 0$. Choose a fixed $\theta > c + \eta$. Then

$$|T_l| \leq N_l^{\theta}$$

for all sufficiently large $l$. The moment estimates from Step 1 give the activation proposition's hypotheses with

$$\alpha = \eta/2 \qquad \text{and} \qquad C_R = 2$$

after discarding a finite initial segment. The entire Gaussian family is independent of $\mathcal{F}_l^{\text{past}}$. Proposition 6.1 therefore supplies events $\Omega_{l,n}$, measurable in the corresponding fresh Gaussian matrices and independent of $\mathcal{F}_l^{\text{past}}$, such that

$$\mathbb{P}(\Omega_{l,n}^c) \leq p_{l,n}^{-D}$$

and

$$H_{l,n}^{\sharp} \mathbf{1}_{\Omega_{l,n}} = 1 \implies Q_{l,n} \leq p_{l,n}^{-\kappa}. \tag{132}$$

This proves the arbitrary-$D$ event assertion.

For $D > c + \eta + 2$, choose $c + \eta < \theta < D - 2$ and write $m_l = \mathbb{E}[\widetilde{R}_l \mid \mathcal{F}_l^{\text{past}}]$. The mean loss is at most $CN_l^{\theta - D}$, so

$$q_l := \frac{m_l^{\sharp} - m_l}{m_l^{\sharp}} \longrightarrow 0, \qquad m_l = (1 + o(1))m_l^{\sharp} \geq N_l^{\eta/2}(1 - o(1)).$$

Step 1 gives $\mathbb{E}[(R_l^{\sharp})^2 \mid \mathcal{F}_l^{\text{past}}] \leq (1 + u_l)(m_l^{\sharp})^2 + m_l^{\sharp}$ with $u_l \to 0$. The retention estimate following Lemma 2.6 yields

$$\frac{\text{Var}(\widetilde{R}_l \mid \mathcal{F}_l^{\text{past}})}{m_l^2} \leq \frac{1 + u_l + (m_l^{\sharp})^{-1}}{(1 - q_l)^2} - 1 = o(1).$$

This proves (154) and (155), retaining the strong Gaussian second moment.

It remains to pass from the fresh Poisson count in (132) to the full block-diagonal terminal statistic. On $\mathcal{G}_l^{\text{past}}$, every squared singular value of $X_l^{\text{past}}$ is at most $N_l/4$. Since $e_n \geq n \geq N_l$ and the logarithmic correction in $t_{l,n}^X$ is $o(N_l)$, for all sufficiently large $l$,

$$t_{l,n}^X \geq \frac{3N_l}{4}, \qquad a_{l,n} \geq \frac{N_l}{2}.$$

Equation (131) also gives

$$B_{l,n} - a_{l,n} \leq CN_l^{1/3+\kappa}. \tag{133}$$

For every $x \in [0, N_l/4]$, the distance from $x$ to the interval of integration is bounded below by a fixed multiple of $N_l$. Hence

$$\begin{aligned} \frac{1}{\pi} \int_{a_{l,n}}^{B_{l,n}} \frac{\widehat{\eta}_{l,n}}{(x-u)^2 + \widehat{\eta}_{l,n}^2} \, du &\leq C \frac{\widehat{\eta}_{l,n}(B_{l,n} - a_{l,n})}{N_l^2} \\ &\leq CN_l^{-4/3-11\kappa}. \end{aligned} \tag{134}$$

The smooth upper statistic vanishes on these values because they lie below $B_{l,n}$. On a retained hard hit, every submatrix squared singular value is at most

$$t_{l,n}^{\sharp} < a_{l,n},$$

so its smooth upper statistic also vanishes, and (132) controls its entire Poisson contribution.

The nonzero squared singular values of $W^T_{l,n}$ are the union of those of the past and fresh blocks. Depending on the rectangular orientations, the list of $\min\{M_n, n\}$ singular values can contain additional zeros. The past block contributes at most $P_l$ nonzero values, and the number of all additional zeros is at most

$$M_n + n \le CN_l.$$

Every such value satisfies (134). Therefore all full-statistic contributions outside $Q_{l,n}$ total at most

$$C(P_l + N_l)N_l^{-4/3-11\kappa} \le CN_l^{-1/3-11\kappa}. \tag{135}$$

Combining (132) and (135), and using $p_{l,n} \asymp N_l$, gives a deterministic constant $C_{\mathrm{act}}$, independent of $l, n$, such that

$$Z_{l,n}(W^T_{l,n}) \le C_{\mathrm{act}} N_l^{-\kappa}$$

on $\mathcal{G}^{\mathrm{past}}_l$ and every retained hit.

Since

$$\delta_l^{-1} = 16(J_{\mathrm{c},l} + 2) = O(\log N_l),$$

we have

$$C_{\mathrm{act}} N_l^{-\kappa} \le \delta_l$$

eventually. The terminal activation conclusion of Proposition 3.2 gives

$$\widetilde{H}_{l,n} = 1 \quad \Longrightarrow \quad y_{l,n}(W^T_{l,n}) \le -J_{\mathrm{c},l} - 1$$

for every retained hit.

**Step 3. Target support and summable good histories.** At the target endpoint, put $W^X_{l,n} = X^{(n)}$. The exact target-support property of the statistic says that $y_{l,n}(W^X_{l,n}) < 0$ only if

$$\|X^{(n)}\|^2_{\mathrm{op}} < t^X_{l,n} = e_n - bd_n(\log n)^{1/3}.$$

Thus

$$\chi^+_n < -b(\log n)^{1/3},$$

or equivalently,

$$\frac{\chi^+_n}{s_n} < -b.$$

Thus the target-support hypothesis in Lemma E.3 holds with $\mathcal{L} = L_l(b)$. Apply that lemma with

$$\mathcal{G} = \mathcal{F}^{\mathrm{past}}_l, \quad \mathcal{H} = \mathcal{G}^{\mathrm{past}}_l, \quad I = T_l, \quad J = J_{\mathrm{c},l}, \quad R = R_{\mathrm{c}}.$$

The retained-count moment package is (154)– (155). Take the lemma's final-layer parameter to be the specified quantity

$$\varepsilon_{\mathrm{fin},l} = |T_l| R_{\mathrm{c}}^{-J_{\mathrm{c},l}};$$

For the prescribed $\zeta > 0$, the lemma now proves (156) and (157) and preserves the retained count's conditional moments.

Finally,

$$\|X^{\mathrm{past}}_l\|_{\mathrm{op}} \le \|X^{\mathrm{past}}_l\|_{\mathrm{HS}},$$

and the entries have second absolute moment one. Since

$$M_{P_l} \le CP_l$$

for all sufficiently large $l$,

$$\mathbb{E}\|X_l^{\text{past}}\|_{\text{HS}}^2 = M_{P_l} P_l \le CP_l^2.$$

Markov's inequality gives

$$\mathbb{P}((\mathcal{G}_l^{\text{past}})^c) \le \frac{4CP_l^2}{N_l}.$$

The $N_l$ are distinct increasing powers of two and $P_l \le (\log N_l)^d$. Therefore these probabilities are summable. □

## Appendix K. Dyadic tail bounds

### K.1. Right-tail bound.

The right-tail assertion is proved in Proposition 8.2 of the main article, using the fixed-size tail bound, the upper-excursion estimate between grid points and the right-endpoint interpolation in Lemma 8.1.

### K.2. Left-tail bound.

The same proposition gives the left-tail assertion. Its proof uses the lower-excursion estimate and left-endpoint interpolation, with the $(\log N)^{1/3}$ normalization.

School of Science, Nanjing University of Posts and Telecommunications, Nanjing 210023, China
*Email address*: `B25100020@njupt.edu.cn`

